\documentclass[11pt]{amsart}

\title{Fully exact and fully dualizable module categories}
\date{\today}

\author{Azat M. Gainutdinov}
\address{Institut Denis Poisson, CNRS, Université de Tours, Parc de Grandmont, 37200 Tours, France}
\email{azat.gainutdinov@cnrs.fr}

\author{Robert Laugwitz}
\address{School of Mathematical Sciences,
University of Nottingham, University Park, Nottingham, NG7 2RD, UK}
\email{robert.laugwitz@nottingham.ac.uk}

\usepackage{amsmath,amsfonts,amsthm,amssymb}
\usepackage[shortalphabetic,nobysame,abbrev
]{amsrefs}
\usepackage[english]{babel}
\usepackage{url}
\usepackage{fancyhdr}
\usepackage{graphicx}
\usepackage{microtype}
\usepackage[
colorlinks=true,
linkcolor=black, 
anchorcolor=black,
citecolor=black, 
urlcolor=black, 
]{hyperref}
\usepackage[sort]{cleveref} 
\usepackage{xcolor}
\usepackage[margin=2.5cm]{geometry}
\usepackage[all]{xy}
\usepackage{dsfont}
\usepackage[shortlabels]{enumitem}
\usepackage{pifont}
\usepackage{tikz-cd}
\usetikzlibrary{calc}
\usetikzlibrary{arrows}

\makeatletter
\newcommand{\superimpose}[2]{%
  {\ooalign{$#1\@firstoftwo#2$\cr\hfil$#1\@secondoftwo#2$\hfil\cr}}}
\makeatother

\newcommand{\longrightleftarrows}{\resizebox{18pt}{7pt}{$\rightleftarrows$}}

\newcommand{\adj}[4]{#1\colon #2~\longrightleftarrows~ #3\colon #4}

\newcommand{\op}[1]{\operatorname{#1}}
\newcommand{\oop}{\mathrm{op}}

\newcommand{\bimod}[1]{#1\text{-}\mathbf{bimod}}
\newcommand{\bimodint}[3]{#1\text{-}\mathbf{bimod}_{#2}\text{-}#3}
\newcommand{\bimodints}[2]{#1\text{-}\mathbf{bimod}_{#2}}
\newcommand{\Cat}{\mathbf{Cat}}

\newcommand{\lmod}[1]{#1\text{-}\mathbf{mod}}
\newcommand{\rmod}[1]{\mathbf{mod}\text{-}#1}
\newcommand{\lMod}[1]{#1\text{-}\mathbf{Mod}}
\newcommand{\Mod}{\mathbf{mod}}

\newcommand{\lemod}[1]{#1\text{-}\mathbf{ex}}

\newcommand{\fexmod}[1]{#1\text{-}\mathbf{fex}}
\newcommand{\sepmod}[1]{#1\text{-}\mathbf{sep}}
\newcommand{\invmod}[1]{#1\text{-}\mathbf{inv}}

\newcommand{\perf}[1]{#1\text{-}\mathbf{perf}}
\newcommand{\opfexmod}[1]{#1\text{-}\mathbf{opfex}}
\newcommand{\lmodint}[2]{#1\text{-}\mathbf{mod}_{#2}}

\newcommand{\rmodint}[2]{\mathbf{mod}_{#1}\text{-}#2}

\newcommand{\lcomod}[1]{#1\text{-}\mathbf{comod} }

\newcommand{\lact}{\triangleright}
\newcommand{\ract}{\triangleleft}

\newcommand{\inner}[1]{\langle#1\rangle}

\newcommand{\rev}{\mathrm{rev}}

\newcommand{\cha}{\operatorname{char}}
\newcommand{\coev}{\operatorname{coev}}

\newcommand{\Drin}{D}
\newcommand{\ev}{\operatorname{ev}}

\newcommand{\End}{\operatorname{End}}
\newcommand{\Fun}{\operatorname{Fun}}

\newcommand{\Hom}{\operatorname{Hom}}
\newcommand{\iHom}{\underline{\operatorname{Hom}}}
\newcommand{\iEnd}{\underline{\operatorname{End}}}

\newcommand\id{{\operatorname{id}}}
\newcommand\Img{\operatorname{Im}}

\newcommand{\isomorph}{\stackrel{\sim}{\longrightarrow}}

\newcommand{\Mor}{\mathbf{Mor}}

\newcommand{\one}{\mathds{1}}

\newcommand{\Pic}{\operatorname{Pic}}
\newcommand{\bPic}{\mathbf{Pic}}

\newcommand{\Res}{\operatorname{Res}}

\newcommand{\Set}[1]{\left\lbrace #1\right\rbrace}

\newcommand{\Vect}{\mathbf{vect}}
\newcommand{\sVect}{\mathbf{svect}}

\newcommand{\action}{centralizer }

\newcommand{\Btwist}[3]{{}^{#1}{#2}^{#3}}
\newcommand{\Dtwist}[4]{{}^{\vphantom{+1}}_{#1}{#2}^{\vphantom{+1}#4}_{#3}}
\newcommand{\sF}{\check{F}}
\newcommand{\smu}{\check{\mu}}
\newcommand{\sR}{\check{R}}

\providecommand{\fr}[1]{\mathfrak{#1}}

\providecommand{\op}[1]{\operatorname{#1}}

\newcommand{\mC}{\mathbb{C}}
\newcommand{\mD}{\mathbb{D}}

\newcommand{\mZ}{\mathbb{Z}}

\newcommand{\cC}{\mathcal{C}}
\newcommand{\otimesop}{\otimes\oop}
\newcommand{\cCop}{\cC^{\otimesop}}
\newcommand{\cCrev}{\cC^{\mathrm{rev}}}
\newcommand{\cD}{\mathcal{D}}

\newcommand{\cO}{\mathcal{O}}

\newcommand{\cV}{\mathcal{V}}

\newcommand{\cM}{\mathcal{M}}
\newcommand{\cN}{\mathcal{N}}
\newcommand{\cQ}{\mathcal{Q}}
\newcommand{\cP}{\mathcal{P}}
\newcommand{\cR}{\mathcal{R}}
\newcommand{\cS}{\mathcal{S}}

\newcommand{\cZ}{\mathcal{Z}}

\newcommand{\bfC}{\mathbf{C}}
\newcommand{\bfD}{\mathbf{D}}
\newcommand{\bfOne}{\mathbf{1}}

\newcommand{\rD}{\mathrm{D}}
\newcommand{\rP}{\mathrm{P}}
\newcommand{\rR}{\mathrm{R}}
\newcommand{\rS}{\mathrm{S}}
\newcommand{\rB}{\mathrm{B}}

\newcommand{\qbin}[2]{\mathchoice%
  {\qbinm{#1}{#2}}{\qbinmm{#1}{#2}}%
  {\qbinmm{#1}{#2}}{\qbinmm{#1}{#2}}}
\newcommand{\qbinm}[2]{\mbox{\footnotesize$\displaystyle
    \genfrac{[}{]}{0pt}{}{#1}{#2}$}}
\newcommand{\qbinmm}[2]{\genfrac{[}{]}{0pt}{}{#1}{#2}}

\Crefname{theorem}{Thm.\!}{Thms.\!}
\Crefname{lemma}{Lem.\!}{Lems.\!}
\Crefname{proposition}{Prop.\!}{Props.\!}
\Crefname{corollary}{Cor.\!}{Cors.\!}
\Crefname{definition}{Def.\!}{Defs.\!}
\Crefname{section}{Sec.\!}{Secs.\!}
\Crefname{exercise}{Ex.\!}{Exs.\!}
\Crefname{example}{Ex.\!}{Exs.\!}
\Crefname{remark}{Rem.\!}{Rems.\!}

\newtheoremstyle{mystyle}
  {0.5cm}                   
  {0.5cm}                   
  {\normalfont}           
  {}                      
  {\itfont\bfseries} 
  {:}                     
  {0.3cm}              
  {\thmname{#1}}

\newtheoremstyle{defstyle}
  {0.5cm}                   
  {0.5cm}                   
  {\normalfont}           
  {}     
  {\normalfont\bfseries}  
  {:}                     
  {0.3cm}              
  {\thmname{#1}\thmnumber{ #2}\thmnote{ (#3)}}

\newtheorem*{rep@theorem}{\rep@title}
\newcommand{\newreptheorem}[2]{%
\newenvironment{rep#1}[1]{%
 \def\rep@title{#2 \ref{##1}}%
 \begin{rep@theorem}}%
 {\end{rep@theorem}}}
\makeatother

\newtheorem{theorem}{Theorem}[section]

\newtheorem{proposition}[theorem]{Proposition}
\newreptheorem{proposition}{Proposition}
\newtheorem{corollary}[theorem]{Corollary}
\newreptheorem{corollary}{Corollary}
\newtheorem{lemma}[theorem]{Lemma}
\newtheorem{conjecture}[theorem]{Conjecture}

\newtheorem*{theorem*}{Theorem}
\newreptheorem{theorem}{Theorem}

\theoremstyle{definition}
\newtheorem{definition}[theorem]{Definition}

\theoremstyle{remark}
\newtheorem{example}[theorem]{Example}
\newtheorem{remark}[theorem]{Remark}

\newtheorem{question}[theorem]{Question}

\theoremstyle{plain}
\newtheorem{introtheorem}{Theorem}

\newtheorem{introproposition}[introtheorem]{Proposition}

\numberwithin{equation}{section}

\newcommand{\cmark}{\ding{51}}%
\newcommand{\xmark}{\ding{55}}%


\begin{document}

\begin{abstract}
Let $\cC$ be a finite braided tensor category over a field $\Bbbk$. We define \emph{fully exact} $\cC$-module categories, a subclass of exact $\cC$-module categories that is stable under the relative Deligne product over~$\cC$. 
In contrast, we demonstrate with examples in both zero and non-zero characteristic of $\Bbbk$ that  the class of exact $\cC$-module categories is not stable under this product.  
We also observe in examples
that fully exact module categories form a dense subset in the class of exact ones.

The monoidal 2-category of fully exact $\cC$-module categories strictly contains those of invertible and separable $\cC$-module categories. 
In fact, we show that each internal algebra of a fully exact $\cC$-module category
is projectively separable, a generalization of separable algebras involving projective objects of $\cC$.
 If $\cC$ is semisimple, a $\cC$-module category is fully exact  if and only if it is separable.

In general, fully exact $\cC$-module categories are not dualizable inside their class, but if they are, they are fully dualizable objects in the monoidal 2-category $\lmod{\cC}$ of finite $\cC$-module categories. We call such module categories \emph{perfect}.
We show that perfect $\cC$-module categories form a rigid monoidal 2-subcategory $\perf{\cC}$ of $\lmod{\cC}$ containing all  fully dualizable objects. Therefore, we propose $\perf{\cC}$ as a model for finite tensor $2$-categories.
If $\cC$ is symmetric, $\cC$-module categories are fully exact if and only if they are perfect.

As a detailed example, we classify fully exact, and hence perfect, module categories over the symmetric tensor category of modules over Sweedler's four-dimensional Hopf algebra and compute their relative Deligne products, and the categories of $1$-morphisms.  In this case, $\perf{\cC}$ has infinitely  many 
(a continuum of) isoclasses of indecomposable objects, 
with non-semisimple finite tensor categories of $1$-endomorphisms, and only finitely many isoclasses of indecomposable objects correspond to separable module categories. For a general quasi-triangular Hopf algebra, we analyze when the category of finite-dimensional vector spaces is fully exact. We show that this is not the case for both Sweedler's Hopf algebra  and Lusztig's factorizable  small quantum group of $\mathfrak{sl}_2$ at an odd root of unity. 
We conjecture a similar statement for small quantum groups of any simple Lie algebra.
\end{abstract}

\maketitle

\setcounter{tocdepth}{1}
\tableofcontents

\section{Introduction}

Tensor categories have a wealth of applications in topology, particularly through the construction of  invariants of links, $3$D manifolds, and $3$D topological field theories (TFTs) \cites{Tur,BK,TV,GPT,CGP,DGGPR}. Monoidal $2$-categories with appropriate dualizability conditions are expected to extend these constructions to dimension $4$ \cites{CF,Mac1,Mac2,Cui}. A special class of semisimple monoidal $2$-categories, called \emph{fusion $2$-categories}, has been introduced in \cite{DR} and used to construct state-sum invariants of $4$D manifolds. Fusion $2$-categories satisfy the strong dualizability properties that objects have left and right duals and all $1$-morphisms have left and right adjoints. The study of fusion $2$-categories has gained significant attention in the literature, see e.g. \cites{DR,KTZ,JR,DHJNPPRY}. 

It was, however, observed in \cite{Reu} that $4$D manifold invariants obtained from semisimple $2$-categories cannot detect exotic smooth structures. Similarly, for $3$D manifolds, invariants constructed from fusion $1$-categories, i.e.\,finite semisimple tensor categories, cannot distinguish certain homotopically equivalent $3$D manifolds (e.g.\ certain lens spaces). However, more recent constructions of invariants from \emph{non-semisimple} tensor categories can distinguish such manifolds \cites{CGP,BGR}. This motivates the search for examples of non-semisimple monoidal $2$-categories that still satisfy strong finiteness and dualizablity conditions similar to those of fusion $2$-categories. 

A large source  of fusion $2$-categories is provided by monoidal $2$-categories of finite \textsl{semisimple} module categories over a braided fusion category. These fusion $2$-categories are sufficient to characterize connected fusion $2$-categories up to equivalence \cites{Dec3,DecS} and general fusion $2$-categories up to Morita equivalence \cite{Dec1}. Beyond the semisimple case, to our knowledge, no examples of monoidal $2$-categories with strong dualizability properties, similar to fusion $2$-categories, appear to be known.

In this paper, we construct a new class of non-semisimple monoidal $2$-categories based on the notions of \emph{fully exact} and \emph{perfect} module categories introduced here. These are subclasses of the class of \emph{exact} module categories of \cites{EO1,EGNO}. In these monoidal $2$-categories, all Hom categories have left and right duals, i.e.\ all $1$-morphisms are dualizable, while the perfect module categories, moreover, have dual objects, and are thus fully dualizable. We provide examples where we can classify fully exact and perfect module categories.

\subsection{Tensor categories and their module categories}

The study of finite module categories over finite tensor categories is a natural categorical analogue of the classical theory of finite-dimensional modules over finite-dimensional $\Bbbk$-algebras \cites{Ost1,EO1,ENO,ENO2,EGNO}. This theory is best understood when $\cC$ is a fusion category, where it is natural to consider semisimple module categories. However, several tensor categories of particular interest, such as representations of quantum groups at roots of unity, are not semisimple. 

It was shown in \cites{Ost1,EO1,EGNO,DSS1} that classifying finite module categories is equivalent to classifying algebras in $\cC$ up to Morita equivalence. Therefore, classifying all finite $\cC$-module categories is a wild problem. When $\cC=\Vect$, this corresponds to the problem of classifying all finite-dimensional $\Bbbk$-algebras up to Morita equivalence. For a fusion category $\cC$ over a field $\Bbbk$ of characteristic zero, one usually restricts to finite semisimple $\cC$-module categories where there are only finitely many indecomposables up to equivalence \cite{EGNO}*{Cor.\,9.1.6}. For a field of arbitrary characteristic, semisimple module categories are replaced by \emph{separable} module categories \cites{DSS2,DR} and again, there are only finitely many indecomposables.

The class of \emph{exact} module categories of \cites{EO1} (see also \cite{EGNO}*{Sec.\,7.5}) has homological properties similar to that of semisimple module categories. In particular, all additive $\cC$-module functors between exact module categories have left and right adjoints. If $\cC$ is a fusion category, exact module categories correspond to semisimple module categories. Another highlight of the theory is that when $\cM$ is an exact $\cC$-module category, the double centralizer theorem holds and implies Morita equivalence of the tensor category $\cC$ and its centralizer $\cC^*_\cM=\Fun_\cC(\cM,\cM)$. 

The theory of exact module categories is much richer in the non-semisimple case, where, at least in examples, a continuous spectrum of indecomposable exact module categories emerges \cite{Mom1}. For example, when $\cC=\lmod{S}$, finite-dimensional modules over the four-dimensional Hopf algebras $S$ of Sweedler, it follows from results of \cites{EO1,Mom1,Will} that indecomposable exact $\cC$-module categories fall into two one-parameter families
$$
\Set{\Vect^\lambda ~|~ \lambda \in \mC} \quad \text{and} \quad \Set{\sVect^\lambda ~|~ \lambda \in  \mC},
$$
on the category $\Vect$ or $\sVect$ of finite-dimensional $\mC$-vector spaces or $\mC$-super vector spaces, plus two non-semisimple categories $\cC$ and $\cD$.  Here, the parameter $\lambda$ characterizes the module category associator; in particular $\lambda=0$ corresponds to the trivial associator.

\medskip
To construct monoidal $2$-categories we start with the  monoidal $2$-category $\lmod{\cC}$, for a braided finite tensor category $\cC$. The monoidal product of $\lmod{\cC}$ is the relative Deligne product $\boxtimes_\cC$ \cites{ENO,GreThe,DSS1,FSS}. We observe that, contrary to claims in the literature,\footnote{See \cite{DN}*{Prop.\,2.10} for the corresponding statement for bimodule categories as well as \cite{SY} (Version 1) and \cite{JY25}*{Lem.\,2.1(ii)} (Version 2).} the $2$-subcategory of \emph{exact} $\cC$-module categories is not closed under $\boxtimes_\cC$ and we provide examples showing this, see \Cref{sec:rel-Del-exact}.  In finite characteristic of $\Bbbk$ dividing the order of an abelian group $G$, $\Vect\boxtimes_\cC\Vect$ is not exact for the symmetric monoidal category $\cC=\Vect_G$ of $G$-graded vector spaces, see \Cref{ex:first}.\footnote{A similar example appears in \cite{CSZ}*{Ex.\,7.6} and can be derived from \cite{DSS2}*{Sec.\,2.5}.} We also provide examples in characteristic zero. For example,  $\sVect\boxtimes_\cC\sVect$ is not exact, for the exact module categories $\sVect$ with trivial associators over $\cC=\lmod{S}$ for Sweedler's four-dimensional Hopf algebra, see \Cref{ex:Sweedler-first}. 

The above failure of $\boxtimes_\cC$ in preserving exact module categories motivates  us to introduce the  full $2$-subcategory $\fexmod{\cC}$ of $\lmod{\cC}$ of \emph{fully exact} module categories, the largest class of module categories that preserves the class of exact module categories by tensoring with $\boxtimes_\cC$ on the right.

\subsection{Fully exact module categories}

For a braided finite tensor category $\cC$, we say that a left $\cC$-module  $\cM$  is \emph{fully exact} if for any exact left $\cC$-module $\cN$, the relative Deligne product $\cN\boxtimes_\cC \cM$ is exact. With this definition, closure under $\boxtimes_\cC$ is easy to show.

\begin{introproposition}[See \Cref{prop:fully-exact-tensor}]
     The full $2$-subcategory $\fexmod{\cC}$ on fully exact $\cC$-modules is a monoidal $2$-subcategory of $\lmod{\cC}$.
\end{introproposition}

To give a characterization of fully exact module categories, we consider the \action functor $A_\cM\colon\cC \to \Fun_{\cC}(\cM,\cM)$, $X\mapsto X\lact (-)$, where the braiding is used to make $X\lact (-)$
a left $\cC$-module functor.
This \action functor is one of the \emph{$\alpha$-induction} functors, denoted by $\alpha^+$ in \cite{DN2}, inspired by constructions in the theory of subfactors, see also \cites{LR,BEK,Ost1,FFRS}.

\begin{introtheorem}[See \Cref{prop:fully-exact-equiv}]\label{thm:intro-A}
Let $\cC$ be a finite braided tensor category over any field $\Bbbk$.
The following are equivalent for a finite left $\cC$-module category $\cM$:
\begin{enumerate}[(i)]
 \item     $\cM$ is fully exact.
  \item For any exact $\cC$-module category $\cN$, the left $\cC$-module category $\Fun_\cC(\cN,\cM)$ is exact.
    \item 
    The centralizer $\cC^*_\cM=\Fun_\cC(\cM,\cM)$ is exact as a left $\cC$-module category.
    \item The \action functor $A_\cM\colon \cC\to \cC^*_\cM$ preserves projective objects. 
\end{enumerate}
\end{introtheorem}
We note that a fully exact $\cC$-module is in particular exact. However, the converse does not hold in general,
as we show in several examples.

Any finite left $\cC$-module $\cM$ is equivalent to $\rmodint{\cC}{A}$, for an algebra $A$ in $\cC$, with the action given by the tensor product of $\cC$ \cites{EO1,EGNO,DSS2}. Such an algebra $A$ is \emph{exact} if $\rmodint{\cC}{A}$ is exact as a $\cC$-module.
Exact algebras in a finite tensor category $\cC$ were conjectured to correspond to semisimple algebras in $\cC$ \cite{EO3}*{Con.\,B.6}.  This conjecture was verified for $\cC=\lmod{H}$, where $H$ is a finite-dimensional Hopf algebra over an algebraically closed field $\Bbbk$ in the same paper. 
A general proof of this conjecture for tensor categories over any field $\Bbbk$ has appeared in \cite{CSZ}. The present paper provides examples of semisimple algebras in braided finite tensor categories  whose tensor product is not semisimple, in both positive and zero characteristic. We also provide examples of such algebras that are commutative. 

The class of fully exact algebras has the property that tensoring with any semisimple algebra will result in a semisimple algebra in $\cC$.
An intrinsic condition of $A$ to be fully exact is that, for at least one projective object $P$ in $\cC$, 
the $A$-bimodule $P\lact A$ is a direct summand of the induced $A$ bimodule $A\otimes P \otimes A$, see \Cref{lem:A-fully-exact-equiv}.

Another intrinsic characterization of a fully exact algebra $A$ is that the algebra 
$A^\psi\otimes A$, where $A^\psi$ is the braided opposite algebra of $A$, is exact, or by~\cite{CSZ}, a semisimple algebra in $\cC$.
This characterization is analogous to the classical result from ring theory that a finitely generated algebra $A$ over a commutative ring $R$ is separable if and only if $A^{\oop}\otimes_R A$ is semisimple over $R$ \cite{Hat}*{Thm.\,2.5}.

\subsubsection*{A non-semisimple example} 
Let $\cC=\lmod{S}$, for Sweedler's Hopf algebra $S$, with symmetric braiding given by that of $\sVect$. We first re-parametrize the indecomposable $\cC$-module categories by two copies of $\mC\rP^1$ by setting, for $\lambda \neq 0$,
$$\cV_\lambda:=\Vect^{1/\lambda},\quad \cS_\lambda:=\sVect^{1/\lambda}, \quad \cV_\infty:=\Vect^0, \quad \cS_\infty:=\sVect^0, \quad \cV_0:=\cC,\quad \cS_0:=\cD.$$ 
As we show in \Cref{prop:S-mod-fully-exact-class}, $\cV_\lambda$ and $\cS_\lambda$ are fully exact if and only if $\lambda \neq \infty$, and we can compute relative Deligne products and functor categories of the fully exact $\cC$-modules.

\begin{introproposition}[See \Cref{prop:Sweedler-products}]
For $\lambda, \mu \in \mC$ we have equivalences of left $\cC$-module categories
\begin{align*}
\cV_\lambda\boxtimes_\cC \cV_\mu &\simeq \cV_{\lambda+\mu}& \cS_\lambda\boxtimes_\cC \cV_\mu&\simeq \cS_{\lambda+\mu}\simeq \cV_\lambda\boxtimes_\cC \cS_\mu,&\cS_\lambda\boxtimes_\cC \cS_\mu &\simeq \cV_{\lambda+\mu},\\
\Fun_\cC(\cV_\lambda,\cV_\mu)&\simeq \cV_{\mu-\lambda},& \Fun_\cC(\cS_\lambda,\cV_\mu)&\simeq \cS_{\mu-\lambda}\simeq 
     \Fun_\cC(\cV_\lambda,\cS_\mu),&    \Fun_\cC(\cS_\lambda,\cS_\mu)&\simeq \cV_{\mu-\lambda}.
\end{align*}
\end{introproposition}

In general, whether a given $\cC$-module category is fully exact is sensitive to the choice of braiding. For example, for $\cC=\lmod{S}$ there is a one-parameter family of braidings corresponding to universal R-matrices $R_t$, with $t\in \mC$. The braiding of $\sVect$ corresponds to the case when $t=0$. The $\cC$-module $\Vect=\cV_\infty$ is only fully exact when $t\neq 0$. For braidings obtained from $R_t$, with $t\neq 0$, instead $\cV_{-2/t}$ fails to be fully exact, see \Cref{rem:vect-Sweedler}.

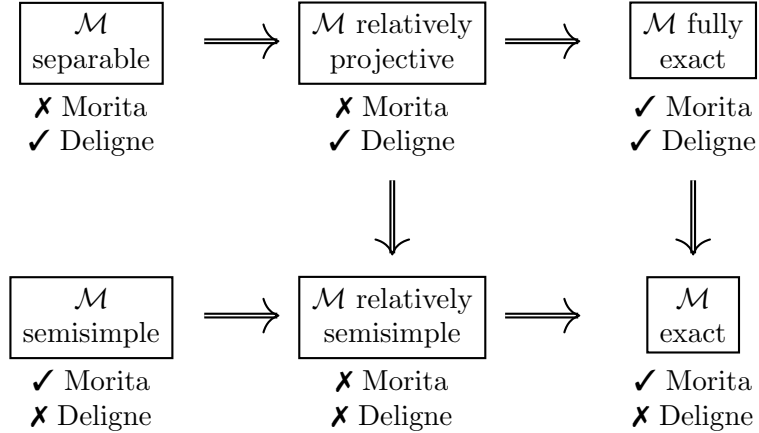
\begin{figure}[ht]
    \centering
\begin{tikzcd}[arrows=Rightarrow]
\node(A)[draw, thick, align=center] at (0,0) {$\cM$\\separable};
\node(B)[draw,thick,align=center] at (4,0) {$\cM$ relatively \\ projective};
\node(C)[draw,thick,align=center] at (8,0) {$\cM$ fully \\ exact};
\node(D)[draw,thick,align=center] at (8,-3.65) {$\cM$ \\ exact};
\node(A2) [align=center ]at (0,-1.1)
{\xmark ~Morita\\
\cmark ~Deligne};
\node(B2) [align=center ]at (4,-1.1)
{\xmark ~Morita\\
\cmark ~Deligne};
\node(C2) [align=center ]at (8,-1.1)
{\cmark ~Morita\\
\cmark ~Deligne};
\node(D2) [align=center ]at (8,-4.75)
{\cmark ~Morita\\
\xmark ~Deligne};
\draw[double,thick,->] (1.5,0.35) to (2.5,0.35);
\draw[double,thick,->] (5.5,0.35) -- (6.5,0.35);
\draw[double,thick,->] (4,-1.5) -- (4,-2.5);
\draw[double,thick,->] (8,-1.5) -- (8,-2.5);
\node(E)[draw, thick, align=center] at (0,-3.65) {$\cM$\\semisimple};
\node(F)[draw,thick,align=center] at (4,-3.65) {$\cM$ relatively \\ semisimple};
\node(E2) [align=center ]at (0,-4.75)
{\cmark ~Morita\\
\xmark ~Deligne};
\node(F2) [align=center ]at (4,-4.75)
{\xmark ~Morita\\
\xmark ~Deligne};
\draw[double,thick,->] (1.5,-3.3) to (2.5,-3.3);
\draw[double,thick,->] (5.5,-3.3) -- (6.5,-3.3);
\end{tikzcd}
    \caption{Different notions of exact $\cC$-module categories and their properties, closure under Morita equivalence (i.e. does not depend on choice of algebra object) or relative Deligne product $\boxtimes_\cC$.}
    \label{fig:overview}
\end{figure}

\Cref{fig:overview} compares the 
class of fully exact module categories to other classes of module categories. Here, \emph{separable} module categories are module categories of the form $\rmodint{\cC}{A}$, for a separable algebra $A$ in $\cC$, i.e.\,where the multiplication map of $A$ has a section as an $A$-bimodule map. Separable algebras are stable under Morita equivalence when $\cC$ is fusion \cite{DSS2}*{Sec.\,2.5}. We show in \Cref{ex:sep-not-Morita} that this fails in the non-semisimple case. Moreover, in the above example of $\cC=\lmod{S}$, only $\cV_0$ and $\cS_0$ are separable, showing that the class of fully exact module categories is significantly larger. In addition, $\cV_0$ and $\cS_0$ are also the only \emph{relatively projective} $\cC$-module categories, i.e.\,which can be represented as modules over an algebra $A$ in $\cC$ such that $A$ is relatively projective as an $A$-bimodule in $\cC$, see a precise definition in \Cref{def:rel-proj-bimodule}. The lower row of implications of \Cref{fig:overview} displays sufficient criteria for a $\cC$-module category $\cM$ to be exact. Here, $\cM$ is \emph{relatively semisimple} if there exists  $M$ in $\cM$ such that any object of $\cM$ is a direct summand of $X\lact M$, for some $X$ in $\cC$, see \Cref{def:M-rel-ss}. We finally note that the class of fully exact module categories is the only class in \Cref{fig:overview}  that is  stable under both Morita equivalence and the relative Deligne product.

\subsection{Relation to separable algebras} 
\label{secIntro:sep}
Applying results of \cite{DSS2}*{Sec.\,2.5}, for a $\Bbbk$-linear fusion category $\cC$, a finite left $\cC$-module $\cM$ is fully exact if and only if it is separable as both notions are equivalent to $\cC^*_\cM$ being semisimple, see \Cref{cor:semisimple-fully-exact}.  
If, in addition, $\Bbbk$ is an algebraically closed field of characteristic zero, all six notions in \Cref{fig:overview} are equivalent, see the second part of \Cref{cor:semisimple-fully-exact}. Elsewhere, we make no assumptions on the field $\Bbbk$ unless stated otherwise.

The equivalence of (i) and (iii) in \Cref{thm:intro-A} shows that fully exact $\cC$-module categories provide a more natural analog of separable module categories for a \emph{non-semisimple} finite braided tensor category $\cC$,  while considering
usual separable module categories 
in this generality appears too restrictive. 
Indeed, by \Cref{lem:A-fully-exact-equiv}, an algebra $A$ in $\cC$ is fully exact 
if and only if $P_{\one}\otimes A$ is a direct summand of $A\otimes P_{\one} \otimes A$ as an $A$-bimodule, for the projective cover $P_\one$ of the tensor unit $\one$.
 When $\cC$ is semisimple, $A$ being fully exact
 is equivalent to this direct summand condition for $P_\one=\one$, 
  i.e.\ to $A$ being a separable algebra. One could thus also call fully exact $\cC$-module categories \textit{projectively separable}. Another natural generalization of separable algebras could be ``$A$ is projective as an $A$-bimodule". However, by 
 \Cref{cor:A-proj-sep} existence of such a non-zero $A$ in $\cha\Bbbk=0$ would then imply that $\cC$ is necessarily semisimple and thus $A$ is separable in the usual sense.

\subsubsection*{When is $\Vect$ fully exact?}
\mbox{}
Let $\cC=\lmod{H}$ be the braided tensor category of finite-dimensional modules over a finite-dimensional quasitriangular Hopf algebra $H$ with universal R-matrix $R=R^{(1)}\otimes R^{(2)}$. The canonical fiber functor $F\colon \cC\to \Vect$ makes the category $\Vect$ a  left $\cC$-module category. To describe when the corresponding $\cC$-action is fully exact,  consider the Hopf algebra homomorphism 
\begin{equation*}
\phi_R\colon H^*\to H,\quad  f\mapsto R^{(1)}f(R^{(2)}).
\end{equation*}

\begin{introproposition}[See \Cref{lem:Vect-fully-exact}]\label{prop:D}
Let $H$ be a finite-dimensional Hopf algebra over any field $\Bbbk$.
    The following conditions are equivalent.
\begin{enumerate}
    \item[(i)] $\Vect$ is fully exact.
    \item[(ii)] The restriction functor $\Res_{\phi_R}\colon \lmod{H}\to \lmod{H^*}$ preserves projective objects. 
\end{enumerate}
\end{introproposition}
If $H^*$ is semisimple, $\Vect$ is always fully exact.
In the case of the factorizable Hopf algebra $H=u_q(\mathfrak{sl}_2)$,  Lusztig's finite-dimensional quantum $\mathfrak{sl}_2$ at any odd order of $q$, we show that $\Vect$ is not fully exact, see \Cref{sec:uqsl2}.  We conjecture  that this observation extends to general small quantum groups $u_q(\mathfrak{g})$.

\subsection{Perfect module categories are fully dualizable}

Dualizability in higher monoidal categories is a topic of considerable interest in the study of TFTs due to the cobordism hypothesis  \cites{BD,Lur,DSS2,BMS,DR,BJS}. In this context, an object of a monoidal $2$-category is called \emph{fully dualizable} if it is contained in a monoidal $2$-subcategory where all objects are dualizable, i.e.\ have left and right duals, and all $1$-morphisms have left and right adjoints \cites{DSS2,Pst}.

We say that $\cM$ is \emph{op-fully exact} if, for any exact $\cN$, we have that $\cM\boxtimes_\cC \cN$ is exact. The monoidal $2$-category $\opfexmod{\cC}$ 
of op-fully exact left $\cC$-modules is in fact equivalent to $\fexmod{\cC^{\otimesop}}$, where $\cC^{\otimesop}$ is the opposite tensor category. The key reason for considering op-fully exact module categories is that $\cM$ is op-fully exact if and only if the right dual ${}^*\cM=\Fun_\cC(\cM,\cC)$, or equivalently, the left dual $\cM^*$ is fully exact, see  \Cref{prop:duality-fully-exact}. We also show that $\cM$ is op-fully exact if and only if it is fully exact as a $\cC^\rev$-module category, where $\cC^\rev$ is the tensor category $\cC$ equipped with the inverse braiding. Hence, $\cM$ is op-fully exact if and only if the \action functor $A_\cM^\rev$, defined using the inverse braiding of~$\cC$, preserves projective objects.\footnote{This functor is the other $\alpha$-induction functor $\alpha^-$ in \cite{DN2}.} 
In particular, 
similarly to the above \Cref{prop:D}, $\Vect$ is op-fully exact if and only if the  functor $\Res(\phi_{R^{-1}_{21}})$ preserves projective objects.

The module category $\Vect$ over the Drinfeld center $\cZ(\cC)$ of $\cC=\lmod{\Bbbk G}$, where $G$ is a finite group and $\cha \Bbbk$ divides $|G|$, provides an example of a fully exact module category that is not op-fully exact, see \Cref{ex:double-group}. This shows that, in general, the monoidal $2$-category $\fexmod{\cC}$ is not closed under duals. However, when $\cC$ is symmetric, then $\fexmod{\cC}$ has duals. 

We call a $\cC$-module category $\cM$ \emph{perfect} if both $\cM$ and its dual $\cM^*$ are fully exact. By  \Cref{prop:duality-fully-exact}, $\cM$ is perfect if and only if $\cM$ is fully exact and op-fully exact. Hence, the perfect $\cC$-module categories are the dualizable objects in $\fexmod{\cC}$. What is more, fully exact $\cC$-modules are, in particular, exact and hence have the property that all $1$-morphisms have left and right adjoints. This shows that  perfect $\cC$-modules are \emph{fully dualizable} objects in $\lmod{\cC}$ in the sense of the cobordism hypothesis \cites{Lur,DSS2}. We show that the converse also holds.

\begin{introtheorem}[See \Cref{prop:dualizability}]
Let $\cC$ be a finite braided tensor category over a field $\Bbbk$.
 An object $\cM$ in $\lmod{\cC}$ is fully dualizable if and only if $\cM$ is perfect, i.e.\,$\cM$ and $\cM^*$ are fully exact.\end{introtheorem}

The terminology of perfect $\cC$-modules is motivated by that of perfect complexes of modules over a commutative ring (or, more generally, of quasi-coherent sheaves over a scheme $X$). Perfect complexes are precisely the dualizable objects in the respective categories with the derived version of  the relative tensor product over $R$, cf.\,\Cref{rem:perf-complex}.

Given a monoidal $2$-category, one can also consider the largest \emph{full} $2$-subcategory where \emph{all} $1$-morphisms have left and right adjoint. In the case of $\lmod{\cC}$, for $\cC$ a finite braided tensor category, this $2$-subcategory is $\lemod{\cC}$, the $2$-subcategory of exact left  $\cC$-modules. The results of this paper show that this $2$-subcategory  is not a monoidal $2$-subcategory.
However, by definition, $\lemod{\cC}$ is a right module $2$-category over the monoidal $2$-category $\fexmod{\cC}$ in the sense that the monoidal product $\boxtimes_\cC$ of $\lmod{\cC}$ restricts to a $2$-functor 
$$
\xymatrix{
\lemod{\cC}\times \fexmod{\cC} \ar[r]^-{\boxtimes_\cC}& \lemod{\cC}.
}
$$
In fact, by definition, $\fexmod{\cC}$ is the largest full monoidal $2$-subcategory such that $\lemod{\cC}$ becomes a right module $2$-category by restriction of $\boxtimes_\cC$. 
Similarly, $\lemod{\cC}$ is a left module $2$-category over $\opfexmod{\cC}$. Combining these observations, $\lemod{\cC}$ is a 
bimodule $2$-category over the monoidal $2$-category $\perf{\cC}$ of perfect $\cC$-modules.

\subsection{A diagram of monoidal 2-subcategories}

We will now summarize the relationship of various monoidal $2$-categories considered in this paper.  For this, we denote by $\invmod{\cC}$ the monoidal $2$-subcategory of invertible objects in $\lmod{\cC}$ studied, e.g., in \cites{ENO,DN2}. We extend the characterization of invertible $\cC$-module categories from \cite{ENO}*{Prop.\,4.2} in terms of the \action functor $A_\cM\colon \cC\to \cC^*_\cM$ being an equivalence to the generality of this paper, see \Cref{prop:Minv-AM-equiv}.
We further denote by 
$\sepmod{\cC}$ the full monoidal $2$-subcategory of \emph{separable} left $\cC$-modules which was studied, in the semisimple case, in \cite{DSS2}.

\begin{introtheorem}\label{thm:fully-exact-mon-2-cat}
Let $\cC$ be a finite braided tensor category over a field $\Bbbk$.
We have the following inclusions of monoidal $2$-subcategories 
$$
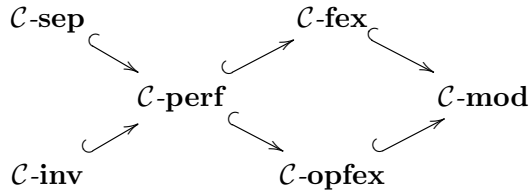
\xymatrix@R=10pt@C=10pt@M=5pt@W=5pt{
\sepmod{\cC}\ar@{^{(}->}[rd]&&\fexmod{\cC}\ar@{^{(}->}[rd]&\\
&\perf{\cC} \ar@{^{(}->}[ru]\ar@{^{(}->}[rd] && \lmod{\cC}\\
\invmod{\cC}\ar@{^{(}->}[ru]&&\opfexmod{\cC}\ar@{^{(}->}[ru]&
}$$
which are full on $1$ and $2$-morphisms. Here, the monoidal $2$-categories $\invmod{\cC}$, $\sepmod{\cC}$, $\perf{\cC}$ and $\lmod{\cC}$ are rigid in the sense of \Cref{def:2rigid}.
\end{introtheorem}

 The inclusion of separable and invertible module categories in $\perf{\cC}$ follows from \Cref{cor:sep-dualizable}, respectively, \Cref{cor:inv-fully-exact}. We know that, in general, all inclusions in \Cref{thm:fully-exact-mon-2-cat} are strict and that $\sepmod{\cC}$ and $\invmod{\cC}$ are not related by inclusion. Indeed, if $\cC$ is a fusion category over an algebraically closed field of characteristic zero, separability is equivalent to semisimplicity of a module category, see \Cref{cor:semisimple-fully-exact} or \cite{DSS2}*{2.5.1}, and not all semisimple module categories are invertible. For $\cC=\lmod{\Bbbk G}$, where $G$ is a non-trivial group and $\Bbbk$ is any field, $\Vect$ is not invertible but perfect, see \Cref{ex:counterexample}. Conversely, in \Cref{sec:Sweedler}, the module categories $\Vect^\lambda$ are invertible but not separable, see \Cref{lem:separable-Sweedler}.

Invertible left $\cC$-module categories were studied in \cite{DN2} and used in the definition of the higher 
Picard group and in constructing  $G$-crossed braided extensions of braided tensor categories for a finite group $G$, see \cite{DN2}*{Sec.\,8.2}. 
The Picard $1$-group $\Pic(\cC)$ for $\cC=\lmod{S}$, where $S$ is Sweedler's Hopf algebra, is known \cite{Car}. It is given by $\mC\times \mZ_2$. 
Our classification of indecomposable fully exact module categories in \Cref{sec:Sweedler} matches this result, as all indecomposable fully exact $\cC$-module  categories are invertible in this example. Moreover, we provide concrete module categories as representatives for each element of this Picard group. In fact, the $2$-categorical Picard group $\bPic(\cC)$ of \cite{DN2}*{Eq.\,(1.4)} is obtained from $\fexmod{\cC}$ (or, equivalently, from $\perf{\cC}$) by restricting to invertible objects,  i.e.\ all indecomposable objects of $\fexmod{\cC}$ in our Sweedler example, and to invertible $1$- and $2$-morphisms. We note that $\bPic(\cC)$ does not contain non-simple $1$-morphisms so it does not see the non-semisimple part of $\fexmod{\cC}$.
Indeed, for general non-semisimple $\cC$, autoequivalences of each object of $\invmod{\cC}\subset\fexmod{\cC}$ correspond to invertible objects of $\cC$, as in $\bPic(\cC)$, while their projective covers live in the $1$-endomorphism categories of 
$\invmod{\cC}$ only.

\subsection{Outlook}

We list here some interesting directions not fully explored in this paper.

\begin{enumerate}
\item \label{item:idemp-comp}
The work \cite{DR}*{Sec.\,A.5} introduces an idempotent completion of a $2$-category, and fusion $2$-categories are required to be idempotent complete. The fusion $2$-categories are also required to be locally semisimple,  i.e.\ $\Fun_\cC(\cM,\cN)$ would need to be semisimple for all objects $\cM,\cN$, which is clearly not the case for a non-semisimple finite tensor category $\cC$. Nevertheless, it would be interesting and important to see whether the monoidal $2$-categories $\fexmod{\cC}$ and $\perf{\cC}$ are idempotent complete in this sense.
\vskip1mm
\item 
As a first step towards \eqref{item:idemp-comp}, it is a natural question to ask what the idempotent completion of the delooping $2$-category $\rB\cC$ with one object and $1$-endomorphism tensor category $\cC$ is. In the fusion case, this is given by $\sepmod{\cC}$ \cite{DR}*{Prop.\,1.3.13} and we expect that this result holds for non-semisimple finite tensor categories.
\vskip1mm
\item  It is of fundamental importance to find examples of finite non-semisimple braided tensor categories~$\cC$, if they exist, for which both classes of exact and fully exact $\cC$-modules agree, i.e.\ $\lemod{\cC}=\fexmod{\cC}$, and thus $\lemod{\cC}=\opfexmod{\cC}=\perf{\cC}$. So far, the examples we have rather suggest that, for $\cha \Bbbk =0$, it might be true that $\lemod{\cC}=\fexmod{\cC}$ if and only if $\cC$ is semisimple. 
\vskip1mm
\item 
It is natural to ask for a definition of fully exact (or perfect) \emph{bimodule} categories over a finite tensor category $\cC$. In \cite{GreThe}*{Thm.\,7.14}, it is outlined how to upgrade the biequivalence $\bimod{\cC}\to \lmod{\cZ(\cC)}$ of \cites{EO1,EGNO} to 
a biequivalence
of monoidal $2$-categories, where the monoidal structure on $\bimod{\cC}$ is the relative Deligne product $\boxtimes_\cC$ and that on $\lmod{\cZ(\cC)}$ is $\boxtimes_{\cZ(\cC)}$, using the braiding of the Drinfeld center $\cZ(\cC)$. As this biequivalence preserves exact module categories, fully exact (or perfect) $\cC$-bimodules correspond to fully exact (respectively, perfect) $\cZ(\cC)$-modules.

With this correspondence, several of our results directly generalize to such fully exact $\cC$-bimodules. 
First, as the biequivalence is monoidal, the relative Deligne product $\boxtimes_\cC$ of fully exact $\cC$-bimodules is again a fully exact $\cC$-bimodule. 
Then, in \Cref{thm:intro-A} (\Cref{prop:fully-exact-equiv}) applied for $\cZ(\cC)$ with $\cC$ not necessarily braided, one would need to replace everywhere exact $\cZ(\cC)$-modules by exact $\cC$-bimodules, and $\Fun_{\cZ(\cC)}$ by $\Fun_\cC$ in (ii) and (iii) because the relative Deligne product in $\bimod{\cC}$ is defined over $\cC$ instead of $\cZ(\cC)$.  Part (iv) has to be replaced by the requirement that both functors $A^l_\cM$ and $A^r_\cM$, as defined in \eqref{eq:Ar}--\eqref{eq:C-linear-Al}, preserve projective objects.

\vskip1mm
\item 
When $\cC$ is a finite ribbon category, then there are equivalences $\cM\simeq \cM^{**}$ for any finite module category $\cM$, see \Cref{rem:double-dual-equiv}. We expect that, in this case, $\perf{\cC}$ and $\lmod{\cC}$ are \emph{pivotal} monoidal $2$-categories in an appropriate sense based on work of \cites{BMS,DR}, see \Cref{rem:double-dual-equiv}. 
This pivotality structure should not be confused with the concept of \emph{pivotal $\cC$-module categories} \cites{SchP,ShiP}.
In  \Cref{rem:pivotal-mod-cat}, we compare the notions of fully exact 
and pivotal module categories  over $\cC=\lmod{S}$, for Sweedler's Hopf algebra $S$.
\vskip1mm
\item 
If $G$ is a finite group such that $\cha \Bbbk$ divides $|G|$, then $\cC=\lmod{\Bbbk G}$ is a non-semisimple symmetric tensor category. It would be interesting to investigate the monoidal $2$-category $\perf{\cC}$ further, i.e.\,to classify its indecomposable objects, based on the classification of exact $\cC$-module categories from \cite{EGNO}*{Cor.\,7.12.20}, and compute their relative Deligne products.
If $G$ is an abelian group, one can further define other non-symmetric braidings on $\cC$ based on roots of unity, see e.g.\,\cite{Maj}*{Ex.\,2.1.6}, which will lead to different monoidal $2$-categories.
\vskip1mm
\item Finally, recalling that the Picard $1$-group $\Pic(\cC) = \mC\times \mZ_2$  for our main non-semisimple example $\cC=\lmod{S}$, where $S$ is 
Sweedler's Hopf algebra, it would be  interesting to identify symmetric  $\mC\times \mZ_2$-graded extensions of $\cC$. Indeed, as noted in~\cite{DN2}*{Sec.\,5.4}, the symmetric $2$-categorical Picard group $\mathbf{Pic}_{\mathrm{sym}}(\cC)$, for symmetric $\cC$, of symmetric $\cC$-module categories is equivalent to the usual Picard $2$-group $\mathbf{Pic}(\cC)$. Moreover, by~\cite{DN2}*{Sec.\,8.4}, all symmetric graded extensions of $\cC$ are controlled by $\mathbf{Pic}_{\mathrm{sym}}(\cC) = \mathbf{Pic}(\cC)$. A symmetric $\mathbb{C}$-graded  extension of $\cC$ is readily provided  by the  category of finite-dimensional modules over the quotient of the infinite-dimensional Sweedler Hopf algebra by $g^2=1$, where $g$ is group-like, and such that the central element $x^2$, where $x$ is the skew primitive generator, acts diagonalizably. This category inherits the symmetric braiding from $\sVect$. We are thus interested in symmetric $\mathbb{Z}_2$-graded extensions of this symmetric $\mathbb{C}$-graded extension of~$\cC$. Our explicit realization of $\mathbf{Pic}_{\mathrm{sym}}(\cC)$ as the $2$-categorical group of the monoidal $2$-category $\fexmod{\cC}=\perf{\cC}$ might provide a clue for such interesting symmetric extensions. 

\end{enumerate}

\subsection{Summary of content}

We start, in \Cref{sec:rel-Del}, by reviewing the monoidal $2$-category $\lmod{\cC}$ for a finite braided tensor category $\cC$. Here, we also discuss $\cC$-module category structures on categories of left $\cC$-module functors.  In \Cref{sec:int-alg}, we work with the monoidal Morita $2$-category $\Mor_\cC$ of internal algebras, bimodules, and bimodule morphisms, which is biequivalent to $\lmod{\cC}$, and describe all required constructions using internal algebras. \Cref{sec:exact-mod-cat} contains the results concerning exact module categories and exact algebras needed in this paper, some of which are likely known but not readily available in the literature. In \Cref{sec:rel-Del-exact}, we explain how, in general, exact module categories are not closed under relative Deligne products and provide examples.

The main definitions and fundamental results on the new class of fully exact module categories are contained in \Cref{sec:fully-exact}. Duality of fully exact module categories is analyzed in the following \Cref{sec:dual-fully-exact}, and perfect modules categories are introduced and studied in \Cref{sec:fully-dualizable}. Next, \Cref{sec:invertible} contains a characterization of invertible module categories and compares this notion to perfect module categories. Finally, \Cref{sec:Hopf} contains classes of examples of fully exact module categories over braided tensor categories of the form $\cC=\lmod{H}$, for $H$ a quasi-triangular Hopf algebra. This includes the result that $\Vect$ is not fully exact for $H=u_q(\fr{sl}_2)$ at an odd root of unity $q$ and a conjecture concerning the general case of the factorizable small quantum groups $u_q(\fr{g})$. We conclude with the detailed study of the  monoidal $2$-category $\fexmod{\cC}$ for the example where $\cC=\lmod{S}$, for $H=S$ the four-dimensional Hopf algebra of Sweedler.  
We have included three appendices for basic notions concerning monoidal categories, module and bimodule categories, and some details on the construction of $\lmod{\cC}$ (or $\bimod{\cC}$) as monoidal $2$-categories, see \Crefrange{sec:basics}{appendix:mod-2-cat}.

\subsection{Conventions}
Throughout this paper, we work over a general field $\Bbbk$. 
Assumptions on the field $\Bbbk$ will be stated when needed. All categories considered will be abelian and $\Bbbk$-linear. 
We call a category \emph{finite} if it is equivalent to a category of finite-dimensional modules over a finite-dimensional $\Bbbk$-algebra. In particular, a finite category is $\Bbbk$-linear, abelian, additive, and has enough projective objects. All functors are at least $\Bbbk$-linear, additive, and right exact. We denote the category of finite-dimensional $\Bbbk$-vector spaces by $\Vect$.

A \emph{finite tensor category} is a finite category $\cC$ which is monoidal with a $\Bbbk$-bilinear monoidal product, left and right duals, and a simple tensor unit $\one=\one_\cC$. In particular, the tensor product is exact in both variables \cite{EGNO}*{Section~4.2} and $\End_\cC(\one)$ is a finite field extension\footnote{Note that $\End_\cC(\one)$ is commutative by \cite{TV}*{Lem.\,1.1} and Schur's Lemma.} of $\Bbbk$. Unless otherwise stated, we assume that the tensor categories considered in this paper are finite tensor categories. If, in addition, $\cC$ is semisimple, we say that $\cC$ is a \emph{fusion category}. The only difference here to definitions in \cite{EGNO} lies in the assumptions on the field $\Bbbk$ in the sense that we do not require $\Bbbk$ to be algebraically closed. We will refer to statements in \cite{EGNO} or other sources that assume that $\Bbbk$ is algebraically closed when the proof does not use this assumption. See also \cite{DSS2}*{Thm.\,2.3.5} where several key results of \cite{EGNO} are reviewed without assumptions on the field $\Bbbk$. 

We note that \cites{DSS1,DSS2} use a different definition of
a finite tensor category where $\one$ is not required to be simple,
but with rigidity it's necessary semisimple,
i.e.\ what is called a \emph{multitensor} category in  \cite{EGNO}. 
Our results (in particular, \Cref{prop:fully-exact-equiv}) can be generalized to multitensor categories.
Such a generalization is straightforward 
 except for the discussion on invertible module categories in \Cref{sec:invertible},
 \Cref{lem:F-faithful} and \Cref{lem:AM-exact-faithful}, as well as \Cref{cor:A-proj-sep}, where the simplicity of the tensor unit was used. We however believe that these statements equally hold for \textsl{indecomposable} multitensor categories.

For algebras $A,B$ in $\cC$, we denote by $\rmodint{\cC}{A}$, $\lmodint{A}{\cC}$, and $\bimodint{A}{\cC}{B}$ the categories of right $A$-modules, left $A$-modules, and $A$-$B$-bimodules in $\cC$. When $A=B$, we simply refer to $A$-$A$-bimodules as $A$-bimodules and simply write $\bimodints{A}{\cC}=\bimodint{A}{\cC}{A}$.

By a \emph{$2$-category} $\bfC$ we mean a bicategory where composition of $1$-morphisms is strictly associative\footnote{Note that the \emph{$2$-categories} we consider are usually also \emph{monoidal $2$-categories} and the monoidal structure might not be strict, see \Cref{appendix:mod-2-cat} for conventions on monoidal $2$-categories.} All $2$-categories considered are $\Bbbk$-linear. That is, a $2$-category is a category enriched over the category $\Cat_\Bbbk$ of $\Bbbk$-linear categories and $\Bbbk$-linear functors. In particular, for any object $X$ in $\bfC$, we have that $\Hom_\bfC(X,X)$ is a strict monoidal category. 
Similarly, a $2$-functor is a functor of categories enriched over $\Cat_\Bbbk$, i.e.\ it strictly preserves composition of $1$-morphisms.
We say that a $2$-subcategory $\bfD$ of $\bfC$ is \emph{full} if  
$\Hom_\bfD(X,Y)=\Hom_\bfC(X,Y)$, for any objects $X,Y$ in $\bfD$.

\subsection{Acknowledgments}
The authors thank A.~Davydov and D.~Nikshych for valuable comments. We also thank D.~Jaklitsch, I.~Runkel, K.~Shimizu, M.~Stroi\'nski, and H.~Yadav for interesting discussions and comments on the subject of this paper. This research was funded by the Royal Society International Exchanges Grant ES\textbackslash R1\textbackslash 231027. Part of this  article was written during the program `CFT: Algebraic, Topological and Probabilistic approaches in Conformal Field Theory' at the Institute Pascal, Orsay, Paris. AMG is supported by C.N.R.S.\ and also partially supported by
the ANR grant NASQI3D ANR-24-CE40-7252.

\section{Module categories and relative Deligne products}
\label{sec:rel-Del}

This section reviews the monoidal $2$-category $\lmod{\cC}$ of finite module categories over a finite braided tensor category $\cC$. 
We also discuss $\cC$-module category structures on categories of module functors, \action functors, and opposite module categories in  \Crefrange{sec:modulecat-rex}{sec:opposite-mod-cats}. 
For background material on (braided) monoidal categories, (bi)module categories and their functors, and monoidal $2$-categories we refer to \Cref{sec:basics,appendix:mod-bimod,appendix:mod-2-cat}, respectively.

\subsection{Monoidal 2-categories of module categories}\label{sec:C-mod}

For a finite tensor category $\cC$,  $\lmod{\cC}$ denotes the $2$-category of finite left $\cC$-module categories $\cM=(\cM,\lact\colon \cC\times \cM \to \cM)$, right exact $\cC$-module functors $(G, s^G)\colon \cM\to \cN$, and $\cC$-module natural transformations,  see \Cref{sec:mod-cats} for definitions. Similarly, $\bimod{\cC}$ denotes the $2$-category of $\cC$-bimodule categories, see \Cref{sec:bimod}.

\smallskip

Recall the Deligne product $\cN\boxtimes\cM$ of finite ($\Bbbk$-linear abelian) categories from  \cite{Del}*{Prop.\,5.13}.
We will now review the \emph{relative Deligne product} $\cN\boxtimes_\cC\cM$, defined, in general, for a right $\cC$-module $\cN$ and a left $\cC$-module $\cM$, see  \cites{ENO,GreThe,Gre,DSS1,DSS2,FSS,DN2}. We require the following definitions.

\begin{definition}\label{def:balancing}
A \emph{$\cC$-balanced functor}
$B\colon \cN\times \cM\to \cP$
is a functor with target category $\cP$ together with a natural isomorphism, called a \emph{$\cC$-balancing}
$$\beta_{N,X,M} \colon B((N\ract X), M)\isomorph B(N, (X\lact M))$$ 
satisfying the coherence condition that the diagram 
\begin{equation}\label{eq:balancing-coherence}
\vcenter{\hbox{
\xymatrix@C=35pt{B(N\ract X\otimes Y, M)\ar[rr]^-{\beta_{N,X\otimes Y,M}}\ar[d]^{B(r^{-1}_{N,X,Y}, \id_M)}&&B(N, X\otimes Y\lact M)\\
B((N\ract X)\ract Y, M)\ar[r]^{\beta_{N\ract X,Y,M}}&B(N\ract X, Y\lact M)\ar[r]^{\beta_{N,X,Y\lact M}}&B(N, X\lact (Y\lact M))\ar[u]^{B(\id_N, l_{X,Y,M})}
}}}
\end{equation}
commutes, and the triangle identity which implies that $\beta_{N,\one,M}=\id$ with our assumptions, cf.\,\cite{ENO}*{Def.\,3.1}, \cite{DSS1}*{Def.\,3.1} or \cite{FSS}*{Def.\,2.7(i)}.
A \emph{$\cC$-balanced natural transformation} $\eta\colon B\to C$, where $(B, \beta^B)$, $(C,\beta^C)\colon \cN\times \cM\to \cP$ are $\cC$-balanced functors is a natural transformation such that the diagram
$$
\xymatrix{B(N\ract X, M)\ar[rr]^-{\beta^B_{N,X,M}}\ar[d]^{\eta_{N\ract X,M}}&&B(N, X\lact M)\ar[d]^{\eta_{N,X\lact M}}\\
C(N\ract X, M)\ar[rr]^-{\beta^C_{N,X,M}}&&C(N, X\lact M)
}
$$
commutes for all objects $N$ in $\cN$, $X$ in $\cC$, $M$ in $\cM$,
 cf.\ \cite{DSS1}*{Def.\,3.1} or \cite{FSS}*{Def.\,2.7(ii)}.

We denote the category of $\cC$-balanced functors $\cN\times \cM\to \cP$ which are right exact in each component, together with $\cC$-balanced natural transformations,  by $\Fun^{\mathrm{bal}}_\cC(\cN\times\cM,\cP)$.
\end{definition}

\begin{definition}\label{def:rel-Del}
Let $\cN$ be a right and $\cM$ be a left $\cC$-module. Then the \emph{relative Deligne product} of $\cN$ and $\cM$ over $\cC$ is a category $\cN\boxtimes_\cC\cM$ together with a $\cC$-balanced functor 
\begin{equation}\label{eq:balanced-P}
    P_{\cN,\cM}=(P,\beta^P)\colon \cN\times \cM\to \cN\boxtimes_\cC \cM,
\end{equation}
right exact in each variable, 
such that for any category $\cP$, the functor
\begin{equation}\label{eq:universal-Del}
(-)\circ P_{\cN,\cM} \colon \Fun (\cN\boxtimes_\cC \cM,\cP)\to  \Fun^{\mathrm{bal}}_\cC(\cN\times\cM,\cP),
\end{equation}
induced by pre-composition with $P_{\cN,\cM}$ is an equivalence of categories. 
\end{definition}

It has been proved in \cite{DSS1}*{Thm.\,3.3} (see also \cite{ENO}*{Prop.\,3.5} in the semisimple case) that the relative Deligne product $\cN\boxtimes_\cC\cM$ exists as a finite abelian category, for $\cC$ a finite tensor category and $\cN$ and $\cM$ finite abelian right (respectively, left) $\cC$-module categories. The proof in \cite{DSS1} relies on the internal algebra realization of module categories which we review in \Cref{sec:MorC-C-mod-monoidal}. 

\begin{proposition}\label{prop:rel-Del-bimod}
If $\cN$, $\cM$ are $\cC$-bimodules, then $\cN\boxtimes_\cC\cM$ is a $\cC$-bimodule.    
\end{proposition}
\begin{proof}
This is proved in \cite{DSS2}*{Thm.\,2.5.5} for separable bimodules when $\cC$ is semisimple, but the proof does not rely on these assumptions.   
\end{proof}

\begin{remark}\label{rem:factorization-canonical-iso}
From the equivalence \eqref{eq:universal-Del} one obtains that for a given $\cC$-balanced functor $F\colon\cN\times\cM\to \cP$, there exists a functor $F'\colon \cN\boxtimes_\cC\cM\to\cP$ with a $\cC$-balanced natural isomorphism 
$e'\colon F'\circ P_{\cN,\cM}\isomorph F$, for the universal $\cC$-balanced functor from \eqref{eq:balanced-P}.
The pair $(F',e')$ is unique up to canonical isomorphism in the following sense. If $F''$ is another functor together with a natural $\cC$-balanced isomorphism $e''\colon F''\circ P_{\cN,\cM}\isomorph F$, then there exists a canonical isomorphism 
$\phi\colon F'\to F''$ such that $\phi\circ P_{\cN,\cM}= (e'')^{-1}\circ e'$.
\end{remark}

\begin{remark}\label{rem:rel-Del-unique}
The relative Deligne product is unique up to equivalence in the following sense. Assume both $\cN\boxtimes_\cC \cM$ and $\cN\boxtimes_\cC'\cM$ have the universal property of the relative Deligne product with respect to balanced functors 
$P_{\cN,\cM}$ and $P'_{\cN,\cM}$, respectively. Then, by applying \eqref{eq:universal-Del} with $\cP=\cN\boxtimes'_\cC \cM$, there exists an equivalence $E\colon \cN\boxtimes_\cC \cM\to \cN\boxtimes_\cC' \cM$ which satisfies that 
the diagram 
$$
\xymatrix@R=10pt{
&\cN\times \cM\ar[dl]_{P_{\cN,\cM}} \ar[dr]^{P_{\cN,\cM}'}&\\
\cN\boxtimes_\cC \cM
\ar[rr]^E&&
\cN\boxtimes_\cC' \cM
}
$$
commutes up to a natural isomorphism of $\cC$-balanced functors 
$e\colon E\circ P_{\cN,\cM}\isomorph P_{\cN,\cM}'.$
If $E'$ is another functor together with an isomorphism of $\cC$-balanced functors $e'\colon E'\circ P_{\cN,\cM}\to P_{\cN,\cM}$, then there is a  canonical natural isomorphism $\phi\colon E\isomorph E'$ in the sense of \Cref{rem:factorization-canonical-iso}.
\end{remark}

It is widely established in the literature that $\bimod{\cC}$ is a monoidal $2$-category with respect to relative Deligne product $\boxtimes_\cC$, see e.g.\ \cite{Gre}*{Thm.\,1.1} or \cite{GreThe}*{Thm.\,0.2.1}  and \cite{DSS2}*{Thm.\,2.2.18}. 
The monoidal $2$-category $\bimod{\cC}$ comes with \emph{associators} 
\begin{equation}\label{eq:ass-Sec2}
A_{\cM,\cN,\cP}\colon \cM\boxtimes_\cC (\cN\boxtimes_\cC \cP)\isomorph (\cM\boxtimes_\cC \cN)\boxtimes_\cC \cP,
\end{equation}
and \emph{unitors}
\begin{equation}\label{eq:unitors-Sec2}
L_\cM\colon \cC\boxtimes_\cC \cM\isomorph \cM, \qquad R_\cM\colon \cM\boxtimes_\cC \cC\isomorph\cM.
\end{equation}
These equivalences are natural up to a choice of coherent natural $2$-isomorphisms\footnote{Called \emph{strong transformations} in \cite{JY21}*{Def.\,4.2.1}.}. For the associator, these natural $2$-isomorphisms are denoted by
\begin{equation}\label{eq:2associator-natural-Sec2}
    (G\boxtimes_\cC H)\boxtimes_\cC K)\circ A_{\cM,\cN,\cP} \xrightarrow{a_{G,H,K}} A_{\cM',\cN',\cP'}\circ (G\boxtimes_\cC (H\boxtimes_\cC K)),
\end{equation}
for  a triple of $\cC$-linear functors $\cM\xrightarrow{G} \cM'$, $\cN\xrightarrow{H} \cN'$, and $\cP\xrightarrow{K} \cP'$.
 The associators and unitors further come with canonical isomorphisms up to which the pentagon and triangle diagrams commute as well as an \emph{interchange isomorphism} 
\begin{equation}
\label{eq:interchange-C-mod}
(H\boxtimes_\cC H')\circ (G\boxtimes_\cC G')\xrightarrow{\phi_{H\boxtimes_\cC H',G\boxtimes_\cC G'}}  (H\circ G)\boxtimes_\cC (H'\circ G'),
\end{equation}
for $\cC$-bimodule functors $\cM\xrightarrow{G}\cN\xrightarrow{H}\cP$ and $\cM'\xrightarrow{G'}\cN'\xrightarrow{H'}\cP'$,
encoding the compatibility of relative Deligne product and composition of $1$-morphisms. The construction of these structural elements of the monoidal $2$-category is explained in \Cref{sec:rel-Del-props}. A list of the structural data of a monoidal $2$-category can be found in \eqref{eq:interchange-general}--\eqref{eq:triangulator-left-right}.

\begin{remark}\label{rem:univ-prop-diff}
    A proof that $\bimod{\cC}$ with tensor product given by $\boxtimes_\cC$ is a monoidal $2$-category is provided in \cite{GreThe}. We note that in this proof, a stricter version of the universal property of the relative Deligne product is used where the equivalence \eqref{eq:universal-Del} is replaced by an isomorphism of categories. The existence of the relative Deligne product in \cite{DSS1} is proved with \eqref{eq:universal-Del} being an equivalence, not an isomorphism.\footnote{In the terminology of $2$-categorical colimits, the universal properties used here and in \cites{Del,DSS1} are those of a \emph{bicolimit}, while the stricter ones used in \cite{GreThe}*{Rem.\,2.1.6} or \cite{ENO}*{Rem.\,3.4} are those of a \emph{pseudo-colimit} in \cite{KelLim}*{(5.7)}. Existence proofs such as the one in \cite{DSS1} are given for the bicolimit object.} Therefore, as we could not find it in the literature, we sketch the proof that $\bimod{\cC}$ is a monoidal $2$-category based on the universal property of \Cref{def:rel-Del} in \Cref{sec:rel-Del-props}.
\end{remark}

Assume now that $\cC$ is braided. Given a left $\cC$-module $\cM$, we can give $\cM$ the structure of a right $\cC$-module, denoted by $\rS_{lr}(\cM)$, with right $\cC$-action defined by $M\ract X:=X\lact M$ and coherence defined in components by
\begin{align}\label{eq:right-coherence}
\vcenter{\hbox{
\xymatrix@R=10pt@C=30pt{
(M\ract X)\ract Y\ar[rr]^{r_{M,X,Y}}\ar@{=}[d]  && M\ract (X\otimes Y)\ar@{=}[d]\\
Y\lact (X\lact M)\ar[r]^{l_{Y,X,M}}&(Y \otimes X)\lact M\ar[r]^{\psi_{Y,X}\lact M}& (X\otimes Y)\lact M.
}}}
\end{align} 
Moreover, we can give $\cM$ the structure of a $\cC$-bimodule, denoted by $\rB(\cM)$ with the same left $\cC$-module structure, the right $\cC$-module structure $\rS_{lr}(\cM)$, and the bimodule coherence defined by 
\begin{align}\label{eq:bimod-coherence}
\vcenter{\hbox{
\xymatrix@C=10pt{
Y\lact( X\lact M)\ar[d]_{l_{Y,X,M}}&(X\lact M)\ract Y\ar@{=}[l] \ar[rrr]^{b_{X,M,Y}}& &&X\lact (M\ract Y)\ar@{=}[r]&X\lact (Y \lact M)\\
Y\otimes X\lact M\ar[rrrrr]^{\psi_{Y,X}\lact M}&&&&& X\otimes Y \lact M.\ar[u]_{l^{-1}_{X,Y,M}}
}}}
\end{align} Indeed, these constructions extend to $2$-functors 
\begin{equation}\label{eq:Slr-B}
    \rS_{lr}\colon \lmod{\cC}\to \rmod{\cC}, \qquad \text{and}\qquad \rB\colon \lmod{\cC}\to \bimod{\cC}.
\end{equation}
Proofs are provided in \Cref{sec:left-to-bimod}.

\begin{definition}\label{def:rel-Del-Cmod}
For two finite left $\cC$-module categories $\cN$, $\cM$, we define $\cN\boxtimes_\cC \cM$ as the relative Deligne product $\rB(\cN)\boxtimes_\cC\rB(\cM)$ of $\cC$-bimodules viewed as a left $\cC$-module.
\end{definition}

\begin{remark}
We make a choice using the braiding in the right $\cC$-module coherence of \eqref{eq:right-coherence}. Alternatively, we could use the inverse braiding but this choice corresponds to replacing $\cC$ by $\cC^{\mathrm{rev}}$, the same tensor category with reversed braiding. With any given choice, we could also use inverse braiding for defining the bimodule coherence in \eqref{eq:bimod-coherence}. These alternative choices give a priori different monoidal structures for $\lmod{\cC}$. In this paper we focus on the chosen one.
\end{remark}

The results of this paper concern the monoidal $2$-category $\lmod{\cC}$ of left $\cC$-module categories described in the following widely accepted result (cf.\,\cites{DN2,DR}).

\begin{theorem}\label{thm:CMod-monoidal-2-cat}
The relative Deligne product makes the $2$-category $\lmod{\cC}$ a monoidal $2$-category.  It is a monoidal $2$-subcategory of $\bimod{\cC}$ via the $2$-functor $\rB$ from \eqref{eq:Slr-B}. 
\end{theorem}
\begin{proof}

We will now sketch the proof that the $2$-functor $\rB\colon \lmod{\cC}\to \bimod{\cC}$ is one of monoidal $2$-categories. 
For the relative Deligne product $\cN\boxtimes_\cC\cM$ of two left $\cC$-module categories $\cM,\cN$, we have that $\rB(\cN\boxtimes_\cC\cM)= \rB(\cN)\boxtimes_\cC\rB(\cM)$ are, by \Cref{def:rel-Del-Cmod}, the same left $\cC$-module categories. We can hence equip $\lmod{\cC}$ with the same associators and unitors as $\bimod{\cC}$, forgetting the right $\cC$-linear structures on their component functors $A_{\cM,\cN,\cP}$, $L_\cM$, and $R_\cM$. This way, $\lmod{\cC}$ becomes a monoidal $2$-category. 

We will equip the identity functor on $\cN\boxtimes_\cC\cM$ with the structure of a right $\cC$-module functor $ I_{\cN,\cM}\colon \rB(\cN\boxtimes_\cC\cM)\to \rB(\cN)\boxtimes_\cC\rB(\cM)$. One checks that universal balancing of \Cref{def:rel-Del} provides a right $\cC$-linear structure
     $$
     s^{I_{\cN,\cM}}_{N\boxtimes_\cC M, X}\colon (N\ract X)\boxtimes_\cC M=(N\boxtimes_\cC M)\ract X\xrightarrow{\beta^P_{N,X,M}}N\boxtimes_\cC( M\ract X)=N\boxtimes_\cC( X\lact M).
     $$
Indeed, this gives  $I_{\cN,\cM}$ a $\cC$-bimodule structure. 
     Now, the equivalences $I_{\cN,\cM}$ equip the $2$-functor 
 $\rB\colon \lmod{\cC}\to \bimod{\cC}$ with the structure of a monoidal $2$-functor. Hence, $\lmod{\cC}$ is a monoidal $2$-subcategory of $\bimod{\cC}$.
\end{proof}

\subsection{Module categories of module functors}
\label{sec:modulecat-rex}

Let $\cC$ be a braided tensor category.
We fix conventions on how to turn categories of (right exact) $\cC$-module functors  $\Fun_\cC(\cM,\cN)$ into $\cC$-module categories where the braiding enters the coherence data. 

\begin{definition}\label{def:Fun-action}
Let $X,Y\in \cC$ and $(G,s^G)\colon \cM\to \cN$ be a $\cC$-module functor. We define
$$
(X\lact G)(M):= X\lact G(M)$$
    and 
    $$s^{X\lact G}_{Y,M}\colon (X\lact G)(Y\lact M)\to (Y\lact (X\lact G))(M)$$
via the commutative diagram
    \begin{align}
    \vcenter{\hbox{
        \xymatrix@R=12pt{
        (X\lact G)(Y\lact M)\ar@{=}[d]\ar[rrrr]^{s^{X\lact G}_{Y,M}}&&&&(Y\lact (X\lact G))(M)\ar@{=}[d]\\
        X\lact G(Y\lact M)\ar[dd]^{X\lact s^G_{Y,M}}\ar[rr]^{(s^G_{X,Y\lact M})^{-1}}&& G(X\lact (Y\lact M))\ar[d]^{G(l^\cM_{X,Y,M})}&&Y\lact (X\lact G(M))\\
&&G(X\otimes Y\lact M)\ar[d]^{s^G_{X\otimes Y, M}}&&       
        \\
        X\lact (Y\lact G(M))\ar[rr]^{l^\cN_{X,Y,G(M)}}&& X\otimes Y\lact G(M)\ar[rr]^{\psi_{X,Y}\lact \id_{G(M)}}&&Y\otimes X\lact G(M)\ar[uu]^{(l^{\cN}_{Y,X,G(M)})^{-1}}.
        }}}\label{eq:FunC-coherence}
    \end{align}
\end{definition}
The family of isomorphism $s^{X\lact G}_{Y,M}$ is natural in $Y$ and $M$ by construction. Note that we will later  require the above coherence $s^{X\lact G}$ in case the module categories $\cM$ and $\cN$ are not strict (see \Cref{sec:Sweedler} where we consider module categories obtained from non-strict monoidal functors) and thus include the module associators in the formula.

\begin{lemma}\label{lem:Hom-C-mod}
\Cref{def:Fun-action} provides a left $\cC$-module structure on  $\Fun_\cC(\cM,\cN)$ with the module associator written in terms of components as:
    $$(l_{X,Y,G})_{M}\colon X\lact(Y\lact G(M))=(X\lact(Y\lact G))(M) \xrightarrow{l^\cN_{X,Y,G(M)}} (X\otimes Y\lact G)(M)=X\otimes Y\lact G(M),$$
    for $G\in \Fun_\cC(\cM,\cN)$ and $X,Y\in \cC$. 
\end{lemma}
\begin{proof}
By \Cref{rem:strictification}, both module categories $\cM$ and $\cN$ are equivalent to strict module categories. To streamline the proof, we will display the proof of the claim that the module associator on the target category $\cN$ induces an associator on $\Fun_\cC(\cM,\cN)$ only for $\cM, \cN$ being strict left $\cC$-module categories, i.e.\ we assume that $l^\cM$ and $l^\cN$ have identity components. 

We first check that the pair $(X\lact G, s^{X\lact G})$ indeed defines a $\cC$-linear functor, i.e.\ satisfies the coherence \eqref{eq:C-linear-coherence}. This follows from commutativity of the outer diagram in: 
$$
\xymatrix{
(X\lact G)(Y\otimes Z\lact M)\ar[rr]^-{s^{X\lact G}_{Y\otimes Z,M}}\ar[rd]|{X\lact s^{G}_{Y\otimes Z,M}} \ar[dd]|{X\lact s^G_{Y,Z\lact M}}
\ar `l/8pt[d] `[dddr] [dddr]_-{s^{X\lact G}_{Y,Z\lact M}} && Y\otimes Z \lact (X\lact G)(M)&\\
&X\otimes Y\otimes Z\lact G(M)\ar[ur]|{\psi_{X,Y\otimes Z}\lact G(M)}\ar[dr]|{\psi_{X,Y}\otimes Z\lact G(M)}&\\
X\otimes Y\lact G(Z\lact M)\ar[ru]|{X\otimes Y\lact s^G_{Z,M}}\ar[dr]|{\psi_{X,Y}\lact G(Z\lact M)}&&Y\otimes X\otimes Z\lact G(M)\ar[uu]|{Y\otimes \psi_{X,Z}\lact G(M)}&\\
&Y\otimes X\lact G(Z\lact M)\ar[ru]|{Y\otimes X\lact s^G_{Z,M}}
\ar@{->}`r[rru]_-{Y\lact s^{X\lact G}_{Z, M}}`u[uuu][uuur]
&&
}
$$
In this diagram, we use that all module associators are trivial. The triangles commute by definition of $s^{X\lact G}$, coherence of $s^G$, and the hexagon axiom for $\psi_{X,Y\otimes Z}$ in the strict case. The inner square clearly commutes. Note that $X\lact G$ is right exact provided that $G$ is right exact using that $\lact$ is exact in both components. This shows well-definedness of the $\cC$-action on objects. It is clear that this action extends to morphisms of $\cC$-linear functors. 

Moreover, the proposed module associator is indeed a morphism in $\Fun_\cC(\cM,\cN)$. By construction, it is natural in all arguments. In the strict case, the module associator is given by identity morphisms in $\cM$, this only requires checking that $s^{X\lact (Y\lact G)}$ and $s^{X\otimes Y\lact G}$ are the same morphism. This follows from the following calculation:
\begin{align*}
s^{X\lact (Y\lact G)}_{Z,M}&=(\psi_{X,Z}\otimes Y\lact G(M))(X\lact s^{Y\lact G}_{Z,M})\\
    &=(\psi_{X,Z}\otimes Y\lact G(M))(X\otimes \psi_{Y,Z}\lact G(M))(X\otimes Y \lact s^G_{Z,M})\\
    &=(\psi_{X\otimes Y,Z}\lact G(M))(X\otimes Y \lact s^G_{Z,M})
    =s_{Z,M}^{X\otimes Y\lact G}\ .
\end{align*}
Here, the first, second and last step use \Cref{def:Fun-action} while the third step uses the hexagon axioms for the braiding $\psi$.
\end{proof}

\begin{remark}\label{rem:FunC-bimod-compare}
    For the $\cC$-bimodule categories $\rB(\cM), \rB(\cN)$ from \Cref{lem:S-bimod},  $\Fun_\cC(\rB(\cM),\rB(\cN))$ has a $\cC$-$\cC$-bimodule structure by \cite{DSS2}*{Section~2.4}. It turns out that this $\cC$-bimodule structure is equivalent to $\rB(\Fun_\cC(\cM,\cN))$ for the left $\cC$-module structure on $\Fun_\cC(\cM,\cN)$ from \Cref{def:Fun-action}. An equivalence of left $\cC$-modules from $\Fun_\cC(\rB(\cM),\rB(\cN))$ to $\rB(\Fun_\cC(\cM,\cN))$ 
   is given by the identity functor $\id\colon \Fun_\cC(\cM,\cN)\to \Fun_\cC(\cM,\cN)$ with $\cC$-linear structure $(s^{\id}_{X,G})_M:=s^G_{X,M}$ for any $\cC$-linear functor $(G,s^G)\colon \cM\to \cN$ and $X$ in $\cC$.
\end{remark}

\subsection{Centralizer functors} \label{sec:action-functor}

Following \cites{EGNO,EO1} we define the \emph{centralizer} of a finite left $\cC$-module category $\cM$ to be $$\cC_\cM^*:=\Fun_\cC(\cM,\cM),$$
which is a strict monoidal  category with $G\otimes F=G\circ F$,
for $F,G\in \cC_\cM^*$, with the $\cC$-linear structure defined in \eqref{eq:compose-C-lin}.
We will equip the action functor $\lact \colon \cC\to \End(\cM)$ with the structure of a monoidal functor valued in the centralizer of $\cM$. In the context of $\cC$-bimodule categories, two such monoidal functors appear in \cite{ENO}*{Section~4.1} and are reviewed in \Cref{sec:bimod}.

\begin{definition}\label{def:AC-C-mod-functor}
Let $\cM$ be a left $\cC$-module of a braided tensor category $\cC$. For an object $X$ in $\cC$, we define the \emph{\action functor} to be the functor
$$A_\cM\colon \cC\to \cC^*_\cM, \quad X\mapsto X\lact \id_\cM,$$
with the left $\cC$-linear structure for $X\lact \id_\cM$ introduced in \Cref{def:Fun-action}.
\end{definition}

\begin{lemma}\label{lem:AC-braided}
Given a $\cC$-module category $\cM$, the functor 
$A_\cM$
is a monoidal functor. The action of $\cC$ on $\cC_\cM^*$ of \Cref{lem:Hom-C-mod} is induced by this functor in the sense of \Cref{ex:tensor-act}.
\end{lemma}
\begin{proof}
We use the left module coherence 
$$\mu^{A_\cM}:=l_{X,Y,M}\colon A_\cM(X)\circ A_\cM(Y)(M)\isomorph A_\cM(X\otimes Y)(M)$$
to define the monoidal structure on the functor $A_\cM$. The compatibility \eqref{eq:mu-coherence} for $\mu^{A_\cM}$ is equivalent to the left module coherence \eqref{eq:mod-pentagon-l} and the unit axiom for $\mu^{A_\cM}$ holds by \eqref{eq:mod-fun-unit}.
Thus, $A_\cM$ is a monoidal functor. 

We now compare the $\cC$-actions. For $G\in \Fun_\cC(\cM,\cM)$ we have 
$$(X\lact G)(M)=X\lact G(M)=((X\lact \id_\cM)\circ G)(M)= (A_\cM(X)\circ G)(M).$$
By comparison, it follows that the $\cC$-module structure on $A_\cM(X)\circ G$, obtained by composition, coincides with the $\cC$-module structure of $X\lact G$ of \eqref{eq:FunC-coherence}.
\end{proof}

\begin{remark}\label{rem:alpha-induction}
    The monoidal functor $A_\cM$ is also called \emph{$\alpha$-induction} in the literature \cites{Ost1,FFRS,DN2} inspired by constructions in the theory of subfactors \cite{BEK}. Comparing to the conventions of \cite{DN2}*{(4.4)} we find that $\alpha_\cM^+=A_\cM\circ I^+$, for the monoidal equivalence $I^+$ from \eqref{eq:I+-}. 
\end{remark}

The following lemma is proven in \cite{EGNO}*{Rem.\,4.3.10}.\footnote{Our assumptions on the target $\cD$ are more general than the assumptions in \cite{EGNO}. The source $\cC$ does not need to be finite.}
\begin{lemma}\label{lem:F-faithful}
    Let $F\colon \cC \to \cD$ be an exact monoidal functor where $\cC$ is
    a tensor category and $\cD$ is a $\Bbbk$-linear abelian monoidal category. Then $F$ is faithful. In particular, for an object $X$ in $\cC$, if $F(X)=0$ then $X=0$ 
\end{lemma}
\begin{proof}
    Let $X$ be a non-zero object in $\cC$. The non-zero morphism $\coev_X\colon \one_\cC \to X\otimes X^*$ is a monomorphism since $\one_\cC$ is a simple object. Now 
    $$\one_\cD\cong F(\one_\cC)\xrightarrow{F(\coev_X)}F(X\otimes X^*)\cong F(X)\otimes F(X^*) $$
    is a monomorphism since $F$ is left exact. Hence, $F(X)$ is non-zero. Therefore, $F$ is faithful since if $F(f)=0$, then $\Img F(f)\cong F(\Img f)=0$, using left and right exactness of $F$. Hence, $\Img f=0$ and $f=0$.
\end{proof}

\begin{corollary}
\label{lem:AM-exact-faithful}
    The \action functor $A_\cM$ is exact and faithful. In particular, if  
    for an object $X$ in $\cC$, $A_\cM(X)=0$ then $X=0$.   
\end{corollary}
\begin{proof}
First, $A_\cM$ is exact by assumption that for any module category $\cM$, $\lact$ is exact in both variables. Thus, the claim follows from \Cref{lem:F-faithful}.
\end{proof}

\subsection{Opposite module categories}
\label{sec:opposite-mod-cats}

We will now recall opposite module categories following \cite{DSS2}*{Def.\,2.4.4} to associate right module categories to left module categories.

\begin{definition}\label{def:*M}
Let $\cN$ be a right $\cC$-module category. Define the \emph{right opposite} ${}^\#\cN$ to be the opposite category $\cN^\oop$ with the structure of a left $\cC$-module category given by 
\begin{align}\label{eq:act-left-*M}
    X\lact N=N\ract {}^*X,
\end{align}
where ${}^*X$ is the \emph{right} dual of $X$.\footnote{\cite{DSS2} denote the right dual of $X$ by $X^*$. Our conventions on duality are defined in \Cref{sec:duality}.} The module coherence is given by the morphism 
$${}^{\#}l_{X,Y,N}\in \Hom_{\cN^\oop}(X\lact (Y\lact N),X\otimes Y\lact N)$$
corresponding to the following morphism in $\cN$:
\begin{align}\label{eq:opposite-coherence}
\vcenter{\hbox{
    \xymatrix@R=10pt{
    X\otimes Y\lact N\ar@{=}[d]\ar[rrrr]^{{}^\#l_{X,Y,N}} &&&&X\lact (Y\lact N)\ar@{=}[d]\\
    N\ract {}^*(X\otimes Y) \ar@{=}[rr]&&N\ract {}^*Y\otimes {}^*X\ar[rr]^{r^{-1}_{N,{}^*Y,{}^*X}} &&  (N\ract {}^*Y)\ract {}^*X.
    }}}
\end{align}
Here, $r_{N,X,Y}$ is the given coherence of the right $\cC$-module structure on $\cN$. Similarly, given a left $\cC$-module category $\cM$, we can define a right $\cC$-module category ${}^\#\cM$ with the action 
\begin{align}\label{eq:act-right-*M}
    M\ract X={}^*X\lact M.
\end{align}
\end{definition}

\begin{definition}[{\cite{DSS2}*{Def.\,2.4.4}}]\label{def:M*}
For right $\cC$-module category $\cN$, we define the \emph{left opposite} $\cN^\#$ to be the opposite category $\cN^\oop$ with the structure of a left $\cC$-module category given by 
\begin{align}\label{eq:act-left-M*}
    X\lact N=N\ract X^*,
\end{align}
where $X^*$ is the left dual of $X$. The module coherence $r^\#_{N,X,Y}$ is obtained from $r^{-1}_{N,X,Y}$ similarly to how ${}^\#l_{X,Y,N}$ is obtained from $r^{-1}_{N,X,Y}$ in \Cref{def:*M} but using left duals instead. Analogously, we define a left opposite right $\cC$-module $\cM^\#$ for a given right $\cC$-module $\cM$ with action
\begin{align}\label{eq:act-right-M*}
    M\ract X=X^*\lact M.
\end{align}
\end{definition}
See also \cite{FSS}*{Section~2.4} where these opposites are used in the context of bimodule categories.

\section{Internal algebras and constructions of module categories}\label{sec:int-alg}

In this section, we describe all required constructions in the monoidal $2$-category $\lmod{\cC}$, for $\cC$ a finite braided tensor category, via internal algebras in $\cC$. 
In \Cref{sec:fun-bimodules}, we recall how left $\cC$-modules and their functor categories are realized by algebras and bimodules in $\cC$, cumulating in the biequivalence of $\lmod{\cC}$ with the Morita bicategory $\Mor_\cC$ in \Cref{thm:morita}. In \Cref{sec:MorC-monoidal} we recall the monoidal structure on $\Mor_\cC$, which in \Cref{sec:MorC-C-mod-monoidal} is linked to the monoidal structure of $\lmod{\cC}$, see \Cref{prop:rel-Del-modules}. Finally, duality in $\lmod{\cC}$ is discussed in \Cref{sec:duals}.

\subsection{Module categories and module functors via internal algebras}
\label{sec:fun-bimodules}

We will now explain the key tool of using internal algebras to study module categories following \cites{Ost1,EO1,EGNO,DSS2}.

    Let $A$ be an algebra in $\cC$. Then $\rmodint{\cC}{A}$ denotes the left $\cC$-module category of right $A$-modules in $\cC$ with left $\cC$-action of $X$ in $\cC$ on a right $A$-module $(V,a^r_V)$  given by 
$$X\lact (V, a^r_V)=(X\otimes V, a_{X\otimes V}^r=\id_X\otimes a_V^r).$$
We will often denote an $A$-module $(V,a^r_V)$ simply by $V$. In particular, the regular $A$-module $(A, m_A)$ is  simply denoted by $A$. 
The module coherence is given by the associator of $\otimes$, which we assume to be given by identities as $\cC$ is strict. 
Similarly, one defines the \emph{right} $\cC$-module category $\lmodint{A}{\cC}$ of \emph{left} $A$-modules in $\cC$.

Recall that if $\cC$ is a finite tensor category, then any finite left $\cC$-module category $\cM$ is equivalent to  $\rmodint{\cC}{A}$ for some algebra $A$ in $\cC$, see \cite{EGNO}*{Section~7.10} or \cite{DSS1}*{Thm.\,2.18} where no assumptions on $\Bbbk$ are made. 

\begin{remark}\label{rem:strictification}
If $\cC$ is strict monoidal, then $\rmodint{\cC}{A}$ is strict as a $\cC$-module category.
Hence, the equivalence $\cM\simeq \rmodint{\cC}{A}$ can be viewed as a \emph{strictification} of the finite left $\cC$-module category $\cM$. 
\end{remark}

We will now describe $\Fun_{\cC}(\cM,\cN)$ in terms of internal algebras. For this, define a left $\cC$-action on the category $\bimodint{A}{\cC}{B}$ of internal $A$-$B$-bimodules in $\cC$. For an internal $A$-$B$-bimodule $V$ and $X$ in $\cC$, define 
$X\lact V$ as the object $X\otimes V$
 with the right $B$-action
\begin{align}\label{eq:bimod1}
a^r_{X\otimes V}\colon X\otimes V\otimes B\xrightarrow{\id_X\otimes a^r_{V}}X\otimes V
\end{align}
and left $A$-action
\begin{align}\label{eq:left-A-bimodule}
a^l_{X\otimes V}\colon A\otimes X\otimes V\xrightarrow{\psi^{-1}_{A,X}\otimes \id_V}X\otimes A\otimes V\xrightarrow{\id_X\otimes a^l_{V}} X\otimes V.
\end{align}
One checks that this is a left $A$-action using naturality of the inverse braiding and the hexagon axioms.
The $\cC$-module coherence for $\bimodint{A}{\cC}{B}$ is given by the associativity constraint of $\cC$, i.e.\ trivial by our assumption of strictness. 

The following proposition upgrades the equivalence of categories from \cite{EGNO}*{Prop.\,7.11.1} or \cite{DSS2}*{Prop.\,2.4.10} to one of $\cC$-module categories.
\begin{proposition} \label{prop:fun-A-bimod-B}
Let $A,B$ be algebras in $\cC$ and $\cM=\rmodint{\cC}{A}$, $\cN=\rmodint{\cC}{B}$. We have the following  equivalence of left $\cC$-module categories
    $$E\colon \Fun_\cC(\cM,\cN)\to \bimodint{A}{\cC}{B}, \quad (G,s^G)\mapsto G(A),$$
where the right $B$-module $G(A)$ is equipped with the left $A$-action
\begin{align} \label{eq:EW-left-act}a^l_{E(G)}:=\left(A\otimes G(A)=A\lact G(A)\xrightarrow{(s^{G}_{A,A})^{-1}}G(A\lact A)\xrightarrow{G(m_A)} G(A)\right).
\end{align}
The left $\cC$-linear structure is given by the identity morphisms 
$$s^E_{X,G}=\id\colon E(X\lact G)=X\otimes G(A)\to X\otimes G(A)=X\lact E(G).$$
\end{proposition}

\begin{proof}
For a $\cC$-module functor $(G,s^G)$, note that $G(A)\in \cN=\rmodint{\cC}{B}$ and that $a^l_{E(G)}$ commutes with the right $B$-action since $a^l_{E(G)}$ is, by construction, a morphism in $\rmodint{\cC}{B}$. Thus, $E(G)$ is a well-defined $A$-$B$-bimodule. It is well-known that $E$ gives an equivalence of categories, see \cite{EGNO}*{Prop.\,7.11.1} or \cite{DSS2}*{Prop.\,2.4.10}. An inverse equivalence to $E$ sends an internal $A$-$B$-bimodule $V$ to the functor 
$$G_V= (-)\otimes_A V\colon \cM\to \cN, \quad M\mapsto M\otimes_A V.$$
Using that $\cC$ is strict monoidal, both right $B$-modules $(X\lact M)\otimes_A V$ and $X\lact (M\otimes_A V)$ are coequalizers of the same diagram. 
Thus, $G_V$ becomes a $\cC$-linear functor with coherences induced by identity morphisms.

It remains to show that $E$ is a functor of left $\cC$-modules. Recall the left $\cC$-module structure on $\bimodint{A}{\cC}{B}$ defined above in \eqref{eq:bimod1}--\eqref{eq:left-A-bimodule} and the action on $\Fun_\cC(\cM,\cN)$ from \Cref{def:Fun-action}. Note that 
$$
E(X\lact G)=(X\lact G)(A)= X\otimes G(A)=X\lact E(G)$$
as objects in $\cN=\rmodint{\cC}{B}$. That the left $A$-actions coincide follows from commutativity of the outer diagram in
\begin{align*}
\vcenter{\hbox{
\xymatrix@C=10pt{
A\lact (X\lact G)(A)\ar[d]_{(s^{X\lact G}_{A,A})^{-1}}
\ar@{=}[rr] 
\ar `l/15pt[d] `[dd]_-{a^l_{E(X\lact G)}} [dd]
&&A\otimes X\otimes G(A)\ar[d]^{\psi^{-1}_{A,X}\otimes A}\ar@{=}[r]&A\otimes (X\lact G(A))\ar[dd]^{a^l_{X\otimes G(A)}}\\
(X\lact G)(A\lact A)\ar[d]_{(X\lact G)(m_A)}^{=X\otimes G(m_A)}\ar[rr]^{X\otimes s^G_{A,A}} &&X\otimes A\otimes G(A)\ar[d]^{X \otimes a^l_{E(G)}} \\
(X\lact G)(A)\ar@{=}[rr] &&X\otimes G(A)\ar@{=}[r]&X\lact G(A)
}}}
\end{align*}
The leftmost vertical arrow is the left $A$-action defined in \eqref{eq:EW-left-act}, with $s^{X\lact G}$ defined in \Cref{def:Fun-action}, and the rightmost vertical arrow is the $A$-action from \eqref{eq:left-A-bimodule}. The middle lower square also commutes by \eqref{eq:EW-left-act}.
As the $\cC$-linear structure of $E$ is given by identities, the coherence \eqref{eq:C-linear-coherence} is trivially satisfied and $E$ is an equivalence of $\cC$-module categories.
\end{proof}

\begin{remark}\label{rem:convention1}
    We fixed a choice of convention in  \Cref{def:Fun-action} by using the braiding $\psi$ in the definition of the $\cC$-linear structure $s^{X\lact G}$. This choice corresponds to  using the inverse braiding $\psi^{-1}$ in \eqref{eq:left-A-bimodule}. Instead, one can use the inverse braiding in \Cref{def:Fun-action} to obtain an alternative left $\cC$-module structure $\Fun'_\cC(\cM,\cN)$ on the same category. Correspondingly, the inverse braiding in \eqref{eq:left-A-bimodule} would be replaced with the braiding to produce an analogue of \Cref{prop:fun-A-bimod-B} with this convention. Our choice of convention is justified later in \Cref{rem:convention2}.
\end{remark}

\begin{definition}\label{def:MorC}
For an abelian monoidal category $\cC$ with right exact tensor product, we define the \emph{Morita bicategory} $\Mor_\cC$ to be the $2$-category whose objects are algebras in $\cC$, morphisms $A\to B$ are $A$-$B$-bimodules, and with $2$-morphisms given by bimodule morphisms.
The composition of $1$-morphisms in $\Mor_\cC$ is defined by the relative tensor product 
$$M\otimes_A N=\operatorname{coker} \Big(\xymatrix{M\otimes A \otimes N\ar[rrr]^{a^r_M\otimes N - M\otimes a^l_N}&&&M\otimes N}\Big).$$
This relative tensor product is not strictly associative. However, there are canonical isomorphisms 
$$M\otimes_A (N\otimes_A P)\isomorph (M\otimes_A N) \otimes_A P$$
compatible with the colimit diagrams. Moreover, there are canonical isomorphisms 
$$M\otimes_A A\cong A, \qquad A\otimes_A M\cong M.$$
\end{definition}

\Cref{prop:fun-A-bimod-B}, together with  \cite{EGNO}*{Section~7.10} or \cite{DSS1}*{Thm.\,2.18}, gives the following well-established result, which enables describing the $2$-category $\lmod{\cC}$ of finite $\cC$-modules.

\begin{proposition}\label{thm:morita}
Let $\cC$ be a finite tensor category. The assignment$$\Mod\colon \Mor_\cC\to \lmod{\cC}, \quad A\mapsto \rmodint{\cC}{A}$$
extends to an equivalence of bicategories.
\end{proposition}

\subsection{Operations on internal algebras in a braided tensor category}\label{sec:MorC-monoidal}

If $\cC$ is braided, we can define opposite algebras and tensor products of algebras in $\cC$. We use the following notation for braided opposite and tensor product algebras.

\begin{definition}[$A^\psi$]\label{def:opp-algebras}
Given an algebra $A$ in $\cC$ with product $m_A\colon A\otimes A \to A$, we denote by $A^\psi$ the \emph{$\psi$-opposite algebra} of $A$. It is defined on the same object $A$ with $\psi$-opposite product
$$m^{\psi}_{A}:=m_A\psi_{A,A}\colon A\otimes A\to A.$$
\end{definition}

\begin{definition}[$A\otimes^\psi B$]\label{def:tensor-algebras}
Given two algebras $A,B$ in $\cC$ with products $m_A,m_B$, respectively, we denote by $A\otimes^\psi B$ the algebra given by $A\otimes B$ with product 
$$m_{A\otimes^\psi B}:= (m_A\otimes m_B)(\id_A\otimes \psi_{B,A}\otimes \id_B)\colon (A\otimes B)\otimes (A\otimes B)\to A\otimes B.$$
\end{definition}
The same way, we define $A^{\psi^{-1}}$ and $A\otimes^{\psi^{-1}}B$, using the inverse braiding $\psi^{-1}$ instead of the braiding. For simplicity, we will often denote $A\otimes^\psi B$ simply by $A\otimes B$.
This tensor of algebras gives $\Mor_\cC$ a monoidal structure.

\begin{proposition}
For $\cC$ braided, the tensor product $A\otimes B$ of internal algebras as in \Cref{def:tensor-algebras}, with the corresponding tensor product of bimodules and bimodule morphisms makes $\Mor_\cC$ a  monoidal bicategory. 
\end{proposition}
Note that the tensor product of $\Mor_\cC$ is strictly associative. This follows from strict associativity of the tensor product of $\cC$. Thus, we can say that $\Mor_\cC$ is a \emph{semi-strict} monoidal bicategory, cf.\ \Cref{rem:2-strictification}.

Next, we collect some facts on tensor products on internal algebras. These are not difficult to check and we omit the proofs here.

\begin{lemma}\label{lem:tensor-ops}
    Let $A,B$ be algebras in $\cC$ then 
    \begin{align*}
    \psi^{-1}_{A,B}\colon (A\otimes^\psi B)^\psi &\to B^\psi\otimes^\psi A^\psi, && \text{and} &&
    \psi_{A,B}\colon (A\otimes^\psi B)^{\psi^{-1}}\to B^{\psi^{-1}}\otimes^\psi A^{\psi^{-1}}
    \end{align*}
    define isomorphisms of algebras in $\cC$.
\end{lemma}

\begin{lemma}\label{lem:tensor-ops2}
    Let $A,B$ be algebras in $\cC$ then 
$\psi^{-1}_{A,B}\colon A\otimes^\psi B \to B\otimes^{\psi^{-1}} A$
    defines an isomorphism of algebras in $\cC$.
\end{lemma}

\subsection{Relative Deligne products}
\label{sec:MorC-C-mod-monoidal}

We will now describe the relative Deligne product using internal algebras in a finite braided tensor category $\cC$. 
 For algebras $A,B$ in $\cC$, consider the right $\cC$-module $\cN=\lmodint{A}{\cC}$ and the left $\cC$-module $\cM=\rmodint{\cC}{B}$.
In the following, we will identify the relative Deligne product $\cN\boxtimes_\cC\cM$ of \Cref{def:rel-Del} with the  realization
\begin{equation}\label{eq:rel-Del-realization}
\lmodint{A}{\cC}\boxtimes_\cC\rmodint{\cC}{B}=\bimodint{A}{\cC}{B}
\end{equation}
provided in \cite{DSS1}*{Thm.\,3.3}.
The canonical balanced functor is given explicitly by 
\begin{equation}\label{eq:balancing-AB-bimod}
P_{\cN,\cM}\colon \lmodint{A}{\cC}\times \rmodint{\cC}{B}\to \bimodint{A}{\cC}{B}, \quad (V,W) \mapsto V\otimes W.
\end{equation}
The left $A$-action and right $B$-action are given by 
$$a^l_{V\otimes W}=a^l_{V}\otimes W, \qquad a^r_{V\otimes W}=V\otimes a^r_{W},$$
where $a^l_V$ and $a^r_W$ denote the left $A$-action on $V$ and the right $B$-action on $W$, respectively. The $\cC$-balancing is given by identities.

Under the identification $\lmodint{A}{\cC}=\bimodint{A}{\cC}{\one}$, we have a left $\cC$-module structure on $\lmodint{A}{\cC}$ from \eqref{eq:bimod1}--\eqref{eq:left-A-bimodule}. The left $\cC$-action is given by
\begin{equation}\label{eq:left-A-action}
X\lact (V,a^l_V)=(X\otimes V,a^l_{X\otimes V}), \quad \text{where}\quad  a^l_{X\otimes V}=(X\otimes a^l_V)\circ (\psi^{-1}_{A,X}\otimes V)\colon A\otimes X\otimes V\to X\otimes V,
\end{equation}
with the trivial left $\cC$-module coherence.
This way, $\lmodint{A}{\cC}$ is a $\cC$-bimodule, with the natural right $\cC$-action and trivial bimodule coherence. Similarly, $\rmodint{\cC}{B}=\bimodint{\one}{\cC}{B}$ becomes a $\cC$-bimodule.

\begin{proposition}\label{rem:rel-Del-bimodule-realization}
Let $A,B$ be algebras in $\cC$. Then, as a left $\cC$-module, the relative Deligne product $\lmodint{A}{\cC}\boxtimes_\cC\rmodint{\cC}{B}$  of $\cC$-bimodules from \Cref{prop:rel-Del-bimod} is equivalent to $\bimodint{A}{\cC}{B}$, with the left $\cC$-action defined by \eqref{eq:bimod1}--\eqref{eq:left-A-bimodule}.
\end{proposition}

\begin{proof}
The left $\cC$-action \eqref{eq:left-A-action} induces a left $\cC$-module structure on $\lmodint{A}{\cC}\boxtimes_\cC\rmodint{\cC}{B}$ by \Cref{prop:rel-Del-bimod}. A description of this left $\cC$-action was provided in the second paragraph of the proof of \cite{DSS2}*{Thm.\,2.5.5}. In our context of the bimodule $\cM=\lmodint{A}{\cC}$ described around \eqref{eq:left-A-action}, this left $\cC$-action is constructed as follows.
Consider the monad 
$T=(-)\ract B\colon \cM\to \cM$. We can identify $\bimodint{A}{\cC}{B}$ with the category $\lmodint{T}{\cM}$  of modules over the monad $T$. The monad $T$ is one of right $\cC$-module categories, with trivial $\cC$-linear structure because the bimodule coherence of $\cM$ is strict. If $\big(M,a^T_M\colon T(M)\to M\big)$ is an object of $\lmodint{T}{\cM}$, define 
$$X\lact \big(M,a^T_M\big)=\bigg(X\lact M, a^T_{X\lact M}:= \Big(T(X\lact M)=X\lact T(M)\xrightarrow{X\lact a^T_M} X\lact M\Big)\bigg),$$
with trivial $\cC$-module coherence. 
Under the identification $\bimodint{A}{\cC}{B}= \lmodint{T}{\cM}$, $(M,a^T_M)$ corresponds to a right $B$-module $\big(M,a^r_M\colon M\otimes B\to M\big)$ and the above left $\cC$-action becomes
\begin{equation}\label{eq:induced-action-T}
X\lact \big(M,a^r_M\big)=\bigg(X\lact M, a^r_{X\lact M}:= \Big((X\lact M)\otimes B=X\lact (M\otimes B)\xrightarrow{X\lact a^r_M} X\lact M\Big)\bigg),
\end{equation}
where the left $A$-module $X\lact M$ was defined in \eqref{eq:left-A-action}.
This left $\cC$-action coincides with the left $\cC$-module structure on $\bimodint{A}{\cC}{B}$ defined in \eqref{eq:bimod1}--\eqref{eq:left-A-bimodule}.
 We thus see that \eqref{eq:rel-Del-realization} realizes the relative Deligne product as a left $\cC$-module, not just as a category.  
\end{proof}

For the next preliminary lemma, recall the $2$-functors $\rS_{lr}, \rS_{rl}$ from \Cref{lem:rS}.

\begin{lemma}\label{lem:S-via-algebra}
Let  $A$ be an algebra in $\cC$.
\begin{enumerate}
    \item[(i)]     There is an equivalence of right $\cC$-modules $\Phi_r\colon \rS_{lr}(\rmodint{\cC}{A})\isomorph \lmodint{A^{\psi^{-1}}}{\cC}$.
    \item[(ii)] There is an equivalence of left $\cC$-modules $\Phi_l\colon \rS_{rl}(\lmodint{A}{\cC})\isomorph \rmodint{\cC}{A^{\psi}}$.
\end{enumerate}
\end{lemma}
\begin{proof}
    We prove Part (i), while Part (ii) follows by $\otimes^\oop$-duality, using the module coherence involving the inverse braiding from \eqref{eq:left-coherence}. 
    
    We define a functor $\Phi_r\colon \rS_{lr}(\rmodint{\cC}{A})\to \lmodint{A^{\psi^{-1}}}{\cC}$ sending a right $A$-module $(V,a^r_V)$ to the left $A^{\psi^{-1}}$-module $(V,a^l_V)$, where
    $a^l_V:=a^r_V\psi^{-1}_{A,V}$. This clearly extends to an equivalence of categories, acting as the identity on morphisms.

    Next, we observe that 
    \begin{equation}\label{eq:Phir-C-linear}
    s^{\Phi_r}_{V,X}\colon \Phi_r(V\ract X)=X\otimes V\xrightarrow{\psi_{X,V}} V\otimes X=\Phi_r(V)\ract X
    \end{equation}
    is an isomorphism of left $A^{\psi^{-1}}$-modules. Recall the right $\cC$-module coherence from \eqref{eq:right-coherence}. For $\cM=\rmodint{\cC}{A}$, it specifies to 
    $$r_{V,X,Y}\colon (V\ract X)\ract Y=Y\otimes X\otimes V\xrightarrow{\psi_{Y,X}\otimes \id_V}  X\otimes Y\otimes V=V\ract (X\otimes Y).$$ 
    One checks that naturality of the braiding implies that $(\Phi_r,s^{\Phi_r})$ is a functor of right $\cC$-modules.  
\end{proof}

Recall the  $2$-functor $\rB$ from \eqref{eq:Slr-B} and \Cref{lem:S-bimod}.

\begin{lemma}\label{lem:B-via-algebra}
Let  $A$ be an algebra in $\cC$.
There is an equivalence of $\cC$-bimodules $\Phi_r\colon \rB(\rmodint{\cC}{A})\isomorph \lmodint{A^{\psi^{-1}}}{\cC}$, where the target is the $\cC$-bimodule with left $\cC$-action defined in \eqref{eq:left-A-action}.
\end{lemma}
\begin{proof}
    We upgrade the equivalence $\Phi_r\colon\rS_{lr}(\rmodint{\cC}{A})\isomorph \lmodint{A^{\psi^{-1}}}{\cC}$ of right $\cC$-module categories from \Cref{lem:S-via-algebra}(i) to an equivalence of $\cC$-bimodules as stated. Recall the right $\cC$-linear structure of $\Phi_r$ defined in \eqref{eq:Phir-C-linear}.
The functor $\Phi_r$ also has a left $\cC$-linear structure
$$s^{\Phi_r}_{X,V}\colon \Phi_r(X\lact V)\isomorph X\lact \Phi_r(V)$$
consisting of identity components. Indeed, the identity trivially satisfies the coherence \eqref{eq:C-linear-coherence} and is a morphism of right $A^{\psi^{-1}}$-modules because of the identity 
$$a^l_{X\otimes V}=(X\otimes a^l_{V})\circ (\psi^{-1}_{A,X}\otimes V)=a^r_{X\otimes V}\psi^{-1}_{A,X\otimes Y}.$$

It remains to verify the coherence condition \eqref{eq:bimodule-functor-condition}. For a right $A$-module $V$ and $X,Y$ in $\cC$, this follows from
\begin{align*}
    (X\lact s^{\Phi_r}_{V,Y})\circ s^{\Phi_r}_{X,V\ract Y}\circ \Phi_r(b_{X,V,Y})&=(X\otimes \psi_{Y,V})\circ(\psi_{Y,X}\otimes V)\\
    &=\psi_{Y,X\otimes V}\\
    &=b_{X,\Phi_r(V),Y}\circ (s^{\Phi_r}_{X,V}\ract Y)\circ s^{\Phi_r}_{X\lact V,Y},
\end{align*}
where we use that the bimodule coherence of the target bimodule $\lmodint{A^{\psi^{-1}}}{\cC}$ as well as the left $\cC$-linear structure of $\Phi_r$ have identity components, and the right hexagon for the braiding.
\end{proof}

We recall the relative Deligne product of left $\cC$-module categories from \Cref{def:rel-Del-Cmod}.
\begin{proposition}\label{prop:rel-tensor-bimod}
For algebras $A,B$ in  $\cC$, there is an equivalence of left $\cC$-module categories
$$\rmodint{\cC}{A}\boxtimes_\cC\rmodint{\cC}{B}\simeq \bimodint{A^{\psi^{-1}}}{\cC}{B}.$$
\end{proposition}
\begin{proof}
The equivalence of $\cC$-bimodules $\Phi_r$ from \Cref{lem:B-via-algebra} induces an equivalence of left $\cC$-modules 
$$\Phi_r\boxtimes_\cC \rmodint{\cC}{B}\colon \quad \rmodint{\cC}{A}\boxtimes_\cC\rmodint{\cC}{B}\isomorph \lmodint{A^{\psi^{-1}}}{\cC}\boxtimes_\cC\rmodint{\cC}{B}.$$
By \Cref{rem:rel-Del-bimodule-realization}, the target $\cC$-module is realized by $\bimodint{A^{\psi^{-1}}}{\cC}{B}$. This proves the claim.
\end{proof}

The following equivalence, describing the relative Deligne product by the tensor product of internal algebras from \Cref{def:tensor-algebras}, was observed in \cite{DN}*{Prop.\,3.4}.

\begin{proposition}\label{prop:rel-Del-modules}
    For algebras $A,B$ in $\cC$, there is an equivalence of left $\cC$-module categories 
    $$\rmodint{\cC}{A}\boxtimes_\cC \rmodint{\cC}{B}\simeq \rmodint{\cC}{A\otimes B}.$$
\end{proposition}
\begin{proof}
Consider the functor 
$$G\colon \bimodint{A}{\cC}{B}\to \rmodint{\cC}{A^\psi\otimes B}$$
which sends an $A$-$B$-bimodule $(M, a^l\colon A\otimes M\to M, a^r\colon M\otimes B \to M)$ to the $A^\psi\otimes B$-module $(M,a)$ with the right 
$A^\psi\otimes B$-action defined by
$$a:=\Big(\xymatrix{M\otimes A^\psi\otimes B\ar[rr]^-{\psi_{M, A^\psi}\otimes \id_B}&&A^\psi\otimes  M\otimes B\ar[r]^-{a^l\otimes \id}& M\otimes B\ar[r]^-{a^r}& M}\Big).$$
One checks that this indeed defines a right $A^\psi\otimes B$-module and extends to a functor acting as identities on morphisms. Recall the left $\cC$-module structure on $\bimodint{A}{\cC}{B}$ defined in \eqref{eq:bimod1}--\eqref{eq:left-A-bimodule}. The functor $G$ has the structure of a functor of left $\cC$-module categories with the identity coherence. 
By \Cref{prop:rel-tensor-bimod}, together with $(A^{\psi^{-1}})^\psi=A$, the claim follows. 
\end{proof}

\begin{proposition}\label{prop:Funrex-modules}
     For algebras $A,B$ in a braided tensor category $\cC$, there is an equivalence of left $\cC$-module categories 
    $$\Fun_\cC(\rmodint{\cC}{A},\rmodint{\cC}{B})\simeq \rmodint{\cC}{A^{\psi}\otimes B}.$$
\end{proposition}
\begin{proof}
    The result follows from \Cref{prop:fun-A-bimod-B} and the equivalence of left $\cC$-module categories $\bimodint{A}{\cC}{B}\simeq \rmodint{\cC}{A^\psi\otimes B}$ given in the proof of \Cref{prop:rel-Del-modules}.
\end{proof}

\begin{remark}
With \Cref{prop:rel-Del-modules}, one can extend the  biequivalence of \Cref{thm:morita} between $\lmod{\cC}$ and $\Mor_\cC$ to a monoidal biequivalence. Thus, we can use internal algebras as a tool to study the monoidal $2$-category $\lmod{\cC}$.
\end{remark}

The following result shows that the $\cC$-module structure on functor categories makes them internal Hom categories with respect to the relative Deligne product. 

\begin{proposition}\label{prop:int-hom-cat}
    Let $\cM,\cN,\cP$ be left $\cC$-module categories. Then there is an equivalence of $\cC$-module categories
    $$\Fun_\cC(\cM\boxtimes_\cC \cN, \cP)\simeq \Fun_\cC(\cN, \Fun_\cC(\cM,\cP)).$$
\end{proposition}
\begin{proof}
We use \Cref{prop:rel-Del-modules} and \Cref{prop:Funrex-modules} to see that both relative Deligne products and categories of right exact functors are given by internal bimodules over tensor product algebras, with the (trivial) associator of $\cC$ giving the coherence isomorphisms for the left $\cC$-actions. Indeed, for $\cM=\rmodint{\cC}{A}$,  $\cM=\rmodint{\cC}{B}$, and $\cN=\rmodint{\cC}{C}$, we have equivalences of left $\cC$-module categories
\begin{align*}
    \Fun_\cC(\cM\boxtimes_\cC \cN, \cP)&\simeq \Fun_\cC(\rmodint{\cC}{A\otimes B},\rmodint{\cC}{C})\\
    &\simeq \rmodint{\cC}{(A\otimes B)^\psi\otimes C}\\
    & \simeq \rmodint{\cC}{B^\psi\otimes A^\psi\otimes C}\\
    &\simeq \bimodint{B}{\cC}{A^{\psi}\otimes C}\\
    &\simeq \Fun_\cC(\rmodint{\cC}{B},\rmodint{\cC}{A^{\psi}\otimes C})\\
    &\simeq \Fun_\cC(\cN,\Fun_\cC(\cM,\cP)).
\end{align*}
The third equivalence follows from the first algebra isomorphism of \Cref{lem:tensor-ops}. 
\end{proof}

\begin{example}\label{ex:FunCM}
    There is an equivalence of left $\cC$-module categories $\Fun_\cC(\cC,\cM)\simeq \cM$ given by evaluating $F\colon \cC\to\cM$ on the tensor unit $\one$. Conversely, to an object $M$ of $\cM$, we associate the $\cC$-module functor $F_M=X\lact (-)$, with coherence isomorphism 
    $$s^{F_M}(Y,X)=(l^\cM_{Y,X,M})^{-1}\colon F_M(Y\otimes X)=Y\otimes X\lact M\longrightarrow Y\lact F_M(X)=Y\lact (X\lact M).$$
\end{example}

\subsection{Dual module categories}\label{sec:duals}

In this section, for $\cC$ braided, we define dual $\cC$-module categories and show that relative tensor products with these duals are equivalent to categories of $\cC$-module functors equipped with the $\cC$-module structure from \Cref{def:Fun-action}

First, we combine the construction of turning left into right $\cC$-modules from \Cref{lem:rS} and the opposite module categories from \Cref{def:*M,def:M*} to define dual module categories.

\begin{definition}\label{def:dualsM}
    For a left $\cC$-module category $\cM$, we call $\cM^*:={}^\#\rS_{lr}(\cM)$ the \emph{left dual} $\cC$-module category and ${}^*\cM:=\rS_{rl}(\cM
    ^\#)$ the \emph{right dual}.
\end{definition}

We now describe the $\cC$-module categories $\cM^{\#}$ and ${}^{\#}\cN$ from \Cref{def:M*} by internal algebras following an analogous result for bimodule categories from \cite{DSS2}*{Cor.\,2.4.14}.

\begin{lemma}\label{lem:M-hash=A-mod}
Let $A$ be an algebra in $\cC$. 
\begin{enumerate}
    \item[(i)]
The right duality functor 
$$\rD_r\colon (\rmodint{\cC}{A})^\# \to \lmodint{A}{\cC}, \quad (V,a^r_V)\mapsto( {}^*V, a^l_{{}^*V}),$$
where the left $A$-action is given by
$$a_{{}^*V}^l:=  (\id_{{}^*V}\otimes \ev^r_V)(\id_{{}^*V}\otimes a_V^r \otimes \id_{{}^*V})(\coev_V^r\otimes \id_{A\otimes {}^*V}),$$
defines an equivalence of right $\cC$-module categories.
\item[(ii)]
The left duality functor
$$\rD_l\colon {}^\#(\lmodint{A}{\cC}) \to \rmodint{\cC}{A}, \quad (V,a^l_V)\mapsto (V^*, a_V^r),$$
where the left $A$-action is given by
$$a_V^r:=  (\ev^l_V\otimes \id_{V^*})(\id_{V^*}\otimes a_V^l \otimes \id_{V^*})(\id_{V^*\otimes A}\otimes \coev_V^l),$$
defines an equivalence of left $\cC$-module categories.
\end{enumerate}
\end{lemma}
\begin{proof}
We first prove Part (i). The functor $\rD_r$ is an equivalence by rigidity of $\cC$.
Recall that the right $\cC$-action on $\rmodint{\cC}{A}^\#$ is given by 
$$V\ract X:=X^* \otimes V.$$
Using the monoidal isomorphism $\eta_X\colon {}^*(X^*)\isomorph X$ we see that the isomorphisms $$\rD_r(V\ract X)= {}^*(X^*\otimes V)={}^*V\otimes {}^*(X^*)\xrightarrow{\id\otimes \eta_X} {}^*V\otimes X=\rD_r(V)\ract X,$$
satisfy the coherences \eqref{eq:C-linear-coherence} for $\rD_r$ to be a $\cC$-linear functor. Part (ii) follows by applying Part (i) with $\cCop$ instead of $\cC$.
\end{proof}

We will now give descriptions of the duals $\cM^*$ and ${}^*\cM$ of \Cref{def:dualsM} via modules over internal algebras. 

\begin{proposition}\label{prop:*M*-algebras}
Let $A$ be an algebra in a braided tensor category $\cC$. There are equivalences of left $\cC$-module categories 
\begin{itemize}
    \item[(i)]
    ${}^*\rD\colon {}^*(\rmodint{\cC}{A})\isomorph \rmodint{\cC}{A^{\psi}},\quad (V,a^r_V)\mapsto ({}^*V, a^r_{{}^*V}),$
    where the right $A^\psi$-action is defined by 
    $$a_{{}^*V}^r:=  (\id_{{}^*V}\otimes \ev^r_V)(\id_{{}^*V}\otimes a_V^r \otimes \id_{{}^*V})(\coev_V^r\otimes \id_{A\otimes {}^*V})\psi_{{}^*V,A},$$
    \item[(ii)]
        $\rD^*\colon (\rmodint{\cC}{A})^*\isomorph \rmodint{\cC}{A^{\psi^{-1}}},\quad (V,a^r_V)\mapsto (V^*, a^r_{V^*}),$
    where the right $A^{\psi^{-1}}$-action is defined by 
    $$a^r_{V^*}:=(\ev^l_V\otimes \id_{V^*})(\id_{V^*}\otimes a^r_V\otimes \id_{V^*})(\id_{V^*\otimes V}\otimes \psi_{V^*,A})(\id_{V^*}\otimes \coev^l_V\otimes \id_A).$$
    \end{itemize}
\end{proposition}
\begin{proof}
We define ${}^*\rD$ to be the composed equivalence
$${}^*(\rmodint{\cC}{A})=\rS_{rl}(\rmodint{\cC}{A}^\#)\xrightarrow{~\rS_{rl}(\rD_r)~} \rS_{rl}(\lmodint{A}{\cC})\xrightarrow{~\Phi_l~}\rmodint{\cC}{A^{\psi}},$$
of the equivalences from \Cref{lem:S-via-algebra,lem:M-hash=A-mod}.
Then 
${}^*\rD(V,a^r_V)$ is the right $A^\psi$-module defined on ${}^*V$ with action given by the formula for $a_{{}^*V}^r$ in the statement of (i). Similarly, define ${}^*\rD$ to be the composition
$$(\rmodint{\cC}{A})^*={}^\#\rS_{lr}(\rmodint{\cC}{A})\xrightarrow{~{}^\#(\Phi_r)~} {}^\#(\lmodint{A^{{\psi}^{-1}}}{\cC}) \xrightarrow{~\rD_l~}\rmodint{\cC}{A^{\psi^{-1}}}.$$
Then $\rD^*(V,a^r_V)=(V^*,a^r_{V^*})$. 
\end{proof}

\begin{remark}\label{rem:convention2}
Recall from \Cref{rem:convention1} that we chose a convention on how to define the $\cC$-linear structure on $X\lact G$, for $G\in \Fun_\cC(\cM,\cN)$, that uses the braiding $\psi$. The alternative choice of the inverse braiding gives a $\cC$-module category $\Fun_\cC'(\cM,\cN)$. The analogue of \Cref{prop:Funrex-modules} with this convention gives an equivalence of $\cC$-module categories $$\Fun'_\cC(\cM,\cN)\simeq \rmodint{\cC}{A^{\psi^{-1}}\otimes^{\psi^{-1}}B}\simeq \rmodint{\cC}{B\otimes A^{\psi^{-1}}}\simeq \rmodint{\cC}{(A\otimes B^{\psi})^{\psi^{-1}}},$$
where the last equivalence comes from the isomorphism of algebras in \Cref{lem:tensor-ops2}.
This shows that $\Fun'_\cC(\cM,\cN)\simeq \Fun_\cC(\cN,\cM)^*$ are equivalent left $\cC$-modules.
\end{remark}

\begin{corollary}\label{cor:stars-inverses}
    There are equivalences of left $\cC$-modules ${}^*(\cM^*)\simeq \cM$ and $({}^*\cM)^*\simeq \cM$. 
\end{corollary}
\begin{proof}
    We can choose an algebra $A$ such that $\cM\simeq \rmodint{\cC}{A}$. The claims now follow from \Cref{prop:*M*-algebras} using the equalities of algebras $(A^\psi)^{\psi^{-1}}=A=(A^{\psi^{-1}})^\psi.$
\end{proof}

Together with \Cref{lem:tensor-ops} and \Cref{prop:rel-Del-modules} we find that duals are compatible with relative Deligne products in the following way.

\begin{corollary}\label{cor:tensor-dual}
    There are equivalences of left $\cC$-module categories
    $$(\cM\boxtimes_\cC\cN)^*\simeq \cN^*\boxtimes_\cC\cM^* \qquad \text{and}\qquad {}^*(\cM\boxtimes_\cC\cN)\simeq {}^*\cN\boxtimes_\cC{}^*\cM.$$
\end{corollary}

These results can be used to describe duals of functor categories.

\begin{corollary}\label{cor:functors-*M*}
There are equivalences of left $\cC$-module categories
$$\Fun_\cC(\cM,\cN)^*\simeq \Fun_\cC(\cN^{**},\cM) \qquad \text{and} \qquad {}^*\Fun_\cC(\cM,\cN)\simeq \Fun_\cC(\cN,{}^{**}\cM).$$
\end{corollary}
\begin{proof}
Take algebras $A$, $B$ in $\cC$ such that $\cM=\rmodint{\cC}{A}$ and $\cN=\rmodint{\cC}{B}$. We have equivalences of left $\cC$-module categories
\begin{align*}
    \Fun_\cC(\cM,\cN)^*\simeq (\rmodint{\cC}{A^\psi\otimes B})^* 
    \simeq \rmodint{\cC}{(A^\psi\otimes B)^{\psi^{-1}}}  
    \simeq \rmodint{\cC}{B^{\psi^{-1}}\otimes A},
\end{align*}
using \Cref{prop:*M*-algebras,prop:Funrex-modules}. This is equivalent to $\Fun_\cC(\cN^{**},\cM)$ by \Cref{lem:tensor-ops}. The other equivalence is proved similarly.
\end{proof}

\begin{remark}\label{rem:double-dual-equiv}
If $\cC$ is a ribbon category \cite{EGNO}*{Def.\,8.10.1} (for example, when $\cC$ is pivotal and symmetric), then  $\cM^{**}\simeq \cM$ as $\cC$-module categories. This follows using \Cref{prop:*M*-algebras} and observing that for an algebra $A$ in $\cC$, the ribbon twist $\theta_A$ gives an isomorphism of algebras $A\cong A^{\psi^{-2}}$.
Therefore, we can remove double stars in above \Cref{cor:functors-*M*}.
We expect that the monoidal $2$-category $\lmod{\cC}$ is pivotal provided that $\cC$ is a finite ribbon category, cf.\,\cite{BMS}*{Cor.\,4.9} and \cite{DR}*{Rem.\,2.2.8}.
\end{remark}

\begin{proposition}\label{prop:rel-tensor-as-hom}
    Let $\cC$ be a finite braided tensor category and $\cM,\cN$ left $\cC$-module categories. There are equivalences of $\cC$-module categories 
    \begin{enumerate}
        \item [(i)] $\cM\boxtimes_\cC\cN\simeq \Fun_\cC(\cM^*,\cN),$ 
        \item [(ii)] $\Fun_\cC(\cM,\cN)\simeq {}^*\cM\boxtimes_\cC \cN .$
    \end{enumerate}
\end{proposition}
\begin{proof}
The equivalence (i) follows from \cite{DSS2}*{Cor.\,2.4.11}. Alternatively, the statement follows by combining \Cref{prop:rel-Del-modules,prop:Funrex-modules,prop:*M*-algebras}. The equivalence (ii) follows from (i) using that ${}^*(\cM^*)\simeq \cM$ by \Cref{cor:stars-inverses}.
 \end{proof}

\begin{example}\label{ex:dual-and-CM}
    Let $\cM$ be a left $\cC$-module category. As a special case of \Cref{prop:rel-tensor-as-hom} we obtain equivalences of left $\cC$-module categories 
    $\Fun_\cC(\cM,\cC)\simeq {}^*\cM$ and ${}^*\cM\boxtimes_\cC \cM\simeq \cC^*_\cM$.
\end{example}


\section{Exact module categories}\label{sec:exact-mod-cat}

In this section, we first recall the definition and some fundamental properties of exact module categories and exact algebras in \Cref{sec:exact-algebras}. The following \Crefrange{sec:rel-semisimple}{sec:exactness-centralizer} contain results concerning relative semisimplicity of a left $\cC$-module category $\cM$ as a necessary criterion for exactness of $\cM$, exactness of braided opposite algebras, and exactness of the centralizer $\cC^*_\cM$ as a $\cC$-module. In particular, we prove the lower row of implications in \Cref{fig:overview}.
 These results are either new or not readily available in the literature. Finally, in \Cref{sec:rel-Del-exact} we give examples of exact module categories whose  relative Deligne product is not an exact $\cC$-module, in both zero and non-zero characteristic, which motivates the main concept of this paper, \emph{fully exact} module categories, introduced in \Cref{sec:fully-exact}.

\subsection{Exact module categories and exact algebras}
\label{sec:exact-algebras}

\begin{definition}\label{def:exact-M}
A left $\cC$-module category $\cM$ is \emph{exact} if for any projective object $P$ in $\cC$ and any object $M$ in $\cM$, $P\lact M$ is projective in $\cM$.

We denote the full $2$-subcategory in $\lmod{\cC}$ of exact module categories by $\lemod{\cC}$. 
\end{definition}

\pagebreak

\begin{example}\label{ex:exact}
$~$
\begin{enumerate}[(i)]
    \item     If $\cC$ is a fusion category over any field, then a finite module category $\cM$ is exact if and only if it is semisimple. Indeed, $\one_\cC$ is a projective object and hence, for $\cM$ to be exact, for any object $M$ in $\cM$, $\one_\cC\lact M\cong M$ has to be projective.
    \item  The regular left $\cC$-module category $\cC$, where $\cC$ acts by the tensor product, is exact by \cite{EGNO}*{Ex.\,7.5.5}.
    \item Let $F\colon \cC\to \cD$ be an exact monoidal functor of finite tensor categories. Consider $\cD$ as a left $\cC$-module category as in \Cref{ex:tensor-act}. Then $\cD$ is an exact module category if and only if $F$ preserves projective objects. Exactness follows from rigidity of $\cD$ because projective objects form a tensor ideal, see \Cref{lem:tensor-cat-properties}. This holds, in particular, when $\cD$ is semisimple, e.g.\ $\cD=\Vect$ and $F$ is a fiber functor. Conversely, if $\cD$ is exact and $P$ a projective object in $\cC$, then $P\lact \one \cong F(P)$ has to be projective.   
\end{enumerate}

\end{example}

The following result from \cite{EGNO}*{Prop.\,7.6.9. and 7.9.7} characterizes exact module categories in terms of their $\cC$-module functors. Over a general ground field $\Bbbk$, see \cite{DSS2}*{Thm.\,2.3.5}.

\begin{proposition}\label{prop:exact-adjoints}
Let $\cC$ be a finite tensor category. If a left $\cC$-module $\cM$ is exact, then any additive $\cC$-linear functor $G\colon \cM \to \cN$, for $\cN$ a  left $\cC$-module category, is exact. Moreover, $\cM$ is exact if and only if every additive $\cC$-module functor $\cM\to \cM$ is exact.

In particular, if \emph{$\cM$ is exact}, any $\cC$-linear functor $G\colon \cM\to \cN$ has both a left and right adjoint.
\end{proposition}

We say that an algebra $A$ in $\cC$ is \emph{exact} if $\rmodint{\cC}{A}$ is an exact left $\cC$-module category. The terminology is justified by the following result (which does not require $\cC$ to be braided). 

\begin{lemma}\label{lem:A-left-right-exact} Let $\cC$ be a finite tensor category. 
    An algebra $A$ in $\cC$ is exact if and only if $\lmodint{A}{\cC}$ is exact as a right $\cC$-module category.
\end{lemma}
\begin{proof}
Recall the equivalence of right $\cC$-module categories
$$\lmodint{A}{\cC}\simeq (\rmodint{\cC}{A})^\#$$
from \Cref{lem:M-hash=A-mod}. We need that  
$\cM$ is exact as a left $\cC$-module if and only if $\cM^\#$ is exact as a right $\cC$-module category. But this follows from \cite{EGNO}*{Cor.\,7.6.4, Prop.\,6.1.3} because an object in an exact $\cC$-module category is injective if and only if it is  projective and duals of projective objects in $\cC$ are projective. 
\end{proof}

A direct consequence of this lemma is that passing to opposite module categories, as defined in \Cref{sec:opposite-mod-cats}, preserves exactness.

\begin{corollary}\label{lem:*-autoequiv2}
The assignments $\cM\mapsto \cM^\#$ and $\cM\mapsto {}^\#\cM$ are mutually inverse up to equivalence of $\cC$-module categories and  preserve exact module categories.
\end{corollary}
\begin{proof}
We see, e.g. by \Cref{lem:M-hash=A-mod} that ${}^\#(\cM^\#)\simeq \cM$ when $\cM$ is a left (respectively, a right) $\cC$-module category. 
These assignments preserve exact module categories by \Cref{lem:A-left-right-exact}, choosing an algebra $A$ in~$\cC$ such that $\cM=\rmodint{\cC}{A}$. 
\end{proof}

\subsection{Relatively semisimple algebras and module categories}
\label{sec:rel-semisimple}

Let $\cC$ be a finite tensor category. Let $A$ be an algebra in $\cC$ with multiplication map $m_A\colon A\otimes A\to A$. This section contains some conditions on $A$ that imply exactness of the associated $\cC$-module category $\cM=\rmodint{\cC}{A}$. These conditions will be used later in \Cref{sec:rel-projective}.
We consider the induction functor
\begin{equation}\label{eq:ind-functor}
I\colon \cC\to \rmodint{\cC}{A}, \quad X \mapsto (X\otimes A, a^r_{I(X)}=\id_X\otimes m_A)\ ,\end{equation}
which is a functor of left $\cC$-modules and maps projective objects in $\cC$ to projective objects in $\rmodint{\cC}{A}$. 
Similar to the proof for usual algebras in $\Vect$, one can show that any projective object in $\rmodint{\cC}{A}$ is a direct summand of $I(P)$, for $P$ a projective object  in $\cC$.

\begin{definition}\label{def:proj-alg}
An algebra $A$ in $\cC$ is \emph{relatively semisimple} if any $A$-module in $\cC$ is a direct summand of some induced module $I(X)$, $X\in \cC$. 
\end{definition}

\begin{lemma}\label{lem:rel-semisimple-exact}
    Any relatively semisimple algebra $A$ is exact.
\end{lemma}
\begin{proof}
Assume that $A$ is relatively semisimple. Then an arbitrary object $M$ of $\cM=\rmodint{\cC}{A}$ is a direct summand of $I(X)$ for some object $X$ of $\cC$. For any projective object $P$ in $\cC$, by additivity of the action functor, $P\lact M$ is a direct summand of the projective right $A$-module $P\lact I(X)$, which is isomorphic to $I(P\otimes X)$ as a right $A$-module. The right $A$-module $I(P\otimes X)$ is projective and so are its direct summands. Hence, $P\lact M$ is projective  and  $\cM$ is exact as a left $\cC$-module. This shows that $A$ is an exact algebra. 
\end{proof}

A sufficient condition for an algebra to be relatively semisimple is to be separable.

\begin{definition}\label{def:separable}
    We say that $A$ is \emph{separable} if the multiplication map $m\colon A\otimes A\to A$ has a section as a morphism of $A$-bimodules in $\cC$. 
\end{definition}

\begin{lemma}\label{lem:sep-rel-ss}
    If $A$ is a separable algebra, then it is relatively  semisimple and hence exact.
\end{lemma}
\begin{proof}
Recall the relative tensor product $\otimes_A$ from \Cref{def:MorC}. For any right $A$-module $M$, we have that $M \otimes_A (-)\colon \bimodints{A}{\cC}\to \rmodint{\cC}{A}$ is an additive functor. Hence, if $A$ is separable, we have 
    that $M\cong M\otimes_A A$ is a direct summand of $M\otimes_A (A\otimes A)\cong M \otimes A = I(M)$  as a right $A$-module. Thus, $A$ is relatively semisimple.
\end{proof}

\begin{remark}
When $\cC$ is semisimple, relative semisimplicity of an algebra $A$ is equivalent to $\rmodint{\cC}{A}$ being semisimple. Indeed, any object is a direct summand of an object of the form $I(X)=X\lact A$, which is projective.  
Thus, being relatively semisimple and being separable are properties  of algebras that are preserved under Morita invariance.  See \cite{DSS2}*{Section~2.5} for the semisimple case. 
\end{remark}

Recall the internal Hom object $\iHom(M,N)\in\cC$ defined by the natural isomorphisms
    $$\Hom_\cM(X\lact M,N)\cong \Hom_\cC(X,\iHom(M,N))\ ,$$
for $X$ in $\cC$. Then the internal endomorphism object $\iEnd(M)$ is an algebra in $\cC$ and the internal Hom extends to a $\cC$-linear functor 
$$\iHom(M,-)\colon \cM \to \rmodint{\cC}{\iEnd(M)}, \quad N\mapsto\iHom(M,N),$$
see \cite{EGNO}*{Secs.\,7.9--7.10}. If $\cM$ is exact, then $\iHom(M,-)$ is an equivalence for any $M$.

For endomorphism algebras $\iEnd(M)$, we have the following criterion for relative semisimplicity. 

\begin{lemma}\label{lem:rel-proj-int-hom}
    Let $\cM$ be a left $\cC$-module category and $M$ an object in $\cM$ such that $\iHom(M,-)$ is right exact. Then the algebra $\iEnd(M)$ is relatively semisimple if and only if every object $N$ of $\cM$ is a direct summand of $X\lact M$ for some $X$ in $\cC$. In particular, $\cM$ is exact. 
\end{lemma}
\begin{proof}Denote $A=\iEnd(M)$.
    Under our assumptions, $\iHom(M,-)\colon \cM \to \rmodint{\cC}{A}$ is an equivalence \cite{EGNO}*{Thm.\,7.10.1}.
    Using additivity  of internal Homs, $N$ in $\cM$ is a direct summand of $X\lact M$  if and only if  $$\iHom(M,N)\underset{\oplus}{\leq} \iHom(M,X\lact M)\cong X\lact \iHom(M,M)=X\otimes A,$$
    is a direct summand in $\rmodint{\cC}{A}$, 
    where for the isomorphism, we use \cite{EGNO}*{Lem.\,7.9.4}. Since $\iHom(M,-)$ is essentially surjective, we see that $A$ is relatively semisimple if and only if \emph{every} object $N$ appears as a direct summand of some object $X\lact M$. As $\cM\simeq \rmodint{\cC}{A}$ and we have shown that $A$ is relatively semisimple, $\cM$ is exact by \Cref{lem:rel-semisimple-exact}.    
\end{proof}

\begin{definition}\label{def:M-rel-ss}
We call an exact left $\cC$-module category $\cM$ \emph{relatively semisimple}  if there exists an object $M$ in $\cM$ such that every object $N$ in $\cM$ is a direct summand of $X\lact M$ for some $X$ in $\cC$. 
\end{definition}

Equivalently, by \Cref{lem:rel-proj-int-hom}, $\cM$ is relatively semisimple if and only if there exists a relatively semisimple algebra $A$ in $\cC$  such that $\cM\simeq \rmodint{\cC}{A}$ as left $\cC$-module categories.

\begin{definition}\label{def:separable-mod-cat}
    We say that a left $\cC$-module category $\cM$ is \emph{separable} if it is equivalent as a left $\cC$-module category to $\rmodint{\cC}{A}$ for a separable algebra $A$ in $\cC$. 
\end{definition}
By \Cref{lem:rel-semisimple-exact}, relatively semisimple module categories are exact. Moreover, by  \Cref{lem:sep-rel-ss}, separable module categories are relatively semisimple and hence exact.

\begin{example}\label{ex:rel-semisimple-not-semisimple}
    The regular $\cC$-module $\cC$ is always separable and relatively semisimple since $\cC\simeq\rmodint{\cC}{\one}$ but might not be semisimple. 
\end{example}

\begin{remark}\label{rem:separable1}
    In the semisimple case, the definition of a separable algebra appears in \cite{DSS2}*{Def.\,2.5.1} and it is shown that if $\cM$ is separable, then $\cM$ is semisimple \cite{DSS2}*{Prop.\,2.5.3}, and thus exact by \Cref{ex:exact}(i). The converse is not true in general, see \Cref{cor:semisimple-fully-exact} and the discussion afterwards. 
\end{remark}

If $\cC$ is non-semisimple, being separable or relatively semisimple are properties \emph{not} stable under Morita equivalence of algebras. This is demonstrated in the following example. 

\begin{example}\label{ex:sep-not-Morita}
Assume that $\cC$ is non-semisimple.
Consider the  left $\cC$-module $\cM=\cC$ and two Morita equivalent algebras realizing this $\cC$-module category $\cM$, the tensor unit $\one$ and the algebra $A_P=\iEnd(P)=P\otimes P^*$ for a projective object $P$  in $\cC$. The fact that $\rmodint{\cC}{A_P}\simeq \cC$ is because $\cC$ is an exact $\cC$-module category. 
    The trivial algebra $\one$ in $\cC$ is separable and thus relatively semisimple. However, $A_P$ is \textit{not} relatively semisimple because otherwise by \Cref{lem:rel-proj-int-hom} every object $N$ of $\cC$ is a direct summand of $X \lact P = X\otimes P$, for some $X\in \cC$, which is projective and thus $\cC$ would be semisimple which is a contradiction.  In particular, $A_P$ is also not a separable algebra, otherwise it would be relatively semisimple due to~\Cref{lem:sep-rel-ss}.
\end{example}

\subsection{Exactness of braided opposite algebras}
\label{sec:opposite-exact}

If $\cC$ is braided, exactness of algebras can, equivalently, be checked for the corresponding braided opposite algebras that were defined in \Cref{def:opp-algebras}.

\begin{proposition}\label{prop:A-Psi-exact}
    Let $A$ be an algebra in $\cC$. The following statements are equivalent. 
    \begin{enumerate}
        \item $A$ is exact.
        \qquad (2)~$A^\psi$ is exact.
        \qquad (3)~$A^{\psi^{-1}}$ is exact.
    \end{enumerate}
\end{proposition}
\begin{proof}
    There is an equivalence of categories 
    \begin{equation}\label{eq:equiv-A-Aop-mod}
    \rmodint{\cC}{A}\isomorph \lmodint{A^{\psi}}{\cC},
    \end{equation}
which sends a right $A$-module $M$, with action $a^r_M\colon M\otimes A \to M$ to the left $A^\psi$-module $M$ with the left $A^\psi$-action 
\begin{equation}\label{eq:left-right-op-equiv}
a^l_M=a^r_M\psi_{A,M}\colon A\otimes M \to M,
\end{equation}
noting that $A=A^\psi$ as objects in $\cC$.
For $X\in \cC$, this functor sends the right $A$-module $X\lact M=X\otimes M$ to the $A^\psi$-module $X\otimes M$ with left $A^\psi$-action given by the morphism
$$(\id_X\otimes a^r_M)\psi_{A,X\otimes M}\colon A\otimes (X\otimes M)\to X\otimes M.$$
This left $A^\psi$-module is isomorphic to $M\ract X$, with the left $A^\psi$-action~\eqref{eq:left-right-op-equiv} on $M$ and the standard right $\cC$-action on $\lmodint{A^{\psi}}{\cC}$, via the isomorphism $\psi_{X,M}^{-1}\colon X\otimes M\to M\otimes X$. Hence, if $A$ is exact, $P\in \cC$ a projective object, then $P\lact M$ is projective as a right $A$-module. Then the image of $P\lact M$ under the equivalence \eqref{eq:equiv-A-Aop-mod} is projective as a left $A^\psi$-module. By the above isomorphism $\psi_{P,M}^{-1}$ of left $A^\psi$-modules, $M\ract P$  is also projective as a left $A^{\psi}$-module. 
This shows that  $\lmodint{A^{\psi}}{\cC}$ is exact as a right $\cC$-module. By \Cref{lem:A-left-right-exact}, this implies that $\rmodint{\cC}{A^{\psi}}$ is exact as a left $\cC$-module category. Hence, the algebra $A^{\psi}$ is exact. This shows that (1) implies (2).

Now, we can use the inverse braiding to show that (1) implies (3). Since $(A^\psi)^{\psi^{-1}}=A=(A^{\psi^{-1}})^{\psi}$, it follows that (2) implies (1) and (3) implies (1).
\end{proof}

We immediately get the following result, cf.\ \cite{DSS2}*{Cor.\,2.4.7}.

\begin{corollary}\label{cor:duals-exact}
For $\cM$ a left $\cC$-module category, the following statements are equivalent. 
\begin{enumerate}
    \item $\cM$ is exact 
    \qquad (2)~$\cM^*$ is exact
    \qquad (3)~${}^*\cM$ is exact.
\end{enumerate}
\end{corollary}
\begin{proof}
    Finding an algebra $A$ in $\cC$ such that $\cM\simeq\rmodint{\cC}{A}$, we have  $\cM^*\simeq \rmodint{\cC}{A^{\psi^{-1}}}$ and ${}^*\cM\simeq \rmodint{\cC}{A^{\psi}}$ by \Cref{prop:*M*-algebras}. Now the result follows from \Cref{prop:A-Psi-exact}. 
\end{proof}


\subsection{Exactness of centralizers}
\label{sec:exactness-centralizer}

We will now show that for a finite $\cC$-module category, exactness of $\cC^*_\cM$ implies exactness of $\cM$. For this, we define, for an algebra $A$ in $\cC$, the functor 
\begin{align}
I^b\colon \cC\to \bimodints{A}{\cC}, \quad X \mapsto I^b(X)=A\otimes X \otimes A,\label{eq:functor-Ib}
\end{align}
where the left and right $A$-actions are given by the multiplication of $A$.
Recall from \eqref{eq:bimod1}--\eqref{eq:left-A-bimodule} that when $\cC$ is braided and $A$ an algebra in $\cC$, the category $\bimodints{A}{\cC}$ of internal $A$-bimodules becomes a $\cC$-module category. In this case, the functor $I^b$ becomes a left $\cC$-module functor with the $\cC$-module structure given by the isomorphism of $A$-bimodules
$$
\xymatrix@R=10pt{
I^b(X\lact Y)\ar@{=}[d]\ar[rr]^{s^{I_b}_{X,Y}}&& X\lact I^b(Y)\ar@{=}[d]\\
A\otimes X\otimes Y \otimes A\ar[rr]^{\psi^{-1}_{A,X}\otimes \id}&& X\otimes A\otimes Y\otimes A.
}
$$
Any projective object in $\bimodints{A}{\cC}$ is a direct summand of an object of an induced bimodule $I^b(Q)$, for $Q$ a projective object in $\cC$, cf.\ \cite{EGNO}*{Sec.\,7.8}.

\begin{lemma}\label{lem:Fun-projectives}
Let $\cM,\cN$ be left $\cC$-module categories. Let $F\colon \cM\to \cN$ be a projective object in $\Fun_\cC(\cM,\cN)$. Then $F(M)$ is projective for every $M\in\cM$.
\end{lemma}
\begin{proof}
By \cite{EGNO}*{Cor.\,7.10.5} or \cite{DSS1}*{Thm.\,2.18}, we can assume that there exist algebras $A$ and $B$ in $\cC$, such that $\cM= \rmodint{\cC}{A}$ and $\cN= \rmodint{\cC}{B}$.
Using the equivalence  
 $\bimodint{A}{\cC}{B}\isomorph \Fun_\cC(\rmodint{\cC}{A}, \rmodint{\cC}{B}), V \mapsto G_V,$ 
 from \Cref{prop:fun-A-bimod-B}, it suffices to consider $\cC$-linear functors of the form
 $$G_V\colon\rmodint{\cC}{A}\to \rmodint{\cC}{B}, \quad M\mapsto M\otimes_A V,$$
for an $A$-$B$-bimodule $V$ in $\cC$. Assume first that $V=A\otimes P\otimes B$, for $P$ a projective object in $\cC$ and with the $A$- and $B$-actions given by left and right multiplication, respectively. Recall that any projective $A$-$B$-bimodule is a direct summand of such bimodules. Then, for any $M$ in $\rmodint{\cC}{A}$, we have $G_V(M)\cong I(M\otimes P)$, as a right $B$-module  in $\cC$, for the functor $I$ defined by replacing $A$ by $B$ in \eqref{eq:ind-functor}. Since $M\otimes P$ is a projective object in $\cC$, $G_V(M)$ is a projective object  in $\rmodint{\cC}{B}\simeq \cN$. As $\otimes_A$ is biadditive, it preserves direct sums. If $W$ is only a direct summand of the projective bimodule $V=A\otimes P\otimes B$, then $G_W(M)$ is a direct summand of $G_V(M)$, for any $M$, which is, again, projective in $\rmodint{\cC}{B}$. This proves the claim.  
\end{proof}

For the next lemma, recall the \action functor from \Cref{def:AC-C-mod-functor}. We say that a functor $F$ \emph{preserves projective objects} if $F(P)$ is projective for any projective $P$.

\begin{lemma}\label{lem:AC-proj-M-exact}
    Let $\cM$ be a  left $\cC$-module category. If $A_\cM\colon \cC\to \cC^*_\cM$ preserves projective objects, then $\cM$ is exact as a left $\cC$-module category.  
\end{lemma}
\begin{proof}
    Assume that $A_\cM$ preserves projective objects. Then $A_\cM(P)=P\lact \id_\cM$ is a projective object in $\Fun_\cC(\cM,\cM)$.  Now, by \Cref{lem:Fun-projectives}, $A_\cM(P)$ sends arbitrary objects in $\cM$ to projective objects in $\cM$. Thus,
    $A_\cM(P)(M)=P\lact M$ is projective in $\cM$, for any object $M$ of $\cM$. This shows that $\cM$ is exact as a left $\cC$-module category. 
\end{proof}

\begin{proposition}\label{lem:AC-exact-equiv}
Let $\cM$ be a left $\cC$-module category. Then $\cC^*_\cM$ is exact as a left $\cC$-module category if and only if the functor $A_\cM\colon \cC\to \cC^*_\cM$ preserves projective objects.
\end{proposition}
\begin{proof}
    The forward implication is clear since $A_\cM(X)=X\lact \id_\cM$ by \Cref{def:Fun-action}. And thus, if $\cC^*_\cM$ is exact, $A_\cM(P)$  is projective for any projective object $P$ of $\cC$. Conversely, assume that $A_\cM$ preserves projective objects. By \Cref{lem:AC-proj-M-exact}, this implies that $\cM$ is an exact left $\cC$-module category. Thus, by results of \cite{DSS2}*{Thm.\,2.3.5}, see also \cite{EGNO}*{Section~7.12}, any $\Bbbk$-linear additive $\cC$-module endofunctor of $\cM$ is exact and $\cC^*_\cM$ has left and right duals. Thus, projective objects in $\cC^*_\cM$ form a (right) tensor ideal, see \Cref{lem:tensor-cat-properties}. Hence, the tensor product $A_\cM(P)\circ F$ is projective for any $F\in \cC^*_\cM$. Thus, by \Cref{lem:AC-braided}, $P\lact F=A_\cM(P)\circ F$ is projective and $\cC^*_\cM$ is exact.
\end{proof}

\begin{remark}\label{rem:perfect-conditions}
    The condition that the tensor functor $A_\cM$ preserves projective objects admits several equivalent characterizations, see \cite{Shi5}*{Lem.\,4.3}. This condition is equivalent to the right adjoint of $A_\cM$ being exact, or, admitting a further right adjoint. Such tensor functors are also called \emph{perfect} in the literature, see \cites{BrNa,Shi5,JY25}.
Consider the following commutative diagram of functors obtained from the action functor $\rho_\cM\colon \cC\to \End(\cM)$ from \eqref{eq:rho-lact}: 
\begin{equation}\label{eq:diag-center}
\vcenter{\hbox{
\xymatrix{
\cZ(\cC)\ar[d]^F\ar[rrr]^{\cZ(\rho_\cM)} &&& \End_\cC(\cM)\ar[d]^U.\\
\cC\ar[rrr]^{\rho_\cM}\ar[urrr]^{A_\cM}&&&\End(\cM)
}}}
\end{equation}
Here, $\cZ(\cC)$ is the Drinfeld center of $\cC$ with forgetful functor $F$, and where the functor $U$ forgets the $\cC$-linear structure.
The functor $\cZ(\rho_\cM)$ appeared in \cite{Shi2}*{Thm.\,3.11}. 
For $\cM$ exact, the functor $\rho_\cM$ is perfect \cite{Shi2}*{Thm.\,3.4}. Hence, $\cZ(\rho_\cM)$ is also perfect. This can be seen by the characterization of perfect functors from \cite{Shi5}*{Lem.\,4.3} and using \cite{Will}*{Lem.\,3.14--3.15}. 
Hence, the horizontal functors $\rho_\cM$ and $\cZ(\rho_\cM)$ in \eqref{eq:diag-center} preserve projective objects, while $A_\cM$ does in general not. 
Those exact $\cM$ where $A_\cM$ is perfect form a special class that we discuss in the \Cref{sec:fully-exact}.
\end{remark}

\subsection{Failure of closure of exact module categories under relative Deligne product}
\label{sec:rel-Del-exact}

Exact $\cC$-module categories are closed under the Deligne tensor product $\boxtimes$ (over $\Vect$) by \cite{DSS2}*{Section~3.3.2}.
However, if either $\cC$ is non-semisimple, or $\cC$ is semisimple but $\Bbbk$ is not of characteristic zero, then exact finite $\cC$-module categories are, in general, not closed under the \emph{relative} Deligne tensor product $\boxtimes_\cC$ and, therefore, $\lemod{\cC}$ is not a monoidal $2$-subcategory of $\lmod{\cC}$. We demonstrate this in two examples, the first one over a field of finite characteristic and the second one over $\mC$.

\begin{example}\label{ex:first}
Let $\cha\Bbbk$ divide the order of a finite abelian group $G$. The category $\cC=\Vect_G$ is the category of $G$-graded $\Bbbk$-vector spaces with the standard monoidal structure and symmetric braiding inherited from $\Vect$. In this case, 
we have an equivalence of $\cC$-modules $\Vect\simeq {}^*\Vect$ and that 
$$\Vect\boxtimes_{\Vect_G}\Vect \simeq \cC^*_\Vect\simeq \lmod{\Bbbk G},$$
applying \Cref{prop:rel-tensor-as-hom} and \cite{EGNO}*{Ex.\,7.12.19} in the second equivalence, which is monoidal. Under this equivalence, the \action functor $A_\Vect$ corresponds to a monoidal functor $\Vect_G \to \lmod{\Bbbk G}$. Such a functor cannot preserve projective objects as the tensor unit of $\Vect_G$ is projective by semisimplicity while the tensor unit of $\lmod{\Bbbk G}$ is not projective. By \Cref{lem:AC-exact-equiv}, $\cC^*_\Vect$ is not exact and hence $\Vect\boxtimes_\cC \Vect$ is not exact.
\end{example}

\begin{example}\label{ex:Sweedler-first}
Consider $\cC=\lmod{S}$ for the Sweedler's four-dimensional Hopf algebra  $S$, which is isomorphic to the bosonization Hopf algebra $B\rtimes \mC C_2$, for $B=\mC[x]/(x^2)$ a commutative and cocommutative Hopf algebra in $\sVect$, see \Cref{sec:Sweedler}. Hence, $\sVect$ admits a left $\cC$-module structure through the tensor functor $F\colon \cC\to \sVect$ which forgets the $B$-action. By the fundamental theorem for Hopf modules (in $\sVect$), we have an equivalence of left $\cC$-module categories 
$\sVect\simeq \rmodint{\cC}{B^*},$ 
where $B^*=\mC[y]/(y^2)$. The braiding $\psi$ of $\sVect$   induces a braiding on $\cC$ and $B^*$ is a commutative exact algebra in $\cC$ with respect to this braiding. Hence, $(B^*)^\psi=B^*=(B^*)^{\psi^{-1}}$ as algebras in $\cC$ and, applying \Cref{lem:B-via-algebra}, we  have an equivalence of $\cC$-bimodule categories 
 $$\rB(\sVect)\simeq \rB(\rmodint{\cC}{B^*})\simeq \lmodint{(B^*)^{\psi^{-1}}}{\cC}=\lmodint{B^*}{\cC},$$
 where $\rB(\cM)$ is the $\cC$-bimodule associated to a left $\cC$-module $\cM$, see \eqref{eq:Slr-B}.
 Again using the fundamental theorem of Hopf modules, the latter category is equivalent to $\sVect$ as a right $\cC$-module category. Hence, $\rB(\sVect)$ is equivalent to the $\cC$-bimodule defined by tensor functor $F$ via \Cref{ex:tensor-bimod}, which we simply denote by $\sVect$.

Recall from the proof of \Cref{rem:rel-Del-bimodule-realization} that 
the left $\cC$-module $\sVect\boxtimes_\cC \sVect$ can be identified with the category of modules over the monad 
$$T=(-)\ract B^*\cong (-)\otimes F(B^*)\colon \quad \sVect\to \sVect.$$  Thus, we have an equivalence of left $\cC$-module categories $\sVect\boxtimes_\cC \sVect\simeq \rmodint{\sVect}{B^*}$. The left $\cC$-action on $\rmodint{\sVect}{B^*}$ is inherited from the left $\cC$-action on $\sVect$. For a right $B^*$-module $V$ and an object $X$ of $\cC$, we thus have
$X\lact V = F(X)\otimes V$, where the right $B^*$-action is given by 
$$(x\otimes v)\cdot f= x\otimes vf,$$
for $x\in X,v\in V$, and $f\in B^*$. 
Consider the right $B^*$-module $\one$ given by the tensor unit of $\sVect$, where $y$ acts by zero. Then, for any object $X$ in $\cC$, $y$ still acts by zero on $X\lact \one=F(X)\otimes \one$. Thus, $X\lact \one$ is not projective and $\sVect\boxtimes_\cC \sVect$ is not exact. 
\end{example}

Another example is obtained in \Cref{cor:rel-Del-not-exact-factorizable} where we show that ${}^*\Vect\boxtimes_\cC \Vect$ is not exact, where $\cC=\lmod{u_q(\mathfrak{sl}_2)}$ is the non-degenerate braided tensor category of finite-dimensional modules over the factorizable small quantum group of $\mathfrak{sl}_2$ at a root of unity $q$ of odd order.

\medskip

The above examples show that exact $\cC$-module categories do \emph{not} form a monoidal $2$-subcategory of $\lmod{\cC}$, in both zero and non-zero characteristics, and both symmetric and non-degenerate braiding cases.
Equivalently, $A,B$ being exact algebras does, in general, \emph{not} imply that $A\otimes B$ is an exact algebra. \Cref{ex:Sweedler-first} also shows that the tensor product of commutative exact algebras is not exact in general.

\begin{remark}\label{rem:literature}
In the more general setup of $\cC$-bimodule categories, the result \cite{DN}*{Prop.\,2.10} showing that if $\cM$ and $\cN$ are exact $\cC$-bimodule categories,  then $\cM\boxtimes_\cC \cN$ is an exact $\cC$-bimodule category, is incorrect.\footnote{However, this does not affect the main results of \cite{DN}.} 
Using the monoidal $2$-functor $\rB\colon \lmod{\cC}\to \bimod{\cC}$ from \Cref{thm:CMod-monoidal-2-cat}, \Cref{ex:first,ex:Sweedler-first}  give counterexamples to this statement. 

As a counterexample with no braiding assumption, consider $\Vect$ as a $\cC$-bimodule, via the fiber functor $F\colon \cC\to \Vect$ as in \Cref{ex:tensor-bimod}. Then $\cC=\lmod{H}$ and
$F\colon \cC\to \Vect$ as in \Cref{ex:tensor-bimod}. Then 
    $$\cC^*_\Vect=\Fun_\cC(\Vect,\Vect)\simeq \lmod{H^*},$$
see \cite{EGNO}*{Ex.\,7.12.26} or \Cref{sec:exactness-vect}. By \cite{DSS2}*{Cor.\,2.4.11}, $\Vect\boxtimes_\cC \Vect\simeq \cC^*_\Vect$ as $\cC$-bimodules.
The left and right $\cC$-actions on $\Vect\boxtimes_\cC \Vect$ are inherited from $\Vect$. Thus, $\Vect\boxtimes_\cC \Vect \simeq \cC^*_\Vect$ is not exact as a $\cC$-bimodule provided that $H^*$ is not semisimple, in particular, if $\cha \Bbbk=0$ and $H$ is non-semisimple.

The gap in the proof of \cite{DN}*{Prop.\,2.10} is that while $\Fun_\cC(\cM, \cN)$ is exact as a $\cC^*_\cN$\,-\,$\cC^*_\cM$-bimodule category, this does not imply exactness as a $\cC$-bimodule category unless the functors $A^r_\cM\colon \cC^{\otimesop}\to \cC^*_\cM$ and $A^r_\cN\colon \cC^{\otimesop}\to \cC^*_\cN$ of \Cref{sec:left-right-action-fun} preserve projective objects.
\end{remark}

The failure of the relative Deligne product preserving exact module categories motivates the concept of \emph{fully exact} module categories that we introduce  in the next \Cref{sec:fully-exact}. Fully exact module categories have the important property that the functor $A^r_\cM\colon \cC^{\otimesop}\to \cC^*_\cM$ mentioned in \Cref{eq:Ar} does preserve projective objects, and as a consequence this class is closed under relative Deligne product, and hence provides a monoidal $2$-subcategory of $\lmod{\cC}$. If $\cC$ is a braided fusion category and $\cha \Bbbk=0$, then these $2$-subcategories of exact and fully exact $\cC$-modules coincide.


\section{Fully exact module categories}
\label{sec:fully-exact}

Throughout this section, we assume that $\cC$ is a finite braided tensor category over any field $\Bbbk$. In \Cref{sec:fully-exact-def}, we define fully exact $\cC$-module categories and show that they form a monoidal $2$-subcategory of $\lmod{\cC}$. \Cref{sec:fully-exact-alg} defines fully exact algebras and \Cref{sec:rel-projective} shows that algebras that are relatively projective as bimodules, a generalization of the class of separable algebras, are fully exact.

\subsection{Fully exact module categories}
\label{sec:fully-exact-def}

\begin{definition}\label{def:fully-exact}
We say that a $\cC$-module category $\cM$ is \emph{fully exact} if, for any exact $\cC$-module category $\cN$, 
$\cN\boxtimes_\cC\cM$ is exact as a left $\cC$-module category.

We denote the  full $2$-subcategory in $\lmod{\cC}$ of fully exact $\cC$-module categories by $\fexmod{\cC}$.
\end{definition}

\begin{lemma}\label{rem:fully-exact-exact}
If a left $\cC$-module category $\cM$ is fully exact, then $\cM$ is exact
\end{lemma}
\begin{proof}
    Setting $\cN=\cC$ in \Cref{def:fully-exact}, for the regular $\cC$-module category $\cC$,  $\cC\boxtimes_\cC \cM\simeq \cM$ is exact. 
\end{proof}

\begin{example}\label{ex:fully-exact-first}$~$
 \begin{enumerate}[(i)]
 \item 
    The regular $\cC$-module category $\cC$, the monoidal unit of $\lmod{\cC}$, is fully exact by virtue of the equivalences 
    $\cN\boxtimes_\cC \cC\simeq \cN.$
\item As we will show below in \Cref{prop:rel-proj-fully-exact}, if a left $\cC$-module category $\cM$ is separable, see \Cref{def:separable}, then $\cM$ is fully exact. 
\item Recall from \Cref{ex:first} that, over a field $\Bbbk$ of characteristic dividing the order of the group $G$, for $\cC=\Vect_G$ and the exact $\cC$-module  $\cM=\Vect$, the relative Deligne product $\Vect\boxtimes_\cC\Vect$ is not an exact $\cC$-module. Hence, $\Vect$ is not fully exact in this case.
\item
Even if $\cha \Bbbk=0$, the module category $\Vect$ is also often not fully exact. We will later discuss as examples the symmetric tensor category $\lmod{S}$, where $S$ is Sweedler's four-dimensional Hopf algebra (see \Cref{sec:Sweedler}) and $\lmod{u_q(\mathfrak{sl}_2)}$, where the Hopf algebra $u_q(\mathfrak{sl}_2)$ is factorizable (see \Cref{sec:uqsl2}). 
\end{enumerate}
\end{example}

Our motivation for defining fully exact module categories is their closure under relative Deligne products.

\begin{proposition}\label{prop:fully-exact-tensor}
    If $\cM_1$ and $\cM_2$ are fully exact $\cC$-module categories, then so is their relative Deligne product $\cM_1\boxtimes_\cC \cM_2$. Hence, $\fexmod{\cC}$ is a monoidal $2$-subcategory of $\lmod{\cC}$.
\end{proposition}
\begin{proof}
    Assume that $\cM_1$ and $\cM_2$ are fully exact $\cC$-module categories. Let $\cN$ be an exact $\cC$-module category. Now, $\cN\boxtimes_\cC \cM_1$ is exact, using that $\cM_1$ is fully exact and then  $(\cN\boxtimes_\cC \cM_1)\boxtimes_\cC \cM_2$ is exact since $\cM_2$ is fully exact. 
    Thus, the equivalent category
    $$\cN\boxtimes_\cC (\cM_1\boxtimes_\cC \cM_2)\simeq (\cN\boxtimes_\cC \cM_1)\boxtimes_\cC \cM_2,$$
    via the associator constructed in \Cref{def:2associator}, is exact, implying that $\cM_1\boxtimes_\cC \cM_2$ is fully exact. 
\end{proof}

We can now derive equivalent conditions for fully exact $\cC$-module categories. 

\begin{theorem}
\label{prop:fully-exact-equiv}
Let $\cC$ be a braided finite tensor category over any field $\Bbbk$.
The following are equivalent for a finite left $\cC$-module category $\cM$:
\begin{enumerate}[(i)]
 \item     $\cM$ is fully exact.
  \item For any exact $\cC$-module  $\cN$, the left $\cC$-module category $\Fun_\cC(\cN,\cM)$ from \Cref{def:Fun-action} is exact.
    \item 
    The centralizer $\cC^*_\cM=\Fun_\cC(\cM,\cM)\simeq {}^*\cM\boxtimes_\cC\cM$ is exact as a left $\cC$-module category.
    \item The functor $A_\cM$ preserves projective objects. 
\end{enumerate}
\end{theorem}
\begin{proof}
The equivalence of (i) and (ii) follows from \Cref{cor:duals-exact}, which shows that $\cN$ is exact if and only if ${}^*\cN$ is exact and the equivalence of $\cC$-module categories ${}^*\cN\boxtimes_\cC\cM\simeq \Fun_\cC(\cN,\cM)$ from \Cref{prop:rel-tensor-as-hom}.

Clearly, (ii) implies (iii) by specializing to 
$\cC^*_\cM=\Fun_\cC(\cM,\cM)$ by taking $\cN=\cM$.

The equivalence of (iii) and (iv) was proved in \Cref{lem:AC-exact-equiv}.

To conclude, we show that (iv) implies (ii). Since $\cM$ is exact by \Cref{lem:AC-proj-M-exact}, then for any exact $\cN$ we have $\Fun_\cC(\cN,\cM)$ is exact as a left $\cC^*_\cM$-module category by \cite{EGNO}*{Prop.\,7.12.14}. Thus, as the monoidal functor $A_\cM$ preserves projective objects, $\Fun_\cC(\cN,\cM)$ is also exact as a left $\cC$-module category via restriction of the action along the monoidal functor $A_\cM$ as in \Cref{ex:tensor-act}. As shown in \Cref{lem:AC-braided}, this action agrees with the one defined in \Cref{def:Fun-action}. 
\end{proof}

\begin{corollary}\label{cor:semisimple-fully-exact}
Assume that $\cC$ is a braided fusion category over a  field $\Bbbk$. Then the following are equivalent for a left $\cC$-module category $\cM$:
\begin{enumerate}[(i)]
    \item $\cM$ is separable,
    \item $\cM$ is fully exact,
    \item $\cC^*_\cM$ is semisimple.
\end{enumerate}
If, in addition,  $\Bbbk$ is algebraically closed and $\cC$ has non-zero dimension (e.g.\ if $\cha \Bbbk =0$), then (i)--(iii) are further equivalent to each of the following conditions:
\begin{enumerate}[(i)]
    \item[(iv)] $\cM$ is exact,
    \item[(v)] $\cM$ is semisimple.
\end{enumerate}

\end{corollary}
\begin{proof}
By  \cite{DSS2}*{Thm.\,2.5.4}, (i) and (iii) are equivalent. By \Cref{ex:exact}(i), (iii) is equivalent to $\cC^*_\cM$ being exact, which, by \Cref{prop:fully-exact-equiv}, is equivalent to (ii).
By the same example, (iv) and (v) are equivalent. By  \cite{DSS2}*{Prop.\,2.5.3}, if $\cM$ is  separable, then it is semisimple, so (i) implies~(v).
 Under the stronger set of assumptions that $\Bbbk$ is algebraically closed and $\dim\cC\neq 0$, e.g.\ when $\cha \Bbbk=0$, then (v) implies (i). Indeed, under these assumptions, the Drinfeld center $\cZ(\cC)$ of a fusion category $\cC$ is semisimple, see \cite{EGNO}*{Thm.\,9.3.2}. Hence, by \cite{DSS2}*{Cor.\,2.5.9\,\&\,Prop.\,2.5.10}, this shows that any semisimple left $\cC$-module category $\cM$ is separable. 
\end{proof}

Note that the assumptions on the field $\Bbbk$ for equivalence of all conditions (i)--(v) in \Cref{cor:semisimple-fully-exact} are necessary. \Cref{ex:fully-exact-first}(iii) provides  a module category $\cM$ over a fusion category $\cC$ over an algebraically closed field of characteristic $p>0$ that is exact (semisimple) but not fully exact. In particular, $\cM$ is not separable by \Cref{cor:semisimple-fully-exact}.
It might be possible to relax that $\Bbbk$ is algebraically closed to assuming that $\Bbbk$ is a perfect field in the second part of \Cref{cor:semisimple-fully-exact}, see \cite{DSS2}*{Remark 2.6.10}, and of course keeping the assumption  $\dim \cC \neq 0$.
For example, for $\cC=\Vect$, the second part of \Cref{cor:semisimple-fully-exact} implies that the product of two simple algebras in $\Vect$ is  semisimple over any algebraically closed field. However, this semisimplicity statement is very well known to hold actually for any perfect field.

\begin{remark}\label{rem:separable2}
If $\cC$ is a fusion category over any  field $\Bbbk$, then \cite{DSS2}*{Thm.\,2.5.4} shows that $\cC^*_\cM$ is semisimple if and only if $\cM$ is separable.
When $\cC$ is a braided fusion category, this shows that $\cM$ is a separable left $\cC$-module if and only if $\cC^*_\cM$ is an exact left $\cC$-module, using \Cref{ex:exact}(i). The equivalence of (i) and (iii) in \Cref{prop:fully-exact-equiv} above is an analogue of this statement for non-semisimple tensor categories $\cC$. This suggests that fully exact module categories are a good analogue of separable module categories for non-semisimple $\cC$. In \Cref{lem:A-fully-exact-equiv} below, we give a characterization of fully exact algebras akin to the definition of a separable algebra.
\end{remark}

\begin{proposition}\label{lem:Res-fully-exacts}
For an equivalence $F\colon \cC\to\cD$ of finite braided tensor categories, the monoidal $2$-functor $\Res_F\colon \lmod{\cD}\to \lmod{\cC}$ from \Cref{prop:res} preserves fully exact module categories.    
\end{proposition}
\begin{proof}
    In fact, any monoidal $2$-equivalence $\rR\colon \lmod{\cD}\to \lmod{\cC}$ which preserves exact module categories also preserves fully exact module categories. If $\cM$ is fully exact and $\cN$ is exact, then
    $$R(\cN\boxtimes_\cD\cM)\simeq R(\cN)\boxtimes_\cC R(\cM),$$
    and $R(\cN\boxtimes_\cD\cM)$ is exact. Since any exact $\cC$-module category is equivalent to one of the form $R(\cN)$, $R(\cM)$ is fully exact.
\end{proof}

\subsection{Fully exact algebras}\label{sec:fully-exact-alg}

We now return to the case when $\cC$ is a general braided finite tensor category over any field $\Bbbk$.

\begin{definition}[Fully exact algebras]\label{def:exactproj-A}
We say an algebra $A$ in $\cC$ is \emph{fully exact} if $\rmodint{\cC}{A}$ is fully exact as a left $\cC$-module category. 
\end{definition}

We have the following equivalent conditions for fully exact algebras.

\begin{corollary}\label{cor:proj-exact-algebras}
The following statements are equivalent for an algebra $A$ in $\cC$.
\begin{enumerate}[(i)]
    \item $A$ is a fully exact algebra.
    \item $A^\psi\otimes A$ is an exact algebra.
    \item $A^{\psi^{-1}}\otimes A$ is an exact algebra.
    \item For any exact algebra $B$ in $\cC$, $B\otimes A$ is an exact algebra. 
\end{enumerate}
\end{corollary}
\begin{proof}
    Consider the left $\cC$-module categories $\cM=\rmodint{\cC}{A}$ and $\cN=\rmodint{\cC}{B}$ and recall the $\cC$-linear equivalences $\cC^*_\cM=\Fun_\cC(\cM,\cM)\simeq \rmodint{\cC}{A^\psi\otimes A}$ and $\cN\boxtimes_\cC \cM\simeq \rmodint{\cC}{B\otimes A}$ from \Cref{prop:Funrex-modules} and \Cref{prop:rel-Del-modules}, respectively. The equivalence of (i), (ii), (iv) now follows from \Cref{prop:fully-exact-equiv}. 
    
    The implication that (iv) implies (iii) is clear, setting $B=A^{\psi^{-1}}$. We prove that (iii) implies (ii). Assume that $A^{\psi^{-1}}\otimes A$ is exact. Then, by \Cref{prop:A-Psi-exact}, 
    $$(A^{\psi^{-1}}\otimes A)^\psi\cong A^\psi\otimes A$$ is exact, where the isomorphism of algebras is taken from \Cref{lem:tensor-ops}. This completes the proof.
\end{proof}

\subsection{Relatively projective algebras}
\label{sec:rel-projective}
For an algebra $A$ in $\cC$, we recall the induced $A$-bimodules $I^b(X)=A\otimes X\otimes A$, for $X\in \cC$, introduced in~\eqref{eq:functor-Ib}.

\begin{definition}\label{def:rel-proj-bimodule}
    We say that an algebra $A$ in  $\cC$ is \emph{relatively projective as a bimodule} if the $A$-bimodule $A$ is a direct summand of an induced bimodule $I^b(X)=A\otimes X \otimes A$, for some  $X\in\cC$.
    \end{definition}
    
A special case considered in the literature (cf.\ \cite{DSS2}*{Sec.\,2.5}) is that of separable algebras, cf.\ \Cref{def:separable}. In this case, one can choose $X=\one$ in \Cref{def:rel-proj-bimodule} to see that any separable algebra is relatively projective.

 Let us comment on the terminology of \Cref{def:rel-proj-bimodule}. The functor $I^b\colon \cC\to \bimodints{A}{\cC}$ is left adjoint to the functor $U\colon \bimodints{A}{\cC}\to \cC$ which forgets the bimodule structure. Note that the functor $U$ is additive, exact, faithful and hence $I^b\dashv U$ is \emph{resolvent pair} \cite{ML2}*{Sec.\, IX.6}. Recall the definition of  a \emph{relatively projective} object for a resolvent pair from \cite{ML2}*{Ch.\, IX}.
By \cite{FGS}*{Lem.\, 2.1}, relatively projective objects for the resolvent pair $I^b\dashv U$ correspond precisely to those $A$-bimodules  that are a direct summand of $I^b(X)$, for some object $X$ in $\cC$.

\begin{definition}\label{def:rel-proj-cat}
We say that a left $\cC$-module  $\cM$ is \emph{relatively projective} if it is equivalent as a $\cC$-module category to $\rmodint{\cC}{A}$ for an algebra $A$ in $\cC$ that is relatively projective as an $A$-bimodule. 
\end{definition}

\begin{lemma}\label{lem:rel-proj-rel-ss}
    If $\cM$ is relatively projective, then $\cM$ is relatively semisimple in the sense of \Cref{def:M-rel-ss}.
\end{lemma}
\begin{proof}
For any $X$ in $\cC$, forgetting the left $A$-action of the $A$-bimodule $I^b(X)$ from \eqref{eq:functor-Ib}, we obtain the right $A$-module $I(A\otimes X)$ from \eqref{eq:ind-functor}.
Hence, if $\cM=\rmodint{\cC}{A}$ for $A$ relatively projective as a bimodule, then  $A$ is a direct summand of the right $A$-module $I(A\otimes X)$ because the forgetful functor is additive. This means that $A$ is relatively semisimple as an algebra in $A$, see \Cref{def:proj-alg}, and hence $\cM$ is relatively semisimple.
\end{proof}

Relative projectivity of algebras  as bimodules is  stable under the tensor product $A\otimes B=A\otimes^\psi B$ from \Cref{def:tensor-algebras}.

\begin{proposition}\label{prop:rel-proj-AB}
Assume $\cC$ is braided. If $A$ and $B$ are relatively projective as bimodules, then $A\otimes B$ is relatively projective as a bimodule.
\end{proposition}
\begin{proof}
Assume that both $A$ and $B$ are relatively projective as bimodules. Thus, we find objects $X$ and $Y$ in $\cC$ and bimodule maps 
\begin{align*}
    \iota_A&\colon A\hookrightarrow A\otimes X \otimes A, & \pi_A\colon A\otimes X \otimes A \twoheadrightarrow A, \\
\iota_B&\colon B\hookrightarrow B\otimes Y \otimes B, & \pi_B\colon B\otimes Y \otimes B \twoheadrightarrow B,
\end{align*}
of bimodules over $A$, respectively, over $B$, such that 
\begin{align*}
    \pi_A\iota_A=\id_A, &&     \pi_B\iota_B=\id_B.
\end{align*}
Now consider the following maps
\begin{align*}
    \iota_{A\otimes B}=&(\id_{A}\otimes \psi^{-1}_{X,B}\otimes \psi^{-1}_{A,Y}\otimes \id_B)(\id_{A\otimes X}\otimes \psi^{-1}_{A,B}\otimes \id_{Y\otimes B})(\iota_A\otimes \iota_B)\colon \\ &A\otimes B \hookrightarrow A\otimes B\otimes X\otimes Y \otimes A \otimes B,\\
\pi_{A\otimes B}=&(\pi_A\otimes \pi_B)(\id_{A\otimes X}\otimes \psi_{B,A}\otimes \id_{Y\otimes B})(\id_{A}\otimes \psi_{B,X}\otimes \psi_{Y,A}\otimes \id_B)\colon \\ &  A\otimes B\otimes X\otimes Y \otimes A \otimes B\twoheadrightarrow A\otimes B.
\end{align*}
One checks that, indeed, $\iota_{A\otimes B}$ and $\pi_{A\otimes B}$ are morphisms of $A\otimes B$-bimodules with respect to the regular left and right actions of this tensor product algebra. 
It follows directly by construction that 
$\pi_{A\otimes B}\iota_{A\otimes B}=\id_{A\otimes B}.$
Thus, $A\otimes B$ is relatively projective as a bimodule.
\end{proof}

We obtain the following equivalent characterization of fully exact algebras, for which we recall the $\cC$-action on $A$-bimodules in $\cC$ from \eqref{eq:bimod1}--\eqref{eq:left-A-bimodule}. 

\begin{proposition}\label{lem:A-fully-exact-equiv}
    Let $A$ be an algebra in $\cC$ with product map $m_A\colon A\otimes A \to A$ and $P$ some projective object in $\cC$.
    Then the following are equivalent:
    \begin{enumerate}
        \item $A$ is fully exact.
        \item The map $P\lact m_A$ splits as a morphism in $\bimodints{A}{\cC}$.
        \item $P\lact A$ is projective in 
        $\bimodints{A}{\cC}$.
    \end{enumerate}
\end{proposition}
\begin{proof}
Recall that any projective object in $\bimodints{A}{\cC}$ is a direct summand of an induced bimodule $I^b(Q)$, for $Q$ a projective object in $\cC$, cf.\ \cite{EGNO}*{Sec.\,7.8}. We saw in \Cref{sec:exactness-centralizer} that the functor $I^b$ is $\cC$-linear. In particular, we have isomorphisms of $A$-bimodules
    $Q\lact I^b(X) \cong I^b(Q\otimes X),$
    for objects $Q,X$ in $\cC$.
By \Cref{lem:AC-exact-equiv}, $\rmodint{\cC}{A}$ is fully exact 
if and only if the composition 
$$\cC \xrightarrow{~A_\cM~} \cC^*_\cM\xrightarrow{E} \bimodints{A}{\cC},$$
where $E$ is the equivalence from \Cref{prop:fun-A-bimod-B}, preserves projective objects. Hence, applying \Cref{lem:AC-braided}, $A$ is fully exact if and only if for any projective $P$,
\begin{equation}\label{eq:EA-P-act}
(E\circ A_\cM)(P)=E(P\lact \id_\cM)=P\lact A,
\end{equation}
is a direct summand of $I^b(Q)$ for some projective $Q$.
Equivalently, $P\lact A$ is projective in $\bimodints{A}{\cC}$. Since $m_A\colon A\otimes A\to A$ is a morphism in the left $\cC$-module category $\bimodints{A}{\cC}$, $P\lact m_A$ is also a morphism in this category. Hence, the composition
$$I^b(P)=I^b(P\lact\one)\isomorph P\lact I^b(\one)=P\lact (A\otimes A)\xrightarrow{P\lact m_A}P\lact A,$$
where the first isomorphism is the $\cC$-linear structure of the functor $I^b$ from \Cref{eq:functor-Ib},
is a surjective morphism in $\bimodints{A}{\cC}$ and hence splits. Thus, we can choose $Q=P$ and $P\lact A$ is a direct summand of $P\lact (A\otimes A)$ as an $A$-bimodule in $\cC$. We have now shown that condition (1) is equivalent to conditions (2) or (3)  holding for \emph{all} projective objects $P$. We shown below that (1) is also equivalent to these conditions for at least one projective.

Now recall from~\cite{BGR}*{Prop.\,2.3(1)} that the (thick) tensor ideal of projective objects in $\cC$ is generated (in the sense of taking retracts) by any fixed non-zero projective object $P$. 
Using additivity of the $\lact$ functor, we thus can replace  the action of any projective object in~\eqref{eq:EA-P-act} by the action of $X\otimes P$ for one fixed projective $P$ and some $X\in \cC$.  The statement now follows from the fact that for any $X$, the functor $X\lact (-)$ preserves projective objects in $\cM$.
\end{proof}

The equivalent property (2) of a fully exact algebra $A$ of \Cref{lem:A-fully-exact-equiv} may be referred to by saying that  $A$ is \emph{projectively separable},
recall our discussion in~\Cref{secIntro:sep}.
We emphasize that being separable, or, more generally, relatively projective as a bimodule is a sufficient condition for being fully exact.

\begin{proposition}\label{prop:rel-proj-fully-exact}
    If $A$ is relatively projective as a bimodule, then $A$ is fully exact. In particular, if $A$ is separable, then $A$ is fully exact.
\end{proposition}
\begin{proof}
Assume that $A$ is a direct summand of the $A$-bimodule $I^b(X)$. Let $P$ be a projective object in $\cC$. Then, because the action functor is additive in both components, $P\lact A$,  is a direct summand of the projective $A$-bimodule 
$P\lact I^b(X)\cong I^b(P\otimes X)$. 
Hence, the claim follows from \Cref{lem:A-fully-exact-equiv}.
\end{proof}

We will see later in \Cref{lem:separable-Sweedler} that there are examples of fully exact module categories that are not relatively projective.

\begin{corollary}
    For $\cC$ braided, the full $2$-subcategory on relatively projective
    $\cC$-module categories is a monoidal $2$-subcategory of $\fexmod{\cC}$.
\end{corollary}
\begin{proof}
    It is a direct consequence of \Cref{prop:rel-proj-AB} with $\cM=\rmodint{\cC}{A}$, $\cN=\rmodint{\cC}{B}$, for some choice of algebras $A$ and $B$ that are relatively projective as bimodules, that $\cM\boxtimes_\cC\cN$ is also relatively projective. By \Cref{prop:rel-proj-fully-exact}, we obtain a monoidal $2$-subcategory of the monoidal $2$-category of fully exact $\cC$-module categories. 
\end{proof}

\begin{remark}\label{rem:separable-semisimple}
Assume that $\cC$ is a finite tensor category. We observe that $A$ is projective as an $A$-bimodule if and only if the object $X$ in \Cref{def:rel-proj-bimodule} can be chosen to be projective. 
If $\cC$ is a fusion category, $A$ being separable is equivalent to $A$ being projective as an $A$-bimodule, which is also equivalent to $A$ being relatively projective as an $A$-bimodule. Hence, in the semisimple case, these properties are stable under Morita equivalence of algebras.
When $\cC$ is non-semisimple, the concepts of separable algebra and algebras being relatively projective as bimodules are \emph{not} stable  under Morita equivalence by \Cref{ex:sep-not-Morita}  because  every algebra that is relatively projective as an $A$-bimodule is, in particular, relatively semisimple by \Cref{lem:rel-proj-rel-ss}.
\end{remark}

We can also consider $\cC$-module categories $\cM=\rmodint{\cC}{A}$, where $A$ is \emph{projective} in $\bimodints{A}{\cC}$. This condition can be shown to be stable under Morita equivalence of internal algebras. However, at least in characteristic zero, we see that  there are no non-trivial algebras with this property for $\cC$ a non-semisimple finite braided tensor category.

\begin{corollary}\label{cor:A-proj-sep}
Assume that $\cha\Bbbk=0$. If $A$ is a non-zero algebra in $\cC$ which is projective as an object in $\bimodints{A}{\cC}$, then $\cC$ is semisimple and thus $A$ is separable.
\end{corollary}
\begin{proof}
Assume that $A$ is projective in $\bimodints{A}{\cC}$. Then $A$ is, in particular, relatively projective as an $A$-bimodule. Hence, by \Cref{prop:rel-proj-fully-exact}, $\cM$ is fully exact. Thus, $\cM$ and $\cC^*_\cM$ are exact left $\cC$-module categories. As $\cM$ is exact, $\cC^*_\cM$ is a multitensor category by \cite{DSS2}*{Thm.\,2.3.5}. Therefore, by \Cref{lem:tensor-cat-properties}, projective objects in $\cC^*_\cM$ form an ideal. Since the identity of 
$\cC^*_\cM \cong \bimodints{A}{\cC}$ is projective, every object is projective and $\cC^*_\cM$ is semisimple. 

Now, since $\cM$ is non-zero, we have that no object acts by zero by \Cref{lem:AM-exact-faithful}. Hence, \cite{EGNO}*{Thm.\,7.12.16} provides a $2$-equivalence 
$\lmod{\cC}\simeq \lmod{\cC^*_\cM}$ which preserves exact module categories. This $2$-equivalence maps the regular $\cC$-module $\cC$ to the $\cC^*_\cM$-module $\cN:=\Fun_\cC(\cM,\cC)$ and induces an equivalence
$\cC\simeq \Fun_{\cC}(\cC,\cC)\simeq \Fun_{\cC^*_\cM}(\cN,\cN)$.
If $\Bbbk$ is a  field of characteristic zero, $\cN$ being exact implies that $\Fun_{\cC^*_\cM}(\cN,\cN)$ is exact and hence semisimple by \Cref{cor:semisimple-fully-exact}. Thus,  $\cC$ is semisimple. As explained in \Cref{rem:separable-semisimple}, $A$ is then also separable.
\end{proof}

\section{Duals of fully exact module categories}
\label{sec:dual-fully-exact}

This section investigates duals of fully exact $\cC$-module categories over a finite braided tensor category $\cC$. After discussing preliminaries on duals of $\cC$-modules in \Cref{sec:duals-finite}, we introduce \emph{op-fully exact} $\cC$-module categories  and show that duals of fully exact module categories are op-fully exact in \Cref{sec:opfully}. We see by~\Cref{ex:double-group} that, in general, for non-symmetric $\cC$, fully exact module categories are not closed under duals. We find criteria, in terms of braided opposite algebras, for duals to be fully exact in \Cref{sec:dualfully}. Finally, in \Cref{sec:fully-dualizable}, we define $\cM$ to be a \emph{perfect} $\cC$-module category if $\cM$ is both fully exact and op-fully exact. 
We prove that perfect $\cC$-module categories are precisely the fully dualizable objects (in the sense of \cites{Lur,Pst}) in $\lmod{\cC}$, see \Cref{prop:dualizability}.

\subsection{Dualizability of finite module categories}\label{sec:duals-finite}

We start with a preliminary section to discuss, following \cite{DSS2}*{Section~3.2}, how the left and right dual $\cC$-modules introduced in \Cref{sec:duals} are duals in the monoidal $2$-category $\lmod{\cC}$ in the sense of \Cref{def:left-dual-cat}.

For a left $\cC$-modules $\cM,\cN,\cP$, consider the composition functor 
$$\circ \colon \Fun_\cC(\cM,\cN) \times \Fun_\cC(\cN, \cP)\to \Fun_\cC(\cM,\cP), \quad (G,H) \mapsto H\circ G.$$
We define a $\cC$-balancing for $\circ$ from the $\cC$-linear structure, namely, for $M$ in $\cM$, we define
$$\circ\big((G\ract X),H\big)(M)=H(X\lact G(M))\xrightarrow{(\beta_{G,X,H})_M:=s^H_{X,G(M)}}X\lact HG(M)=\circ\big(G, (X\lact H)\big)(M).$$
Here, we use that $(X\lact G)(M)=X\lact G(M)$ by \Cref{def:Fun-action} and $M\ract X=X\lact M$ by \eqref{eq:right-coherence}. The following computation checks the coherence condition  \eqref{eq:balancing-coherence} for the balancing $\beta$ defined above. Evaluated at an object $M$ in $\cM$, this follows because the upper diagram in 
\begin{align*}
\resizebox{\textwidth}{!}{
\xymatrix@R=15pt{
H(X\otimes Y\lact GM)\ar[rrrrr]^{s^H_{X\otimes Y,GM}=(\beta_{G,X\otimes Y,H})_M} & && & & X\otimes Y\lact HG(M)\\
& && &X\lact (Y\lact HGM)\ar[ru]^{l_{X,Y,HGM}} & \\
& H(Y\lact (X\lact GM))\ar[ldd]^{H(l_{Y,X,GM})}\ar[uul]_{H(r_{G,X,Y})_M}\ar[rr]^{(\beta_{G\ract X,Y,H})_M}_{s^H_{Y,X\lact GM}}&& Y\lact H(X\lact GM)\ar[ru]^{(\beta_{G,X,Y\lact H})_M}_{s^{Y\lact H}_{X,GM}}\ar[rd]^{Y\lact s^H_{X,GM}} && \\
& && &Y\lact (X\lact HGM)\ar[rd]^{l_{Y,X,HGM}} & \\
H(Y\otimes X\lact GM)\ar[rrrrr]_{s^H_{Y\otimes X,GM}}\ar[uuuu]_{H(\psi_{Y,X}\lact GM)}& && & & Y\otimes X\lact HG(M)\ar[uuuu]^{\psi_{Y,X}\lact HG(M)}
}}
\end{align*}
commutes. Here, the right diagram commutes by \eqref{eq:FunC-coherence}, the left diagram commutes by definition of $(r_{G,X,Y})_M=r_{GM,X,Y}$ in \eqref{eq:right-coherence}, the lower diagram commutes by coherence of $s^H$, see \eqref{eq:C-linear-coherence}, and the outer diagram commutes by naturality of $s^H_{-,GM}$ with respect to $\psi_{Y,X}$. Moreover, by strict unitality, it follows that $\beta_{G,\one,H}=s^H_{\one,G(M)}=\id$ as required.
Thus, by descending to the relative Deligne product, we obtain a well-defined functor of $\cC$-module categories 
\begin{equation}\label{eq:rel-circ}
\circ \colon \Fun_\cC(\cM,\cN) \boxtimes_\cC \Fun_\cC(\cN, \cP)\to \Fun_\cC(\cM,\cP).
\end{equation}

By \Cref{prop:rel-tensor-as-hom} and self-duality of $\cC$, there is an equivalence of $\cC$-module categories $\Fun_\cC(\cC,\cM)\simeq \cM$ and, in particular, $\Fun_\cC(\cC,\cC)\simeq \cC$. Together with the equivalence from \Cref{prop:rel-tensor-as-hom}(ii), we obtain a functor
\begin{equation}\label{eq:ev}
\ev^r_\cM:=\big( \cM\boxtimes_\cC {}^*\cM\isomorph \Fun_\cC(\cC,\cM)\boxtimes_\cC \Fun_{\cC}(\cM,\cC)\xrightarrow{~\circ~} \Fun_\cC(\cC,\cC)\isomorph \cC\big).
\end{equation}

\begin{lemma}\label{lem:ev-coherence}
    The following diagrams    of $\cC$-linear functors commute up to natural isomorphism:
    \begin{gather*}
    \xymatrix@C=45pt{
    \Fun_\cC(\cC,\cM)\boxtimes_\cC \Fun_{\cC}(\cM,\cM)\ar[r]^-{\circ}\ar[d]_{\sim}&\Fun_\cC(\cC,\cM)\ar[r]^{\sim}&\cM \\
    \cM\boxtimes_\cC({}^*\cM\boxtimes_\cC \cM)\ar[r]^{A_{\cM,{}^*\cM,\cM}} & (\cM\boxtimes_\cC {}^*\cM)\boxtimes_\cC \cM \ar[r]^-{\ev^r_\cM\boxtimes_\cC \id_\cM}& \cC\boxtimes_\cC \cM\ar[u]_{L_\cM},
    }\\[10pt]
    \xymatrix@C=45pt{
    \Fun_{\cC}(\cM,\cM)\boxtimes_\cC \Fun_\cC(\cM,\cC) \ar[r]^-{\circ }\ar[d]_{\sim}& \Fun_\cC(\cM,\cC)\ar[r]^{\sim}&{}^*\cM\\
    ({}^*\cM\boxtimes_\cC \cM)\boxtimes_\cC {}^*\cM \ar[r]^{A_{{}^*\cM,\cM,{}^*\cM}} & {}^*\cM\boxtimes_\cC (\cM\boxtimes_\cC {}^*\cM)  \ar[r]^-{\id_{{}^*\!\cM}\boxtimes_\cC \ev^r_\cM}& {}^*\cM\boxtimes_\cC \cC\ar[u]_{R_\cM}.
    }
\end{gather*}
Here, $\circ$ is the functor from \eqref{eq:rel-circ} and  $A$, $R$, $L$ denote the associator, right, and left unitor functors of $\lmod{\cC}$, see \eqref{eq:ass-Sec2}--\eqref{eq:unitors-Sec2}.  
\end{lemma}
\begin{proof}
Assume that $\cM=\rmodint{\cC}{A}$. Then we have equivalences of $\cC$-module categories 
$$\Fun_\cC(\cM,\cM)\simeq \rmodint{\cC}{A^\psi\otimes A}, \quad {}^*\cM\simeq \rmodint{\cC}{A^\psi}$$
and the relative Deligne products are simply given by internal right modules over the corresponding tensor product algebras, see \Cref{prop:Funrex-modules} and \Cref{prop:rel-Del-modules}. 
To show commutativity of the first diagram, it suffices to consider an object $V\otimes W\in \rmodint{\cC}{A^\psi\otimes A}$ and 
$$M\otimes V\otimes W\in \rmodint{\cC}{A\otimes A^\psi\otimes A}\simeq \cM\boxtimes_\cC \Fun_\cC(\cM,\cM).$$
On the one hand, applying the lower path of compositions of functors in the diagram of the statement of the lemma, we find that  
$$M\otimes V\otimes W \mapsto (M\otimes_A V)\otimes W.$$
On the other hand, the functor $\circ$ in the upper path is given by 
$$M\otimes V\otimes W \mapsto M\otimes_A(V\otimes W).$$
Since the tensor product is exact and the right $A^\psi$-action is defined on $V$ only, the two resulting right $A$-modules are isomorphic. The isomorphism is compatible with the left $\cC$-module coherence, which for all of the categories of internal modules is given by the associator of $\cC$.

Commutativity of the second diagram can be checked using internal modules in a similar way.
\end{proof}

\begin{lemma}\label{lem:duals-fun}
    The left $\cC$-modules ${}^*\cM$ and $\cM^*$ are right, respectively, left duals in $\lmod{\cC}$ in the sense of \Cref{def:left-dual-cat}. 
\end{lemma}
\begin{proof}
We have already constructed an evaluation functor $\ev_\cM^r$ above. We define 
\begin{equation}\label{eq:coev}\coev_\cM^r=\Big(\cC\xrightarrow{A_\cM}\Fun_\cC(\cM,\cM)\xrightarrow{\sim} {}^*\cM\boxtimes_\cC \cM\Big),\end{equation}
where $A_\cM$ is the functor from \Cref{def:AC-C-mod-functor} and the equivalence $\sim$ is obtained from \Cref{prop:rel-tensor-as-hom}. By construction, this is a left $\cC$-module functor. 
Applying \Cref{lem:ev-coherence}, providing the equivalence $\alpha$ of \Cref{def:left-dual-cat} is equivalent to showing that the composition 
$$\cM\simeq \Fun_\cC(\cC,\cM)\boxtimes_\cC \cC \xrightarrow{\id_\cM\boxtimes_\cC A_\cM}\Fun_\cC(\cC,\cM) \boxtimes_\cC \Fun_\cC(\cM,\cM) \xrightarrow{\circ} \Fun_\cC(\cC,\cM)\isomorph \cM$$
is isomorphic to the identity functor. This can again be checked using the equivalent description of relative Deligne products and $\cC$-linear functor categories via internal algebras. Here, the functor $A_\cM\colon \cC \to \Fun_\cC(\cM,\cM)$ is given by $X\mapsto X\otimes A$, which is an $A$-bimodule in $\cC$ as in  \eqref{eq:bimod1}--\eqref{eq:left-A-bimodule}. Thus, the entire composition of functors sends an object $M$ to $\one \lact M\cong M$. Hence, ${}^*\cM$ is a right dual in the sense of \Cref{def:left-dual-cat}.

Using the second diagram in \Cref{lem:ev-coherence}, constructing the equivalence $\beta$ required in \Cref{def:left-dual-cat} is equivalent to showing that the composition
$$\Fun_\cC(\cM,\cC)\simeq \cC\boxtimes_\cC \Fun_\cC(\cM,\cC)\xrightarrow{A_\cM \boxtimes_\cC \id} \Fun_\cC(\cM,\cM)\boxtimes_\cC  \Fun_\cC(\cM,\cC)  \xrightarrow{\circ} \Fun_\cC(\cM,\cC)$$
is isomorphic to the identity functor, which is clear. 

Now, we can display $\cM^*$ as a \emph{left} dual of $\cM$ in the sense of \Cref{def:left-dual-cat} by application of the identification ${}^*(\cM^{*})\simeq \cM$ of \Cref{cor:stars-inverses}. Under this identification, 
$\ev^l_\cM=\ev^r_{\cM^*}$ 
and $\coev^l_\cM=\coev^r_{\cM^*}$. 
\end{proof}

Recall from \Cref{thm:CMod-monoidal-2-cat} that $\rB\colon \lmod{\cC}\to \bimod{\cC}$ is a functor of monoidal $2$-categories.
From \Cref{rem:FunC-bimod-compare} we know that 
$\rB(\Fun_\cC(\cM,\cN))\simeq \Fun_\cC(\rB(\cM),\rB(\cN)),$
where the target has the $\cC$-$\cC$-bimodule structure from \cite{DSS2}*{Section~2.4}. This is summarized in the following result.

\begin{corollary}\label{cor:C-mod-duals}
    The inclusion $\rB\colon \lmod{\cC} \to \bimod{\cC}$ is a functor of monoidal $2$-categories that preserves duals.
\end{corollary}

\subsection{Op-fully exact module categories and duality}\label{sec:opfully}

In \Cref{def:fully-exact}, we require closure of the exactness property under tensoring on the right. Instead, we can tensor on the left, leading to the following notion of \emph{op-fully exact} module categories. This concept will be crucial to understanding duals of fully exact $\cC$-module categories. 

\begin{definition}\label{def:op-fex}
    We say that a left $\cC$-module category $\cM$ is \emph{op-fully exact} if for any exact $\cC$-module category $\cN$, the $\cC$-module category $\cM\boxtimes_\cC \cN$ is exact. 
\end{definition}

 We denote the  full $2$-subcategory in $\lmod{\cC}$ of op-fully exact $\cC$-module categories by $\opfexmod{\cC}$. As in \Cref{prop:fully-exact-tensor}, it follows that $\opfexmod{\cC}$ is closed under relative Deligne product and, thus, $\opfexmod{\cC}$ is a monoidal $2$-subcategory of $\lmod{\cC}$.  

\begin{proposition}\label{prop:duality-fully-exact}
    Let $\cM$ be a finite left $\cC$-modules category. The following statements are equivalent:
    \begin{enumerate}[(i)]
        \item $\cM$ is fully exact.
        \qquad \qquad (ii) $\cM^*$ is op-fully exact.
        \item[(iii)] ${}^*\cM$ is op-fully exact.
        \qquad  (iv) $\cM^{**}$ is fully exact.
    \end{enumerate}
\end{proposition}
\begin{proof}
    This follows from the $\cC$-linear equivalence $(\cN\boxtimes \cM)^*\simeq \cM^*\boxtimes \cN^*$ from \Cref{cor:tensor-dual} and the fact that $\cM$ is exact if and only if $\cM^*$ is exact, see \Cref{cor:duals-exact}. The same argument works with ${}^*\cM$. Equivalence of (i) and (iv) follows by applying this argument twice.
\end{proof}

\begin{definition}
We call an algebra $A\in \cC$ an \emph{op-fully exact algebra} if $\rmodint{\cC}{A}$ is op-fully exact as a left $\cC$-module category.  
\end{definition}
    
\begin{corollary}\label{lem:opfullyexact}
    The following statements are equivalent for an algebra $A$ in $\cC$. 
    \begin{enumerate}[(i)]
        \item $A$ is fully exact.
        \qquad \qquad \quad (ii) $A^\psi$ is op-fully exact.
        \item[(iii)] $A^{\psi^{-1}}$ is  op-fully exact.
        \qquad (iv) $A^{\psi^2}$ is fully exact.
    \end{enumerate}
\end{corollary}
\begin{proof}
    Setting $\cM=\rmodint{\cC}{A}$, this is a direct consequence of \Cref{prop:duality-fully-exact}.
\end{proof}

Recall, from \Cref{appendix:monoidal}, the braided tensor categories $\cCop$, with opposite tensor product $X\otimes^{\oop}Y=Y\otimes X$ and braiding $\psi^{\cCop}_{X,Y}=\psi_{Y,X}$, and $\cCrev$,  with braiding $\psi_{X,Y}^{\cCrev}=\psi^{-1}_{X,Y}$.

\begin{proposition}\label{prop:op-exact-algebras}
Let $A$ be an algebra in $\cC$. Then the following conditions are equivalent:
\begin{enumerate}[(i)]
    \item $A$ is op-fully exact.
    \item $A$ is fully exact as an algebra in $\cCop$.
    \item $A$ is fully exact as an algebra in $\cCrev$.
    \item $A\otimes A^\psi$ is an exact algebra.
    \item $A\otimes A^{\psi^{-1}}$ is an exact algebra.
\end{enumerate}
\end{proposition}
\begin{proof}
    Assume (i) that $A$ is op-fully exact. Then $A\otimes B= B\otimes^{\oop}A$ is exact for any exact algebra $B$ in $\cC$. This shows (ii) that $A$ is fully exact as an algebra in $\cCop$. Since $\cC=(\cCop)^{\otimesop}$, (ii) also implies (i).
    Now, \Cref{cor:proj-exact-algebras} applied to the braided tensor category $\cCop$ shows that (ii) is equivalent to (iv) and to (v). 

    The isomorphism of algebras 
    $A\otimes B=A\otimes^\psi B\cong B\otimes^{\psi^{-1}}A$
    proved in \Cref{lem:tensor-ops2} shows that (i) and (iii) are equivalent. This completes the proof.
\end{proof}

An analogue of the equivalent characterizations of fully exact algebras in \Cref{lem:A-fully-exact-equiv} also holds for $A$ being op-fully exact. We simply replace $A$ by $A^\psi$ and the multiplication $m_A$ by the opposite one $m_A^{\mathrm{op}}$
in the conditions
(2) and (3) there.

A consequence of \Cref{prop:op-exact-algebras} is that the $2$-categories $\opfexmod{\cC}$ and $\fexmod{\cCop}$ are equivalent by restricting the $2$-equivalence from \eqref{eq:right-C-to-left-Cop}. Using \Cref{lem:M-hash=A-mod} and \Cref{prop:rel-Del-modules}, we can show that this $2$-equivalence is compatible with relative Deligne products.

\smallskip

We can now derive the following criterion for op-fully exact module categories. 
For this, note that the left $\cC$-module $\cM$ is also a left module  over $\cCrev$, as $\cCrev$ is equal to $\cC$ as a tensor category but with the inverse braiding, see \Cref{sec:braided-mon-cat}. The corresponding monoidal functor $A_\cM$,
defined as in~\Cref{def:AC-C-mod-functor} but with the left $\cC$-linear structure using in \Cref{def:Fun-action} the reverse braiding $\psi^{-1}$ replacing the braiding $\psi$,
is denoted by $A_\cM^\rev\colon \cC \to \cC^*_\cM$.

\begin{corollary}
\label{prop:op-fully-equiv}
The following statements are equivalent for a left $\cC$-module category~$\cM$:
\begin{enumerate}[(i)]
    \item $\cM$ is op-fully exact.
    \item The functor $A_\cM^\rev$ preserves projective objects. 
\end{enumerate}  
\end{corollary}
\begin{proof}
\Cref{prop:op-exact-algebras} shows that (i) is equivalent to $\cM=\rmodint{\cC}{A}=\rmodint{\cCrev}{A}$ being fully exact as a left $\cCrev$-module category. By \Cref{prop:fully-exact-equiv}, this is equivalent to (ii).   
\end{proof}

\begin{remark}\label{rem:alpha-induction2}
    Continuing \Cref{rem:alpha-induction}, the monoidal functor $A_\cM^\rev$ corresponds to the other $\alpha$-induction functor $\alpha_\cM^-$ of \cite{DN2}*{(4.4)}, i.e.\,$\alpha_\cM^-=A_\cM^\rev\circ I^-$, for the monoidal equivalence $I^-$  from \eqref{eq:I+-}. 
\end{remark}

The characterization in \Cref{prop:op-fully-equiv} helps us to give an example of a fully exact module category that is \emph{not} op-fully exact.

\begin{example}\label{ex:double-group}
    Let $\cC=\cZ(\lmod{H})$, where $H=\Bbbk G$ is the group algebra of a finite group $G$ such that $\cha \Bbbk$ divides $|G|$. Equivalently, $\cC\simeq \lmod{\Drin(\Bbbk G)}$, where $\Drin(\Bbbk G)$ is the \emph{Drinfeld double} of $G$, see e.g.\,\cite{EGNO}*{Prop.\,7.14.6}. The Hopf algebra $H$ is non-semisimple while the dual $H^*=\Bbbk[G]$ is semisimple. The Drinfeld double $D(\Bbbk G)$ is also a non-semisimple algebra~\cite{Rad}*{Cor.\,13.2.3}. As a coalgebra, $\Drin(\Bbbk G)=\Bbbk[G]\otimes \Bbbk G$, and the Hopf algebra $\Drin(\Bbbk G)$ is quasi-triangular with universal R-matrix 
    $$R=\sum_{g\in G} (\varepsilon_{\Bbbk G} \otimes g)\otimes (\delta_g\otimes 1_{\Bbbk G}),$$
where $\varepsilon_{\Bbbk G}\in (\Bbbk G)^*$ is the counit and $1_{\Bbbk G}$ is the unit of $\Bbbk G$.
Consider the $\cC$-module category $\Vect$ obtained from the fiber functor $F\colon\cC\to \Vect$ via \Cref{ex:tensor-act}. It is well-known, see e.g. \cite{EGNO}*{Ex.\,7.12.26}, that $\cC^*_\Vect\simeq\lmod{\Drin(\Bbbk G)^*}$ are equivalent categories. Since $\Drin(\Bbbk G)=\Bbbk [G]\otimes \Bbbk G$ as a coalgebra, we see that $\Drin(\Bbbk G)^*$ is also non-semisimple. Using the result \Cref{lem:Vect-fully-exact} below, we see that $\Vect$ is fully exact if and only if the restriction functor $\Res_{\phi_R}\colon \lmod{\Drin(H)}\to\lmod{\Drin(H)^*}$, for the morphism $\phi_R$ from \eqref{eq:phiR}, preserves projective objects.
With the above universal R-matrix and the convention $(V\otimes W)^*:=W^*\otimes V^*$, we compute that
$$\phi_R \colon \Big(\Drin(\Bbbk G)^*= (\Bbbk G)^*\otimes (\Bbbk G)^{**}\xrightarrow{ (1_{\Bbbk [G]}\circ \varepsilon_{\Bbbk [G]})\otimes p} (\Bbbk G)^*\otimes \Bbbk G=\Drin(\Bbbk G)\Big),$$
where $p\colon (\Bbbk G)^{**}\to \Bbbk G$ is the pivotal structure of $\Vect$, an isomorphism. Since $(\Bbbk G)^*$ is semisimple, restriction along $\phi_R$ preserves projective objects. Thus, $\Vect$ is fully exact.

By \Cref{prop:op-fully-equiv}, $\cM$ is op-fully exact if and only if $A_\cM^\rev$ preserves projective objects. By \Cref{cor:Vect-op-fex}, this functor is given by restriction along the morphism  
$$ \phi_{R^{-1}_{21}}\colon \Big(\Drin(\Bbbk G)^*= (\Bbbk G)^*\otimes (\Bbbk G)^{**}\xrightarrow{ S_{\Bbbk[G]}\otimes (1_{\Bbbk G}\circ \varepsilon_{(\Bbbk G)^{**}})} (\Bbbk G)^*\otimes \Bbbk G=\Drin(\Bbbk G)\Big), $$
involving the antipode $S_{\Bbbk[G]}$ of $\Bbbk[G]$. Restriction along this homomorphism of algebras does not preserve projective objects. For example, the regular module is sent to a direct sum of simple objects but the category $\cC^*_\cM$ is not semisimple. Thus, $\Vect$ is \emph{not} op-fully exact. This example generalizes to Hopf algebras $H$ in non-zero characteristic which are non-semisimple but where the dual is semisimple. 
\end{example}

\subsection{Criteria for duals to be fully exact}
\label{sec:dualfully}

To conclude this section, we will consider situations in which the notions of fully exact and op-fully exact left $\cC$-modules are equivalent.  

\begin{corollary}\label{cor:C-symmetric-op-fex}
For $\cC$ symmetric, a left $\cC$-module is fully exact if and only if it is  op-fully exact.
\end{corollary}
\begin{proof}
    The symmetric braiding of $\cC$ provides an isomorphism of algebras 
    $$\psi_{A,B}\colon A\otimes B \to B\otimes A.$$
    This gives an equivalence of left $\cC$-module categories $\cN\boxtimes_\cC \cM\simeq \cM\boxtimes_\cC \cN$, for $\cM=\rmodint{\cC}{A}$ and $\cN=\rmodint{\cC}{B}$. Now the claim follows from \Cref{prop:duality-fully-exact}.
\end{proof}

We now consider criteria for algebras to be fully exact and op-fully exact.
Notably, separable algebras provide a class of fully exact algebras that are also op-fully exact.

\begin{lemma}\label{lem:separable-duals}
    Assume $A$ is separable. Then $A^\psi$ is separable and $A$ is op-fully exact.
\end{lemma}
\begin{proof}
We already know from \Cref{prop:rel-proj-fully-exact} that $A$ is fully exact. 
    Fix an inclusion $i$ and projection $p$ of $A$-bimodules
    $$i\colon A\to A\otimes A, \qquad p\colon A\otimes A\to A,$$
    such that $i p=\id_A$. Consider the morphisms
    $$j= \psi^{-1}_{A,A}i\colon A^\psi \to A^\psi\otimes A^\psi, \qquad q= p\,\psi_{A,A}\colon A^\psi \to A^\psi\otimes A^\psi,$$
    using that $A=A^\psi$ as objects in $\cC$ (but not as algebras). One checks that $j$ and $q$ are morphisms of $A^\psi$-bimodules. Hence, $A^\psi$ is separable and hence fully exact. By \Cref{lem:opfullyexact}, this means that $A$ is op-fully exact. 
\end{proof}

\begin{remark}\label{rem:sep-indep-braiding}
    Note that the definition of a separable algebra $A$ (or, more generally, that of an algebra that is relatively projective as a bimodule) does not involve the braiding $\psi$ of $\cC$. Thus, if $A$ is separable (respectively, relatively projective as a bimodule), then $\rmodint{\cC}{A}$ is fully exact and op-fully exact for $\cC$ equipped with \emph{any} braiding.
\end{remark}

\begin{corollary}\label{cor:sep-dualizable}
If $\cM$ is a separable left $\cC$-module, then $\cM$ is fully exact and op-fully exact. 
\end{corollary}
\begin{proof}
    This follows directly from \Cref{def:separable-mod-cat} and \Cref{lem:separable-duals}.
\end{proof}

\begin{proposition}
    Assume that one of the following conditions holds for an algebra $A$~in~$\cC$:
    \begin{enumerate}[(i)]
        \item $A$ is commutative in $\cC$;
        \item $\psi_{A,A}^{2}=\id_{A\otimes A}$;
        \item $A$ is contained in the M\"uger center of $\cC$;
        \item the algebras $A$ and $A^\psi$ are Morita equivalent.
    \end{enumerate}
Then $A$ is fully exact if and only if it is op-fully exact in $\cC$.
\end{proposition}
\begin{proof}
Assume that  $A$ is a  fully exact algebra.

Assume (i) so that  $A=A^\psi$. By \Cref{lem:opfullyexact}, $A=A^\psi=A$ is also op-fully exact. 

Assume (ii) so that $\psi_{A,A}=\psi_{A,A}^{-1}$. As objects in $\cC$, $A=A^\psi$ and this assumption implies that
$A^\psi\otimes A=A^\psi\otimes^{\psi^{-1}} A$ are equal as algebras. But $A^\psi\otimes^{\psi^{-1}} A\cong A\otimes^{\psi}  A^\psi$ by \Cref{lem:tensor-ops2}. Thus, $A$ is op-fully exact by \Cref{prop:op-exact-algebras}.

Assuming (iii) implies (ii), so $A$ is also op-fully exact in this case.

Assuming (iv), we have $\cM:=\rmodint{\cC}{A}\simeq \rmodint{\cC}{A^\psi}$ as left $\cC$-module categories. But then 
$${}^*\cM\boxtimes_\cC \cM\simeq \cM\boxtimes_\cC {}^*\cM.$$
In particular, the exact algebra $A^\psi\otimes A$ is Morita equivalent to the algebra $A\otimes A^\psi$. This implies that $A\otimes A^\psi$ is also exact and $A$ is op-fully exact by \Cref{prop:op-exact-algebras}.

Thus, under any of the assumptions (i)--(iv), we have checked that $A$ being  fully exact implies that $A$ is op-fully exact. The converse follows by the same reasoning since $A$ is op-fully exact if and only if $A$ is fully exact when viewed as an algebra in $\cCop$, by \Cref{prop:op-exact-algebras} and all assumptions (i)--(iv) also hold for $A$ viewed as an algebra in $\cCop$.
\end{proof}

 We now recall the definition of a braided module category, see e.g.~\cite{DN2}*{Def.\,4.1}.

\begin{proposition}\label{prop:braided-mod-cat}
Let $\cM$ be a left $\cC$-module that admits a structure of a braided module category. Then $\cM$ is fully exact if and only if $\cM$ is op-fully exact. 
\end{proposition}
\begin{proof}
    By \cite{DN2}*{Prop.\,4.9}, a braided $\cC$-module category structure corresponds to a monoidal isomorphism $A_\cM\cong A_\cM^\rev$. Hence, these functors preserve projective objects at the same time. 
\end{proof}

\begin{remark}\label{rem:ZC-mod}
 According to the work \cite{DN2}*{Prop.\,4.9 and Thm.4.11} there is a $2$-equivalence between  the $2$-categorical version of Drinfeld center $\cZ(\lmod{\cC})$ of \cite{BN}*{Sec.\,3} and the $2$-category $\cC$-$\mathbf{brmod}$ of braided $\cC$-modules. Therefore, the latter has a monoidal $2$-category structure. According to 
 \cite{DN2}*{Rem.4.13} the tensor product of braided module categories $\cM$ and $\cN$ is just $\cM\boxtimes_\cC \cN$ and equipped with a canonical braiding. We thus conclude that product of (op-)fully exact braided $\cC$-modules is again (op-)fully exact and braided.
\end{remark}

We conclude this subsection by observing that the argument from \Cref{lem:separable-duals} does not directly generalize to relatively projective or fully exact algebras and pose the following question. 

\begin{question}
    If $\cha \Bbbk=0$, are fully exact module categories stable under duality, i.e.\ is $\cM$ fully exact if and only if $\cM^*$ is fully exact? Equivalently, is the algebra $A^\psi\otimes A$ in $\cC$ exact if and only if $A\otimes A^{\psi}$ is exact?
\end{question}
When $\cha \Bbbk = p$, the answer is negative by \Cref{ex:double-group}. For $\cC$ a symmetric category, the answer is positive by \Cref{cor:C-symmetric-op-fex}, see also \Cref{sec:Sweedler} for a detailed example. 

\subsection{Perfect module categories}
\label{sec:fully-dualizable}

\Cref{ex:double-group} shows that fully exact module categories are, in general, not closed under duals. In order to produce a monoidal $2$-categories with duals, we can restrict to left $\cC$-modules that are \emph{both} fully exact and op-fully exact. Recall from \Cref{def:left-dual-cat} that an object of a monoidal $2$-category is called \emph{dualizable} if it has a left and right dual. Further recall that all objects in $\lmod{\cC}$ are dualizable, see \Cref{sec:duals-finite}.

\begin{definition}\label{def:dualizable-fully-exact}
    A left $\cC$-module category $\cM$ is 
    called \emph{perfect} if it is dualizable in $\fexmod{\cC}$. 
We denote the full $2$-subcategory of $\lmod{\cC}$ on all perfect module categories by $\perf{\cC}$.
\end{definition}
As $\cM^*$ is fully exact if and only if ${}^*\cM$ is fully exact by \Cref{prop:duality-fully-exact}, we know that $\cM$ is perfect if and only if $\cM$ and $\cM^*$ are fully exact.

Since the classes of fully exact and op-fully exact $\cC$-modules are closed under tensor product, $\perf{\cC}$  is a monoidal $2$-subcategory of $\lmod{\cC}$.
Combining the previous results \Cref{prop:fully-exact-equiv}, \Cref{prop:op-exact-algebras}, \Cref{prop:op-fully-equiv}, and \Cref{prop:rel-tensor-as-hom} we obtain equivalent criteria for perfect module categories.

\begin{corollary}\label{cor:dualizable-fully-exact-criterion}
The following are equivalent for a left $\cC$-module category $\cM$:
\begin{enumerate}[(i)]
    \item $\cM$ is perfect.
    \item  $\cM$ is both fully exact and op-fully exact.
    \item Both functors $A_\cM, A_\cM^\rev\colon \cC \to \cC^*_\cM$ preserve projective objects.
    \item Both left $\cC$-modules ${}^*\cM\boxtimes_\cC \cM$ and $\cM\boxtimes_\cC {}^*\cM$ are exact.
    \item Both left $\cC$-modules $\cM^*\boxtimes_\cC \cM$ and $\cM\boxtimes_\cC \cM^*$ are exact.
\end{enumerate}
\end{corollary}

\begin{remark}\label{rem:perf-complex}
The terminology of perfect module categories is motivated by analogy with the terminology of \emph{perfect complexes} of modules over a commutative ring $R$ or quasi-coherent sheaves over a scheme $X$ 
\cite{stacks-project}*{Secs.\,\href{https://stacks.math.columbia.edu/tag/0656}{15.75} and \href{https://stacks.math.columbia.edu/tag/0716}{20.3}}.
 It can be shown that perfect complexes are precisely the dualizable objects in the derived category $D(R)$ of $R$-modules with respect to the derived relative tensor product over $R$. More generally, this observation applies to the derived category $D(\cO_X)$ of quasi-coherent sheaves over $X$ with respect to the derived relative tensor product 
 \cite{stacks-project}*{\href{https://stacks.math.columbia.edu/tag/0FP7}{Sec.\,20.50}}. See also \cite{BFN}*{Prop.\,3.6} and \cite{HR}*{Lem.\,4.3} for a generalization of such a statement to stacks. 
In particular, perfect module categories give a categorical version of perfect modules over $R$, i.e.\,perfect complexes concentrated in degree zero.

By \Cref{cor:dualizable-fully-exact-criterion}, perfect module categories $\cM$ are equivalently characterized by both \action functors $A_\cM$ and $A_\cM^\rev$ being \emph{perfect} tensor functors in the terminology of \cites{BrNa,Shi5}, see \Cref{rem:perfect-conditions}.
\end{remark}

Recall that in any $2$-category, we can define left and right adjoints of $1$-morphisms, see \cite{JY21}*{Def.\, 6.1.1}.

\begin{definition}\label{def:2rigid}
    A monoidal $2$-category is \emph{rigid} if all of its objects are dualizable and all $1$-morphisms have left and right adjoints.\footnote{In an $(\infty,2)$-categorical setting, in \cite{Lur}*{Def.\,2.3.16} this property is called \emph{having duals}. We choose the terminology of rigid monoidal $2$-category over fully dualizable monoidal $2$-category to avoid confusion with the property of being a fully dualizable object in a Morita $4$-category of monoidal $2$-categories, see \cite{Dec2}. Note that in \cite{DJF}*{Sec.\,4.1}, what we call a rigid monoidal $2$-categories is called \emph{fully rigid}.}
\end{definition}
For example, $\perf{\cC}$ is a rigid monoidal $2$-category.

\begin{definition}\label{def:fully-dualizable}
We say that an object of a monoidal $2$-category $\bfC$ is \emph{fully dualizable} if it is contained in a rigid monoidal $2$-subcategory $\bfD$ of $\bfC$.   An equivalent condition for an object $\cM$ of $\bfC$ to be fully dualizable is that $\cM$ is dualizable and the left and right evaluation and coevaluation $1$-morphisms have left and right adjoints, and these adjoints have further adjoints, etc.,  
cf.\,the beginning of \cite{DSS2}*{Ch.\,1}.  
\end{definition}

We note that $\bfD$ in \Cref{def:fully-dualizable} is not required to be a \emph{full} monoidal $2$-subcategory. The following is the main result of this section.

\begin{theorem}
\label{prop:dualizability} 
 An object $\cM$ in $\lmod{\cC}$ is fully dualizable if and only if $\cM$ is perfect, i.e.\,$\cM$ and $\cM^*$ are fully exact.
\end{theorem}
\begin{proof}
We know at this point that $\perf{\cC}$ is a fully dualizable monoidal $2$-category that is a full $2$-subcategory of $\lmod{\cC}$. Hence, all perfect $\cC$-modules are fully dualizable.

Conversely, assume that $\cM$ is a fully dualizable object in $\lmod{\cC}$. In particular, the coevaluation functor $\coev^r_\cM$ from \eqref{eq:coev} has a right adjoint $R$, which itself has a further right adjoint.  Hence, as $\coev^r_\cM$ is a functor between finite abelian categories, this implies that $R$ is exact by \Cref{lem:exact-vs-adjoints}. 
 But, by construction, the \action functor $A_\cM$ corresponds to $\coev^r_\cM$ up to composition with an equivalence. Hence, $A_\cM$ has an exact right adjoint. By \Cref{rem:perfect-conditions} this implies that $A_\cM$ preserves projective objects and hence $\cM$ is fully exact by \Cref{prop:fully-exact-equiv}. If $\cM$ is fully dualizable then so is $\cM^*$. Hence, $\cM^*$ is fully exact and using the equivalence ${}^*(\cM^*)\simeq \cM$ from \Cref{cor:stars-inverses},  $\cM$ is also op-fully exact by \Cref{prop:duality-fully-exact}. Together, we see that $\cM$ is perfect by \Cref{def:dualizable-fully-exact}.
\end{proof}

\begin{remark}
\cite{Lur}*{Claim 2.3.19} introduces the fully dualizable part $\bfC^{\mathrm{fd}}$ of a symmetric monoidal $2$-category $\bfC$. In the case of a symmetric finite tensor category $\cC$ and $\bfC=\lmod{\cC}$, we observe using \Cref{prop:dualizability} that $\bfC^{\mathrm{fd}}$ is given by the rigid monoidal $2$-category $\perf{\cC}$.  
\end{remark}

\begin{remark}
Consider the (full) monoidal $2$-subcategory of $\lmod{\cC}$ generated under $\boxtimes_\cC$, direct sums and summands  by both classes, fully exact and op-fully exact $\cC$-module categories. If the two classes do not agree, there exists a left $\cC$-module category $\cM$ which is op-fully exact but not fully exact. Such an $\cM$ is provided in \Cref{ex:double-group}. Hence, using \Cref{prop:rel-tensor-as-hom}, $\cC^*_\cM\simeq {}^*\cM\boxtimes_\cC\cM$ is not exact. The dual ${}^*\cM$ is fully exact but not op-fully exact. Hence, tensoring an op-fully exact and a fully exact $\cC$-module gives a non-exact $\cC$-module in this case.
\end{remark}

\section{Invertible module categories}\label{sec:invertible}

This section discusses the relationship between fully exact  and invertible module categories. 

\subsection{Characterization of invertible module categories}\label{sec:invertible-def}

We say that a finite left $\cC$-module category $\cM$ is \emph{invertible} if there exists another left $\cC$-module category $\cN$, called an \emph{inverse}, and equivalences of left $\cC$-module categories $\cM\boxtimes_\cC \cN\simeq \cC$ and $\cN\boxtimes_\cC \cM \simeq \cC$, cf. \Cref{def:inv} for the monoidal $2$-category $\lmod{\cC}$. 
We require the following lemma for which the idea of the proof appeared in \cite{ENO}*{Prop.\,4.2} in the context of invertible bimodule categories.

\begin{lemma}\label{lem:F-equiv}
    Let $F\colon \cC\to \cD$ be  a monoidal functor between monoidal categories $\cC$ and $\cD$,
 with $\cC$ braided, and consider $\cD$ as a left $\cC$-module category as in \Cref{ex:tensor-act}. Then $\cD$ is equivalent to $\cC$ as a left $\cC$-module if and only if $F$ is an equivalence. 
\end{lemma}
\begin{proof}
    Let $G\colon \cC\to \cD$ be an equivalence of left $\cC$-modules. 
    Then, using the $2$-functor $\rB$ 
    defined in \Cref{lem:S-bimod},
    we obtain an equivalence of $\cC$-bimodules
    $\rB(G)\colon \rB(\cC)\to \rB(\cD)$. Therefore, there exists an object $C$ in $\cC$ such that $G(C)\cong \one_\cD$. Now, 
    $$\one _\cD\cong G(C)\cong G(C\otimes \one_\cC)\cong F(C)\otimes G(\one_\cC),$$
    and, using the right $\cC$-linear structure,     
    $\one_\cD\cong 
G(\one_\cC)\otimes F(C).$
    This shows that $G(\one_\cC)$ is an invertible object in $\cD$.
    Further, by $\cC$-linearity, we have the natural family of isomorphisms
    $$
    s^G_{X,\one}\colon G(X)\xrightarrow{\;\cong\;} F(X)\otimes G(\one_\cC)
    $$
and hence the functor $G$ is isomorphic to the composition of $F$ with the equivalence given by tensoring with an invertible object. This implies that $F$ is  an equivalence as well.

The converse direction is clear since $F\colon \cC\to \cD$ is, in particular a functor of left $\cC$-modules.
\end{proof}

Following the argument in \cite{ENO}*{Prop.\,4.2} we can prove a characterization of invertible module categories. 

\begin{theorem}\label{prop:Minv-AM-equiv}
    Let $\cM$ be a finite left $\cC$-module category. Then the following are equivalent:
    \begin{enumerate}[(i)]
        \item $\cM$ is invertible as a left $\cC$-module category.
        \item The centralizer functor $A_\cM$ from \Cref{def:AC-C-mod-functor} is an equivalence.
        \item The evaluation functor of a right dual of $\cM$ is an equivalence.
        \item The coevaluation functor of a right dual of $\cM$ is an equivalence.
        \item The evaluation functor of a left dual of $\cM$ is an equivalence.
        \item The coevaluation functor of a left dual of $\cM$ is an equivalence.
    \end{enumerate}
\end{theorem}
\begin{proof}
Assume that $\cM$ is an invertible $\cC$-module category. Recall the equivalence of left $\cC$-module categories $\Fun_\cC(\cM,\cM)\simeq {}^*\cM\boxtimes_\cC\cM$ from \Cref{prop:rel-tensor-as-hom}. Since $\cM$ is invertible, with inverse $\cM^*\simeq {}^*\cM$ by \Cref{lem:duals-inv}, we have that $\Fun_\cC(\cM,\cM)\simeq \cC$ are equivalent left $\cC$-module categories. By \Cref{lem:AC-braided}, the left $\cC$-action on $\Fun_\cC(\cM,\cM)$ is induced by the tensor functor $A_\cM$. 
Hence, by \Cref{lem:F-equiv}, $A_\cM$ is an equivalence.  This shows that (i) implies (ii).

We now prove that (ii) implies (i).
For this, assume that $A_\cM\colon \cC\to \End_\cC(\cM)$ is an equivalence. Thus, $\cC\simeq {}^*\cM \boxtimes_\cC\cM$ and $\cM$ has a left inverse ${}^*\cM$. 
Also, by \Cref{lem:AM-AMr-compare}, the monoidal functor $A^r_{\rB(\cM)}$ is an equivalence,
and therefore the pull-back $2$-functor $\Res_{A^r_{\rB(\cM)}}$ from \Cref{lem:restriction-2-functor} is a $2$-equivalence. 
Furthermore,
by \cite{EGNO}*{Thm.\,7.12.11}, the double centralizer functor~$\mD_\cM$ in~\eqref{eq:can-double-dual} is an equivalence. 
Using this result requires that $\cM$ is \emph{faithful}, i.e.\ that any non-zero object acts by a non-zero functor $X\lact \id_\cM\colon \cM\to \cM$. But this follows from \Cref{lem:AM-exact-faithful}. Therefore, by \Cref{prop:centralizer-correspondence}, $A^l_{\rB(\cM)}$ is also an equivalence. Thus, by \Cref{lem:Al-identified}, the functor $A_\cM^\rev$ is an equivalence.
Suppose that $\cM=\rmodint{\cC}{A}$. Then this gives an equivalence of left $\cC$-module categories
$$A_\cM^\rev\colon \cC\to (\cC^\rev)^*_\cM\simeq \rmodint{\cC}{A^{\psi^{-1}}\otimes^{\psi^{-1}} A}\simeq \rmodint{\cC}{A \otimes A^{\psi^{-1}}}\simeq \cM\boxtimes_\cC \cM^*.$$
Here, we have applied, in order, the equivalences from \Cref{prop:Funrex-modules} (for $\cCrev$), \Cref{lem:tensor-ops2} via the isomorphism of algebras $A\otimes B\simeq B\otimes^{\psi^{-1}} A$, and \Cref{prop:*M*-algebras} to the target category.

Now, we have equivalences of left $\cC$-module categories 
$$\cC\simeq {}^*\cM\boxtimes_\cC\cM, \quad \text{and}\quad  \cC\simeq \cM\boxtimes_\cC\cM^*.$$
Combining these shows that 
$$ {}^*\cM\simeq  {}^*\cM\boxtimes_\cC \cC \simeq{}^*\cM\boxtimes_\cC (\cM\boxtimes_\cC\cM^*) \simeq ({}^*\cM\boxtimes_\cC \cM)\boxtimes_\cC\cM^*\simeq \cC\boxtimes_\cC \cM^*\simeq \cM^*$$
and, hence, that $\cM$ is invertible, i.e.\ (i) holds. 

By construction, the coevaluation functor $\coev^r_\cM$ of \eqref{eq:coev} is the composition of $A_\cM$ with an equivalence. This shows that (ii) and (iv) are equivalent.

We may use the left duality to define
\begin{gather*}
    \ev'\colon \big(\cM^*\boxtimes_\cC \cM\simeq ({}^*\cM\boxtimes_\cC \cM)^*\xrightarrow{(\coev^r_\cM)^*} \cC^*\simeq \cC\big), 
    \\
    \coev' \colon \big(\cC^*\simeq \cC\xrightarrow{(\ev^r_\cM)^*} (\cM\boxtimes_\cC {}^*\cM)^*\simeq \cM\boxtimes_\cC \cM^*\big). 
\end{gather*}
It follows that $(\cM^*,\ev',\coev')$ define a left dual for $\cM$. 
Now, $\coev^r_\cM$ is an equivalence if and only $\ev'$ is an equivalence, i.e.\ (iv) and (v) are equivalent. 
Further, $\ev^r_\cM$ is an equivalence if and only $\coev'$ is an equivalence, i.e.\ (iii) and (vi) are equivalent. 

Using \Cref{lem:duals-inv}, (i) implies (iii) which is equivalent to (vi).
It remains to show that (vi) implies (i). Under the identification $\cM={}^*(\cM^*)$, we have $\coev^l_{\cM}=\coev^r_{\cM^*}$. Thus, $\coev^l_{\cM}$ being an equivalence is equivalent to $\cM^*$ being invertible. Applying \Cref{lem:duals-inv}, it follows that ${}^*(\cM^*)= \cM$
 is invertible, i.e.\ that (i) holds.
\end{proof}

\subsection{Invertible module categories are perfect}

We are now ready to prove that invertible module categories are fully and op-fully exact, i.e.\,perfect, see \Cref{def:dualizable-fully-exact}. 

\begin{corollary}\label{cor:inv-fully-exact}
    If $\cM$ is invertible then $\cM$ is perfect.
\end{corollary}
\begin{proof}
Assume that $\cM$ is invertible. Then, by \Cref{prop:Minv-AM-equiv}, the tensor functor $A_\cM\colon \cC\to \cC^*_\cM$ is an equivalence. Thus, it preserves projective objects, and $\cM$ is fully exact. Now, $\cM^*$ is also invertible by \Cref{lem:duals-inv}. Thus, $\cM^*$ is fully exact and hence  $\cM$ is also op-fully exact by \Cref{prop:duality-fully-exact}. 
\end{proof}

\begin{remark}
In particular, we have provided an alternative proof of the statement in \cite{DN2}*{Prop.\,5.1} that any finite invertible module category is exact. This statement now follows from combining \Cref{cor:inv-fully-exact} and \Cref{rem:fully-exact-exact}.
\end{remark}

The following example shows that the converse of \Cref{cor:inv-fully-exact} does not hold.

\begin{example}\label{ex:counterexample}
    Let $G$ be a finite group and $\cC:=\lmod{\Bbbk G}$, and $F\colon \cC\to \Vect$ be the associated fiber functor, where $\Bbbk$ is any field. As in \Cref{ex:first}, this makes $\Vect$ an indecomposable left $\cC$-module category. We give $\cC$ the symmetric braiding inherited from $\Vect$. The \action functor 
    $$A_{\Vect}\colon \cC\to  \Fun_\cC(\Vect,\Vect)\simeq \lcomod{\Bbbk G}$$
    factors through $F$ since $V\lact \id_{\Vect}$, for any $\Bbbk G$-module $V$, corresponds to the trivial $\Bbbk G$-comodule. The category $\lcomod{\Bbbk G}$ is semisimple and hence $A_{\Vect}$ preserves projective objects. This shows that $\Vect$ is fully exact. As $\cC$ is symmetric, $\Vect$ is also op-fully exact by \Cref{cor:C-symmetric-op-fex}.  However, $A_{\Vect}$ is not an equivalence for a non-trivial group $G$. Hence, by \Cref{prop:Minv-AM-equiv}, $\Vect$ is \emph{not} invertible. This example applies independently of the characteristic of $\Bbbk$.
\end{example}

The results of this section amount to the following result. 
\begin{corollary}\label{cor:inv-2-cat}
    Invertible left $\cC$-module categories form a  full monoidal $2$-subcategory $\invmod{\cC}$ of $\perf{\cC}$.
\end{corollary}

In general, not every indecomposable perfect left $\cC$-module category is invertible, as \Cref{ex:counterexample} shows. In contrast, in \Cref{sec:Sweedler} we will see a non-trivial example where all indecomposable perfect module categories are invertible.

\section{Hopf algebra examples}\label{sec:Hopf}

We conclude the paper with some results on fully exact module categories over braided finite tensor categories obtained from quasitriangular Hopf algebras.  In \Cref{sec:exactness-vect}, we provide equivalent conditions for $\Vect$ to be fully exact in~\Cref{lem:Vect-fully-exact}, op-fully exact and perfect in~\Cref{cor:Vect-op-fex}, and invertible in~\Cref{cor:Vect-inv}. We show in \Cref{sec:uqsl2}  that even in examples with a non-degenerate braiding, such as the factorizable small quantum group $u_q(\mathfrak{sl}_2)$ at any odd root of unity $q$, its exact module category $\Vect$ is \textsl{not} fully exact, see~\Cref{prop:vect-sl2}.  We conjecture that this result holds for general factorizable small quantum groups.  In \Cref{sec:Sweedler}, we classify fully exact and perfect module categories for the finite tensor category $\cC=\lmod{S}$ for Sweedler's four-dimensional Hopf algebra~$S$. We further compute relative Deligne products, duals, and functor categories in this case. This requires calculations of categories of $\cC$-linear functors between twisted $\cC$-module categories $\Vect^J$, where $J$ is a Hopf $2$-cocycle, see \Cref{sec:vectJ}. Some proofs are only sketched and more details will appear in a forthcoming paper \cite{GL}.


\subsection{Conditions for fully exactness of \texorpdfstring{$\Vect$}{Vect}}\label{sec:exactness-vect}

Let $\cC$ be a braided finite tensor category with $\Vect$ as a left $\cC$-module. Then we have an exact faithful monoidal  functor, called \emph{fiber functor},
\begin{equation}\label{eq:F-fiber}
F\colon \cC\to \Fun(\Vect,\Vect)\simeq \Vect,
\end{equation}
see \Cref{lem:F-faithful}.
Hence, $\cC\simeq \lmod{H}$ for $H$ a quasi-triangular Hopf algebra over $\Bbbk$, see e.g. \cite{EGNO}*{Thm.\,5.2.3} together with \cite{Maj}*{Prop.\,9.4.2}. We denote the product, unit, coproduct and counit of~$H$ by 
$$m_H\colon H\otimes H\to H,\quad u_H\colon \one \to H, \quad\Delta_H\colon H\to H\otimes H, \quad\text{and}\quad \varepsilon_H \colon H\to \one.$$ We will sometimes use Sweedler's notation to write the coproduct by $\Delta_H(h)=h_{(1)}\otimes h_{(2)}$, for $h\in H$ and the universal R-matrix of $H$ by $R=R^{(1)}\otimes R^{(2)}\in H\otimes H$, where sums are omitted. The braiding of $\cC$ is given by 
$$
\psi_{V,W}\colon V\otimes W \to W\otimes V, \quad v\otimes w \mapsto R^{(2)}w\otimes R^{(1)}v,
$$
for left $H$-modules $V$, $W$ and any $v\in V, w\in W$. We also consider the tensor category $\lcomod{H}$ of finite-dimensional left $H$-comodules $(C,\delta)$, where we denote the left $H$-coaction by
\begin{equation}\label{eq:Sweedler-coact}
\delta\colon C\to H\otimes C,
\qquad
c\mapsto c^{(-1)}\otimes c^{(0)}.
\end{equation}

The following lemma displays a well-known equivalence, see \cite{EGNO}*{Ex.\,7.12.26} and its generalization in \Cref{lem:equiv-FunVect-twisted} below.

\begin{lemma}\label{lem:FunC-H*-mod}
There is an equivalence of categories
$$E\colon \Fun_\cC(\Vect,\Vect)\longrightarrow \lcomod{H}, \quad (G, s^G)\mapsto (G(\Bbbk),\delta^G),$$
where, denoting the regular $H$-module by $H$,
\begin{equation}\label{eq:delta-G}
\delta^G = s^G_{H,\Bbbk}\circ 
G(u_H)
\colon \quad G(\Bbbk)\to H\otimes G(\Bbbk),
\end{equation}
where we identify
the target of $G(u_H)$ with 
$G(F(H)) = G(H\lact \Bbbk)$,
with its quasi-inverse 
$$T\colon \lcomod{H}\longrightarrow \Fun_\cC(\Vect,\Vect), \quad (C,\delta)\mapsto ( (-)\otimes C, s^\delta),$$
where the components of $s^\delta$ are defined as,
 recall $F$ from~\eqref{eq:F-fiber},
\begin{equation}
    s^\delta_{W,\Bbbk}\colon F(W)\otimes C\to F(W)\otimes C, \quad w\otimes c \mapsto c^{(-1)}w\otimes c^{(0)}, 
\end{equation}
for any $H$-module $W$, and elements $w\in W$, $c\in C$. For a general object 
$V=\Bbbk^{\oplus n}$ in $\Vect$, $s^\delta_{W,V}$ is given by the direct sum of $n$ copies of the map $s^\delta_{W,\Bbbk}$. 
\end{lemma}

Now consider the dual Hopf algebra $H^*$ with product 
\begin{equation}\label{eq:dual-product}
fg(h)=f(h_{(2)})g(h_{(1)}), \qquad \forall f,g\in H^*, h\in H,
\end{equation}
and coproduct $\Delta(f)=f_{(1)}\otimes f_{(2)}$ defined by 
\begin{equation}\label{eq:dual-coproduct}
f(hk)= f_{(1)}(h)f_{(2)}(k), \qquad \forall h,k\in H.
\end{equation}

The following result is standard.

\begin{lemma}\label{lem:dualize}
    There is an equivalence of categories 
    $$\lcomod{H}\to \lmod{H^*}, \quad (C,\delta)\mapsto (C,a^l_C),$$
    where, for a left $H$-comodule $(C,\delta)$, we define 
    \begin{equation}\label{eq:H-dual-action}
    a^l_C\colon H^*\otimes C\to C, \quad f\otimes c\mapsto f\cdot c:=f(c^{(-1)})c^{(0)},
    \end{equation}
    for all $f\in H^*$, $c\in C$, using the notation of \eqref{eq:Sweedler-coact}.
\end{lemma}

In the following, we will utilize the morphism 
\begin{equation}\label{eq:phiR}
\phi_R\colon H^*\to H,\quad  f\mapsto R^{(1)}f(R^{(2)}),
\end{equation}
obtained from the universal R-matrix.
Before stating the next lemma, recall that for a Hopf algebra~$H$, a \emph{left $H$-comodule algebra} is an algebra $A$ internal to the tensor category of left $H$-comodules. The category $\lmod{A}$ is then a left module over $\cC=\lmod{H}$, see\,\cite{AM}*{Sec.\,1.5} or \cite{EGNO}*{Ex.\,7.8.3}. The left $\cC$-action is defined, for $W$ in $\lmod{H}$ and $V$ in $\lmod{A}$, by 
$$W\lact V=(W\otimes V, a^l_{W\lact V}),$$
with the $A$ action 
\begin{equation}\label{eq:comod-act}
a^l_{W\lact V}\colon A\otimes W\otimes V\to W\otimes V,\quad a\otimes w\otimes v\mapsto a^{(-1)}w\otimes a^{(0)}v,
\end{equation}
for $a\in A,w\in W,v\in V$, using the notation \eqref{eq:Sweedler-coact} for the $H$-coaction on $A$.

\begin{lemma}\label{lem:H-com-alg}
The morphism 
$\phi_R\colon H^*\to H$ is a homomorphism of Hopf algebras. Hence, $H^*$ is a left $H$-comodule algebra via the coaction
\begin{equation}\label{eq:delta-star}
\delta_{H^*}\colon H^*\to H\otimes H^*, \quad f\mapsto \phi_R(f_{(1)})\otimes f_{(2)}.
\end{equation}
In particular, $\lmod{H^*}$ is a $\cC$-module via the $\cC$-action defined by
$W\lact V=(W\otimes V, a^l_{W\lact V}),$ 
for $W\in \lmod{H}=\cC$ and $V\in \lmod{H^*}$, where  
\begin{align}
\begin{split}\label{eq:C-action-H*-mod}
    a^l_{W\lact V}\colon H^*\otimes W\otimes V &\to W\otimes V,\\
    f\otimes w\otimes v&\mapsto \phi_R(f_{(1)})w\otimes f_{(2)}v=f_{(1)}(R^{(2)})R^{(1)}w\otimes f_{(2)}v,
    \end{split}
\end{align}
for $f\in H^*,w\in W,v\in V$, and module associator inherited from $\Vect$.
\end{lemma}
\begin{proof}
    The statement that $\phi_R$ is a homomorphism of Hopf algebras can be found in \cite{Rad}*{Prop.\,12.2.11} with different conventions. Thus, $H^*$ is an $H$-comodule algebra by co-restriction of the regular $H^*$-coaction along $\phi_R$.  This $H$-comodule structure gives a $\cC$-module structure on $\lmod{H^*}$ as in \eqref{eq:comod-act}, which one checks coincides with the stated one.
\end{proof}

\begin{proposition}\label{prop:Fun-H*-mod}
    The composition of the equivalences from \Cref{lem:FunC-H*-mod,lem:dualize}  defines  an equivalence of left $\cC$-module categories
    $$D\colon \Fun_\cC(\Vect,\Vect)\isomorph \lmod{H^*},$$
    where the $\cC$-actions are defined in \Cref{def:Fun-action} for the source and in \Cref{lem:H-com-alg} for the target.
\end{proposition}
\begin{proof}
    We first consider $\delta^{X\lact G}$ as defined in \eqref{eq:delta-G}, for the functor $X\lact G$ defined in \Cref{def:Fun-action}, for $G\in \cC^*_\Vect$ and $X\in \cC$.  We compute
\begin{align*}
    \delta^{X\lact G}(x\otimes v)&=s^{X\lact G}_{H,\Bbbk}(x\otimes 1_H\otimes  v)\\
    &=(\psi_{X,H}\otimes \id_{G(\Bbbk)})\circ (\id_X\otimes s^G_{H,\Bbbk})(x\otimes 1_H\otimes v)\\
    &=(\psi_{X,H}\otimes \id_{G(\Bbbk)})(x\otimes v^{(-1)}\otimes v^{(0)})
    =R^{(2)}v^{(-1)}\otimes R^{(1)}x\otimes v^{(0)},
\end{align*}
for $v\in G(\Bbbk), x\in X$.
Here, we have used \eqref{eq:FunC-coherence} in the second equality, using the trivial coherence $l^\Vect$ of $\Vect$. We use adapted Sweedler's notation 
$\delta^G(v)=v^{(-1)}\otimes v^{(0)},$
for $v\in G(\Bbbk)$, in the second equality. This describes the $H$-comodule $E(X\lact G)=(X\otimes G(\Bbbk),\delta^{X\lact G})$, where $E$ is the equivalence of \Cref{lem:FunC-H*-mod}.  

Now we apply the equivalence from \Cref{lem:dualize} to $E(X\lact G)=(X\otimes G(\Bbbk), \delta^{X\lact G})$. For this, recall that the $H^*$-module associated to $E(G)$ is defined on the vector space $G(\Bbbk)$ with left $H^*$-action from \eqref{eq:H-dual-action}. By the same formula, the $H^*$-action on $X\otimes G(\Bbbk)$ induced by $\delta^{X\lact G}$ is given by
\begin{align*}
a^l_{E(X\lact G)}\colon H^*\otimes X\otimes G(\Bbbk)&\to X\otimes G(\Bbbk), \\
f\otimes x\otimes v &\mapsto f(R^{(2)}v^{(-1)}) R^{(1)}x\otimes v^{(0)}\\
&= f_{(1)}(R^{(2)})f_{(2)}(v^{(-1)}) R^{(1)}x\otimes v^{(0)}\\
&= f_{(1)}(R^{(2)})R^{(1)}x\otimes f_{(2)}\cdot v.
\end{align*}
The first equality applies the dual coproduct from \eqref{eq:dual-coproduct} and the second equality applies \eqref{eq:H-dual-action} to $f_{(2)}\cdot v$. 
This equals the $H^*$-action $a^l_{X\lact G(\Bbbk)}$ defined in \eqref{eq:C-action-H*-mod}. Hence, $D(X\lact G)=X\lact D(G)$ as left $H^*$-modules. Thus, the functor $D$ is one of left $\cC$-modules with trivial $\cC$-linear structure.  
\end{proof}

For a homomorphism of algebras $\phi\colon A\to B$, we denote the restriction functor by
\begin{equation}\label{eq:Res-phi}
\Res_{\phi}\colon \lmod{B}\to \lmod{A},
\end{equation}
where $\Res_\phi(V)=V$ as a $\Bbbk$-vector space with action $av=\phi(a)v$, for all $v\in V$, $a\in A$. The functor $\Res_\phi$ is faithful and exact. Moreover, if $\phi$ is a homomorphism of Hopf algebras, $\Res_\phi$ is a monoidal functor with trivial monoidal structure. 

\begin{proposition}\label{lem:Vect-fully-exact}
The following are equivalent for the left $\cC$-module category $\Vect$. 
\begin{enumerate}
    \item[(i)] $\Vect$ is fully exact.
    \item[(ii)] The functor $\Res_{\phi_R}\colon \lmod{H}\to \lmod{H^*}$ for $\phi_R$ from \eqref{eq:phiR} preserves projective objects.
    \item[(iii)]  $\Res_{\phi_R}(H)$ is projective as a left $H^*$-module. 
\end{enumerate}
\end{proposition}

\begin{proof} 
To prove the equivalence of (i) and (ii), consider the commutative diagram of functors 
\begin{equation}\label{eq:AVect-diag}
\vcenter{\hbox{
\xymatrix@R=10pt{
\cC\ar[rr]^{A_\Vect}\ar[rd]_-{\Res_{\phi_R}}&&\cC^*_{\Vect}, \ar^-D[ld]\\
&\lmod{H^*}&}}}
\end{equation}
where $A_\Vect$ is the \action functor from \Cref{def:AC-C-mod-functor}, $D$ is the equivalence of \Cref{prop:Fun-H*-mod}, and $\Res_{\phi_R}$ is defined in \eqref{eq:Res-phi}.
Indeed, $A_\Vect(X)=X\lact \id_\Vect$, but $\id_\Vect$ corresponds to the trivial $H^*$-module under the equivalence $D$ since $\id_\Vect(\Bbbk)=\Bbbk$ and the trivial $\cC$-linear structure on $\id_\Vect$ produces the trivial $H^*$-module.  As $D$ is a $\cC$-linear equivalence, 
$$D(X\lact \id_\Vect)\cong X\lact D(\id_\Vect)=X\lact \Bbbk\cong \Res_{\phi_R}(X),$$
where the last equivalence uses \eqref{eq:C-action-H*-mod}, applied to the trivial $H^*$-module $\Bbbk$. Now, equivalence of (i) and (ii) follows as $\Vect$ is fully exact if and only if $A_\Vect$ preserves projective objects, see \Cref{prop:fully-exact-equiv}. 

 The equivalence of (ii) and (iii) is standard homological algebra. That (ii) implies (iii) is clear and the converse uses that $H$ is a projective generator for $\lmod{H}$ and that $A_\Vect$ is, in particular, right exact by \Cref{lem:AM-exact-faithful}.
\end{proof}

 Similarly, we obtain equivalent conditions for $\Vect$ to be  op-fully exact, see \Cref{def:op-fex}.
\begin{corollary}\label{cor:Vect-op-fex}
    The left $\cC$-module $\Vect$ is op-fully exact or, equivalently, ${}^*\Vect$ is fully exact if and only if the functor $\Res_{\phi_{R^{-1}_{21}}}$ preserves projective objects. Moreover, $\Vect$ is perfect if both $\Res_{\phi_R}$ and $\Res_{\phi_{R^{-1}_{21}}}$ preserve projective objects.
\end{corollary}
\begin{proof}
By \Cref{prop:op-fully-equiv}, $\Vect$ is op-fully exact if and only if $A_\Vect^\rev$ preserves projective objects. This functor corresponds, under equivalence $D$ in \eqref{eq:AVect-diag}, to  restriction along the morphism 
$\phi_{R^{-1}_{21}}$ since the reverse braiding on $\cC$ is obtained by acting with the R-matrix $R^{-1}_{21}$ instead of $R$. Thus, the result follows from \Cref{lem:Vect-fully-exact}. Combining both assumptions, $\Vect$ is perfect, see \Cref{def:dualizable-fully-exact} and \Cref{cor:dualizable-fully-exact-criterion}.
\end{proof}

We note that whether $\Vect$ is fully exact depends on the choice of R-matrix for $H$. See the later \Cref{rem:vect-Sweedler} that shows a Hopf algebra $S$ with a class of braidings $R_t$, for $t\in \mC$, where $\Vect$ is fully exact if and only if $t\neq 0$.
We observe the following corollaries.

\begin{corollary}\label{cor:C-semisimple-Vect-fully-exact}
    Assume $H$ is cocommutative and equip $\cC=\lmod{H}$ with the symmetric braiding given by the universal R-matrix $R=1\otimes 1$. If $H^*$ is not semisimple, then $\Vect$ is not fully exact. 
\end{corollary}
\begin{proof}
    As $R=1\otimes 1$, we have that $\phi_R= u_H\circ \varepsilon_H$. This implies that $\Res_{\phi_R}=\Res_{\varepsilon_H}\circ \Res_{u_H}$. This functor sends an object $X$ of $\cC$ first to the underlying vector space and then to $\one^{\oplus \dim_\Bbbk X}$. If $\lmod{H^*}$ is not semisimple, then the object $\one^{\oplus \dim_\Bbbk X}$ is not projective for non-zero $X$. By \Cref{lem:Vect-fully-exact}, $\Vect$ is not fully exact. 
\end{proof}

\begin{lemma} \label{lem:Res-equiv}
For a morphism of algebras $\phi\colon A\to B$, the functor $\Res_\phi$ from \eqref{eq:Res-phi} is an equivalence if and only if $\phi$ is an isomorphism of algebras.
\end{lemma}
\begin{proof} It is clear that if $\phi$ is an isomorphism, $\Res_\phi$ is an equivalence. Conversely, assume that $\Res_\phi$ is an equivalence. By the Eilenberg--Watts theorem, this is equivalent to the $A$-$B$-bimodule $X=\Res_{\phi}(B)$ being invertible. This bimodule $X$ is defined on $B$, with the regular right $B$-action and left $A$-action given by restriction along $\phi$.  We denote $X$ by ${}_\phi B$.

For the inverse $B$-$A$-bimodule $Y$, there exists an isomorphism of $A$-bimodules ${}_\phi B \otimes_B Y\cong A$. Using the canonical isomorphism $B\otimes_B Y\cong Y$ of right $A$-modules, we have that $Y\cong A$ as a right $A$-module. For simplicity, we identify $Y=A$ as right $A$-modules. Thus, the left $B$-module structure on $Y$ is given by the algebra homomorphism $\tau\colon B\to A$ defined by $\tau(b)=b\cdot 1_A$. Hence, we write $Y={}_\tau A$. 
With this identification, the isomorphism 
$\alpha \colon {}_\phi B\otimes_B {}_\tau A \isomorph A$
of  $A$-bimodules is given by the left $B$-action on $A$. That is  $\alpha(b\otimes a)=b\cdot a=\tau(b)a.$ But $\alpha$ is an isomorphism of left $A$-modules. Hence, for any $a\in A$,
$$a=a\cdot \alpha(1\otimes 1)=\alpha(a\cdot 1\otimes 1)=\alpha(\phi(a)\otimes 1)=\tau\phi(a).$$
This implies that $\tau\phi=\id_A.$ Similarly, $\phi\tau=\id_B$ and $\phi$ is an isomorphism. 
\end{proof}

\begin{corollary}\label{cor:Vect-inv}
 The  left  $\cC$-module category $\Vect$ is invertible  if and only if $\phi_R$ is injective (or surjective). In this case, $\cC^*_\Vect\simeq \cC$ as tensor categories and as left $\cC$-module categories. 
    In particular, by \Cref{cor:inv-fully-exact}, $\Vect$ is fully exact and op-fully exact  (see \Cref{def:op-fex}) in this case.
\end{corollary}
\begin{proof}
    By \Cref{prop:Minv-AM-equiv}, $\Vect$ is invertible if and only if $A_\Vect$ is an equivalence. By the commutative diagram  \eqref{eq:AVect-diag}, this is equivalent to $\Res_{\phi_R}$ being an equivalence. By \Cref{lem:Res-equiv}, this is equivalent to $\phi_R$ being an isomorphism. 
    Since the Hopf algebras $H$ and $H^*$ are finite-dimensional of the same dimension, $\phi_R$ is an isomorphism if and only if it is injective (or surjective).
\end{proof}

 An equivalent characterization of invertibility of $\phi_R$ is that the R-matrix defines a non-degenerate copairing when viewed as a map $\Bbbk \to H\otimes H, 1\mapsto R$. 
The next example in~\Cref{sec:uqsl2} 
shows that $\Vect$ might not be fully exact even when the braiding is non-degenerate, i.e.\ when the Hopf algebra $H$ is factorizable  in the sense that $RR_{21}$ defines a non-degenerate copairing, see e.g. \cite{Rad}*{Def.\,12.4.1}.

\subsection{Cocyle twists of \texorpdfstring{$\Vect$}{Vect}}\label{sec:vectJ}
 We first recall that a Hopf $2$-cocycle is an invertible element $J\in H\otimes H$ 
satisfying
\begin{equation}\label{eq:Hopf-cocycle}
    (\id\otimes \Delta)J\cdot J_{23}=(\Delta\otimes \id)J\cdot J_{12} \ ,
\end{equation}
that is normalized, i.e.\,$(\varepsilon\otimes \id)J=1=(\id\otimes\varepsilon)J$.
Then we can twist the monoidal structure of the fiber functor $F\colon \cC \to \Vect$ by acting with $J$ to a monoidal functor $F^J\colon \cC\to \Vect$, see e.g. \cite{EGNO}*{Sec.\,5.14}. The resulting $\cC$-module category structure on $\Vect$ is denoted by $\Vect^J$. Its module associator is defined by the action of $J$, i.e.
\begin{equation}\label{eq:J-ass}
    l^J_{V,U,W}\colon V\lact (U\lact W)\isomorph V\otimes U \lact W, \quad v\otimes (u \otimes w) \mapsto J\cdot (v\otimes u)\otimes w,
\end{equation}
for objects $V,U$ in $\cC$, $W\in \Vect^J$, and $v\in V,u\in U,w\in W$.

In \Cref{prop:Fun-H*-mod-twisted}, we will generalize \Cref{prop:Fun-H*-mod} to compute the categories of $\cC$-linear functors between such  $\cC$-module categories.

 \begin{definition}\label{def:H-twisted-coalg}
     The coalgebra $\Btwist{L}{H}{J}$ is defined on the vector space $H$ with the counit $\varepsilon_H$ and the twisted coproduct
    \begin{equation}\label{eq:Delta-twist}
    ({}^{L}\Delta_H^J)(h)=L^{-1}\cdot \Delta_H (h)\cdot J,\qquad \text{for}\quad h\in H\ ,
    \end{equation}
    which is coassociative due to the $2$-cocycle condition on $J$ and $L$.
 \end{definition}

We define a left $\cC$-action on $\lcomod{\Btwist{L}{H}{J}}$ as follows. For a left $H$-module $(W,a_W)$ and a left $\Btwist{L}{H}{J}$-comodule $(V,\delta_V)$, $W\lact V$ is the $\Btwist{L}{H}{J}$-comodule $(W\otimes V, \delta_{W\lact V})$, where 
\begin{equation}\label{eq:twisted-coaction}
\delta_{W\lact V}(w\otimes v)=(R^L)^{(2)}v^{(-1)} \otimes ((R^L)^{(1)}\cdot w)\otimes v^{(0)},
\end{equation}
for $w\in W$, $v\in V$, where we denote $\delta_V(v)=v^{(-1)}\otimes v^{(0)}$, and  
\begin{equation}\label{eq:R-twisted}
R^L=(R^L)^{(1)}\otimes (R^L)^{(2)}=L^{-1}_{21}\cdot R\cdot L.
\end{equation}
The module associator for this $\cC$-action is inherited from $\Vect^L$ and thus given by the action of $L$.

\smallskip 

For the next definition, recall our convention of \eqref{eq:dual-product} on the dual algebra $H^*$.

\begin{definition}\label{def:H-twisted-alg}
Given two $2$-cocycles $J$ and $L$, dualizing the coproduct $\Btwist{L}{H}{J}$ induces a product on $H^*$:
\begin{equation}\label{eq:m-twisted}
    \Dtwist{L}{m}{J}{}(f\otimes g)(h):=g\bigl(L^{(-1)}h_{(1)}J^{(1)}\bigr)f\bigl(L^{(-2)}h_{(2)}J^{(2)}\bigr),
\end{equation}
for $f,g\in H^*$ and $h\in H$, where $J=J^{(1)}\otimes J^{(2)}$ and $L^{-1}=L^{(-1)}\otimes L^{(-2)}$. We denote the resulting algebra structure on the vector space $H^*$, with product $\Dtwist{L}{m}{J}{}$ and the same unit as $H^*$, by $\Dtwist{L}{H}{J}{*}$.
\end{definition}

We obtain the following generalization of the equivalence $E$ from \Cref{lem:FunC-H*-mod}.
\begin{lemma}\label{lem:equiv-FunVect-twisted}
    There is an equivalence of categories
\begin{equation}\label{eq:equiv-Fun-twisted}
\Btwist{L}{E}{J}\colon \Fun_\cC(\Vect^J,\Vect^L)\isomorph \lcomod{\Btwist{L}{H}{J}},\quad (G,s^G)\mapsto (G(\Bbbk),\delta^G),
\end{equation}
for the coalgebra $\Btwist{L}{H}{J}$ from \Cref{def:H-twisted-coalg} and with $\delta^G$ defined as in \eqref{eq:delta-G}. 
\end{lemma}
\begin{proof}
Write $\delta^G(v)=v^{(-1)}\otimes v^{(0)}\in F(H)\otimes G(\Bbbk)$. 
Identifying $G(V) := V \otimes G(\Bbbk)$ as vector spaces, we use naturality
of the $\cC$-linear structure $s^G_{X,\Bbbk}$, to show that $s^G_{X,\Bbbk}(x\otimes v)=v^{(-1)}\cdot x\otimes v^{(0)}$,
for any $X\in\cC$, $x\in X$ and $v\in G(\Bbbk)$.
Then we evaluate the coherence condition \eqref{eq:C-linear-coherence} on projective generators, setting $X=Y=H$ and $M=\Bbbk$. 
First, the term $s^G_{H\otimes H,\Bbbk}$ of \eqref{eq:C-linear-coherence}, which is an automorphism of $F(H)\otimes F(H)\otimes G(\Bbbk)$, is the following linear map: $g\otimes h \otimes v \mapsto \Big(\Delta_H(v^{(-1)})\cdot(g\otimes h)\Big)\otimes v^{(0)}$.
Then, the equation \eqref{eq:C-linear-coherence} evaluated on the element $1_H\otimes 1_H\otimes v$  gives the equation
$$ \Delta_H(v^{(-1)})\otimes v^{(0)} = 
L^{(1)}v^{(-1)}J^{-(1)}\otimes L^{(2)}v^{(0)(-1)} J^{-(2)}\otimes v^{(0)(0)}\ ,$$
for any $v\in G(\Bbbk)$, and we used the non-trivial module associators of $\Vect^J$ and $\Vect^L$ given by the action of $J$ and $L$, recall~\eqref{eq:J-ass}.
This is equivalent to the comodule coassociativity condition for the map $\delta^G$,
recall the coproduct~\eqref{eq:Delta-twist}.
Moreover,  $s^G_{\one,\Bbbk}=\id_{G(\Bbbk)}$ is equivalent 
to the counit axiom for~$\delta^G$. Hence 
$(G(\Bbbk),\delta^G)$
is a  $\Btwist{L}{H}{J}$-comodule and thus $\Btwist{L}{E}{J}$ is a well-defined functor. 
A quasi-inverse for $\Btwist{L}{E}{J}$ is obtained by mapping a $\Btwist{L}{H}{J}$-comodule $(C,\delta)$ to the pair $((-)\otimes C,s^\delta)$, with $s^\delta$ defined by the same formula \eqref{eq:F-fiber}. Checking the coherence condition \eqref{eq:C-linear-coherence} for $s^{\delta}$ is straightforward.
\end{proof}

\begin{proposition}\label{prop:Fun-H*-mod-twisted}
For two $2$-cocycles $J$ and $L$, we have an equivalence of left $\cC$-module categories 
\begin{equation}\label{eq:twisted-dualizing-equiv}
    \Fun_\cC(\Vect^J,\Vect^L)\isomorph \lmod{\Dtwist{L}{H}{J}{*}\,}.
\end{equation}
Here, we equip $\lmod{\Dtwist{L}{H}{J}{*}\,}$ with the following left $\cC$-module  structure:
For $V\in\cC$ and $W\in \lmod{\Dtwist{L}{H}{J}{*}\,}$, we define $V\lact W$ to be $V\otimes W$, as a vector space, with left $\Dtwist{L}{H}{J}{*}$-action defined by 
\begin{equation}\label{eq:H*-action-VW}
    f\cdot (v\otimes w) = \phi_{R^L}(f_{(1)})\cdot v \otimes f_{(2)}\cdot w,
\end{equation}
for $v\in V, w\in W$. The module associator $l^L$ is defined by the action of $L$, cf.\,\eqref{eq:J-ass}.
\end{proposition}
\begin{proof}
The proof follows the same strategy as that of \Cref{prop:Fun-H*-mod}. 
Recall that the twisted algebra $\Dtwist{L}{H}{J}{*}$ from \Cref{def:H-twisted-alg} is obtained by dualizing the coalgebra $\Btwist{L}{H}{J}$. Hence,  \Cref{lem:dualize} generalizes to an equivalence of categories 
\begin{equation}\label{eq:twisted-dualizing-equiv-pre}
    \lcomod{\Btwist{L}{H}{J}}\isomorph \lmod{\Dtwist{L}{H}{J}{*}\,},
\end{equation}
by dualizing a coaction to an action of the dual by the same formula as in \eqref{eq:H-dual-action}. Composing with the equivalence of \Cref{lem:equiv-FunVect-twisted}, we obtain an equivalence
$$
D\colon \Fun_\cC(\Vect^J,\Vect^L)\isomorph \lmod{\Dtwist{L}{H}{J}{*}\,}.
$$
We have to show that the equivalence $D$ is $\cC$-linear. As in the proof of \Cref{prop:Fun-H*-mod}, but with the non-trivial module associator $l^L$, we compute
 \begin{align*}
    \delta^{X\lact G}(x\otimes v)
    &=s^{X\lact G}_{H,\Bbbk}(x\otimes 1_H\otimes  v)\\
&=(l^L_{H,X,G(\Bbbk)})^{-1}\circ (\psi_{X,H}\otimes \id_{G(\Bbbk)})\circ l^L_{X,H,G(\Bbbk)}\circ (\id_X\otimes s^G_{H,\Bbbk})(x\otimes 1_H\otimes v)\\
    &=(l^L_{H,X,G(\Bbbk)})^{-1}\circ (\psi_{X,H}\otimes \id_{G(\Bbbk)})(L^{(1)}x\otimes L^{(2)}v^{(-1)}\otimes v^{(0)})\\
    &=L^{-(1)}R^{(2)} L^{(2)}v^{(-1)}\otimes L^{-(2)}R^{(1)}L^{(1)} x\otimes v^{(0)}
    =(R^L)^{(2)} v^{(-1)}\otimes (R^L)^{(1)}  x\otimes v^{(0)},
\end{align*}
for $v\in G(\Bbbk), x\in X$,
where we use \eqref{eq:FunC-coherence} in the second equality and the definition of $R^L$ from \eqref{eq:R-twisted} in the final equality. Next, dualizing the $\Btwist{L}{H}{J}$-coaction $\delta^{X\lact G}$ via the equivalence \eqref{eq:twisted-dualizing-equiv-pre} recovers the $\Dtwist{L}{H}{J}{*}\,$-action of \eqref{eq:H*-action-VW}. This shows that the equivalence $D$
maps $X\lact (G,s^G)$ to $X\lact D(G,s^G)$. Thus, $D$ is an equivalence of $\cC$-module categories with trivial $\cC$-linear structure. The associator of $\lmod{\Dtwist{L}{H}{J}{*}\,}$ is inherited from the associator of $\Vect^L$ in \eqref{lem:Hom-C-mod} as it is the image of the associator of $\Fun_\cC(\Vect^J,\Vect^L)$, defined in \Cref{lem:Hom-C-mod}, under the equivalence $D$.
\end{proof}

The module associator $l^L$ 
from \Cref{prop:Fun-H*-mod-twisted} being a morphism of $\Dtwist{L}{H}{J}{*}$-modules is equivalent to
\begin{equation}\label{eq:twisted-comod-algebra}
    (L^{-1}\otimes 1)\cdot ((\Delta\otimes \id)\delta(f))\cdot (L\otimes 1)=(\id\otimes \delta)\delta(f)\ ,
\end{equation}
where, analogue to \eqref{eq:delta-star}, we define a homomorphism of algebras 
\begin{equation}\label{eq:delta-twisted}
\delta \colon \Dtwist{L}{H}{J}{*}\to H\otimes \Dtwist{L}{H}{J}{*}, \quad f\mapsto \phi_{R^L}(f_{(1)})\otimes f_{(2)}.
\end{equation}
We observe that \eqref{eq:twisted-comod-algebra} is equivalent to  $\Dtwist{L}{H}{J}{*}$ being a $\Btwist{L}{H}{L}$-comodule algebra with respect to $\delta$ in~\eqref{eq:delta-twisted}. Here, the Hopf algebra $\Btwist{L}{H}{L}$ is $H$ with coproduct twisted by $L$ as in \Cref{def:H-twisted-coalg}.

\begin{remark}\label{rem:twisted-fully-exact}
    It follows from \Cref{prop:Fun-H*-mod-twisted} that the equivalence of (i) and (ii) in \Cref{lem:Vect-fully-exact} has the following 
    generalization:
    $\Vect^J$ is fully exact if and only if $\Res_{\phi_{R^J}}$ preserves projective objects,
    with $R^J$ defined in~\eqref{eq:R-twisted} and recall~\eqref{eq:phiR}. We omit the proof because it essentially repeats the one of \Cref{lem:Vect-fully-exact}.
    Furthermore, using \eqref{eq:R-twisted}, we have  $(R^J)_{21}^{-1}=(R_{21}^{-1})^J$. Hence, ${}^*\Vect^J$ is fully exact if and only if  $\Res_{\phi_{(R_{21}^{-1})^J}}$ preserves projective objects, 
   a direct generalization of 
   \Cref{cor:Vect-op-fex}. Moreover, $\Vect^J$ is invertible if and only if 
 $\phi_{R^J}$ is invertible, 
 generalizing \Cref{cor:Vect-inv}. 
\end{remark}

\subsection{Small quantum \texorpdfstring{$\mathfrak{sl}_2$}{sl2}}
 \label{sec:uqsl2}
Here, we consider the case of a non-semisimple factorizable Hopf algebra.
Let $q$ be a primitive root of unity of odd order $p \geq 3$. The small quantum group $u_q = u_q(\mathfrak{sl}_2)$ is the $\mathbb{C}$-algebra generated by $E,F,K$ with the defining relations
\[ KE = q^2EK, \quad KF = q^{-2}FK, \quad EF - FE = \frac{K - K^{-1}}{q - q^{-1}}, \quad E^p = F ^p = 0, \quad K^p = 1. \]
Its coproduct counit are given by\footnote{We do not give the antipode because it is not used in this section.}
\begin{equation}\label{HopfStructureUqsl2}
\begin{split}
\Delta(E) &= 1 \otimes E + E \otimes K, \qquad  \Delta(F) = F \otimes 1 + K^{-1} \otimes F, \qquad \Delta(K) = K \otimes K\ ,\\
\varepsilon(E)&=0,\qquad \varepsilon(F)=0, \qquad \varepsilon(K)=1.
\end{split}
\end{equation}
It is well-known~\cite{CPr} or \cite{KS}*{Sec.\,3.3.3} that this algebra admits $p$ isomorphism classes of simple modules, each of dimension $1\leq s \leq p$. The one simple module of dimension $p$ is projective. All the other $p-1$ simple modules have projective covers of dimension $2p$ each \cite{CPr}*{Sec.\,3.8}.

We now consider the dual Hopf algebra $u_q^*$.
 Let $\rho : u_q(\mathfrak{sl}_2) \to \mathrm{End}(\mathbb{C}^2)$ be the fundamental representation defined by
$$\rho(E) = 
\begin{pmatrix}
0 & 1\\
0 & 0
\end{pmatrix},\qquad  
\rho(F) = 
\begin{pmatrix}
0 & 0\\
1 & 0
\end{pmatrix}, 
\qquad \rho(K) = 
\begin{pmatrix}
q & 0\\
0 & q^{-1}
\end{pmatrix}.
$$
Write  $\rho = 
\left(\begin{smallmatrix}
a & b\\
c & d
\end{smallmatrix}\right),$
so that $a,b,c,d \in u_q^*$.  We first notice that these functionals on the PBW basis are
\begin{equation}
    \label{eq:abcd-PBW}
    \begin{split}
a(F^mE^nK^j) &= q^j\delta_{m,0}\delta_{n,0} \ ,\qquad
b(F^mE^nK^j) = q^{-j}\delta_{m,0}\delta_{n,1} \ ,\\
c(F^mE^nK^j) &= q^j\delta_{m,1}\delta_{n,0} \ ,\qquad
d(F^mE^nK^j) = q^{-j}(\delta_{m,0}\delta_{n,0} + \delta_{m,1}\delta_{n,1}) \ ,
\end{split}
\end{equation}
for all $0\leq m,n,j \leq p-1$. By induction, we prove the following formula for the coproduct on the PBW basis:
\begin{multline*}
  \Delta(F^m E^n K^j)
  =\sum_{r=0}^m\sum_{s=0}^n\,q^{2(n-s)(r-m)+r(m - r) + s(n - s)}
  \qbin{m}{r}\qbin{n}{s}
  F^r E^{n-s} K^{r-m+j}\otimes F^{m-r} E^s K^{n-s+j}.
\end{multline*}
Here, we use $q$-binomials
$$\qbin{m}{n}=\frac{[m]!}{[n]![m-n]!}, \qquad \text{where} \qquad [n]!=\prod_{k=1}^n[k], \qquad \text{and} \qquad [k]=\frac{q^{k}-q^{-k}}{q-q^{-1}}.$$
With this coproduct formula and with the convention~\eqref{eq:dual-product} on the product in $H^*$, we calculate that 
\begin{align*}
bc(F^m E^n K^j) &= \delta_{m,1} \delta_{n,1} = cb(F^m E^n K^j)\ ,\\
db(F^m E^n K^j) &= q^{-2j-1} ( \delta_{m,0} \delta_{n,1} + \delta_{m,1} \delta_{n,2}) \ ,\\
da(F^m E^n K^j) &= \delta_{m,0} \delta_{n,0} + q^{-1}\delta_{m,1} \delta_{n,1}\ , 
\end{align*}
and similar for the other products. With this result, we see that $u_q^*$ is generated by $b,c,d$ with the defining relations
\begin{equation}\label{relationsBarUiDual}
bc = cb, \:\:\:\: db = q^{-1} bd, \:\:\:\: dc = q^{-1} cd, \:\:\:\: b^p=c^p=0, \:\:\:\: d^p=1,
\end{equation}
together with $da -ad = (q^{-1} - q)cb$, and where $1$ in the last relation is the counit of $u_q$.
We thus see that the element $a$ is not independent due to invertibility of $d$: $a=d^{-1} + q^{-1}d^{-1}bc$.

Now we can describe projective covers over $u_q^*$.
From the relations~\eqref{relationsBarUiDual} we see that  $u_q^*$ admits precisely $p$ isomorphism classes of simple modules that are one-dimensional $\mathbb{C}_\xi$ indexed by $p$th roots $q^\xi$ of $1$ which are eigenvalues of $d$, for $\xi=0,\ldots,p-1$, and with trivial action of $b$ and $c$.
The corresponding primitive idempotents are $e_\xi = \frac{1}{p} \sum_{n=1}^p q^{-\xi n} d^n$, i.e.\ $d e_\xi = q^\xi e_\xi$, and the  projective modules $u_q^* e_\xi$ cover $\mathbb{C}_\xi$. 
Because $b$ and $c$ commute, these are of dimension $p^2$ each, using the penultimate relations in~\eqref{relationsBarUiDual}. 
The projective covers $P_\xi$ of $\mC_\xi$ are direct summands of $u_q^* e_\xi$ and thus have dimension at most $p^2$. By the Wedderburn--Artin theorem, $\sum_{\xi=0}^{p-1}\dim_\mC P_\xi\leq p\cdot p^2$ has to equal $p^3=\dim_\mC u_q^*$. And therefore $P_\xi=u_q^* e_\xi$ for all $\xi$ and every projective $u_q^*$-module has dimension at least $p^2$.

The Hopf algebra $u_q$ is known to be factorizable \cite{Lyu}*{Cor.\,A.3.3} with the R-matrix~\cite{Lyu}*{Secs.\,A.2--A.3} of the form
$R \in u_q(\mathfrak{b}_+)\otimes u_q(\mathfrak{b}_-)$ where $u_q(\mathfrak{b}_+)$ is the Hopf subalgebra generated by $E$ and $K$ while $u_q(\mathfrak{b}_-)$ is the one generated by $F$ and $K$. Therefore, the image of $u_q^*$ under the Hopf algebra map $\phi_R$ from~\eqref{eq:phiR} is $u_q(\mathfrak{b}_+)\subset u_q$. The corresponding restriction functor $\Res_{\phi_R}\colon \lmod{u_q}\to \lmod{u_q^*}$  sends the $p$ and $2p$ dimensional projective $u_q$-modules to $u_q^*$-modules of the same dimension which cannot be projective  because of the inequality $2p< p^2$. 
By \Cref{lem:Vect-fully-exact} (ii), we conclude the following result. 

\begin{proposition}\label{prop:vect-sl2}
The left $\cC$-module category $\Vect$ is not fully exact for $\cC=\lmod{u_q(\mathfrak{sl}_2)}$.
\end{proposition}

As a corollary, we obtain the following example of a relative Deligne product of exact $\cC$-module categories over a non-degenerate braided tensor category $\cC$ that is not exact. 
\begin{corollary}\label{cor:rel-Del-not-exact-factorizable}
The left $\cC$-module category ${}^*\Vect\boxtimes_\cC \Vect$ is not exact.
\end{corollary}
\begin{proof}
    By \Cref{prop:vect-sl2}, $\Vect$ is not fully exact and hence ${}^*\Vect\boxtimes_\cC \Vect$ is not exact by \Cref{prop:fully-exact-equiv}. Moreover,  ${}^*\Vect$ is exact since $\Vect$ is exact, see \Cref{prop:A-Psi-exact}.
\end{proof}

A similar analysis can be done for general simple Lie algebras $\mathfrak{g}$.
The R-matrix for any $u_q (\mathfrak{g})$ has   
the form $R \in u_q(\mathfrak{b}_+)\otimes u_q(\mathfrak{b}_-)$, and therefore, similarly to the above $\mathfrak{sl}_2$ case, the image of $\phi_R: u_q (\mathfrak{g})^* \to u_q (\mathfrak{g})$ is the Borel subalgebra $u_q(\mathfrak{b}_+)$. In particular, we see from~\Cref{cor:Vect-inv} that the module category $\Vect$  is not invertible. Furthermore, we can conjecture:

\begin{conjecture}
Let $\cC$ be the braided tensor category of finite-dimensional modules over $u_q(\mathfrak{g})$, for $q$ a root of unity of odd order, and with the standard factorizable R-matrix, see \cite{Lyu}*{Secs.\,A.2--A.3}. Then the $\cC$-module category $\Vect$ is not fully exact.
\end{conjecture}

\subsection{Fully exact module categories for Sweedler's Hopf algebra}\label{sec:Sweedler}

In this section, we provide a classification of fully exact and perfect module categories over the category of (finite-dimensional) modules over Sweedler's four-dimensional Hopf algebra, see \Cref{prop:S-mod-fully-exact-class}. We also compute the relative Deligne product and functor categories for this example in \Cref{prop:Sweedler-products}.

We assume that $\Bbbk=\mC$ for the rest of this section. Denote by $\sVect$ the symmetric tensor category of \emph{super vector spaces} over $\mC$.  As a braided tensor category, $\sVect$ is equivalent to the category of $\mC C_2$-modules, for $C_2=\inner{g~|~g^2=1}$, with braiding given by the R-matrix 
\begin{equation}\label{eq:R-sVect}
R=\frac{1}{2}(1\otimes 1+g\otimes 1+1\otimes g-g\otimes g).
\end{equation}

\begin{definition}[Sweedler Hopf algebra]
The \emph{Sweedler Hopf algebra} is the algebra
$$S=\mC\inner{x,g~|~x^2=0,~g^2=1,~gx=-xg},$$
with coproduct $\Delta$, counit $\varepsilon$, and antipode $S$ determined on generators by
\begin{gather*}
\Delta(x)=x\otimes 1+ g\otimes x, \qquad \Delta(g)=g\otimes g,\\
\varepsilon(x)=0, \qquad \varepsilon(g)=1,\qquad S(g)=g, \qquad S(x)=xg.
\end{gather*}
\end{definition}

The tensor category $\cC=\lmod{S}$ is non-semisimple. It has two simple objects $\mC_{+}$ and $\mC_-$, which are one-dimensional and where $g$ acts by $+1$, respectively, $-1$, and $x$ acts by zero. The projective covers of these two simple objects are $Se_{\pm}$, for the idempotents 
$e_{\pm}=(1\pm g)/2$.

\smallskip

Consider the super Hopf algebra 
\begin{equation}\label{def:B}
B=\mC[x]/(x^2)\ ,
\end{equation}
where $x$ has degree 1, with primitive coproduct $\Delta(x)=x\otimes 1+1\otimes x$. Equivalently, $B$ is a commutative and cocommutative Hopf algebra in $\sVect$. As a Hopf algebra, $S$ is isomorphic to the bosonization $B\rtimes \mC C_2$.
The super Hopf algebra $B$ is quasi-triangular with respect to any of the R-matrices 
\begin{equation} \sR_t =  1\otimes 1+ t x\otimes x,\label{eq:Rt-super}
\end{equation}
for $t \in \mC$. This makes the Hopf algebra $S$ quasi-triangular with a family of R-matrices given by 
\begin{align}\label{eq:Rt}
    R_t:=\frac{1}{2}(1\otimes 1+g\otimes 1+1\otimes g-g\otimes g)(1\otimes 1+tx\otimes gx), \qquad t\in \mC.
\end{align}
Note that $R_0=R$ for the R-matrix from \eqref{eq:R-sVect}. One checks that $(R_t)_{21}\cdot R_t=1\otimes 1$. Hence, the R-matrix $R_t$ induces a symmetric braiding on $\cC=\lmod{S}$.

\begin{definition}\label{def:Ct}
The symmetric  tensor category obtained by equipping $\cC$ with the braiding induced by the R-matrix $R_t$ from~\eqref{eq:Rt} is denoted by $\cC_t$.
 We   denote the symmetric tensor category $\cC_0$ simply by $\cC$.
\end{definition}

We note that the R-matrices $R_t$ are, in particular, Hopf $2$-cocycles in $S\otimes S$, recall~\eqref{eq:Hopf-cocycle}.  Similarly, $\sR_t$ is a Hopf $2$-cocycle in $B\otimes B$ but internal to $\sVect$.

\begin{lemma}\label{lem:Ct-equivalences}
For any $t,\lambda \in \mC$, we have an equivalence of braided tensor categories
\begin{equation}\label{eq:Ct-equivalence}
I^\lambda_t=(\id_\cC,\iota^\lambda)\colon \cC_{t}\to\cC_{t+2\lambda},
\end{equation}
with monoidal structure defined by
$$\iota_{V,W}^\lambda\colon V\otimes W \to V\otimes W, \quad v\otimes w\mapsto R_\lambda \cdot v\otimes w,$$
for objects $V,W$ of $\cC$ and elements $v\in V,w\in W$. In particular,
we get an equivalence of braided tensor categories $I_t^{-t/2}\colon \cC_t\isomorph \cC_0$. 
\end{lemma}
\begin{proof}
Because $R_\lambda$ is a $2$-cocycle for the Hopf algebra $S$, it follows that $I^\lambda_t$ is a monoidal functor. It is evidently an equivalence.
The identity
\begin{equation}\label{eq:twist-R}
    (R_\lambda)_{21}^{-1}\cdot R_t \cdot R_\lambda= R_{t+2\lambda}
\end{equation}
of R-matrices for the Hopf algebra $S$ is easily verified and implies that $I^\lambda_t$ 
is a braided monoidal equivalence.
\end{proof}

A classification of equivalence classes of indecomposable exact left module categories over $\cC$ can be obtained from the more general results \cite{EO1}*{Thm.\,4.10} and \cite{Mom1}*{Section~8.1}, see also \cite{Will}*{Sec.\,5.3} for a detailed classification of exact module categories in terms of mixed associators which we will use. The indecomposable exact left module categories fall into the two one-parameter families
$$
\Set{\Vect^\lambda ~|~ \lambda \in \mC} \quad \text{and} \quad \Set{\sVect^\lambda ~|~ \lambda \in  \mC},
$$
of semisimple module categories 
plus two non-semisimple module categories $\cC$ and $\cD=\lmod{D}$. Here, $D=\mC[x]/(x^2)$, which is a subalgebra of $S$. The restriction of the coproduct of $S$, $\left.\Delta \right|_D\colon D\to S\otimes D$ makes $D$ a left $S$-coideal subalgebra. In particular, $D$ is a left $S$-comodule algebra,  which induces a left $\cC$-action on $\cD$  by~\eqref{eq:comod-act}.

Denote by $F\colon \cC\to \Vect$ the fiber functor that forgets the $S$-action, together with the monoidal structure given by identity components. 
For each $\lambda \in \mC$,  $\Vect^\lambda=\Vect$ as a category, with the left $\cC$-action given via \Cref{ex:tensor-act} by the tensor functor 
\begin{equation}\label{eq:F-lambda}
F^\lambda=(F,\mu^{\lambda}) \colon \cC \to \Vect, \qquad \mu^{\lambda}_{V,W}\colon F(V)\otimes F(W)\to F(V\otimes W), \quad v\otimes w\mapsto R_\lambda\cdot v\otimes w,
\end{equation}
for $v\in F(V)$, $w\in F(W)$.
Similarly, there is a fiber functor $\sF\colon \cC\to \sVect$, since $\cC=\lmodint{B}{\sVect}$.  We define $\sVect^\lambda=\sVect$ as a category, with the left $\cC$-action given via \Cref{ex:tensor-act} by the  forgetful functor $(\sF,\smu^\lambda)\colon \cC \to \sVect$ with the monoidal structure
\begin{gather*}
\smu^\lambda_{V,W} \colon \sF(V)\otimes \sF(W)\to \sF(V\otimes W), 
\quad v\otimes w\mapsto \sR_\lambda\cdot v\otimes w,
\end{gather*}
where $\sR_\lambda$ is the $2$-cocycle of the super Hopf algebra $B$, see \eqref{def:B}--\eqref{eq:Rt-super}. 

In the following we will re-parametrize the set of exact module categories as follows: We denote, for $\lambda\in \mC \setminus \Set{0}$,
$$\cV_\lambda:=\Vect^{\lambda^{-1}}, \quad \cV_0:=\cC, \quad \cV_\infty:=\Vect^0,$$
and, similarly, 
$$\cS_\lambda:=\sVect^{\lambda^{-1}}, \quad \cS_0:=\cD, \quad \cS_\infty:=\sVect^0.$$
 This way, the indecomposable exact $\cC$-module categories of $\cC$ are parametrized as a disjoint union
 $$\Set{\cV_\lambda ~|~ \lambda \in \mC \rP^1} \sqcup \Set{\cS_\lambda ~|~ \lambda \in \mC \rP^1},$$
 indexed by two copies of the complex projective space $\mC\rP^1=\mC\cup \Set{\infty}$. That no two of these left $\cC$-module categories are equivalent follows from \cite{Will}*{Props.\,5.20\,\&\,5.21}. 

\begin{lemma}\label{lem:D-separable}
    The left $\cC$-module category $\cD$ is separable (see  \Cref{def:separable,def:separable-mod-cat}).
\end{lemma}
\begin{proof}
    It can be shown that $\cD\simeq \rmodint{\cC}{I}$, for the algebra $I=\Bbbk[y]/(y^2-1)$ in $\cC$. An an object, $I=\mC_+\oplus \mC_-$, where $1_I$ and $y$ span the direct summands $\mC_+$ and $\mC_-$, respectively.  

    A section for the multiplication map is defined by 
    $$s\colon I\to I\otimes I,\qquad  1\mapsto \frac{1}{2}(1\otimes 1+y\otimes y), \quad y\mapsto \frac{1}{2}(y\otimes 1 +1\otimes y),$$
    which clearly commutes with the $S$-action.
    This defines a homomorphism of $I$-bimodules since
    $$y\cdot \frac{1}{2}(1\otimes 1+y\otimes y)=\frac{1}{2}(y\otimes 1+1\otimes y)=\frac{1}{2}(1\otimes 1+y\otimes y)\cdot y.$$
    Hence, $I$ is a separable algebra in $\cC$. 
    \end{proof}

We are now ready to classify fully exact $\cC$-modules in the Sweedler case $H=S$.

\begin{proposition}\label{prop:S-mod-fully-exact-class}
For $\lambda \in \mC\rP^1$, the left $\cC$-module categories $\cV_\lambda$ and $\cS_\lambda$ are fully exact, and thus perfect if and only if $\lambda\neq \infty$. 
\end{proposition}
\begin{proof}
As $\cC$ is symmetric, the classes of fully exact and op-fully exact $\cC$-module categories coincide by \Cref{cor:C-symmetric-op-fex}. Therefore, $\cV_\lambda$ and $\cS_\lambda$ are fully exact if and only if they are perfect.

 Consider the morphism $\phi_R\colon S^*\to S$ from \eqref{eq:phiR}. Using the explicit formula for $R=R_0$ from \eqref{eq:R-sVect}, we see that the image of $\phi_R$ is equal to the subalgebra $\mC C_2\subset S$ generated by $g$, and in particular $\phi_R(x)=0$.  Therefore, the functor $\Res_{\phi_R}\colon \cC\to \lmod{S^*}\simeq \cC$ does not preserve projective objects because its image consists of semisimple objects. By \Cref{lem:Vect-fully-exact}, $\cV_\infty=\Vect$ is not fully exact.

Clearly, $\cV_0=\cC$ is fully exact by \Cref{ex:fully-exact-first}(i).
For $\lambda\neq 0$, to show that  $\cV_{\lambda^{-1}}=\Vect^\lambda$ is fully exact it is enough by \Cref{rem:twisted-fully-exact} to show that $\phi_{R^{J}}$, for $J=R_\lambda$, is an isomorphism. Here $R^J = (R_\lambda)_{21}^{-1}\cdot R_0 \cdot R_\lambda=R_{2\lambda}$ by \eqref{eq:twist-R}. 
We then compute that the homomorphism $\phi_{R_{2\lambda}}$ is indeed an isomorphism.

We have shown in \Cref{ex:Sweedler-first} that $\cS_\infty\boxtimes_\cC \cS_\infty$ is not exact and hence $\cS_\infty$ is not fully exact by \Cref{def:fully-exact}. The non-semisimple $\cC$-module category $\cS_0=\cD$ is fully exact by \Cref{lem:D-separable} and \Cref{prop:rel-proj-fully-exact}.
The proof that $\cS_\lambda$ is fully exact for all $\lambda \in \mC\setminus \Set{0}$ uses a generalization of the proof for $\cV_\lambda$ in the previous paragraph, but working internally to $\sVect$. 
Details are left to \cite{GL}.
\end{proof}

For $\cC=\lmod{S}$, we can compute the relative Deligne products, duals, and functor categories for all indecomposable fully exact $\cC$-module categories.

\begin{proposition}\label{prop:Sweedler-products}
For $\lambda, \mu \in \mC$ we have the following equivalences of left $\cC$-module categories:
\begin{align*}
\cV_\lambda\boxtimes_\cC \cV_\mu &\simeq \cV_{\lambda+\mu}, &\cS_\lambda\boxtimes_\cC \cV_\mu&\simeq \cS_{\lambda+\mu}, &  \cV_\lambda^*&\simeq \cV_{-\lambda}\simeq {}^*\cV_\lambda, \\ \cV_\lambda\boxtimes_\cC \cS_\mu&\simeq \cS_{\lambda+\mu}, &\cS_\lambda\boxtimes_\cC \cS_\mu &\simeq \cV_{\lambda+\mu},& \cS_\lambda^*&\simeq \cS_{-\lambda}\simeq {}^*\cS_\lambda,
\end{align*}
\begin{align*}
\Fun_\cC(\cV_\lambda,\cV_\mu)&\simeq \cV_{\mu-\lambda},& \Fun_\cC(\cS_\lambda,\cV_\mu)&\simeq \cS_{\mu-\lambda},\\
     \Fun_\cC(\cV_\lambda,\cS_\mu)&\simeq \cS_{\mu-\lambda},&
    \Fun_\cC(\cS_\lambda,\cS_\mu)&\simeq \cV_{\mu-\lambda}.
\end{align*}
\end{proposition}

\begin{proof} 
We first describe how the categories of $\cC$-linear functors involving $\cV_\lambda=\Vect^{\lambda^{-1}}, \cV_\mu=\Vect^{\mu^{-1}}$, for $\lambda,\mu\in \mC$, are obtained. We set $H=S$, $J=R_{\lambda^{-1}}$, and $L=R_{\mu^{-1}}$, for the $2$-cocycles from \eqref{eq:Rt}.
In this case, we compute that the algebra $\Dtwist{L}{H}{J}{*}$ of \Cref{def:H-twisted-alg} is isomorphic to the algebra
\begin{align}\label{eq:H*mu}
S^*_{d}=\mC\langle h,y\rangle /(h^2=1, hy=-yh, y^2=d 1),
\end{align}
for $d=\lambda^{-1}-\mu^{-1}$. 
If  $\lambda\neq \mu$, 
then $S^*_{d}$
is semisimple as isomorphic to the matrix algebra $M_2(\mC)$. The $\Btwist{L}{H}{L}$-coaction from \eqref{eq:delta-twisted} is just an $S$-coaction as, in this case, we have $\Btwist{L}{\Delta}{L}=\Delta$ for the twisted coproduct from \eqref{eq:Delta-twist}.  We compute that the $S$-coaction $\delta\colon S^*_{d}\to S\otimes S^*_{d},$
is determined on generators by
$$\delta(y)=g\otimes y -2\mu^{-1} x\otimes 1, \qquad \delta(h)=g\otimes h.$$
By setting $\hat{g}=h$ and $\hat{x}=\frac{iyh}{\sqrt{d}}$, we can identify the algebra $S^*_{d}$ with the $S$-comodule algebra $A_{\xi}=\mC\inner{\hat{g},\hat{x}}/(\hat{g}^2=1, \hat{g}\hat{x}=-\hat{x}\hat{g}, \hat{x}^2=1)$ of \cite{Will}*{Eq.\,(5.3)},\footnote{\cite{Will} uses a different presentation of Sweedler's Hopf algebra. The isomorphism is given by identifying $x$ in our presentation with $xg$ in that of \cite{Will}.}
with $\xi=-\frac{2\mu^{-1} i}{\sqrt{d}}$. Under this identification, $\delta$ corresponds to the $S$-coaction 
$$\delta\colon A_\xi\to S\otimes A_\xi,\qquad \hat{g}\mapsto g\otimes \hat{g}, \qquad \hat{x}\mapsto 1\otimes \hat{x}+\xi xg\otimes \hat{g}. $$
By \Cref{prop:Fun-H*-mod-twisted}, we have an equivalence of left $\cC$-module categories $\Fun_\cC(\cV_\lambda,\cV_\mu)\simeq \lmod{S^*_d}$, where the $\cC$-module associator is non-trivial, given by the action of
$L=R_{\mu^{-1}}$.
The resulting $\cC$-module category $\lmod{S^*_d}$ can be identified with  $\left(\lmod{A_{\xi}}\right)^{-\mu^{-1}}$ in the notation of \cite{Will}*{Cor.\,5.19}.
Hence, using \cite{Will}*{Prop.\,5.23}, 
 computing 
 $\mu^{-1}-\frac{1}{4}\bigl(-\frac{2\mu^{-1} i}{\sqrt{\lambda^{-1}-\mu^{-1}}}\bigr)^2 
= \frac{1}{\mu-\lambda}$,
 we have an equivalence of left $\cC$-module categories 
$$\Fun_\cC(\cV_\lambda,\cV_\mu)=\Fun_\cC(\Vect^{\lambda^{-1}},\Vect^{\mu^{-1}})\simeq \Vect^{(\mu-\lambda)^{-1}}=\cV_{\mu-\lambda}.$$
If $\lambda=\mu$,  we first recall the proof of \Cref{prop:S-mod-fully-exact-class} showing that $\phi_{R^{J}}$, for $J=R_{\lambda^{-1}}$, is an isomorphism. Then by \Cref{rem:twisted-fully-exact}, the $\cC$-module $\cV_\lambda$ is invertible, and thus by \Cref{prop:Minv-AM-equiv} its centralizer functor $\cC \to \cC^*_{\cV_\lambda}$ is a monoidal equivalence inducing 
 a $\cC$-module equivalence by \Cref{lem:F-equiv}.

We can further compute that $\cV_\lambda^*={}^*\cV_\lambda=\cV_{-\lambda}$. Thus,
$$\cV_\lambda\boxtimes_\cC \cV_\mu \simeq \Fun_\cC(\cV_\lambda^*,\cV_\mu)\simeq \Fun_\cC(\cV_{-\lambda},\cV_\mu)\simeq \cV_{\lambda+\mu}.$$

For the equivalences involving $\cS_\lambda$, similar but more technical methods are used. Details for this are left to appear in \cite{GL}. 
\end{proof}

For Sweedler's Hopf algebra we make the following notable observation. 

\begin{corollary}\label{cor:Sweedler-fully-exact-invertible}
    All indecomposable fully exact $\cC$-module categories $\cV_\lambda$, $\cS_\lambda$, for $\lambda \in \mC$, are invertible.
\end{corollary}
\begin{proof}
\Cref{prop:Sweedler-products} shows that for these module categories $\cM=\cV_\lambda$ or $\cS_\lambda$, with $\lambda\in \mC$, we have $\cM^*\boxtimes_\cC \cM\simeq \cV_0 \simeq \cM\boxtimes_\cC \cM^*$. Since $\cV_0=\cC$, this shows that $\cM$ is invertible. 
\end{proof}

\begin{remark}
The group $\pi_0(\invmod{\cC})$ of all objects in $\invmod{\cC}$ up to equivalence is the Picard $1$-group of~$S$. This group is $\mC \times C_2$ as computed in \cite{Car}. By \Cref{prop:Sweedler-products} and \Cref{cor:Sweedler-fully-exact-invertible}, we recover this result and find representatives of the elements of the Picard $1$-group in terms of concrete module categories.
\end{remark}

 We finish by observing that there are only two separable module categories over $\cC$.

\begin{lemma}\label{lem:separable-Sweedler}
The only indecomposable relatively projective left $\cC$-module categories are the separable module categories $\cC=\cV_0$ and $\cD=\cS_0$. In particular, these are the only indecomposable separable $\cC$-module categories.
\end{lemma}
\begin{proof}
    Recall from \Cref{def:Ct} that $\cC_t$ denotes the braided tensor category $\cC$ with braiding defined by the R-matrix $R_t$ of \eqref{eq:Rt}. For $t,\lambda \in\mC\setminus\Set{0}$, using the techniques explained in the proof of \Cref{prop:Sweedler-products}, we obtain an equivalence of left $\cC$-module categories 
    $$\Fun_{\cC_t}(\cV_\lambda,\cV_\lambda)\simeq \lmod{S^*_0},$$
    with the notation from \eqref{eq:H*mu},
    where the $\cC$-action on $\lmod{S^*_0}$ is given by the $S$-coaction from \eqref{eq:twisted-coaction}, which corresponds to the $S$-coaction from \Cref{lem:H-com-alg} using the twisted R-matrix $R_{t+2\lambda^{-1}}=(R_{\lambda^{-1}})^{-1}_{21}\cdot R_t\cdot R_{\lambda^{-1}}$. Thus, for $\lambda=-2/t$, it follows that $\cV_\lambda$ is \emph{not} fully exact as a $\cC_t$-module category by \Cref{lem:Vect-fully-exact}. A similar argument shows that $\cS_{-2/t}$ is not fully exact as a $\cC_{t}$-module category. 
    
    However, we observe that the definition of an algebra being relatively projective as a bimodule in \Cref{def:rel-proj-bimodule} does not depend on the choice of braiding of $\cC$. Thus, if $\cV_\lambda$ was relatively projective in the sense of \Cref{def:rel-proj-cat}, then it would be fully exact by \Cref{prop:rel-proj-fully-exact} for \emph{any} choice of braiding on $\cC$. Hence, for $\lambda \in \mC\setminus\Set{0}$, the $\cC$-module categories $\cV_\lambda$ and $\cS_\lambda$ cannot be relatively projective.

    For $\lambda=0$, it is clear that $\cV_0=\cC$ is separable  and $\cS_0=\cD$ is separable by \Cref{lem:D-separable}.
\end{proof}

\begin{remark}\label{rem:vect-Sweedler}
Denote by $\cC_t$ the symmetric tensor category $\cC$ with braiding obtained from the R-matrix $R_t$ of \eqref{eq:Rt}. By \Cref{eq:Ct-equivalence}, there is an equivalence  of symmetric tensor categories $\cC\isomorph \cC_t$. We note that $\Vect$ is fully exact as a $\cC_t$-module category for all $t\neq 0$. This follows from \Cref{lem:Vect-fully-exact} as the map $\phi_{R_t}$, with $t\neq 0$, is an isomorphism. However, $\Vect$ is not fully exact as a $\cC_0$-module. Hence, full exactness of $\Vect$ depends on the choice of the braiding, even when the braidings are twist equivalent.

As shown in the proof of \Cref{lem:separable-Sweedler}, the module category $\Vect^{-t/2}$ is not fully exact as a left $\cC_t$-module.
The monoidal equivalence $I:=I_t^{-t/2}$ from \Cref{eq:Ct-equivalence}
induces a monoidal $2$-equivalence $\Res_{I}\colon \lmod{\cC}\to \lmod{\cC_t}$ via \Cref{prop:res} which preserves fully exact module categories by \Cref{lem:Res-fully-exacts}. This equivalence maps the non-fully exact $\cC$-module $\Vect$ to the non-fully exact $\cC_t$-module $\Vect^{-t/2}$. This follows as, by definition, $\Res_{I}(\Vect)$ is the $\cC_t$-module category defined on $\Vect$ with action given (as in \Cref{ex:tensor-act}) by the monoidal functor $F\circ I$  which is equal to the monoidal functor $F^{-t/2}$ from \eqref{eq:F-lambda} that defines the $\cC$-action on $\Vect^{-t/2}$.
 \end{remark}

\begin{remark}\label{rem:pivotal-mod-cat}
    We see from the example $\cC=\lmod{S}$ that the concepts of fully exact $\cC$-module categories and pivotal $\cC$-module categories \cite{ShiP} are not directly related. We see from \cite{ShiP}*{Table\,1} that the fully exact $\cV_\lambda$, for all $\lambda\in \mC$, and the non-fully exact $\Vect$ are all pivotal $\cC$-module categories. However,  the fully exact $\cS_\lambda$, for all $\lambda\in \mC$, and the non-fully exact $\sVect$ are all not pivotal $\cC$-module categories.
\end{remark}

\appendix

\section{Monoidal and braided monoidal categories}
\label{sec:basics}
We collect conventions and basic definitions concerning  monoidal categories here.

\subsection{Monoidal categories}\label{appendix:monoidal}

A monoidal category $\cC$ comes with the datum of an associator, a natural isomorphism with components 
$$\alpha_{X,Y,Z}\colon X\otimes (Y\otimes Z)\to (X\otimes Y)\otimes Z$$
satisfying the usual pentagon
and triangle
axioms, see e.g.\ \cite{EGNO}*{Prop.\ 2.2.3--2.2.4}.
However, we will assume that $\cC$ is strict and $\alpha$ consists of identity morphisms. Note that  $\cC$ can always be replaced with an equivalent strict monoidal category~\cite{ML}. 

For a category $\cM$, $\cM^\oop$ denotes the opposite category with morphism spaces 
$$\Hom_{\cM^\oop}(X,Y)=\Hom_\cM(Y,X).$$
If $\cC$ is a monoidal category, then $\cCop$ denotes the category $\cC$ with \emph{opposite} tensor product 
\begin{equation}
X\otimes^{\oop} Y= Y\otimes X.\label{eq:tensor-op}
\end{equation}

A monoidal functor $F\colon \cC\to \cD$ is a functor together with a natural isomorphism
$$\mu^F\colon F(X)\otimes F(Y) \to F(X\otimes Y)$$
which satisfies the coherence condition that the diagram 
\begin{align}
    \label{eq:mu-coherence}
    \vcenter{\hbox{\xymatrix{
    F(X)\otimes F(Y)\otimes F(Z)\ar[rr]^{\mu^F_{X,Y}\otimes F(Z)}\ar[d]_{F(X)\otimes \mu^F_{Y,Z}}&& F(X\otimes Y)\otimes F(Z)\ar[d]^{\mu^F_{X,Y\otimes Z}}\\
    F(X)\otimes F(Y\otimes Z)\ar[rr]^{\mu^F_{X,Y\otimes Z}}&& F(X)\otimes F(Y\otimes Z)
    }}}
\end{align}
commutes for any objects $X,Y,Z$ in $\cC$. We assume that all monoidal functors are strictly unital, i.e. that $F(\one_\cC)=\one_\cD$. This assumption is justified in \cite{EGNO}*{Rem.\,2.4.6}.

A \emph{monoidal natural transformation} $\eta\colon F\to G$, where $F$ and $G$ are monoidal functors, is a natural transformation satisfying that the diagram 
\begin{equation}
\xymatrix{
F(X)\otimes F(Y)\ar[d]^{\eta_X\otimes \eta_Y}\ar[rr]^{\mu^F_{X,Y}}&&F(X\otimes Y)\ar[d]^{\eta_{X\otimes Y}}\\
G(X)\otimes G(Y)\ar[rr]^{\mu^F_{X,Y}}&&G(X\otimes Y)
}    
\end{equation}
commutes for all objects $X,Y$ of $\cC$.

\subsection{Duals in monoidal categories}
\label{sec:duality}

We use the following conventions on duals in monoidal categories. A \emph{left dual} of an object $X$ of $\cC$ is an  object $X^*$ together with morphisms 
$$\ev^l_X\colon X^*\otimes X \to \one, \qquad \coev^l_X\colon \one \to X\otimes X^*$$
satisfying the usual snake identities. Similarly, a \emph{right dual} is an object ${}^*X$ together with morphisms 
$$\ev^r_X\colon X\otimes {}^*X \to \one, \qquad \coev^r_X\colon \one \to {}^*X\otimes X$$
satisfying similar identities. The monoidal category $\cC$ is \emph{rigid} if every object as a left and right dual. In this case, left duals assemble into a monoidal equivalence 
$$(-)^*\colon \cC\to \cCop.$$
Because $\cC$ is strict monoidal, we can assume that the canonical isomorphism ${}^*(X\otimes Y)={}^*Y \otimes {}^*X$ is an identity. Similarly, there is a right duality functor. There are canonical natural isomorphism
$$\eta^l_X\colon X\isomorph ({}^*X)^*, \qquad \eta^r_X\colon X\isomorph {}^*(X^*), \qquad $$
which are monoidal isomorphisms. 

The following lemma is well-known, see \cite{EGNO}*{Section~4, Section~6.1}. 
\begin{lemma}\label{lem:tensor-cat-properties}
    The following statements hold for $\cC$ an abelian rigid monoidal category:
    \begin{enumerate}[(i)]
        \item The tensor product $\otimes\colon \cC\times \cC\to \cC$ is exact in both variables;
        \item For every projective object, the left (and right) dual is projective;
        \item Projective objects form a $2$-sided tensor ideal in $\cC$.
    \end{enumerate}
\end{lemma}

\subsection{Braided monoidal categories}\label{sec:braided-mon-cat}

Throughout most of the paper, $\cC$ will be a braided monoidal category with braiding $\psi\colon \otimes \isomorph \otimes^\oop$, or, in components
$$\psi_{X,Y}\colon X\otimes Y\isomorph Y\otimes X=X\otimes^\oop Y,$$
for objects $X,Y$ in $\cC$,
satisfying the two hexagon axioms, see e.g.\ \cite{EGNO}*{Def.\ 8.1.1}.
The components of the inverse braiding are denoted by 
$$\psi^{-1}_{X,Y}=(\psi_{Y,X})^{-1}\colon X\otimes Y\isomorph Y\otimes X.$$
The inverse braiding also defines a  braiding on $\cC$, sometimes referred to as the \emph{reverse braiding}. We denote by $\cCrev$ the braided monoidal category $\cC$, with the same monoidal product but the reverse braiding $\psi^{-1}$. Note that the monoidal category $\cCop$ from \eqref{eq:tensor-op} inherits a braiding from $\cC$ defined by 
$$\psi^{\cCop}_{X,Y}:=\psi_{Y,X}\colon X\otimes^{\oop} Y\isomorph Y\otimes^{\oop} X.$$

For a braided monoidal category $\cC$, it is well-known that the identity functor $\id_\cC\colon \cC \to \cC$ has two natural structures of  monoidal equivalences
\begin{equation}\label{eq:I+-}
I^{\pm }=(\id_\cC,\mu^\pm) \colon \cCop\to \cC
\end{equation}
where the monoidal structure is given by the braiding or the inverse braiding, $\mu^\pm_{X,Y} =\psi^{\pm 1}_{X,Y}$. These functors are functors of braided monoidal categories since 
    $$\mu^{\pm}_{Y,X}\circ \psi_{X,Y}=\psi_{X,Y}^{\cCop}\circ \mu^{\pm}_{X,Y}.$$

\section{Module and bimodule categories}\label{appendix:mod-bimod}

We now review key definitions and constructions concerning (bi)module categories.

\subsection{Module categories}
\label{sec:mod-cats}

Let $\cC$ be a finite tensor category as defined in the beginning of \Cref{sec:rel-Del}. 
A \emph{left $\cC$-module category} is a finite category $\cM$ together with a functor 
$\lact\colon \cC\times \cM\to \cM,$
which is right exact in the $\cC$-component,\footnote{By rigidity of $\cC$, $\lact$ is always exact in the $\cM$ component \cite{EGNO}*{Ex.\,7.3.2} and by \cite{DSS1}*{Cor.\,2.26}, $\lact$ is also exact in the $\cC$ component.} and a natural isomorphism $l\colon \lact (\id_\cC\times \lact)\to \lact (\otimes \times \id_\cM$), called the \emph{module associator}, with components 
$$l_{X,Y,M}\colon X\lact (Y\lact M) \isomorph X\otimes Y\lact M$$
satisfying that the diagram
\begin{align}\label{eq:mod-pentagon-l}
\vcenter{\hbox{
\xymatrix{
& X\lact (Y \lact (Z\lact M))\ar[dl]_-{X\lact l_{Y,Z,M}}\ar[rd]^-{l_{X,Y,Z\lact M}}&\\
X\lact (Y \otimes  Z)\lact M)\ar[d]_-{l_{X,Y\otimes Z,M}}&&(X\otimes Y)\lact (Z \lact M)\ar[d]^{l_{X\otimes Y,Z,M}}\\
X\otimes (Y \otimes Z)\lact M \ar@{=}[rr]&&(X\otimes Y)\otimes Z \lact M
}}}
\end{align}
commutes. 
For simplicity, we assume that $\cM$ is strictly unital, i.e.\ that $\one \lact M=M$ for all $M$ in $\cM$. Indeed, by \Cref{rem:strictification}, every finite $\cC$-module category is equivalent to a strict module category which is, in particular, strictly unital. Since the monoidal category $\cC$ is also strictly unital, we have
\begin{equation}\label{eq:mod-fun-unit}
l_{X,\one,M}=\id_{X\lact M}=l_{\one,X,M}.
\end{equation}
We will often refer to finite  $\cC$-module categories simply as \emph{$\cC$-module categories} or \emph{$\cC$-modules}. The functor $\lact$ is called the \emph{action functor}. The corresponding monoidal functor 
\begin{equation}\label{eq:rho-lact}
    \rho_\cM\colon \cC\to \End(\cM), \quad X\mapsto X\lact (-)
\end{equation}
will also be referred to as action functor.

A \emph{right $\cC$-module category} or \emph{right $\cC$-module} is a finite category with a functor 
$$\ract \colon \cM\times \cC\to \cM,$$
right exact in $\cC$, 
 a natural isomorphism $r\colon \ract(\ract \times \id_\cC) \to \ract(\id_\cM \times \otimes)$ with components 
$$r_{M,X,Y}\colon (M\ract X)\ract Y\isomorph M\ract X\otimes Y$$
satisfying the analogous coherence diagram
\begin{align}\label{eq:mod-pentagon-r}
\vcenter{\hbox{
\xymatrix{
& ((M\ract X)\ract Y)\ract Z \ar[dl]_-{ r_{M,X,Y}\ract Z}\ar[rd]^-{r_{M\ract X,Y,Z}}&\\
(M\ract (X \otimes  Y))\ract Z \ar[d]_-{r_{M,X\otimes Y,Z}}&&(M\ract X)\ract (Y \otimes Z)\ar[d]^{r_{M,X,Y\otimes Z}}\\
M\ract (X\otimes Y) \otimes Z \ar@{=}[rr]&&M\ract X\otimes (Y\otimes Z)
}}},
\end{align}
and the action is again assumed to be strictly unital, i.e.\ $M\ract \one=M$ for all $M$ in $\cM$.

\begin{example}\label{ex:tensor-act}
For a right exact monoidal functor $F\colon \cC\to \cD$, we can endow $\cD$ with the structure of a left $\cC$-module category with action defined by 
$$C\lact D=F(C)\otimes D, \qquad \forall C\in \cC, D\in \cD,$$
and module associator
\begin{align}\label{eq:tensor-action-ass}
\vcenter{\hbox{
\xymatrix@R=5pt{
C\lact (C'\lact D)    \ar[rr]^{l^F_{C,C',D}}\ar@{=}[d]&&  C\otimes C'\lact D\ar@{=}[d]\\
F(C)\otimes F(C')\otimes D\ar[rr]^{\mu^F_{C,C'}\otimes D} && F(C\otimes  C')\otimes D.
}}}
\end{align}
In this case, we say that $\cC$ acts on $\cD$ via \emph{restriction along $F$}.
\end{example}

For a left $\cC$-module category, the module coherence
$l$ 
defines the coherence $(\mu^\lact_{X,Y})_M:=l_{X,Y,M}$ of the monoidal functor 
$$\lact \colon \cC\to \Fun(\cM).$$
For a right $\cC$-module category $\cN$, a coherence for the functor 
$$\ract\colon \cC^{\otimes\oop}\to \Fun(\cN)$$
can be defined by 
$(\mu^{\ract}_{Y,X})_N=r_{N,X,Y}.$

\medskip

A \emph{left $\cC$-module functor} $G\colon \cM\to \cN$ is defined by a pair $(G, s^G)$, where $G\colon \cM \to \cN$ is a ($\Bbbk$-linear) functor, $s^G\colon G(\lact)\to \lact \circ G$ is a natural isomorphism with components 
$$s^G_{X,M}\colon G(X\lact M)\to X\lact G(M),$$
that satisfying the coherences conditions of \cite{EGNO}*{Def.\,7.2.1}. That is, the diagram
\begin{align}
\label{eq:C-linear-coherence}
\vcenter{\hbox{
\xymatrix{G(X\lact (Y\lact M))\ar[d]_{G(l^\cM_{X,Y,M})}\ar[rr]^{s^G_{X,Y\lact M}}&&X\lact G(Y\lact M)\ar[d]^{X\lact s^G_{Y,M}}\\
G(X\otimes Y\lact M)\ar[dr]^{s^G_{X\otimes Y,M}}&&X\lact (Y\lact G(M))\ar[dl]_-{l^\cN_{X,Y,G(M)}}\\
&X\otimes Y\lact G(M)&
}}}
\end{align}
commutes. We will often denote the pair $(G,s^G)$ simply by $G$ when the $\cC$-module structure is clear from context. All module functors considered will be right exact.

Since we assume that $\cM$ and $\cN$ are strictly unital, it follows that $s_{\one,M}^G=\id_{G(M)}$, using the identity proved in \cite{Will}*{Prop.\,A.18} in the general case.

Right $\cC$-module functors are defined analogously. We will often refer to $\cC$-module functors simply as \emph{$\cC$-linear functors}.
The following result has been proved in \cite{DSS1}*{Sec.\,2.1}.

\begin{lemma}\label{lem:exact-vs-adjoints}
    A left $\cC$-module functor $G\colon \cM\to \cN$ has a left (or right) adjoint as a left $\cC$-module functor if and only if $G$ is left (respectively, right) exact.
\end{lemma}

A \emph{left $\cC$-module transformation} $\eta\colon G\to H$ between two left $\cC$-module functors 
$G$ and $H$ from $\cM$ to $\cN$ is a natural transformation $\eta$ satisfying 
\begin{equation}\label{eq:C-transform}
\vcenter{\hbox{
\xymatrix{
G(X\lact M) \ar[rr]^{\eta_{X\lact M}}\ar[d]^{s^G_{X,M}}&&H(X\lact M)\ar[d]^{s^H_{X,M}}\\
X\lact G(M) \ar[rr]^{X\lact \eta_{M}}&&X\lact H(M).
}}}
\end{equation}
The $2$-category of left $\cC$-module categories, right exact left $\cC$-module functors, and left $\cC$-module transformations  is denoted by $\lmod{\cC}$. For two left $\cC$-module categories $\cM$ and $\cN$ we set
$$\Fun_\cC(\cM,\cN):=\Hom_{\lmod{\cC}}(\cM,\cN).$$

Similarly, \emph{right $\cC$-module transformations} and the $2$-category $\rmod{\cC}$ are defined, and we denote 
$$\Fun(\cM,\cN)_\cC:=\Hom_{\rmod{\cC}}(\cM,\cN),$$
for right $\cC$-module categories $\cM$ and $\cN$. 

We note that the composition of two left $\cC$-linear functors $F\colon \cM\to \cN$, $G\colon \cN\to \cP$ has the $\cC$-linear structure defined by 
\begin{align}\label{eq:compose-C-lin}
\vcenter{\hbox{
\xymatrix@R=15pt{
GF(X\lact M)\ar[dr]_-{G(s^F_{X,M})}\ar[rr]^{s^{GF}_{X,M}}&&X\lact GF(M)\ .\\
&G(X\lact F(M))\ar[ru]_-{s^G_{X,F(M)}}&
}}}
\end{align}

Left $\cC^{\otimes\oop}$-module categories correspond to right $\cC$-module categories. 
Indeed, there is a $2$-equivalence
\begin{equation}\label{eq:right-C-to-left-Cop}
R\colon \rmod{\cC}\to \lmod{\cCop}, \quad (\cM,\ract, r)\mapsto (\cM,\lact^{\oop}, l^{\oop}),
\end{equation}
where we define  the left $\cCop$-action $\lact^{\oop}$ by $X\lact^{\oop} N=N\ract X$ and the left $\cCop$-module coherence  
\begin{equation} \label{eq:left-Cop-action}
l^{\cCop}_{Y,X,N}\colon Y\lact^{\oop} (X\lact^{\oop} N)=(N\ract X)\ract Y\xrightarrow{r^\cC_{N,X,Y}} N\ract X\otimes Y=Y\otimes X\lact^{\oop} N,
\end{equation}
for a given right $\cC$-action $\ract$ and right $\cC$-module coherence $r^\cC$.
On $1$-morphisms, $R$ is defined as follows. For a given right $\cC$-linear functor $(G, s^{G})\colon \cM\to \cN$, the same functor $G$ becomes a left $\cCop$-linear functor $R(\cM)\to R(\cN)$ with the left $\cCop$-linear structure
$$s^G_{X,M}\colon G(X\lact^{\oop}M)=G(M\ract X) \xrightarrow{s^G_{M,X}} G(M)\ract X=X\lact^{\oop}G(M).$$
On $2$-morphisms, $R$ acts as the identity. Note that $R$ is a strict $2$-functor in the sense that $R(G\circ H)=R(G)\circ R(H)$ for two compatible $1$-morphisms.

Restricting left $\cC$-modules along a monoidal functor gives a $2$-functor, generalizing \Cref{ex:tensor-act}.
\begin{lemma}\label{lem:restriction-2-functor}
    A  monoidal functor $F\colon \cC\to \cD$ induces a $2$-functor 
    $\Res_F\colon \lMod{\cD}\to \lMod{\cC}.$ Moreover, if $F$ preserves projective objects, the $\Res_F$ preserves exact module categories. 
\end{lemma}
\begin{proof}
    This construction is well-known but we sketch the data here for completeness. For a left $\cD$-module category $\cM$, define $\Res_F(\cM)$ to be the same category $\cM$ with left $\cC$-module action $C\lact^F M=F(C)\lact M$ and coherence $l^{\Res_F(\cM)}$ defined by the following diagram:
    \begin{align}
    \vcenter{\hbox{
        \xymatrix@R=10pt@C=45pt{
        C\lact^F (C'\lact^F M)\ar@{=}[d]\ar[rr]^{l^{\Res_F(\cM)}_{C,C',M}}&& C\otimes C'\lact^F M\ar@{=}[d]\\
        F(C)\lact (F(C')\lact M)\ar[r]^-{l^\cM_{F(C),F(C'),M}}&F(C)\otimes F(C')\lact M\ar[r]^-{\mu^F_{C,C'}\lact M}&F(C\otimes C')\lact M
        }}}
    \end{align}
For a $\cD$-linear functor $(G,s^G)\colon \cM\to \cN$, the same underlying functor $G$ defines a $\cC$-linear functor $\Res_F(G)\colon \Res_F(\cM)\to \Res_F(\cN)$ with $\cC$-linear structure $s^{\Res_F(G)}$ defined by $s^{\Res_F(G)}_{C,M}:=s^G_{F(C),M}$.
On $2$-morphisms, $\Res_F$ acts as the identity since every left $\cD$-module transformation restricts to a left $\cC$-module transformation. One checks that these assignments are indeed well-defined and $2$-functorial. 

Finally, by construction, if $F$ preserves projective objects and $\cM$ is exact, then $F(P)\lact M$ is projective for any projective object $P$ of $\cC$ and any object $M$ in $\cM$. Hence, $\Res_F(M)$ is exact.
\end{proof}

\subsection{Bimodule categories}
\label{sec:bimod}

A \emph{$\cC$-bimodule} is a simultaneous left and right $\cC$-module $\cM$ with a \emph{$\cC$-bimodule coherence}, i.e.\ a natural isomorphism 
$b\colon \ract \circ (\lact\times \id) \to \lact \circ (\id\times \ract)$ with components
\begin{equation}\label{eq:bimod-coh}
    b_{X,M,Y}\colon (X\lact M) \ract Y\to X\lact (M\ract Y)
\end{equation}
satisfying the following coherence diagrams (cf.\ \cite{EGNO}*{Def.\,7.1.7} or \cite{capucci2022actegories}*{Def.\,4.3.1})
\begin{gather}\label{eq:bimodule-coh1}
\vcenter{\hbox{
    \xymatrix{
    &(X\lact (Y\lact M))\ract Z\ar[dl]_{b_{X,Y\lact M,Z}} \ar[dr]^{l_{X,Y,M}\ract Z}&\\
    X\lact ((Y\lact M)\ract Z)\ar[d]_{X\lact b_{Y,M,Z}}&& (X\otimes Y\lact M)\ract Z \ar[d]^{b_{X\otimes Y,M,Z}}\\
    X\lact (Y\lact (M\ract Z)))\ar[rr]^{l_{X,Y,M\ract Z}}&& X\otimes Y\lact (M\ract Z).
    }
    }}
\\
\label{eq:bimodule-coh2}
\vcenter{\hbox{
    \xymatrix{
    &((X\lact M)\ract Y)\ract Z\ar[dl]_{b_{X,M,Y}\ract Z}\ar[rd]^{r_{X\lact M,Y,Z}}&\\
(X\lact (M\ract Y))\ract Z\ar[d]_{b_{X,M\ract Y,Z}}&&(X\lact M)\ract Y\otimes Z\ar[d]^{b_{X,M,Y\otimes Z}}\\
X\lact ((M\ract Y)\ract Z)\ar[rr]^{X\lact r_{M,Y,Z}}&&X\lact (M\ract Y\otimes Z),
    }
    }}
\end{gather}
and, since we assume that both the left and right $\cC$-module structures are strictly unital, 
\begin{equation}
b_{\one,M,X}=\id_{M\ract X}, \qquad b_{X,M,\one}=\id_{X\lact M}.
\end{equation}
A \emph{functor of $\cC$-bimodule categories} $G\colon \cM\to \cN$ is a (right exact $\Bbbk$-linear) functor of left and right $\cC$-module categories for the respective left and right $\cC$-module structures satisfying 
\begin{equation}\label{eq:bimodule-functor-condition}
\vcenter{\hbox{
\xymatrix{
    &G((X\lact M)\ract Y)\ar[dl]_{s^G_{X\lact M,Y}}\ar[rr]^{G(b^\cM_{X,M,Y})}&&G(X\lact (M\ract Y))\ar[dr]^{s^G_{X,M\ract Y}}&\\
    G(X\lact M)\ract Y\ar[dr]_{s^G_{X,M}\ract Y}&&&& X\lact G(M\ract Y)\ar[dl]^{X\lact s^G_{M,Y}},\\
    &(X\lact G(M))\ract Y\ar[rr]^{b^\cN_{X,G(M),Y}}&& X\lact (G(M)\ract Y)&
    }}}
\end{equation}
where $b^\cM$ and $b^\cN$ are the bimodule coherences of $\cM$, respectively, $\cN$, and $s^G_{X,M}$ and $s^G_{M,Y}$ are the components of the left, respectively, right $\cC$-linear structure of $G$.
Finally, \emph{$\cC$-bimodule transformations} are simultaneous left and right $\cC$-bimodule transformations. We thus obtain a $2$-category $\bimod{\cC}$ of $\cC$-bimodules.

\begin{example}\label{ex:tensor-bimod}
For a monoidal functor $F\colon \cC\to \cD$, we can extend the left $\cC$-module structure on $\cD$ from \Cref{ex:tensor-act} to a $\cC$-bimodule with the right action  
$$D\ract C=D\otimes F(C), \qquad \forall C\in \cC, D\in \cD,$$
and with a similar module associator to  \eqref{eq:tensor-action-ass}.
The bimodule associator is given by the associator of the monoidal category $\cD$ and therefore given by identity morphisms as we assume that $\cD$ is a strict monoidal category. There is also an analogue of \Cref{lem:restriction-2-functor} for bimodule categories. 
\end{example}

\subsection{From left modules to bimodules}\label{sec:left-to-bimod}
We now assume that $\cC$ is a braided monoidal category.

\begin{lemma}\label{lem:rS}
\begin{enumerate}[(i)]
    \item We have a $2$-functor 
$$\rS_{lr}\colon \lmod{\cC}\to \rmod{\cC}$$
which sends a left $\cC$-module $\cM$ to the same category, with right action given by 
$M\ract X:= X\lact M$
and coherence $r$ defined in \eqref{eq:right-coherence}, sends a left $\cC$-module functor $(G,s^G)$ to the same functor~$G$ with right $\cC$-linear structure given by the same component isomorphisms $s^G_{X,M}$ viewed as morphisms $G(M\ract X)\to G(M)\ract X$, and does not change the $\cC$-linear transformations.

\item Similarly, we have a $2$-functor 
$$\rS_{rl}\colon \rmod{\cC}\to \lmod{\cC}$$
which sends a left $\cC$-module $\cM$ to the same category, with left action given by 
\begin{align*}
    X\lact M:= M\ract X,
\end{align*}
with coherence $l$ defined by the components
\begin{align}\label{eq:left-coherence}
\vcenter{\hbox{
\xymatrix@R=10pt{
X\lact (Y\lact M)\ar@{=}[d]\ar[rr]^{l_{X,Y,M}}&& (X\otimes Y)\lact M\ar@{=}[d]\\
(M\ract Y)\ract X\ar[dr]^{r_{M,Y,X}}  && M\ract (X\otimes Y).\\
&M\ract (Y \otimes X)\ar[ru]^{M\ract \psi^{-1}_{Y,X}}.&
}}}
\end{align}
\item The $2$-functors defined in (i) and (ii) are mutual inverses, i.e.\ $\rS_{rl}\rS_{lr}=\id$ and $\rS_{lr}\rS_{rl}=\id$.
\end{enumerate}
\end{lemma}
\begin{proof}
It is not hard to verify that $\rS_{lr}$ extends to a $2$-functor which does not change the
$\cC$-linear
functors and $\cC$-linear natural transformations. The statements for $\rS_{rl}$ are proved the same way. By construction, $\rS_{rl}\rS_{lr}(\cM)=\cM$ with the original coherence $l_{M,X,Y}$. As both $2$-functors do not change
$\cC$-linear functors and $\cC$-linear natural transformations,
Part (iii) follows. 
\end{proof}

The following result was observed in \cite{DN2}*{Sec.\,3.3}. We give a proof for the reader's convenience.

\begin{lemma}\label{lem:S-bimod}
The $2$-functor $\rS_{lr}$ from \Cref{lem:rS} extends to a $2$-functor
$$\rB\colon \lmod{\cC}\to \bimod{\cC},$$
with same left and right module coherences and the bimodule coherence of \eqref{eq:bimod-coherence}.
\end{lemma}
\begin{proof}
The axiom \eqref{eq:bimodule-coh1} follows from commutativity of the outer diagram in 
\begin{align*}
\resizebox{1.1\textwidth}{!}{    
\xymatrix@C=3pt{
&(X\otimes Y\lact M)\ract Z\ar[rrrr]^{b_{X\otimes Y,M,Z}}\ar[rd]^{l_{Z,X\otimes Y,M}} &&&&X\otimes Y\lact (M\ract Z)\ar[ld]_{l_{X\otimes Y,Z,M}}&  \\
*+[r]{(X\lact (Y\lact M))\ract Z} \ar[dr]^{l_{Z,X,Y\lact M}}\ar[ru]^{l_{X,Y,Z}\ract Z} \ar@/_4pc/[dddrrr]_-{b_{X,Y\lact M,Z}}
&&Z\otimes X\otimes Y\lact M \ar[rr]^{\psi_{Z,X\otimes Y\lact M}}\ar[rd]_-{\psi_{Z,X}\otimes Y\lact M} &&X\otimes Y\otimes Z\lact M && *+[l]{X\lact (Y\lact (M\ract Z))}\ar[lu]_{l_{X,Y,M\ract Z}}\ar[ld]_{l_{X,Y\otimes Z,M}}.\\
    &\ar[ru]^{l_{Z\otimes X,Y,M}}Z\otimes X\lact (Y\lact M)\ar[dr]^{\psi_{Z,X}\lact (Y\lact M)}  && X\otimes Z \otimes Y\lact M\ar[ru]_-{X\otimes \psi_{Z,Y}\lact M} &&X\lact (Y\otimes Z\lact M)\ar[lu]_{l_{X,Y\otimes Z,M}}&\\
    &&X\otimes Z\lact (Y\lact M)\ar[ru]^{l_{X\otimes Z,Y,M}}&&\ar[lu]_{l_{X,Z\otimes Y,M}}X\lact (Z\otimes Y\lact M)\ar[ru]^{X\lact \psi_{Z,Y}\lact M} &&\\
    &&&X\lact ((Y\lact M)\ract Z)\ar[lu]_{l_{X,Z,Y\lact M}}\ar[ru]^{X\lact l_{Z,Y,M}}\ar@/_4pc/[rrruuu]_-{X\lact b_{Y,M,Z}}&&&
    }}
\end{align*}
In this diagram, the rectangles involving only $l$ commute by \eqref{eq:mod-pentagon-l}, the distorted squares involving $b$ commute by definition of $b$, see \eqref{eq:bimod-coherence}, the squares involving the braiding $\psi$ commute by naturality of $l$ applied to the braiding, and the central triangle commutes by the hexagon for the braiding.

We can use \eqref{eq:right-coherence} to rewrite 
$$b_{X,M,Y}=l^{-1}_{X,Y,M}r_{M,X,Y}=r_{M,Y,X}^{-1}(M\ract \psi_{X,Y}) r_{M,X,Y}.$$
Now, the other condition \eqref{eq:bimodule-coh2} follows from the first for the $\otimes$-opposite monoidal category $\cC^{\otimes\oop}$ whose left module categories correspond to right $\cC$-module categories, see \eqref{eq:right-C-to-left-Cop}.
\end{proof}

\subsection{Left and right \action functors}
\label{sec:left-right-action-fun}

For a $\cC$-bimodule category $\cM$, there are two monoidal functors 
\begin{align}\label{eq:Ar}
    A^r_\cM&\colon \cCop \to \Fun_{\cC}(\cM,\cM),  && X\mapsto (-)\ract X,\\
        A^l_\cM&\colon \cC \to \Fun(\cM,\cM)_\cC,  && X\mapsto X\lact (-).\label{eq:Al}
\end{align}
To define these functors, we need to specify the left $\cC$-linear structure of $A_\cM^r(X)$, and the right $\cC$-linear structure of $A_\cM^l(X)$, for a given object $X$ in $\cC$. These are defined using the bimodule coherence $b$ of $\cM$, see \eqref{eq:bimod-coh}, via
\begin{align}
\label{eq:C-linear-Ar}
    s^{A^r_\cM(X)}_{Y,M}\colon (Y\lact M)\ract X &\xrightarrow{b_{Y,M,X}} Y\lact (M\ract X),\\
\label{eq:C-linear-Al}
    s^{A^l_\cM(X)}_{M,Y}\colon X\lact (M\ract Y) &\xrightarrow{b_{X,M,Y}^{-1}} (X\lact M)\ract Y.
\end{align}
The monoidal structures are defined using the right (respectively, left) $\cC$-module coherences, namely,
\begin{gather}\label{eq:monoidal-structure-Ar}
\vcenter{\hbox{
\xymatrix@R=8pt{
\big(A^r_\cM(X)\circ A^r_\cM(Y)\big)(M)\ar@{=}[d]\ar[rr]^{(\mu^{A^r_\cM}_{X,Y})_M }&&A^r_\cM(Y\otimes X)(M)\ar@{=}[d]\\
(M\ract Y)\ract X\ar[rr]^{r_{M,Y,X}}&& M\ract Y\otimes X
}}},\\
\label{eq:monoidal-structure-Al}
\vcenter{\hbox{
\xymatrix@R=8pt{
    \big(A^l_\cM(X)\circ A^l_\cM(Y)\big)(M)\ar@{=}[d] \ar[rr]^{(\mu^{A^l_\cM}_{X,Y})_M }&&A^l_\cM(X\otimes Y)(M)\ar@{=}[d]\\
    X\lact (Y\lact M)\ar[rr]^{l_{X,Y,M}}&& X\otimes Y\lact M.
    }}}
\end{gather}
The coherence axioms of these monoidal functor structures follow from those of the left (respectively, right) $\cC$-module structures.

\smallskip

We will now use the equivalence of categories $I^\pm\colon \cCop\to \cC$  from \eqref{eq:I+-} to compare the monoidal functor $A_\cM$ associated to a left $\cC$-module category (see \Cref{def:AC-C-mod-functor}) with the monoidal functors $A_{\rB(\cM)}^r$ and $A_{\rB(\cM)}^l$ associated to the bimodule $\rB(\cM)$ from \Cref{lem:S-bimod} (see \eqref{eq:Ar} and \eqref{eq:Al}).

\begin{lemma}\label{lem:AM-AMr-compare}
    The monoidal functors $A_\cM\circ I^-$ and $A_{\rB(\cM)}^r$ are equal.
\end{lemma}
\begin{proof}
We first compare the functors. For an object $X$ in $\cC$, the $\Bbbk$-linear functors $\id_\cM\ract X$ and $X\lact \id_\cM$ are equal as the right action is defined in terms of the left action. The left $\cC$-linear structures are equal since
\begin{align}\label{eq:sAr-BM}
     s^{A^r_\cM(X)}_{Y,M}=b_{Y,M,X}=l^{-1}_{Y,X,M}(\psi_{X,Y}\lact M)l_{X,Y,M}=s^{X\lact \id_\cM}_{Y,M},
\end{align}
where the first equality is \eqref{eq:C-linear-Ar}, the second equality is \eqref{eq:bimod-coherence}, and the last equality is \eqref{eq:FunC-coherence}. As functors, we have $A_{\rB(\cM)}^r\circ I^-=A_\cM^r$, since the underlying functor of $I^-$ is $\id_\cC$. 

Next, we use the definitions of monoidal structures and compute
\begin{align*}
(\mu^{A^r_{\rB(\cM)}\circ I^-}_{X,Y})_M&=(A_{\rB(\cM)}^r(\mu^+_{X,Y}) )_M\circ  (\mu^{A^r_\cM}_{I^-(X),I^-(Y)})_M\\
&=(\psi^{-1}_{X,Y}\lact M)\circ r_{M,Y,X}\\
&=(\psi^{-1}_{X,Y}\lact M)(\psi_{X,Y}\lact M)l_{X,Y,M}\\
&=l_{X,Y,M}=\mu^{A_\cM}_{X,Y},
\end{align*}
using \eqref{eq:right-coherence} in the third equality. Thus, the monoidal structures are equal.
\end{proof}

We now turn attention to the left \action functor $A_{\rB(\cM)}^l$. Recall that $A_\cM^\rev$ denotes the \action functor for $\cM$ viewed as a left $\cC^\rev$-module category.
In the following, we will denote the functor obtained by evaluating a $2$-functor, such as $\rS_{lr}$,
 on a Hom category by the same symbol.
 
\begin{lemma}\label{lem:Al-identified}
The following diagram of monoidal functors commutes
$$
\xymatrix@C=2pt{
\cC\ar[rr]^{A^\rev_\cM}\ar[d]_-{A_{\rB(\cM)}^l}&& \cC^*_\cM\ar[d]^-{\;\;\;\rS_{lr}}\\
 \Fun(\rB(\cM),\rB(\cM))_\cC\ar@{=}[rr]&& \Fun(\rS_{lr}(\cM),\rS_{lr}(\cM))_\cC,
}$$
where $\rS_{lr}$ is the evaluation of the $2$-functor $\rS_{lr}\colon \lmod{\cC}\to \rmod{\cC}$ from \Cref{lem:rS} on the Hom category $\cC^*_\cM=\Fun_\cC(\cM,\cM)$, which is an equivalence.
\end{lemma}
\begin{proof}
We first compare the underlying linear functors $A_{\rB(\cM)}^l$ and $\rS_{lr}\circ  A^\rev_\cM$. To an object $X$ in $\cC$, both functors associate the functor $X\lact \id_\cM$. Recall from \Cref{lem:rS} that the $2$-functor $\rS_{lr}$ is invertible and given by the identity assignments on the level of $1$- and $2$-morphisms. That is, given a left $\cC$-module functor $(G,s^l)\colon \cM\to \cM$, the same data defines  a \emph{right} $\cC$-module functor $(G,s^r)\colon \rS_{lr}(\cM)\to \rS_{lr}(\cM)$, where
$$\xymatrix@R=5pt{
G(M\ract X)\ar@{=}[d]\ar@/^1pc/[rr]_{\overset{s^r_{M,X}}{\rotatebox{90}{=~}}}&&G(M)\ract X\ar@{=}[d]\\
G(X\lact M)\ar@/_1pc/[rr]^{s^l_{X,M}}&&X\lact G(M).
}$$
Thus, we see that the left $\cC$-linear structure of $s^{X\lact \id}$ of $X \lact \id_\cM$ under the functor $\rS_{lr}\circ  A^\rev_\cM$ is given by
\begin{align}\label{eq:s-Al-identified}
s^{X\lact \id}_{Y,M}\overset{\text{\eqref{eq:FunC-coherence}}}{=}l_{Y,X,M}^{-1}(\psi^{-1}_{X,Y}\lact M)l_{X,Y,M}\overset{\text{\eqref{eq:bimod-coherence}}}{=}b_{X,M,Y}^{-1}
\overset{\text{\eqref{eq:C-linear-Ar}}}{=} s^{A^l_\cM(X)}_{M,Y}.
\end{align}
Hence, the left $\cC$-linear structures  of $A_{\rB(\cM)}^l(X)$ and $\rS_{lr}\circ  A^\rev_\cM(X)$ coincide. 

It remains to compare the monoidal structures. The $2$-functor $\rS_{lr}$ strictly preserves the monoidal structure of composition (the composition of $1$-morphisms). Further, we have 
$$(\mu^{A^l_\cM}_{X,Y})_M
\overset{\text{\eqref{eq:monoidal-structure-Al}}}{=}
l_{X,Y,M}=\mu^{A_\cM}_{X,Y}=\mu^{A_\cM^\rev}_{X,Y},$$
where the second equality uses \Cref{lem:Hom-C-mod}, and we observe that $A_\cM^\rev$ and $A_\cM$ have the same monoidal structures. 
Thus, the functors are equal as monoidal functors and the claimed diagram commutes. 
\end{proof}

\subsection{Double centralizer functor}
For later use, we record the following proposition relating $A_\cM$ and $A_\cM^\rev$. For this, we recall the \textit{double centralizer} 
functor\footnote{The functor $\mD$ is part of the double centralizer theorem of \cite{EGNO}*{Thm.\,7.12.11} for exact module categories, which shows when $\mD$ is an equivalence.} 
\begin{equation}\label{eq:can-double-dual}
    \mD=\mD_\cM\colon \cC \to \Fun_{\cC^*_\cM}(\cM,\cM).
\end{equation}
Here, we use the left $\cC^*_\cM$-action on $\cM$ given by $G\lact M=G(M)$, for objects $G=(G,s^G)$ in $\cC^*_\cM$ and $M$ in $\cM$. For an object $X$ in $\cC$, $\mD(X)=(X\lact \id_\cM, s^{\mD(X)})$, where
\begin{equation}\label{eq:sDX}
\vcenter{\hbox{
\xymatrix@R=8pt{
\mD(X)(G\lact M)\ar@{=}[d]\ar[rr]^{s^{\mD(X)}_{G, M}}&& G\lact (\mD(X)(M))\ar@{=}[d]\\
X\lact G(M)\ar[rr]^{(s^G_{X,M})^{-1}}&& G(X\lact M)
}}}
\end{equation}
The monoidal structure of $\mD$ is given by
\begin{align}\label{eq:monoidal-structure-mD}
\vcenter{
\hbox{
   \xymatrix@R=8pt{
   (\mD(X)\circ \mD(Y))(M)\ar@{=}[d]\ar[rr]^{(\mu^{\mD}_{X,Y})_M}&&\mD(X\otimes Y)(M)\ar@{=}[d]\\
   X\lact (Y\lact M)\ar[rr]^{l_{X,Y,M}}&& X\otimes Y\lact M
   } }}
\end{align}
when evaluated on an object $M$ in $\cM$.
Further recall the $2$-functor $\Res_F$ for a monoidal functor $F$
defined in \Cref{lem:restriction-2-functor}.

\begin{proposition}\label{prop:centralizer-correspondence}
   For $\cM$ a left $\cC$-module category, the diagram 
\begin{align*}
    \xymatrix{
\cC\ar[rr]^{\mD_\cM}\ar[d]_-{A_{\rB(\cM)}^l}&&\Fun_{\cC^*_\cM}(\cM,\cM)\ar[d]^-{\Res_{A_{\rB(\cM)}^r}~}\\
\Fun(\rB(\cM),\rB(\cM))_\cC\ar[rr]^{R}&&\Fun_{\cCop}(\rB(\cM),\rB(\cM))
    }
\end{align*}
of monoidal functors commutes. Here, the functor $R$ is the equivalence obtained by evaluating
the $2$-equivalence from \eqref{eq:right-C-to-left-Cop} on the Hom category $\Fun(\rB(\cM),\rB(\cM))_\cC$.
\end{proposition}
\begin{proof}
    We first note that $\Res_{A_{\rB(\cM)}^r}$ preserves finite module categories and right exact functors. As $\Res_{A_{\rB(\cM)}^r}$ is even a $2$-functor, it is strict monoidal when restricted to a functor on endofunctor categories. Thus, every functor in the diagram is a monoidal functor. 

To check that the diagram of functors commutes, let $X$ be an object in $\cC$. We will identify the left $\cCop$-linear structure obtained by restricting the $\cC^*_\cM$-linear structure $s^*$ of $X\lact \id_\cM$ along the monoidal functor $A_{\rB(\cM)}^r$. This results in the $\cCop$-linear structure 
$$X\lact (Y\lact M)\xrightarrow{s^*_{Y\lact \id, M}=(s^{Y\lact \id}_{X,M})^{-1}}Y\lact (X\lact M),$$
using the $\cC$-linear structure $s^{Y\lact \id}_{X,M}$ from \Cref{def:Fun-action}. We derive that 
\begin{align*}
    (s^{Y\lact \id}_{X,M})^{-1}&= l_{Y,X,G(M)}^{-1}(\psi^{-1}_{X,Y}\lact \id)l_{X,Y,G(M)}=s^{A_{\rB(\cM)}^l(X)}_{Y,M},
\end{align*}
which is precisely the \emph{right} $\cC$-linear structure $s^{A_{\rB(\cM)}^l(X)}_{M,Y}$ as computed in \eqref{eq:s-Al-identified}, considered here as a \emph{left} $\cCop$-linear structure. As on morphisms, none of the functors change the data of natural transformations, this shows that the diagram of functors commutes. 

It remains to compare the monoidal structures. The functor $\Res_{A^r_{\rB(\cM)}}$ is strict monoidal and the monoidal structure of the composition $\Res_{A^r_{\rB(\cM)}}\!\!\!\circ\mD_\cM$ is given by \eqref{eq:monoidal-structure-mD}. This is equal to the monoidal structure of $A^l_{\rB(\cM)}$ from \eqref{eq:monoidal-structure-Al}. The  equivalence $R$ is also strict monoidal as it comes from restricting a $2$-functor to a $1$-morphism category. Thus, we have a commutative diagram of monoidal functors. 
\end{proof}

\section{Monoidal \texorpdfstring{$2$}{2}-categories}\label{appendix:mod-2-cat}

We briefly review the definition of a monoidal $2$-category before including some basic definitions and results on duals in monoidal $2$-categories in \Cref{sec:2duals}. 
\Cref{sec:rel-Del-props} then reviews the construction of the monoidal $2$-category structure of $\bimod{\cC}$, with relative Deligne product, for a finite tensor category $\cC$. Finally, in \Cref{sec:2-mon-Res} we show that equivalences of the braided tensor categories induce equivalences of monoidal $2$-categories of module categories.

By a \emph{monoidal $2$-category} we mean a monoidal bicategory \cites{GPS,SPThe,JY21} which is a $2$-category, i.e.\,composition of $1$-morphisms is strictly associative. 
Such a monoidal $2$-category $\bfC$ comes equipped with the following data. 
\begin{itemize}
    \item The tensor product $\Box\colon \bfC\times \bfC \to \bfC$ is a \emph{pseudofunctor} \cite{JY21}*{Sec.\,4.1}. In particular, $\Box$ is not strictly compatible with composition of $1$-morphisms but up to natural $2$-isomorphisms, called \emph{interchange isomorphisms}
\begin{equation}
\label{eq:interchange-general}
\phi_{H\Box H',G\Box G'}\colon (H\Box H')\circ (G\Box G')\isomorph  HG\Box H'G',
\end{equation}
for $1$-morphisms $\cM\xrightarrow{G}\cN\xrightarrow{H}\cP$ and $\cM'\xrightarrow{G'} \cN'\xrightarrow{H'}\cP'$.
    \item A monoidal unit, i.e.\,an object $\bfOne$ in $\bfC$.
    \item \emph{Associator} equivalences
    \begin{equation}\label{eq:2associator-general} 
    A_{\cM,\cN,\cP}\colon \cM\Box (\cN\Box \cP)\isomorph (\cM\Box \,\cN)\Box \cP,
    \end{equation}
    that are natural up to a coherent system of $2$-isomorphisms
\begin{equation}\label{eq:2associator-natural}
    ((G\Box H)\Box K))\circ A_{\cM,\cN,\cP} \xrightarrow{a_{G,H,K}} A_{\cM',\cN',\cP'}\circ (G\Box (H\Box K)),
\end{equation}
for  $1$-morphisms $\cM\xrightarrow{G} \cM'$, $\cN\xrightarrow{H} \cN'$, and $\cP\xrightarrow{K} \cP'$.
    \item \emph{Left} and \emph{right unitor} equivalences 
    \begin{equation}\label{eq:2unitors-general} 
    L_\cM\colon \bfOne\Box \cM\isomorph \cM, \qquad R_\cM\colon \cM\Box \bfOne\isomorph\cM,
    \end{equation}
    that are natural up to coherent $2$-isomorphisms
\begin{equation}\label{eq:2unitors-natural} 
L_{\cM'}\circ (\bfOne\Box G)\xrightarrow{l_{\cM}} G \circ L_{\cM}, \qquad R_{\cM'}\circ ( G \Box\bfOne)\xrightarrow{r_{\cM}} G \circ R_{\cM},
\end{equation}
for any $1$-morphism $\cM\xrightarrow{G} \cM'$.
\item 
\emph{Pentagonator} $2$-isomorphisms
\begin{equation}\label{eq:pentagonator-general}
    (A_{\cM,\cN,\cP}\Box \cQ)\circ A_{\cM,\cN\Box \cP,\cQ}\circ ( \cM\Box A_{\cN,\cP,\cQ})\xrightarrow{p_{\cM,\cN,\cP,\cQ}}
    A_{\cM\Box \cN,\cP,\cQ}\circ A_{\cM,\cN,\cP\Box \cQ},
\end{equation}
for any four objects $\cM,\cN,\cP,\cQ$ in $\bfC$.
\item \emph{Triangulator} $2$-isomorphisms
\begin{gather}\label{eq:triangulator-middle}
    (R_\cM\Box \cN)\circ A_{\cM,\one,\cN}\xrightarrow{t^m_{\cM,\cN}} \cM\Box L_\cN, \\
    \label{eq:triangulator-left-right}
    R_{\cM\Box \cN}\circ A_{\cM,\cN,\one}\xrightarrow{t^r_{\cM,\cN}} \cM\Box R_\cN, \qquad   (L_\cM\Box \cN) \circ A_{\one,\cM,\cN}\xrightarrow{t^l_{\cM,\cN}}  L_{\cM\Box \cN},
\end{gather} for any two objects $\cM,\cN$ in $\bfC$.
\end{itemize}

The above data satisfies coherences\footnote{Found, for example, in \cite{JY21}*{Sec.\,12.1}.} which we do not specify here in detail, see also \Cref{rem:general-coherence}.

\begin{remark}\label{rem:2-strictification}
It is often convenient to work with \emph{semi-strict} monoidal $2$-categories, see e.g. \cite{BN}*{Lem.\,4} or \cite{DR}*{Def.\,2.1.1}. A semi-strict monoidal $2$-category has a strictly associative and unital tensor product $\Box$ so that the $1$-morphisms $A_{\cM,\cN,\cP},L_\cM,R_\cM$ of \eqref{eq:2associator-general} and \eqref{eq:2unitors-general} are identities and the associated natural $2$-isomorphisms $a,l,r$, as well as the pentagonator and triangulators $p, t^m,t^r,t^l$  defined in \eqref{eq:2associator-natural}, \eqref{eq:2unitors-natural}--\eqref{eq:triangulator-left-right} have identity components. However, the interchange isomorphisms $\phi$ of \eqref{eq:interchange-general} may still be non-trivial.
Using the strictification theorem for tricategories \cite{GPS}*{Thm.\,8.1} in the one-object case, any monoidal $2$-category is equivalent to a semi-strict monoidal $2$-category. 
\end{remark}

We note that the monoidal $2$-category $\lmod{\cC}$ is not semi-strict but equivalent to a semi-strict monoidal $2$-category.

\subsection{Duals and inverses in monoidal 2-categories}
\label{sec:2duals}
In this section, we let $\bfC$ denote  a monoidal $2$-category. We define duals and inverses in $\bfC$ and prove some elementary lemmas that we were unable to find in the literature. In the semi-strict case, the following definition appears in \cite{DR}*{Def.\,2.1.5}.

\begin{definition}\label{def:left-dual-cat}
 A \emph{left dual} of an object $\cM$ in $\bfC$ is an object $\cM^*$ together with $1$-morphisms $\ev^l=\ev^l_\cM\colon \cM^*\Box\cM\to \bfOne$  and $\coev^l=\coev^l_\cM\colon \bfOne\to \cM\Box\cM^*$, called \emph{left evaluation} and \emph{left coevaluation}, and $2$-isomorphisms 
\begin{gather*}
\alpha\colon R_\cM\circ (\cM\Box \ev^l)\circ A^{-1}_{\cM,\cM^*,\cM}\circ (\coev^l\Box \cM)\circ L^{-1}_\cM\isomorph\id_\cM, \\
\beta\colon \id_{\cM^*}\isomorph L_\cM\circ (\ev^l\Box \cM)\circ A_{\cM^*,\cM,\cM^*}\circ (\cM^*\Box \coev^l)\circ R_\cM^{-1}.
\end{gather*}
Right duals are defined similarly by changing the tensor product order. We denote the right dual by ${}^*\cM$ and the right evaluation and coevaluation $1$-morphisms by
$\ev^r=\ev^r_\cM\colon \cM\Box{}^{*}\cM\to \bfOne$  and $\coev^r=\coev^r_\cM\colon \bfOne\to {}^*\cM\Box\cM$. An object with a left dual and right dual is called \emph{dualizable}.
\end{definition}

For objects $\cM$ and $\cN$ in $\bfC$ we write $\cM\simeq \cN$ if the objects are equivalent, see e.g.\,\cite{JY21}*{Sec.\,6.2}.

\begin{lemma}\label{lem:dual-unique}
If $\cM^*$ and $\cM'$ are both left  duals for $\cM$, then $\cM'\simeq \cM^*$.
\end{lemma}
\begin{proof}
The proof adapts the well-known argument for duals in monoidal categories replacing identities by $2$-isomorphisms. To streamline the proof, we show that two left dual objects $\cM'$ and $\cM^*$ are equivalent in a \emph{semi-strict} monoidal $2$-category where $A,L,R$ are trivial. By \Cref{rem:2-strictification}, the statement then also holds in $\bfC$. 
    We denote $\ev=\ev^l$, $\coev=\coev^l$ in this proof and write $\ev'$ and $\coev'$ for the left evaluation and left coevaluation $1$-morphisms of the other dual object $\cM'$. We show that the $1$-morphisms
    \begin{gather*}
   H_1:=(\ev\Box \cM')\circ (\cM^*\Box \coev')\colon \cM^*\to \cM',\\
   H_2:=(\ev'\Box \cM^*)\circ (\cM'\Box \coev)\colon \cM'\to \cM^*   
    \end{gather*}
define an equivalence in $\bfC$. We have a chain of $2$-isomorphisms
\begin{align*}
H_1\circ H_2&=H_1\circ (\ev'\Box \cM^*)\circ (\cM'\Box \coev)\\
&\cong (\ev'\Box \cM^*)\circ (\cM'\Box \cM \Box H_1)\circ (\cM'\Box \coev)\\
&= (\ev'\Box \cM^*)\circ (\cM'\Box \cM \Box \ev\Box \cM')\circ (\cM'\Box \cM \Box\cM^*\Box \coev')\circ (\cM'\Box \coev)\\
&\cong (\ev'\Box \cM^*)\circ (\cM'\Box \cM \Box \ev\Box \cM')\circ (\cM'\Box \coev\Box \cM\Box\cM')\circ (\cM'\Box \coev')\\
&\cong (\ev'\Box \cM^*)\circ (\cM'\Box \coev')\\
&\cong \id_{\cM'}.
\end{align*}
 Here, we first apply the interchange isomorphisms of \eqref{eq:interchange-general} twice before applying the $2$-isomorphism of the shape $\beta^{-1}$ in \Cref{def:left-dual-cat} for the dual $\cM^*$ and then the analogous $2$-isomorphism for $\cM'$. 
 The isomorphism $H_2H_1\cong \id_{\cM^*}$ is proved similarly. 
\end{proof}

\begin{definition}\label{def:inv}
We say that an object $\cM$ in $\bfC$ is \emph{invertible} if there exists an object $\cN$, called an \emph{inverse}, and equivalences $\cM\Box \cN\simeq \bfOne$ and $\cN\Box \cM \simeq \bfOne$. 
\end{definition}

Inverses are unique up to equivalence if they exist. The next lemma characterizes when objects with a left dual are invertible. 

\begin{lemma}\label{lem:duals-inv}
Let $\cM$ be an object of $\bfC$ and $\cM^*$ a left dual of $\cM$. Then $\cM$ is invertible if and only if the $1$-morphisms $\ev$ and $\coev$ from \Cref{def:left-dual-cat} are equivalences. In particular, $\cM^*\simeq {}^*\cM$ is the inverse of $\cM$.
\end{lemma}
\begin{proof}
It is clear that if $\ev$ and $\coev$ are equivalences, then  $\cN=\cM^*$ satisfies the conditions of \Cref{def:inv}. To prove the converse, fix equivalences 
$\adj{\phi}{\cM\Box \cN}{\bfOne}{\phi'}$ and $\adj{\eta}{ \cN\Box \cM}{\bfOne}{\eta'}.$

We want to show that $\cN$ is a left dual of $\cM$.
Using \Cref{rem:2-strictification}, it suffices to prove this for semi-strict $\bfC$ so we can omit the  associators and unitors from the notation.

First, note that the $1$-morphism 
$$\gamma:=(\eta\Box \cN)\circ(\cN\Box \phi')\colon \cN\to \cN$$
is an equivalence and we choose a inverse equivalence $\gamma'$. 
Now we define 
\begin{gather*}
\ev:= \eta\circ (\gamma'\Box \cM)\colon \cN\Box \cM\to \bfOne,\qquad 
\coev:= \phi'\colon \bfOne\to \cM\Box \cN.
\end{gather*}
By construction, both $\ev$ and $\coev$ are equivalences. 
We have to check that the triple $(\cN,\ev,\coev)$ defines a left dual for $\cM$ in the sense of \Cref{def:left-dual-cat}. Indeed, we have
\begin{align*}
(\ev\Box \cN)\circ (\cN\Box \coev) &= (\eta\Box\cN) \circ (\gamma'\Box \cM\Box \cN) \circ (\cN\Box \phi')\\
&\simeq (\eta\Box\cN) \circ  (\cN\Box \phi')\circ \gamma'=\gamma \circ \gamma'\simeq \id_\cN,
\end{align*}
where the first isomorphism is an instance of the interchange isomorphism from \eqref{eq:interchange-Cmod}. This produces the isomorphism $\alpha$ required in \Cref{def:left-dual-cat}.
Moreover, we consider
\begin{align*}
(\cM\Box \ev\Box \cN)\circ
( \coev \Box\cM\boxtimes \cN)\circ \phi'
 &=(\cM\Box \eta\Box \cN)\circ (\cM\Box\gamma'\Box \cM\Box \cN)\circ (\phi'\Box\cM\Box \cN)\circ \phi'\\
&\simeq (\cM\Box \eta \Box \cN)\circ (\cM\Box \cN\Box\phi')\circ (\cM\Box \gamma')\circ \phi' \\
&=(\cM\Box \gamma)\circ (\cM\Box \gamma')\circ \phi'\simeq \phi',
\end{align*}
where both isomorphisms $\simeq$ are interchange isomorphism. Now, pre-composing with $\phi$ produces the second isomorphism $\beta$ proving that $\cN$ is a left dual in the sense of \Cref{def:left-dual-cat}.  Reversing the roles of $\cM$ and $\cN$, we can also show that $\cM$ is a left dual of $\cN$. But this is equivalent to $\cN$ being a right dual for $\cM$.
\end{proof}

\subsection{Relative Deligne Products and monoidal 2-categories}\label{sec:rel-Del-props}

This section contains a sketch of the proof that $\bimod{\cC}$ is a monoidal $2$-category. This is used in \Cref{thm:CMod-monoidal-2-cat}, when $\cC$ is braided, to show that $\lmod{\cC}$ is a monoidal $2$-category. 
In this section, we assume that all categories are $\Bbbk$-linear and finite abelian as in \Cref{sec:rel-Del}. Recall the relative Deligne product $\cN\boxtimes_\cC\cM$ and its universal property from \Cref{def:rel-Del}.  For $\cC$-bimodules $\cN,\cM$, it has the structure of a $\cC$-bimodule by \Cref{prop:rel-Del-bimod}.

Recall the canonical balanced functor $P_{\cM,\cN}\colon \cM\times \cN\to \cM\boxtimes_\cC \cN$ from \Cref{def:rel-Del}.  When $\cM$ and $\cN$ are $\cC$-bimodules, the functor $P_{\cM,\cN}$ and its $\cC$-balancing are compatible with the left and right $\cC$-actions. 
That is, there are natural isomorphisms
$$P_{\cM,\cN}(X\lact M,N)\xrightarrow{s^P_{X,M,N}} X\lact P_{\cM,\cN}( M,N), \qquad P_{\cM,\cN}( M,N\ract X)\xrightarrow{s^P_{M,N,X}} P_{\cM,\cN}( M,N)\ract X,$$
for any objects $M$ in $\cM$, $N$ in $\cN$, and $X$ in $\cC$ satisfying the coherences that $P_{\cM,\cN}(-,N)$ is a functor of left $\cC$-modules for any $N$ and that $P_{\cM,\cN}(M,-)$ is a functor of right $\cC$-modules for any $M$. 
Moreover, for every fixed $Y$ in $\cC$, the universal $\cC$-balancing  defines a left $\cC$-module transformation 
$$\beta^P_{-,Y,N}\colon P_{\cM,\cN}((-)\ract Y,N)\isomorph P_{\cM,\cN}(-,Y\lact N),$$ for every $N$ in $\cN$ and a right $\cC$-module transformation 
$$\beta^P_{M,Y,-}\colon P_{\cM,\cN}(M\ract Y,-)\isomorph P_{\cM,\cN}(M,Y\lact (-))),$$ for every $M$ in $\cM$.

Given two $\cC$-bimodule functors $G\colon \cM\to \cN$, $G'\colon \cM'\to \cN'$,  one checks that the functor $P_{\cM',\cN'}\circ (G\times G')$ is $\cC$-balanced. Hence, due to compatibility of $P_{\cM,\cN}$ with the $\cC$-bimodule structures, there exist an induced $\cC$-bimodule functor $G\boxtimes_\cC G'$ making the diagram 
$$
\xymatrix@R=15pt{
\cM\times\cN \ar[rr]^{G\times G'}\ar[d]_{P_{\cM,\cN}}&&\cM'\times \cN'\ar[d]^{P_{\cM',\cN'}}\\
\cM\boxtimes_\cC\cN \ar[rr]^{G\boxtimes_\cC G'}&&\cM'\boxtimes_\cC \cN'
}
$$
commute up to a $\cC$-balanced natural isomorphism.  
The resulting functor $G\boxtimes_\cC G'$ is unique up to canonical isomorphism  as explained in \Cref{rem:factorization-canonical-iso}. 
This uniqueness up to canonical isomorphism implies that, for additional $\cC$-bimodule functors $H\colon \cN\to \cP$ and $H'\colon \cN'\to \cP'$, there is a canonical isomorphism 
\begin{equation}\label{eq:interchange-Cmod} 
(H\boxtimes_\cC H')\circ (G\boxtimes_\cC G')\xrightarrow{\phi_{H\boxtimes_\cC H',G\boxtimes_\cC G'}}  (H\circ G)\boxtimes_\cC (H'\circ G'),
\end{equation}
which we refer to as an \emph{interchange isomorphism}, cf.\,\eqref{eq:interchange-C-mod}.
The canonical isomorphisms of the form $\phi_{(H,H'),(G,G')}$ make $\boxtimes_\cC\colon \bimod{\cC}\times \bimod{\cC}\to \bimod{\cC}$ a pseudofunctor in the sense of \cite{JY21}*{Sec.\,4.1}. In particular, for $\cC$-bimodule transformations $\theta\colon G\to H$, $\theta'\colon G'\to H'$, there are unique induced $\cC$-bimodule transformations 
$$
\theta\boxtimes_\cC \theta'\colon G\boxtimes_\cC G'\to H\boxtimes_\cC H'.
$$
The relative Deligne product of $\cC$-bimodule transformation  is strictly functorial with respect to vertical composition of natural transformations. That is, given additional $\cC$-bimodule transformations $\tau\colon H\to K$, $\tau'\colon H'\to K'$, 
$$(\tau\boxtimes_\cC \tau')(\theta\boxtimes_\cC \theta')=(\tau\theta) \boxtimes_\cC(\tau'\theta')\colon G\boxtimes_\cC G'\to K\boxtimes_\cC K'.$$
\smallskip

The relative Deligne product is not strictly associative. To construct associators, we extend the universal property of $\cN\boxtimes_\cC\cM$ to relative Deligne products of three $\cC$-bimodules. By composition, the equivalences from  \eqref{eq:universal-Del} induce two equivalences of categories
\begin{align}
\label{eq:multibalanced-universal}
\vcenter{\hbox{
\xymatrix@R=20pt@C=15pt{
*+[r]{\Fun(\cM\boxtimes_\cC (\cN\boxtimes_\cC \cP), \cQ)}\ar@<3ex>[d]^-{(-)\circ P_{\cM,\cN\boxtimes_\cC\cP}}
&&&&
*+[l]{\Fun((\cM\boxtimes_\cC \cN)\boxtimes_\cC \cP, \cQ)}\ar@<-3ex>[d]_-{(-)\circ P_{\cM\boxtimes_\cC\cN,\cP}}
\\
*+[r]{\Fun_\cC^{\mathrm{bal}}(\cM\times (\cN\boxtimes_\cC \cP), \cQ)}\ar@/_1pc/[drr]_-{(-)\circ (\cM\times P_{\cN,\cP})}
&&&&*+[l]{\Fun_\cC^{\mathrm{bal}}((\cM\boxtimes_\cC\cN)\times \cP), \cQ)}
\ar@/^1pc/[dll]^-{(-)\circ (P_{\cM,\cN}\times \cP)}
\\
&&\Fun^{\op{mulbal}}_\cC(\cM\times \cN \times \cP,\cQ)&&
}}}
\end{align}
Here,  $\Fun^{\op{mulbal}}(\cN\times \cM\times \cP,\cQ)$ denotes the category of $\cC$-\emph{multibalanced functors} and $\cC$-\emph{multibalanced natural transformations}, cf.\,\cite{GreThe}*{Def.\,2.1.3}. Here, a $\cC$-multibalanced functor consists of the data $(B,\beta^1,\beta^2)$, where $B\colon \cM\times \cN\times \cP\to \cQ$ is a functor that is right exact in each component, together with two natural isomorphisms
\begin{gather*}
    B(M\ract X,N,P)\xrightarrow{\beta^1_{M,X,N,P}} B(M,X\lact N,P),\qquad
    B(M,N\ract X,P)\xrightarrow{\beta^2_{M,N,X,P}} B(M,N,X\lact P),
\end{gather*}
such that for a fixed object $P$ in $\cP$, $\beta^1_{-,-,-,P}$ is a $\cC$-balancing for the functor $B(-,-,-,P)$ and for a fixed object $M$ in $\cM$, $\beta^2_{M,-,-,-}$ is a $\cC$-balancing for $B(M,-,-,-)$, and that the diagram 
\begin{equation}
\vcenter{\hbox{
\xymatrix@R=20pt@C=52pt{
B(M\ract X,N\ract Y,P)\ar[d]_{\beta^1_{M,X,N\ract Y,P}}\ar[rr]^{\beta^2_{M\ract X,N,Y,P}} && B(M\ract X,N,Y\lact P)\ar[d]^{\beta^1_{M,X,M,Y\lact P}}\\
B(M,X\lact(N\ract Y),P)\ar[r]^{(\id_M,b^{-1}_{X,N,Y},\id_P)}&B(M,(X\lact N)\ract Y,P)\ar[r]^{\beta^2_{M,X,Y\lact N}}&B(M,X\lact N,Y\lact P).
}}}
\end{equation}
commutes. A $\cC$-multibalanced natural transformation is simply a natural transformation that is compatible with all $\cC$-balancings obtained by fixing objects in $\cM$ and $\cP$. 

\smallskip

 Now, for three $\cC$-bimodule categories $\cM,\cN,\cP$, there is an equivalence 
\begin{equation}
\label{def:2associator}
A_{\cM,\cN,\cP}\colon \cM\boxtimes_\cC (\cN \boxtimes_\cC \cP)\isomorph(\cM\boxtimes_\cC \cN) \boxtimes_\cC \cP
\end{equation}
of $\cC$-bimodule categories making the diagram
\begin{align}\label{diag:associator}
\vcenter{\hbox{
\xymatrix{
&\cM \times \cN \times \cP\ar[dl]_{\cM\times P_{\cN,\cP}}\ar[dr]^{P_{\cM,\cN}\times \cP}& \\
\cM \times (\cN \boxtimes_\cC \cP)\ar[d]_{P_{\cM,\cN\boxtimes_\cC \cP}}&&(\cM \boxtimes_\cC \cN) \times \cP
\ar[d]^{P_{\cM\boxtimes_\cC \cN,\cP}}\\
\cM \boxtimes_\cC (\cN \boxtimes_\cC \cP)\ar[rr]^{A_{\cM,\cN,\cP}}&&(\cM \boxtimes_\cC \cN) \boxtimes_\cC \cP}}}
\end{align}
commute up to a $\cC$-multibalanced natural isomorphism. This equivalence is obtained as follows. We set $\cQ=(\cM\boxtimes_\cC\cN)\boxtimes_\cC \cP$ in \eqref{eq:multibalanced-universal} and define $A_{\cM,\cN,\cP}$ as a pre-image of the $\cC$-multibalanced functor 
$\cM\times \cN\times \cP\xrightarrow{P_{\cM\boxtimes_\cC\cN,\cP}\circ (\cM\times P_{\cN,\cP})}(\cM\boxtimes_\cC\cN)\boxtimes_\cC \cP$
under the composed functor on the left. Hence, $A_{\cM,\cN,\cP}$ is unique up to canonical isomorphism.

The associators of \eqref{def:2associator} are only natural (in the components $\cM,\cN$, and $\cP$) in the weaker sense that for any triple of $\cC$-bimodule functors $G\colon \cM\to \cM'$, $H\colon \cN\to \cN'$, $K\colon \cP\to \cP'$, there is a canonical isomorphism 
\begin{equation}\label{eq:A-naturality-constraint}
\vcenter{\hbox{
    \xymatrix@R=8pt{
&
\ar@{=>}[dd]|-{a_{G,H,K}}&\\
\cM\boxtimes_\cC (\cN \boxtimes_\cC \cP)\ar@/^1.6pc/[rr]^{((G\boxtimes_\cC H)\boxtimes_\cC K))\circ A_{\cM,\cN,\cP}}\ar@/_1.6pc/[rr]_{ A_{\cM',\cN',\cP'}\circ (G\boxtimes_\cC (H\boxtimes_\cC K))}&&(\cM'\boxtimes_\cC \cN')\boxtimes_\cC  \cP')
\\
&&}}}
\end{equation}
which satisfies the required coherences with respect to composition of $\cC$-bimodule functors.\footnote{In general form, these coherence are e.g.\,listed in \cite{JY21}*{Sec.\,4.2}.} 

\smallskip

To explain the weak pentagon axioms for the associators, we extend the universal property of $\boxtimes_\cC$ to the product of four $\cC$-bimodules in the expected way. There are two equivalences
\begin{equation}\label{eq:4-universal-prop}
\vcenter{\hbox{
\xymatrix@R=15pt{
*+[r]{\Fun(\cM\boxtimes_\cC (\cN \boxtimes_\cC (\cP\boxtimes_\cC \cQ)),\cR)}
\ar@/_0.5pc/[rd]_-{(-)\circ P_1}
&&*+[l]{\Fun(((\cM\boxtimes_\cC \cN) \boxtimes_\cC \cP)\boxtimes_\cC \cQ,\cR)}
\ar@/^0.5pc/[ld]^-{(-)\circ P_2}
\\
&\Fun^{\op{mulbal}}(\cM\times \cN\times  \cP\times \cQ,\cR)&
}}}
\end{equation}
defined by composing with 
\begin{gather*}
P_1= P_{\cM,\cN \boxtimes_\cC (\cP\boxtimes_\cC \cQ)}\circ(\cM\times P_{\cN, \cP\boxtimes_\cC \cQ})\circ (\cM\times \cN\times P_{\cP,\cQ}),\\
P_2=P_{(\cM\boxtimes_\cC\cN) \boxtimes_\cC \cP,\cQ}\circ(P_{\cM\boxtimes_\cC \cN,\cP}\times \cQ)\circ (P_{\cM,\cN}\times \cP\times \cQ),
\end{gather*}
where the target consist of $\cC$-multibalanced functors from a product of four $\cC$-bimodules, defined similarly to the above. 

Now there are two compositions of instances of the associator $1$-morphisms that provide equivalences between the two bracketings of the relative Deligne product of four $\cC$-bimodule categories $\cM,\cN,\cP,\cQ$ displayed in the following diagram:
\begin{equation}\label{eq:2associator}
\vcenter{\hbox{
\xymatrix@R=8pt{
&
\ar@{=>}[dd]|-{p_{\cM,\cN,\cP,\cQ}}&\\
\cM\boxtimes_\cC (\cN \boxtimes_\cC (\cP\boxtimes_\cC \cQ))\ar@/^1.8pc/[rr]^{(A_{\cM,\cN,\cP}\boxtimes_\cC \cQ)\circ A_{\cM,\cN\boxtimes_\cC \cP,\cQ}\circ ( \cM\boxtimes_\cC A_{\cN,\cP,\cQ})}\ar@/_1.8pc/[rr]_{A_{\cM\boxtimes_\cC \cN,\cP,\cQ}\circ A_{\cM,\cN,\cP\boxtimes_\cC \cQ}}&&((\cM\boxtimes_\cC \cN)\boxtimes_\cC  \cP)\boxtimes_\cC \cQ.
\\
&&}}}
\end{equation}
These compositions of associators are hence canonically isomorphic via the \emph{pentagonator} $p_{\cM,\cN,\cP,\cQ}$ that appears in the diagram. 
Because $p_{\cM,\cN,\cP,\cQ}$ is canonical, it will satisfy the coherence of the pentagonator of a monoidal $2$-category \cite{SPThe}*{Def.\,2.1, SM1.a, SM1.b}.\footnote{The associativity coherence for composing $1$-morphisms that appears in the diagrams of \cite{SPThe} are trivial here.}

Similarly, we obtain equivalences 
\begin{equation}\label{eq:2unitors}
\cC\boxtimes_\cC \cM\xrightarrow{L_{\cM}} \cM, \quad\text{ and }\quad  \cM\boxtimes_\cC \cC\xrightarrow{R_{\cM}} \cM
\end{equation}
which serve as left and right \emph{unitors}. 
The uniqueness of these equivalences up to canonical isomorphism provides isomorphisms encoding naturality of $R_\cM$ and $L_\cM$ in $\cM$ as well as isomorphisms up to which the unitality triangles involving $A_{\cM,\cN,\cP}$, $L_\cM$, $R_\cM$ commute. These, again, satisfy the required coherences by similar arguments as used before.

\begin{remark}\label{rem:general-coherence}
The structures obtained above give $\bimod{\cC}$ the structure of a monoidal $2$-category. In fact, the uniqueness property of the relative Deligne product implies that \emph{any} well-defined formal diagram obtained by combining interchange isomorphism, naturality constraints for $A,L,R$, pentagonators, and unitality isomorphisms via the operations of the $2$-category $\bimod{\cC}$ ($\boxtimes_\cC$, vertical and horizontal composition of $2$-morphisms) have to be equal. This coherence statement makes it unnecessary, in the scope of this paper, to identify a short list of axioms that is sufficient to define a monoidal $2$-category. However, such axiomatizations have appeared in the literature, e.g. \cites{GPS,SPThe,JY21}. 
\end{remark}

\subsection{Monoidal \texorpdfstring{$2$}{2}-equivalences induced by braided tensor equivalences}
\label{sec:2-mon-Res}

In this section, we show how the restriction functor along an equivalence of finite braided tensor categories provides monoidal $2$-equivalence.

\begin{proposition}\label{prop:res}
    Let $F\colon \cC \to \cD$ be an equivalence  of finite braided tensor categories. Then the restriction functor from \Cref{lem:restriction-2-functor} gives a monoidal $2$-functor 
    $\Res_F\colon \lmod{\cD}\to \lmod{\cC}$ which is a $2$-equivalence. The $2$-equivalence $\Res_F$ preserves exact module categories. 
\end{proposition}
\begin{proof}
Clearly, if $F$ is an equivalence of braided tensor categories, then  $\Res_F$ is an equivalence of $2$-categories. Since $F$ is an equivalence, it preserves projective objects and $\Res_F$ preserves exact module categories by \Cref{lem:restriction-2-functor}. 

It remains to check that $\Res_F$ is compatible with the monoidal $2$-category structures. For this, we fix an inverse monoidal equivalence $G\colon \cD\to \cC$ of braided tensor categories. Such an equivalence can be chosen such that the natural isomorphisms $\epsilon \colon \id\to FG$, $\eta\colon GF\to \id$ are monoidal natural isomorphisms \cite{TV}*{Exercise~1.4.6}. This induces a $2$-equivalence $\Mor_G\colon \Mor_\cD\to \Mor_\cC$ that sends an algebra $A$ in $\cD$ to the algebra $G(A)$ in $\cD$, using the monoidal structure of $G$. Moreover, for a fixed algebra $A$, the monoidal functor $F$ induces an equivalence of left $\cC$-module categories 
\begin{gather*}
F_A\colon \rmodint{\cC}{G(A)}\isomorph  \Res_F(\rmodint{\cD}{A}), \qquad 
(V,a_V\colon V\otimes G(A)\to V)\mapsto (F(V), a_{F(V)}),\\
    a_{F(V)}:=\Big(F(V)\otimes A \xrightarrow{\id_{F(V)}\otimes \epsilon_A}F(V)\otimes FG(A)\xrightarrow{\mu^F_{V,G(A)}}F(V\otimes G(A))\xrightarrow{F(a_V)}F(V) \Big).
\end{gather*}
One checks that $(F(V), a_{F(V)})$ is indeed a right $A$-module using that $\epsilon$ is a monoidal transformation. As $F$ is an equivalence, this functor $F_A$ is an equivalence. The $\cC$-linear structure for $F_A$ is defined by 
$$s^{F_{A}}_{X,V}\colon  F_A(X\lact V)=F(X\otimes V)\xrightarrow{(\mu^F_{X,V})^{-1}} F(X)\otimes F(V)=X\lact F_A(V).$$
Denoting $\cM=\rmodint{\cD}{A}$, the functor $F_A$ gives an equivalence of $\cC$-module categories $\rmodint{\cC}{G(A)}\isomorph \Res_F(\cM)$.
In particular, with $A=\one_\cD$, we obtain an equivalence $\mu_0\colon \cC\isomorph \Res_F(\cD)$ combining $F_\one$ with the equivalence induced by the isomorphism of algebras $G(\one_\cD)\cong \one_\cC$.

Now, since the monoidal structure of $G$ is compatible with the braidings, we have that $$(\mu^G_{A,B})^{-1}\colon G(A\otimes B) \isomorph G(A)\otimes G(B)$$
is an isomorphism of algebras in $\cC$. This isomorphism of algebra induces an equivalence of left $\cC$-module categories 
$$\rmodint{\cC}{G(A)\otimes G(B)}\isomorph \rmodint{\cC}{G(A\otimes B)}\isomorph \Res_F\big(\rmodint{\cC}{G(A\otimes B)}\big),$$
where the second equivalence is the equivalence $F_{A\otimes B}$ from above. Hence, by \Cref{prop:rel-Del-modules}, we have an equivalence of left $\cC$-module categories 
$$\mu^{\Res_F}_{\cM,\cN}\colon \Res_F(\cM)\boxtimes_\cC\Res_F(\cN)\isomorph \Res_F(\cM\boxtimes_\cC\cN),$$
for $\cM=\rmodint{\cD}{A}$ and $\cN=\rmodint{\cD}{B}$. 
Now, the relative Deligne product is defined by a universal property which induces the other structural data (associators, unitors, pentagonators, \ldots) as explained in \Cref{sec:rel-Del-props}. This way, there is a natural isomorphism 
$$
\xymatrix@R=8pt{
&
\ar@{=>}[dd]|-{\mu^F_{\cM,\cN,\cP}}&\\
\Res_F(\cM)\boxtimes_\cC \big(\Res_F(\cN)\boxtimes_\cC \Res_F(\cP)\big)\ar@/^1.8pc/[rr]^-{\Res_F(A_{\cM,\cN,\cP})\circ\mu^{\Res_F}_{\cM,\cN\boxtimes_\cD \cP}\circ\big(\Res_F(\cM)\boxtimes_\cC \mu^{\Res_F}_{\cN,\cP}\big)}
\ar@/_1.8pc/[rr]_{\mu^{\Res_F}_{\cM\boxtimes_\cD \cN,\cP}\circ\big( \mu^{\Res_F}_{\cM,\cN}\boxtimes_\cC\Res_F(\cP)\big)\circ A_{\Res_F(\cM),\Res_F(\cN),\Res_F(\cP)}}
&&\Res_F\big( (\cM\boxtimes_\cD\cN\boxtimes_\cD) \cP\big),\\
&&
}
$$
witnessing the compatibility of $\mu^{\Res_F}$ with the associator $A_{\cM,\cN,\cP}$ for the relative Deligne product from \eqref{eq:2associator}. Further, there are natural isomorphisms 
$$
\xymatrix@R=8pt{
&
\ar@{=>}[dd]|-{\rho^F_{\cM}}&\\
\Res_F(\cM)\boxtimes_\cC \cC \ar@/^1.8pc/[rr]^-{\Res_F(R_\cM)\circ\mu^{\Res_F}_{\cM,\cC}\circ\big(\Res_F(\cM)\boxtimes_\cC \mu_0\big)}\ar@/_1.8pc/[rr]_-{R_{\Res_F(\cM)}} && \Res_F(\cM),\\
&&
} \quad 
\xymatrix@R=8pt{
&
\ar@{=>}[dd]|-{\lambda^F_{\cM}}&\\
\cC\boxtimes_\cC\Res_F(\cM) \ar@/^1.8pc/[rr]^-{\Res_F(L_\cM)\circ\mu^{\Res_F}_{\cC,\cM}\circ\big(\mu_0\boxtimes_\cC \Res_F(\cM)\big)}\ar@/_1.8pc/[rr]_-{L_{\Res_F(\cM)}} && \Res_F(\cM),\\
&&
}
$$
witnessing the compatibility with the unitors from \eqref{eq:2unitors}. The $2$-natural transformations $\mu^F, \rho^F$, and $\lambda^F_\cM$ obtained this way are unique and hence satisfy all required coherences. This shows that $\Res_F$ has the structure of a monoidal $2$-functor.\footnote{Axioms for a monoidal $2$-functor were given in \cite{BN}*{Def.\,16}, see also \cite{DN2}*{Def.\,2.10}. The concrete list of coherence conditions is not important for this paper.} \end{proof}

\begin{proposition}\label{prop:Cop}
    The monoidal $2$-categories $\lmod{\cCop}$ and $\lmod{\cC}$ are equivalent as monoidal $2$-categories. The equivalence preserves exact module categories.  
\end{proposition} 
\begin{proof}
   Recall the monoidal equivalences
    $I^{\pm }=(\id_\cC,\mu^\pm) \colon \cCop\to \cC$ from \eqref{eq:I+-}. Now, applying \Cref{prop:res} to these equivalences induces equivalences of monoidal $2$-categories as required.
\end{proof}

\begin{remark}
\cite{DN2}*{Prop.\,3.7} shows that for tensor subcategories, induction is a monoidal $2$-functor. As restriction and induction are adjoint $2$-functors in an appropriate sense, we  expect that $\Res_F$ is only lax monoidal when $F$ is not an equivalence.  
\end{remark}

\bibliography{biblio}

\begin{bibdiv}
\begin{biblist}

\bib{AM}{article}{
      author={Andruskiewitsch, Nicol\'{a}s},
      author={Mombelli, Juan~Mart\'{\i}n},
       title={On module categories over finite-dimensional {H}opf algebras},
        date={2007},
        ISSN={0021-8693},
     journal={J. Algebra},
      volume={314},
      number={1},
       pages={383\ndash 418},
         url={https://doi.org/10.1016/j.jalgebra.2007.04.006},
}

\bib{BD}{article}{
      author={Baez, John~C.},
      author={Dolan, James},
       title={Higher-dimensional algebra and topological quantum field theory},
        date={1995},
        ISSN={0022-2488},
     journal={J. Math. Phys.},
      volume={36},
      number={11},
       pages={6073\ndash 6105},
         url={https://doi.org/10.1063/1.531236},
}

\bib{BEK}{article}{
      author={B\"{o}ckenhauer, Jens},
      author={Evans, David~E.},
      author={Kawahigashi, Yasuyuki},
       title={On {$\alpha$}-induction, chiral generators and modular invariants
  for subfactors},
        date={1999},
        ISSN={0010-3616},
     journal={Comm. Math. Phys.},
      volume={208},
      number={2},
       pages={429\ndash 487},
         url={https://doi.org/10.1007/s002200050765},
}

\bib{BGR}{article}{
      author={Berger, Johannes},
      author={Gainutdinov, Azat~M.},
      author={Runkel, Ingo},
       title={Non-semisimple link and manifold invariants for symplectic
  fermions},
        date={2023},
     journal={arXiv e-prints},
       pages={ArXiv:2307.06069},
}

\bib{BJS}{article}{
      author={Brochier, Adrien},
      author={Jordan, David},
      author={Snyder, Noah},
       title={On dualizability of braided tensor categories},
        date={2021},
        ISSN={0010-437X},
     journal={Compos. Math.},
      volume={157},
      number={3},
       pages={435\ndash 483},
         url={https://doi.org/10.1112/s0010437x20007630},
}

\bib{BK}{book}{
      author={Bakalov, Bojko},
      author={Kirillov, Alexander, Jr.},
       title={Lectures on tensor categories and modular functors},
      series={University Lecture Series},
   publisher={American Mathematical Society, Providence, RI},
        date={2001},
      volume={21},
        ISBN={0-8218-2686-7},
         url={https://doi.org/10.1090/ulect/021},
}

\bib{BMS}{article}{
      author={Barrett, John~W.},
      author={Meusburger, Catherine},
      author={Schaumann, Gregor},
       title={Gray categories with duals and their diagrams},
        date={2024},
        ISSN={0001-8708},
     journal={Adv. Math.},
      volume={450},
       pages={Paper No. 109740, 129},
         url={https://doi.org/10.1016/j.aim.2024.109740},
}

\bib{BN}{article}{
      author={Baez, John~C.},
      author={Neuchl, Martin},
       title={Higher-dimensional algebra. {I}. {B}raided monoidal
  {$2$}-categories},
        date={1996},
        ISSN={0001-8708},
     journal={Adv. Math.},
      volume={121},
      number={2},
       pages={196\ndash 244},
         url={https://doi.org/10.1006/aima.1996.0052},
}

\bib{BrNa}{article}{
      author={{Br}ugui\`eres, Alain},
      author={{Na}tale, Sonia},
       title={Central exact sequences of tensor categories, equivariantization
  and applications},
        date={2014},
        ISSN={0025-5645},
     journal={J. Math. Soc. Japan},
      volume={66},
      number={1},
       pages={257\ndash 287},
         url={https://doi.org/10.2969/jmsj/06610257},
}

\bib{BFN}{article}{
      author={Ben-Zvi, David},
      author={Francis, John},
      author={Nadler, David},
       title={Integral transforms and {D}rinfeld centers in derived algebraic
  geometry},
        date={2010},
        ISSN={0894-0347,1088-6834},
     journal={J. Amer. Math. Soc.},
      volume={23},
      number={4},
       pages={909\ndash 966},
         url={https://doi.org/10.1090/S0894-0347-10-00669-7},
}

\bib{Car}{article}{
      author={{Ca}rnovale, Giovanna},
       title={The {B}rauer group of modified supergroup algebras},
        date={2006},
        ISSN={0021-8693},
     journal={J. Algebra},
      volume={305},
      number={2},
       pages={993\ndash 1036},
         url={https://doi.org/10.1016/j.jalgebra.2006.06.002},
}

\bib{Cui}{article}{
      author={{Cu}i, Xingshan},
       title={{Higher Categories and Topological Quantum Field Theories}},
        date={2016},
     journal={arXiv e-prints},
       pages={ArXiv:1610.07628},
}

\bib{CF}{incollection}{
      author={Crane, Louis},
      author={Frenkel, Igor~B.},
       title={Four-dimensional topological quantum field theory, {H}opf
  categories, and the canonical bases},
        date={1994},
      volume={35},
       pages={5136\ndash 5154},
         url={https://doi.org/10.1063/1.530746},
        note={Topology and physics},
}

\bib{capucci2022actegories}{article}{
      author={Capucci, Matteo},
      author={Gavranović, Bruno},
       title={Actegories for the working amthematician},
        date={2022},
     journal={arXiv e-prints},
       pages={ArXiv:2203.16351},
}

\bib{CGP}{article}{
      author={Costantino, Francesco},
      author={Geer, Nathan},
      author={Patureau-Mirand, Bertrand},
       title={Quantum invariants of 3-manifolds via link surgery presentations
  and non-semi-simple categories},
        date={2014},
        ISSN={1753-8416},
     journal={J. Topol.},
      volume={7},
      number={4},
       pages={1005\ndash 1053},
         url={https://doi.org/10.1112/jtopol/jtu006},
}

\bib{CPr}{article}{
      author={Chari, Vyjayanthi},
      author={Premet, Alexander},
       title={Indecomposable restricted representations of quantum {${\rm
  sl}_2$}},
        date={1994},
        ISSN={0034-5318},
     journal={Publ. Res. Inst. Math. Sci.},
      volume={30},
      number={2},
       pages={335\ndash 352},
         url={https://doi.org/10.2977/prims/1195166137},
}

\bib{CSZ}{article}{
      author={Coulembier, Kevin},
      author={Stroi\'nski, Mateusz},
      author={Zorman, Tony},
       title={Simple algebras and exact module categories},
        date={2025},
     journal={arXiv e-prints},
       pages={ArXiv:2501.06629},
}

\bib{Dec3}{article}{
      author={{D\'{e}}coppet, Thibault~D.},
       title={Compact semisimple 2-categories},
        date={2023},
        ISSN={0002-9947,1088-6850},
     journal={Trans. Amer. Math. Soc.},
      volume={376},
      number={12},
       pages={8309\ndash 8336},
         url={https://doi.org/10.1090/tran/9044},
}

\bib{Dec2}{article}{
      author={{D\'{e}}coppet, Thibault~D.},
       title={On the dualizability of fusion 2-categories},
        date={2023},
     journal={arXiv e-prints},
       pages={ArXiv:2311.16827},
}

\bib{Dec1}{article}{
      author={{D\'{e}}coppet, Thibault~D.},
       title={Drinfeld centers and {M}orita equivalence classes of fusion
  2-categories},
        date={2025},
        ISSN={0010-437X},
     journal={Compos. Math.},
      volume={161},
      number={2},
       pages={305\ndash 340},
         url={https://doi.org/10.1112/S0010437X24007553},
}

\bib{Del}{incollection}{
      author={Deligne, P.},
       title={Cat\'{e}gories tannakiennes},
        date={1990},
   booktitle={The {G}rothendieck {F}estschrift, {V}ol. {II}},
      series={Progr. Math.},
      volume={87},
   publisher={Birkh\"{a}user Boston, Boston, MA},
       pages={111\ndash 195},
}

\bib{DHJNPPRY}{article}{
      author={D\'ecoppet, Thibault~D.},
      author={Huston, Peter},
      author={Johnson-Freyd, Theo},
      author={Nikshych, Dmitri},
      author={Penneys, David},
      author={Plavnik, Julia},
      author={Reutter, David},
      author={Yu, Matthew},
       title={The classification of fusion 2-categories},
        date={2024},
     journal={arXiv e-prints},
       pages={ArXiv:2411.05907},
}

\bib{DN}{article}{
      author={Davydov, Alexei},
      author={Nikshych, Dmitri},
       title={The {P}icard crossed module of a braided tensor category},
        date={2013},
        ISSN={1937-0652},
     journal={Algebra Number Theory},
      volume={7},
      number={6},
       pages={1365\ndash 1403},
         url={https://doi.org/10.2140/ant.2013.7.1365},
}

\bib{DN2}{article}{
      author={Davydov, Alexei},
      author={Nikshych, Dmitri},
       title={Braided {P}icard groups and graded extensions of braided tensor
  categories},
        date={2021},
        ISSN={1022-1824},
     journal={Selecta Math. (N.S.)},
      volume={27},
      number={4},
       pages={Paper No. 65, 87},
         url={https://doi.org/10.1007/s00029-021-00670-1},
}

\bib{DR}{article}{
      author={Douglas, Christopher~L.},
      author={Reutter, David~J},
       title={Fusion $2$-categories and a state-sum invariant for 4-manifolds},
        date={2018},
     journal={arXiv e-prints},
       pages={ArXiv:1812.11933},
        note={To appear in Mem. AMS.},
}

\bib{DGGPR}{article}{
      author={De~Renzi, Marco},
      author={Gainutdinov, Azat~M.},
      author={Geer, Nathan},
      author={Patureau-Mirand, Bertrand},
      author={Runkel, Ingo},
       title={3-{D}imensional {TQFT}s from non-semisimple modular categories},
        date={2022},
        ISSN={1022-1824},
     journal={Selecta Math. (N.S.)},
      volume={28},
      number={2},
       pages={Paper No. 42, 60},
         url={https://doi.org/10.1007/s00029-021-00737-z},
}

\bib{DecS}{article}{
      author={D\'{e}coppet, Thibault~D.},
      author={Sanford, Sean},
       title={Compact semisimple tensor 2-categories are {M}orita connected},
        date={2025},
        ISSN={1073-7928},
     journal={Int. Math. Res. Not. IMRN},
      number={11},
       pages={Paper No. rnaf135, 34},
         url={https://doi.org/10.1093/imrn/rnaf135},
}

\bib{DSS1}{article}{
      author={Douglas, Christopher~L.},
      author={Schommer-Pries, Christopher},
      author={Snyder, Noah},
       title={The balanced tensor product of module categories},
        date={2019},
        ISSN={2156-2261},
     journal={Kyoto J. Math.},
      volume={59},
      number={1},
       pages={167\ndash 179},
         url={https://doi.org/10.1215/21562261-2018-0006},
}

\bib{DSS2}{article}{
      author={Douglas, Christopher~L.},
      author={Schommer-Pries, Christopher},
      author={Snyder, Noah},
       title={Dualizable tensor categories},
        date={2020},
        ISSN={0065-9266},
     journal={Mem. Amer. Math. Soc.},
      volume={268},
      number={1308},
       pages={vii+88},
         url={https://doi.org/10.1090/memo/1308},
}

\bib{EGNO}{book}{
      author={Etingof, Pavel},
      author={Gelaki, Shlomo},
      author={Nikshych, Dmitri},
      author={Ostrik, Victor},
       title={Tensor categories},
      series={Mathematical Surveys and Monographs},
   publisher={American Mathematical Society, Providence, RI},
        date={2015},
      volume={205},
        ISBN={978-1-4704-2024-6},
         url={https://doi.org/10.1090/surv/205},
}

\bib{ENO}{article}{
      author={Etingof, Pavel},
      author={Nikshych, Dmitri},
      author={Ostrik, Victor},
       title={Fusion categories and homotopy theory},
        date={2010},
        ISSN={1663-487X},
     journal={Quantum Topol.},
      volume={1},
      number={3},
       pages={209\ndash 273},
         url={https://doi.org/10.4171/QT/6},
        note={With an appendix by Ehud Meir},
}

\bib{ENO2}{article}{
      author={Etingof, Pavel},
      author={Nikshych, Dmitri},
      author={Ostrik, Viktor},
       title={On fusion categories},
        date={2005},
        ISSN={0003-486X},
     journal={Ann. of Math. (2)},
      volume={162},
      number={2},
       pages={581\ndash 642},
         url={https://doi.org/10.4007/annals.2005.162.581},
}

\bib{EO3}{article}{
      author={Etingof, Pavel},
      author={Ostrik, Victor},
       title={On the {F}robenius functor for symmetric tensor categories in
  positive characteristic},
        date={2021},
        ISSN={0075-4102},
     journal={J. Reine Angew. Math.},
      volume={773},
       pages={165\ndash 198},
         url={https://doi.org/10.1515/crelle-2020-0033},
}

\bib{EO1}{article}{
      author={Etingof, Pavel},
      author={Ostrik, Viktor},
       title={Finite tensor categories},
        date={2004},
        ISSN={1609-3321},
     journal={Mosc. Math. J.},
      volume={4},
      number={3},
       pages={627\ndash 654, 782\ndash 783},
         url={https://doi.org/10.17323/1609-4514-2004-4-3-627-654},
}

\bib{FFRS}{article}{
      author={Fr\"{o}hlich, J\"{u}rg},
      author={Fuchs, J\"{u}rgen},
      author={Runkel, Ingo},
      author={Schweigert, Christoph},
       title={Correspondences of ribbon categories},
        date={2006},
        ISSN={0001-8708},
     journal={Adv. Math.},
      volume={199},
      number={1},
       pages={192\ndash 329},
         url={https://doi.org/10.1016/j.aim.2005.04.007},
}

\bib{FGS}{article}{
      author={Faitg, M.},
      author={Gainutdinov, A.~M.},
      author={Schweigert, C.},
       title={Davydov-{Y}etter cohomology and relative homological algebra},
        date={2024},
        ISSN={1022-1824},
     journal={Selecta Math. (N.S.)},
      volume={30},
      number={2},
       pages={Paper No. 26, 80},
         url={https://doi.org/10.1007/s00029-024-00917-7},
}

\bib{FSS}{article}{
      author={Fuchs, J\"{u}rgen},
      author={Schaumann, Gregor},
      author={Schweigert, Christoph},
       title={A trace for bimodule categories},
        date={2017},
        ISSN={0927-2852},
     journal={Appl. Categ. Structures},
      volume={25},
      number={2},
       pages={227\ndash 268},
         url={https://doi.org/10.1007/s10485-016-9425-3},
}

\bib{GreThe}{book}{
      author={Greenough, Justin},
       title={Bimodule categories and monoidal 2-structure},
   publisher={ProQuest LLC, Ann Arbor, MI},
        date={2010},
        ISBN={978-1124-30199-0},
  url={http://gateway.proquest.com/openurl?url_ver=Z39.88-2004&rft_val_fmt=info:ofi/fmt:kev:mtx:dissertation&res_dat=xri:pqdiss&rft_dat=xri:pqdiss:3430786},
        note={Thesis (Ph.D.)--University of New Hampshire},
}

\bib{Gre}{article}{
      author={Greenough, Justin},
       title={Relative centers and tensor products of tensor and braided fusion
  categories},
        date={2013},
        ISSN={0021-8693},
     journal={J. Algebra},
      volume={388},
       pages={374\ndash 396},
         url={https://doi.org/10.1016/j.jalgebra.2013.04.014},
}

\bib{DJF}{article}{
      author={Gaiotto, Davide},
      author={Johnson-Freyd, Theo},
       title={Condensations in higher categories},
        date={2019},
     journal={arXiv e-prints},
       pages={ArXiv:1905.09566},
}

\bib{GL}{misc}{
      author={Gainutdinov, Azat~M.},
      author={Laugwitz, Robert},
       title={Monoidal 2-categories from twisted fiber functors (in
  preparation)},
}

\bib{GPT}{article}{
      author={Geer, Nathan},
      author={Patureau-Mirand, Bertrand},
      author={Turaev, Vladimir},
       title={Modified quantum dimensions and re-normalized link invariants},
        date={2009},
        ISSN={0010-437X},
     journal={Compos. Math.},
      volume={145},
      number={1},
       pages={196\ndash 212},
         url={https://doi.org/10.1112/S0010437X08003795},
}

\bib{GPS}{article}{
      author={Gordon, R.},
      author={Power, A.~J.},
      author={Street, Ross},
       title={Coherence for tricategories},
        date={1995},
        ISSN={0065-9266},
     journal={Mem. Amer. Math. Soc.},
      volume={117},
      number={558},
       pages={vi+81},
         url={https://doi.org/10.1090/memo/0558},
}

\bib{Hat}{article}{
      author={Hattori, Akira},
       title={Semisimple algebras over a commutative ring},
        date={1963},
        ISSN={0025-5645,1881-1167},
     journal={J. Math. Soc. Japan},
      volume={15},
       pages={404\ndash 419},
         url={https://doi.org/10.2969/jmsj/01540404},
}

\bib{HR}{article}{
      author={Hall, Jack},
      author={Rydh, David},
       title={Perfect complexes on algebraic stacks},
        date={2017},
        ISSN={0010-437X,1570-5846},
     journal={Compos. Math.},
      volume={153},
      number={11},
       pages={2318\ndash 2367},
         url={https://doi.org/10.1112/S0010437X17007394},
}

\bib{JR}{article}{
      author={Johnson-Freyd, Theo},
      author={Reutter, David},
       title={Minimal nondegenerate extensions},
        date={2024},
        ISSN={0894-0347},
     journal={J. Amer. Math. Soc.},
      volume={37},
      number={1},
       pages={81\ndash 150},
         url={https://doi.org/10.1090/jams/1023},
}

\bib{JY25}{article}{
      author={{Ja}klitsch, David},
      author={{Ya}dav, Harshit},
       title={$\otimes$-{F}robenius functors and exact module categories},
        date={2025},
     journal={arXiv e-prints},
       pages={ArXiv:2501.16978},
}

\bib{JY21}{book}{
      author={Johnson, Niles},
      author={Yau, Donald},
       title={2-dimensional categories},
   publisher={Oxford University Press, Oxford},
        date={2021},
        ISBN={978-0-19-887138-5; 978-0-19-887137-8},
         url={https://doi.org/10.1093/oso/9780198871378.001.0001},
}

\bib{KelLim}{article}{
      author={Kelly, G.~M.},
       title={Elementary observations on {$2$}-categorical limits},
        date={1989},
        ISSN={0004-9727},
     journal={Bull. Austral. Math. Soc.},
      volume={39},
      number={2},
       pages={301\ndash 317},
         url={https://doi.org/10.1017/S0004972700002781},
}

\bib{KS}{book}{
      author={Klimyk, Anatoli},
      author={Schm\"{u}dgen, Konrad},
       title={Quantum groups and their representations},
      series={Texts and Monographs in Physics},
   publisher={Springer-Verlag, Berlin},
        date={1997},
        ISBN={3-540-63452-5},
         url={https://doi.org/10.1007/978-3-642-60896-4},
}

\bib{KTZ}{article}{
      author={Kong, Liang},
      author={Tian, Yin},
      author={Zhou, Shan},
       title={The center of monoidal 2-categories in {$3+1{\rm D}$}
  {D}ijkgraaf-{W}itten theory},
        date={2020},
        ISSN={0001-8708},
     journal={Adv. Math.},
      volume={360},
       pages={106928, 25},
         url={https://doi.org/10.1016/j.aim.2019.106928},
}

\bib{Lur}{incollection}{
      author={Lurie, Jacob},
       title={On the classification of topological field theories},
        date={2009},
   booktitle={Current developments in mathematics, 2008},
   publisher={Int. Press, Somerville, MA},
       pages={129\ndash 280},
}

\bib{Lyu}{article}{
      author={{Ly}ubashenko, Volodymyr~V.},
       title={Invariants of {$3$}-manifolds and projective representations of
  mapping class groups via quantum groups at roots of unity},
        date={1995},
        ISSN={0010-3616},
     journal={Comm. Math. Phys.},
      volume={172},
      number={3},
       pages={467\ndash 516},
         url={http://projecteuclid.org/euclid.cmp/1104274312},
}

\bib{LR}{incollection}{
      author={Longo, R.},
      author={Rehren, K.-H.},
       title={Nets of subfactors},
        date={1995},
      volume={7},
       pages={567\ndash 597},
         url={https://doi.org/10.1142/S0129055X95000232},
        note={Workshop on Algebraic Quantum Field Theory and Jones Theory
  (Berlin, 1994)},
}

\bib{ML2}{book}{
      author={Mac~Lane, Saunders},
       title={Homology},
      series={Die Grundlehren der mathematischen Wissenschaften, Band 114},
   publisher={Springer-Verlag, Berlin-G\"{o}ttingen-Heidelberg; Academic Press,
  Inc., Publishers, New York},
        date={1963},
}

\bib{ML}{book}{
      author={Mac~Lane, Saunders},
       title={Categories for the working mathematician},
      series={Graduate Texts in Mathematics, Vol. 5},
   publisher={Springer-Verlag, New York-Berlin},
        date={1971},
}

\bib{Mac1}{article}{
      author={Mackaay, Marco},
       title={Spherical {$2$}-categories and {$4$}-manifold invariants},
        date={1999},
        ISSN={0001-8708,1090-2082},
     journal={Adv. Math.},
      volume={143},
      number={2},
       pages={288\ndash 348},
         url={https://doi.org/10.1006/aima.1998.1798},
}

\bib{Mac2}{article}{
      author={Mackaay, Marco},
       title={Finite groups, spherical 2-categories, and 4-manifold
  invariants},
        date={2000},
        ISSN={0001-8708,1090-2082},
     journal={Adv. Math.},
      volume={153},
      number={2},
       pages={353\ndash 390},
         url={https://doi.org/10.1006/aima.1999.1909},
}

\bib{Maj}{book}{
      author={{Ma}jid, Shahn},
       title={Foundations of quantum group theory},
   publisher={Cambridge University Press, Cambridge},
        date={2000},
        ISBN={0-521-64868-8},
         url={http://dx.doi.org/10.1017/CBO9780511613104},
        note={Paperback Edition, originally published 1995},
}

\bib{Mom1}{article}{
      author={{Mo}mbelli, Mart\'{\i}n},
       title={Module categories over pointed {H}opf algebras},
        date={2010},
        ISSN={0025-5874},
     journal={Math. Z.},
      volume={266},
      number={2},
       pages={319\ndash 344},
         url={https://doi.org/10.1007/s00209-009-0571-2},
}

\bib{Ost1}{article}{
      author={Ostrik, Victor},
       title={Module categories, weak {H}opf algebras and modular invariants},
        date={2003},
        ISSN={1083-4362},
     journal={Transform. Groups},
      volume={8},
      number={2},
       pages={177\ndash 206},
}

\bib{Pst}{article}{
      author={Pstragowski, Piotr},
       title={On dualizable objects in monoidal bicategories},
        date={2022},
     journal={Theory Appl. Categ.},
      volume={38},
       pages={Paper No. 9, 257\ndash 310},
}

\bib{Rad}{book}{
      author={Radford, David~E.},
       title={Hopf algebras},
      series={Series on Knots and Everything},
   publisher={World Scientific Publishing Co. Pte. Ltd., Hackensack, NJ},
        date={2012},
      volume={49},
        ISBN={978-981-4335-99-7; 981-4335-99-1},
}

\bib{Reu}{article}{
      author={{Re}utter, David},
       title={Semisimple four-dimensional topological field theories cannot
  detect exotic smooth structure},
        date={2023},
        ISSN={1753-8416},
     journal={J. Topol.},
      volume={16},
      number={2},
       pages={542\ndash 566},
         url={https://doi.org/10.1112/topo.12288},
}

\bib{SPThe}{book}{
      author={Schommer-Pries, Christopher~John},
       title={The classification of two-dimensional extended topological field
  theories},
   publisher={ProQuest LLC, Ann Arbor, MI},
        date={2009},
        ISBN={978-1109-46779-6},
        note={Thesis (Ph.D.)--University of California, Berkeley, available at
  \url{https://arxiv.org/pdf/1112.1000}},
}

\bib{SchP}{article}{
      author={{Sch}aumann, Gregor},
       title={Pivotal tricategories and a categorification of inner-product
  modules},
        date={2015},
        ISSN={1386-923X,1572-9079},
     journal={Algebr. Represent. Theory},
      volume={18},
      number={6},
       pages={1407\ndash 1479},
         url={https://doi.org/10.1007/s10468-015-9547-6},
}

\bib{Shi5}{article}{
      author={Shimizu, Kenichi},
       title={The relative modular object and {F}robenius extensions of finite
  {H}opf algebras},
        date={2017},
        ISSN={0021-8693},
     journal={J. Algebra},
      volume={471},
       pages={75\ndash 112},
         url={https://doi.org/10.1016/j.jalgebra.2016.09.017},
}

\bib{Shi2}{article}{
      author={Shimizu, Kenichi},
       title={Further results on the structure of (co)ends in finite tensor
  categories},
        date={2020},
        ISSN={0927-2852},
     journal={Appl. Categ. Structures},
      volume={28},
      number={2},
       pages={237\ndash 286},
         url={https://doi.org/10.1007/s10485-019-09577-7},
}

\bib{ShiP}{article}{
      author={Shimizu, Kenichi},
       title={Relative {S}erre functor for comodule algebras},
        date={2023},
        ISSN={0021-8693,1090-266X},
     journal={J. Algebra},
      volume={634},
       pages={237\ndash 305},
         url={https://doi.org/10.1016/j.jalgebra.2023.07.015},
}

\bib{stacks-project}{misc}{
      author={{{St}acks Project Authors}},
       title={\textit{Stacks Project}},
        date={2018},
        note={\url{https://stacks.math.columbia.edu}, accessed Dec 12, 2025},
}

\bib{SY}{article}{
      author={Shimizu, Kenichi},
      author={Yadav, Harshit},
       title={{Commutative exact algebras and modular tensor categories}},
        date={2024},
     journal={arXiv e-prints},
       pages={ArXiv:2408.06314},
}

\bib{Tur}{book}{
      author={Turaev, V.~G.},
       title={Quantum invariants of knots and 3-manifolds},
      series={De Gruyter Studies in Mathematics},
   publisher={Walter de Gruyter \& Co., Berlin},
        date={1994},
      volume={18},
        ISBN={3-11-013704-6},
}

\bib{TV}{book}{
      author={Turaev, Vladimir},
      author={Virelizier, Alexis},
       title={Monoidal categories and topological field theory},
      series={Progress in Mathematics},
   publisher={Birkh\"{a}user/Springer, Cham},
        date={2017},
      volume={322},
        ISBN={978-3-319-49833-1; 978-3-319-49834-8},
         url={https://doi.org/10.1007/978-3-319-49834-8},
}

\bib{Will}{thesis}{
      author={Willprecht, Jan-Ole},
       title={On deformations of module categories over finite tensor
  categories},
        type={Ph.D. Thesis},
    language={English},
        date={2023},
        note={University of Hamburg, available at
  \url{https://ediss.sub.uni-hamburg.de/handle/ediss/10651} (retreived August
  18, 2025)},
}

\end{biblist}
\end{bibdiv}
\bibliographystyle{amsrefs}

\end{document}